\documentclass{amsart}
\usepackage{graphicx}
\usepackage{hyperref}
\usepackage{latexsym}
\usepackage{amsmath}
\usepackage{amsthm}
\usepackage{amssymb}
\usepackage{graphicx}
\usepackage{pdfpages}
\usepackage[labelformat=empty]{caption}

\newtheorem{theorem}{Theorem}[section]
\newtheorem{lemma}[theorem]{Lemma}
\newtheorem{cor}[theorem]{Corollary}
\newtheorem{prop}[theorem]{Proposition}

\theoremstyle{definition}

\theoremstyle{remark}
\newtheorem{remark}[theorem]{Remark}

\newcommand{\vip}{\vskip.2cm}
\newcommand{\R}{\mathbb{R}}
\newcommand{\E}{\mathbb{E}}
\newcommand{\Et}{\E_\theta}

\newcommand{\indiq}{{{\mathbf 1}}}
\newcommand{\intot}{{\int_0^t}}

\newcommand{\bZ}{{\mathbf Z}}
\newcommand{\bI}{{\mathbf I}}
\newcommand{\bJ}{{\mathbf J}}
\newcommand{\bl}{\boldsymbol{\ell}}
\newcommand{\bM}{{\boldsymbol{M}}}
\newcommand{\cJ}{{\boldsymbol{J}}}
\newcommand{\bX}{{\mathbf X}}
\newcommand{\bL}{{\mathbf L}}
\newcommand{\bx}{\boldsymbol{x}}
\newcommand{\bV}{\boldsymbol{V}}

\newcommand{\bun}{{\mathbf 1}_N}
\newcommand{\omg}{\indiq_{\Omega_{N,K}}}
\newcommand{\e}{{\varepsilon}}

\newcommand{\rr}{{\mathbb{R}}}

\newcommand{\cF}{{\mathcal F}}
\newcommand{\cA}{{\mathcal A}}

\newcommand{\cE}{{\mathcal E}}
\newcommand{\cV}{{\mathcal V}}
\newcommand{\cW}{{\mathcal W}}

\newcommand{\cP}{{\mathcal P}}
\newcommand{\cD}{{\mathcal D}}

\newcommand{\cX}{{\mathcal X}}

\newcommand{\cY}{{\mathcal Y}}
\hypersetup{pdfmenubar=true}

\begin{document}

\title[Central limit theorem for Hawkes process]
{Central limit theorem  for a partially observed interacting system of Hawkes processes}

\author{Chenguang LIU}

\address{Sorbonne Universit\'e- LPSM, Campus Pierre et Marie Curie, 4 place Jussieu, 75252
PARIS CEDEX 05, LPSM, F-75005 Paris, France.}

\email{LIUCG92@gmail.com}

\begin{abstract}
We observe the actions of a $K$ sub-sample of $N$ individuals up to time $t$ for some large $K\le N$.
We model the relationships of individuals by i.i.d. Bernoulli($p$)-random variables, where $p\in (0,1]$ is an unknown parameter. The rate of action of each individual depends on some unknown parameter $\mu> 0$ and on the sum of some function $\phi$ of the ages of the actions of the individuals which influence him. 
The parameters $\mu$ and $\phi$ are considered as nuisance parameters.
The aim of this paper is to obtain a central limit theorem  
for the estimator of $p$ that we introduced in \cite{D}, both in the subcritical
and supercritical cases.
\end{abstract}


\maketitle

\tableofcontents
\section*{Acknowledgements}
I would like to thank my supervisors, N. Fournier and S. Delattre. Thanks for their great support to this research. This paper would not be possible without their wisdom ideas and  patiently guidance. 
\label{key}\section{Introduction} 
\subsection{Setting }
We consider some unknown parameters $p$ $\in(0,1], \mu >0$ and $\phi:[0,\infty)\to [0,\infty)$. We always assume that the function $\phi$ is measurable and locally integrable. For  $N\ge 1$, we consider an i.i.d. family $(\Pi^{i}(dt,dz))_{i=1,...,N}$ of Poisson measures on $[0,\infty)\times [0,\infty)$ with intensity $dtdz$, together with $(\theta_{ij})_{i,j=1,...,N}$, a family  of i.i.d. Bernoulli($p$) random variables independent of the family $(\Pi^{i}(dt,dz))_{i=1,...,N}$. We consider the following system: for all $i\in \{1,...,N\},$ all $t\geq 0,$
\begin{align}\label{sssy}
Z_{t}^{i,N}=\int^{t}_{0}\int^{\infty}_{0}\boldsymbol{1}_{\{z\le\lambda_{s}^{i,N}\}}\Pi^{i}(ds,dz), \hbox{ where }
\lambda_{t}^{i,N}=\mu+\frac{1}{N}\sum_{j=1}^{N}\theta_{ij}\int_{0}^{t-}\phi(t-s)dZ_{s}^{j,N}.
\end{align}
In this paper, $\int_{0}^{t}$ means $\int_{[0,t]}$, and $\int_{0}^{t-}$ means $\int_{[0,t)}$. The solution $((Z_{t}^{i,N})_{t\ge 0})_{i=1,...,N}$ is a family of counting processes. 
By \cite[Proposition 1]{A}, the system $(1.1)$ has a unique $(\mathcal{F}_{t})_{t\ge 0}$-measurable c\`adl\`ag 
solution, where   
$$\mathcal{F}_{t}=\sigma(\Pi^{i}(A):A\in\mathcal{B}([0,t]\times [0,\infty)),i=1,...,N)\vee 
\sigma(\theta_{ij},i,j=1,...,N),$$ as soon as $\phi$ is locally integrable.

\subsection{An illustrating example}
Let us provide an interpretation of the process $((Z_{t}^{i,N})_{t\ge 0})_{i=1,...,N}$. It describes the activity of $N$ individuals along the time.
Each individual $j \in \{1,\dots,N\}$ influences  the set of individuals
$S_j=\{i \in \{1,\dots,N\} : \theta_{ij}=1\}$.
The only possible action of the individual $j$ is to send a message to all the individuals of $S_j$.
Here $Z^{i,N}_t$ stands for the number of messages sent by $i$ during $[0,t]$.

\smallskip

The rate $\lambda^{i,N}_t$ at which $i$ sends messages can be decomposed as the sum of two effects:

$\bullet$ he sends {\it new} messages at rate $\mu$;

$\bullet$ he {\it forwards} the messages he received, after some delay (possibly infinite) depending on the 
age of the message, which induces a sending rate of the form 
$\frac{1}{N}\sum_{j=1}^{N}\theta_{ij}\int_{0}^{t-}\phi(t-s)dZ_{s}^{j,N}$.

\smallskip

If for example $\phi=\boldsymbol{1}_{[0,K]}$, then  $N^{-1} \sum_{j=1}^N \theta_{ij}\int_0^{t-}  \phi(t-s)dZ_{s}^{j,N}$ is
precisely the number of messages that the $i$-th individual received between time $t-K$ and time $t$,
divided by $N$.

\renewcommand\theequation{\arabic{equation}}\addtocounter{equation}{0}

\subsection{Motivations and main goals}
In the real world, the number of individuals is often large. So it is necessary to construct consistent estimators in the asymptotic where $N$ and $t$ tend simultaneously to infinity. 
In our context, we only observe the activity of some (or all) individuals, we do not know at all the graph corresponding to the relationships between individuals.
Our goal is to estimate $p$, which can be seen as the main characteristic of the graph of interactions, since
it represents the proportion of open edges. In  \cite{A}, Delattre and Fournier  consider the case when one observes the  whole sample $(Z_{s}^{i,N})_{i=1...N,0\le s\le t}$ and they propose some estimator of the unknown parameter $p$. 
In \cite{D}, we  build some estimator of $p$ when observing $(Z_{s}^{i,N})_{\{i=1,...,K,\ 0\le s\le t\}}$ with $1\ll K \le N$ and with $t$ large. %The paper \cite{A} thus considers the special case where $K=N.$
In this work, we establish a central limit theorem for this estimator, which allows to construct an asymptotic confidence interval of the parameter $p$.

\subsection{Assumptions}

We will work under  the following conditions:
 for some $q\geq 1$,
\renewcommand\theequation{{$H(q)$}}
\begin{equation}
\mu \in (0,\infty),\quad \Lambda\in(0,\infty), \quad \Lambda p \in [0,1),
\quad \int_0^\infty s^q\phi(s)ds <\infty \quad \hbox{and} 
\quad \int_0^\infty (\phi(s))^2ds <\infty
\end{equation}
\renewcommand\theequation{\arabic{equation}}\addtocounter{equation}{-1}
or
\renewcommand\theequation{{$A$}}
\begin{equation}
\mu \in (0,\infty), \quad \Lambda p \in (1,\infty) \quad \hbox{and} \quad \phi(s)=e^{-bs}
\quad \hbox{for some unknown } b>0. 
\end{equation}
Here $b$ is a positive constant. Since $\Lambda= 1/b$, we thus assume that
$p>b$.
\renewcommand\theequation{\arabic{equation}}\addtocounter{equation}{-1}

\subsection{The result in subcritical case}\label{TRISC}
Here we will assume $H(q)$ for some $q\geq 1$.
We first recall the estimator we built in \cite{D}.
For $N\geq 1$ and for $((Z^{i,N}_t)_{t\geq 0})_{i=1,\dots,N}$ the solution of system \eqref{sssy}, we set
$\bar{Z}^{N}_{t}=N^{-1}\sum_{i=1}^N Z^{i,N}_t$, and $\bar{Z}^{N,K}_{t}=K^{-1}\sum_{i=1}^K Z^{i,N}_t.$ Next, we introduce:
$$\varepsilon _{t}^{N,K}=\frac 1t(\bar{Z}_{2t}^{N,K}-\bar{Z}_{t}^{N,K}),\qquad \mathcal{V}_{t}^{N,K}
=\frac{N}{K}\sum_{i=1}^{K}\Big[\frac{Z_{2t}^{i,N}-Z_{t}^{i,N}}t-\varepsilon_{t}^{N,K}\Big]^{2}-\frac{N}{t}\varepsilon_{t}^{N,K}.$$ 
And for $\Delta>0$ such that $t/(2\Delta)\in \mathbb{N}^{*}.$
$$
\mathcal{X}_{t,\Delta}^{N,K}:=\mathcal{W}_{\Delta,t}^{N,K}-\frac{N-K}{K}\varepsilon_{t}^{N,K},
$$
where
$$
\mathcal{W}_{\Delta,t}^{N,K}=2\mathcal{Z}_{2\Delta,t}^{N,K}-\mathcal{Z}_{\Delta,t}^{N,K},\qquad
\mathcal{Z}^{N,K}_{\Delta,t}=\frac{N}{t}\sum_{a=\frac{t}{\Delta}+1}^{\frac{2t}{\Delta}}(\bar{Z}_{a\Delta}^{N,K}-\bar{Z}_{(a-1)\Delta}^{N,K}-\Delta\varepsilon_{t}^{N,K})^{2}.
$$
We then introduce the function $\Psi^{(3)}$ defined by
$$
\Psi^{(3)}(u,v,w)=\frac{u^2(1-\sqrt{\frac{u}{w}})^2}{v+u^2(1-\sqrt{\frac{u}{w}})^{2}}\quad \text{if $u>0$, $v>0$, $w>0$}
\quad \text{and \ $\Psi^{(3)}(u,v,w)=0$ otherwise.}
$$
We set
$$
\hat p_{N,K,t} = \Psi^{(3)}(\varepsilon_{t}^{N,K},\mathcal{V}_{t}^{N,K},\mathcal{X}_{t,\Delta_{t}}^{N,K}),
$$
with the choice
\begin{equation} \label{Deltat}
\Delta_t=(2 \lfloor t^{1-4/(q+1)}\rfloor)^{-1}t
\end{equation}
 It was shown in \cite[Theorem 2.1]{D}
that under $(H(q))$ for some $q>3$, for some constants $C,c>0$ 
(depending only on $q,p,\mu,\phi$), for all $\e\in(0,1)$,
all $1\leq K \leq N$, all $t\geq 1$,
$$
P(|\hat p_{N,K,t}-p|>\e) \leq \frac C\epsilon \Big(\frac 1{\sqrt K} + \frac NK \sqrt{\frac{\Delta_t}t}
+ \frac{N}{t\sqrt K} \Big) + C N e^{-cK}.
$$
We also showed using a toy model in \cite[Section 13]{D} that this rate of convergence is likely to be optimal.
Finally, to have an idea of the orders of magnitude, we recall that roughly, in the subcritical case
(where $\Lambda p \in (0,1)$), each individual has around $1$ jump per unit of time, in the sense that, see 
\cite[Remark 2]{A}, under $H(1)$,
$t^{-1}\bar{Z}^{N}_t$ goes in probability to the constant $(1-\Lambda p)^{-1}\mu>0$ as $N\to\infty$ and $t\to \infty$.
Hence, when observing a sample of $K$ individuals during $[0,t]$, one observes around $K t$ jumps.
Here is the main result of the present paper in the subcritical case.

\begin{theorem}\label{mainsubcr}
We assume that $p>0$ and that $H(q)$ holds for some $q> 3$. Define $\Delta_t$ by \eqref{Deltat}.
We set $c_{p,\Lambda}:=(1-\Lambda p)^2/(2\Lambda^2)$.
 We always work in the asymptotic  $(N,K,t)\to (\infty,\infty,\infty)$
and in the regime $\frac 1{\sqrt K} + \frac NK \sqrt{\frac{\Delta_t}t}+ \frac{N}{t\sqrt K}+ Ne^{-c_{p,\lambda} K} \to 0$.

(i) In the regime
with dominating term $\frac 1{\sqrt K}$, i.e. when
$[\frac 1{\sqrt K}]/[ \frac NK \sqrt{\frac{\Delta_t}t}+ \frac{N}{t\sqrt K}]\to \infty$,
it holds that
$$
\sqrt K \Big(\hat p_{N,K,t}-p \Big) \stackrel{d}{\longrightarrow} 
\mathcal{N}\Big(0,  \frac{p^2(1-p)^{2}}{\mu^4}\Big).
$$

(ii) In the regime
with dominating term $\frac{N}{t\sqrt K}$, i.e. when
$[\frac{N}{t\sqrt K}]/[\frac 1{\sqrt K}+\frac NK \sqrt{\frac{\Delta_t}t}]\to \infty$,
we have
$$
\frac{t\sqrt{K}}{N}\Big(\hat p_{N,K,t}-p \Big)\stackrel{d}{\longrightarrow} 
\mathcal{N}\Big(0,\frac{2(1-\Lambda p)}{\mu^2 \Lambda^4}\Big).
$$

(iii) In the regime
with dominating term $\frac{N}{K}\sqrt{\frac{\Delta_t}t}$, i.e. when
$[\frac{N}{K}\sqrt{\frac{\Delta_t}t}]/[\frac 1{\sqrt K}+ \frac{N}{t\sqrt K} ]\to \infty$, 
imposing moreover that $\lim_{N,K\to\infty} \frac{K}{N}=\gamma\in [0,1]$,
$$
\frac{K}{N}\sqrt{\frac{t}{\Delta_{t}}}\Big(\hat p_{N,K,t}
-p\Big)\stackrel{d}{\longrightarrow} \mathcal{N}\Big(0,
\frac{3(1-p)^2}{2\mu^2\Lambda^2}\Big((1-\gamma)(1-\Lambda p)^{3}+\gamma(1-\Lambda p)\Big)^2\Big).
$$
\end{theorem}

We decided not to study the regimes where there are two or three dominating terms. 
We believe this is not very restrictive in practise. Furthermore, the study would be much more tedious,
because it would be very difficult to study the correlations between the different terms.

\begin{remark}
This result allows us to construct an asymptotic confidence interval for $p$.
We define 
\begin{align*}
    \hat \mu_{N,K,t}:=\Psi^{(1)}(\varepsilon_{t}^{N,K},\mathcal{V}_{t}^{N,K},\mathcal{X}_{\Delta_{t},t}^{N,K}), \quad  \hat \Lambda_{N,K,t}:=\Psi^{(2)}(\varepsilon_{t}^{N,K},\mathcal{V}_{t}^{N,K},\mathcal{X}_{\Delta_{t},t}^{N,K})
\end{align*}
where
\begin{gather*}
\Psi^{(1)}(u,v,w):=u\sqrt{\frac{u}{w}},\quad  \Psi^{(2)}(u,v,w):=\frac{v+[u-\Psi^{(1)}(u,v,w)]^{2}}{u[u-\Psi^{(1)}(u,v,w)]}
\end{gather*}
if $u>0$, $v>0$, $w>u$
and \ $\Psi^{(1)}(u,v,w)=\Psi^{(2)}(u,v,w)=0$ otherwise.
By \cite[Theorem 2.1]{D}, we have, 
in the regime $\frac 1{\sqrt K} + \frac NK \sqrt{\frac{\Delta_t}t}+ \frac{N}{t\sqrt K}+Ne^{-c_{p,\Lambda} K} \to 0$,
$$
\Big(\hat \mu_{N,K,t},  \hat \Lambda_{N,K,t},\hat p_{N,K,t}\Big)\stackrel{P}{\longrightarrow} (\mu,\Lambda,p). 
$$
Hence by Theorem \ref{mainsubcr},
in the regime (i), (ii) or (iii), for $0< \alpha<1,$
$$
 \lim P\Big(|\hat p_{N,K,t}-p|\le I_{N,K,t,\alpha}\Big)= 1-\alpha
$$
where
\begin{align*}
&I_{N,K,t,\alpha}= (\Phi)^{-1}(1-\frac{\alpha}{2})
\Big(\frac{1}{\sqrt{K}}\frac{\hat p_{N,K,t}(1-\hat p_{N,K,t})}{\hat p_{N,K,t}}+\frac{N}{t\sqrt{K}}\frac{\sqrt{2(1-\hat \mu_{N,K,t})}}{\hat \mu_{N,K,t}(\hat \Lambda_{N,K,t})^2}\\
&\hskip9cm +\frac{N}{K}\sqrt{\frac{\Delta_t}{t}}\sqrt{\frac{3(1-\hat p_{N,K,t})^2}{2\hat \mu_{N,K,t}^2\hat \Lambda_{N,K,t}^2}}\Big)
\end{align*}
 and $\Phi(x)=\frac{1}{\sqrt{2\pi}}\int^x_{-\infty}e^{-\frac{s^2}{2}}ds.$
\end{remark}

Concerning the case $p=0$, the following result shows that $\hat p_{N,K,t}$ is not always consistent.

\begin{prop}\label{pzero}
We assume that $p=0$ and that $H(q)$ holds for some $q> 3$. We set $c_{p,\Lambda}:=(1-\Lambda p)^2/(2\Lambda^2)$.
We always work in the asymptotic  $(N,K,t)\to (\infty,\infty,\infty)$
and in the regime $ \frac NK \sqrt{\frac{\Delta_t}t}+ \frac{N}{t\sqrt K}+Ne^{-c_{p,\Lambda}K} \to 0$.

(i) If
$[\frac{N}{t\sqrt K}]/[\frac NK \sqrt{\frac{\Delta_t}t}]^2\to \infty$,
we have
$$
\hat p_{N,K,t} \stackrel{P}{\longrightarrow} 0.
$$

(ii) If
$[\frac{N}{K}\sqrt{\frac{\Delta_t}t}]^2/[ \frac{N}{t\sqrt K} ]\to \infty$, we have
$$
\hat p_{N,K,t} \stackrel{d}{\longrightarrow} X
$$
where $P(X=1)=P(X=0)=\frac{1}{2}$.
\end{prop}

\subsection{The result in the supercritical case}
Here we will assume $A$ and first recall the estimator we built in \cite{D},
$\bar{Z}_{t}^{N,K}$ being defined as previously. We set
\begin{gather}
\label{UP}\mathcal{U}_{t}^{N,K}:=\Big[\frac{N}{K}\sum_{i=1}^{K}\Big(\frac{Z_{t}^{i,N}-\bar{Z}_{t}^{N,K}}{\bar{Z}_{t}^{N,K}}\Big)^{2}-
\frac{N}{\bar{Z}_{t}^{N,K}}\Big]\boldsymbol{1}_{\{\bar{Z}_{t}^{N,K}>0\}}\\
 \label{PU}\hbox{and} \quad \mathcal{P}_{t}^{N,K}:=\frac{1}{\mathcal{U}_{t}^{N,K}+1}\boldsymbol{1}_{\{\mathcal{U}_{t}^{N,K}\ge 0\}}.
\end{gather}

It was shown in \cite[Theorem 2.3]{D}
that we assume $A$ (actually for a much more general class of functions $\phi$), for all $\eta>0$,
for some constant $C_\eta>0$
(depending only on $\eta,p,\mu,b$), for all $\e\in(0,1)$, all $\eta>0$, for all
$1\leq K \leq N$, all $t\geq 1$,
$$
P(| \cP_t^{N,K}-p|>\e) \leq \frac {C_\eta e^{\eta t}}\epsilon \Big(\frac N{\sqrt K e^{\alpha_0 t}} 
+ \frac 1{\sqrt K} \Big),
$$
where $\alpha_0=p-b >0$ (it is determined by the equation $p\int_{0}^{\infty}e^{-\alpha_{0}t}\phi(t)dt=1$).
We also showed using a toy model in \cite[Section 13]{D} that this rate of convergence is likely to be optimal.
Finally, to have an idea of the orders of magnitude, we recall that assumption $A$, see
\cite[Remark 5]{A},  for any $\eta>0$,
$\lim_{(N,t)\to(\infty,\infty)} P(\bar{Z}^{N}_t\in [e^{(\alpha_0-\eta)t},e^{(\alpha_0+\eta)t}])=1$.
Hence, when observing a sample of $K$ individuals during $[0,t]$, one observes around $Ke^{\alpha_0 t}$ jumps.
Here is the main result of the present paper in the supercritical case.

\begin{theorem}\label{mainsupsup}
We assume $(A)$ and set $\alpha_{0}=p-b$. In the regime where $(N,K,t)\to(\infty,\infty,\infty)$
with $\frac N{\sqrt K e^{\alpha_0 t}} + \frac 1{\sqrt K} \to 0$ with dominating term 
$\frac N{\sqrt K e^{\alpha_0 t}}$ (i.e. with $[\frac N{\sqrt K e^{\alpha_0 t}}]/ [\frac 1{\sqrt K}]\to \infty$),
it holds that,
\begin{align*}
\frac{e^{\alpha_0 t}\sqrt{K}}{N}\Big(\mathcal{P}_{t}^{N,K}-p\Big)\stackrel{d}{\longrightarrow}\mathcal{N}\Big(0,\frac{2(\alpha_0)^4p^2}{\mu ^2}\Big).
\end{align*}
\end{theorem}

While our result in the subcritical case is rather general and satisfying, there are many restrictions
in the supercritical case. First, we have not been able to deal with general functions $\phi$.
Second, we did not manage to prove a central limit theorem concerning a large Bernoulli random matrix
(and its Perron-Frobenius eigenvalue and eigenvector) that would allow us to study the second regime
where  $[\frac 1{\sqrt K}]/[\frac N{\sqrt K e^{\alpha_0 t}}]\to\infty$.

\subsection{Reference and fields of application}
Hawkes processes were first introduced as an birth-immigration model by Hawkes in \cite{C}. The 
properties of one dimensional Hawkes processes have been well
studied, see e.g. Chapter 12 of Daley and Vere-Jones in  \cite{M} for the stability of the  process,  
Br\'emaud and Massouli\'e in \cite{O} for the analysis of the Bartlett spectrum of the process.
Some limit theorems of some large systems modeled by interacting Hawkes processes also have also
been estiablished by Delattre, Fournier and Hoffmann, \cite{E}.
In \cite{bdhm2}, Bacry, Delattre, Hoffmann and Muzy prove a law of large numbers and a functional 
central limit theorem for finite dimensional Hawkes processes observed over a time interval $[0,T]$, 
as $T\to\infty.$
Zhu proves some large deviation principles for Markovian nonlinear Hawkes processes in the subcritical 
case in \cite{G} 
and central limit theorem of  stationary and ergodic nonlinear Hawkes process in \cite{GLZ}. 

%\cite{WZC} 
%show the the application of Hawkes process in dynamic events. 

\vip

Hawkes processes have a lot of applications:

\vip

$\bullet$  earthquake seismology, see e.g. Ogata \cite{YO},
 
$\bullet$  finance  about market orders modelling, see e.g. Bauwens and Hautsch \cite{Q} or Lu and Abergel \cite{LXF}, 
 
$\bullet$ neuroscience, see e.g. Br\'emaud-Massouli\'e \cite{R},
  
$\bullet$ criminology, see e.g. Mohler,  Short, Brantingham,  Schoenberg and Tita \cite{S},

$\bullet$ genomics, see e.g. Gusto and Schbath \cite{T}.

$\bullet$ social networks interactions, see Blundell et al. \cite{bhb} and Zhou et al. \cite{zzs}.

\vip

For more examples see the references \cite{E}.

\subsection{Plan of the paper}
Sections \ref{sesub} to \ref{finsecsub} are devoted to the study of the subcritical case.
After some preliminaries stated in Section \ref{sesub}, we study some random matrix in Section \ref{sec3},
establish some limit theorems for the first and second estimator in Section \ref{first theorem}, 
and for the third
one in Section \ref{sec5}. We conclude the study of the subcritical case in Section \ref{finsecsub}.

\vip

Concerning the supercritical case, we study the random matrix in Section \ref{secsup}, the stochastic
processes in Section \ref{misecsup}, and conclude the proof in Section \ref{finsecsup}.

\vip

An appendix containing some technical results lies at the end of the paper.

\subsection{Important notation}
In the whole paper, we denote by $\Et$ the conditionnal expectation knowing $(\theta_{ij})_{i,j=1,\dots N}$.

\section{Preliminaries for the subcritical case}\label{sesub}
\subsection{Some notations}\label{subimn}
For $r\in [1,\infty)$ and $\boldsymbol{x}\in \R^N$, we set $\|\boldsymbol{x}\|_{r}=(\sum_{i=1}^{N}|x_{i}|^{r})^{\frac{1}{r}}$, and$\ \|\boldsymbol{x}\|_{\infty}=\max_{i=1...N}|x_{i}|$. For $M$ a $N\times N$ matrix, we denote by $|||M|||_{r}$ is the operator norm associated to $ \|\cdot \|_{r}$, that is $|||M|||_{r}=\sup_{\boldsymbol{x}\in R^{n}}\|M\boldsymbol{x}\|_{r}/\|\boldsymbol{x}\|_{r}$. We  have the special cases
$$
|||M|||_{1}=\sup_{j=1,...,N}\sum_{i=1}^{N}|M_{ij}|,\quad |||M|||_{\infty}=\sup_{i=1,...,N}\sum_{j=1}^{N}|M_{ij}|.
$$
We also have the inequality
$$
|||M|||_{r}\le|||M|||_{1}^{\frac{1}{r}}|||M|||_{\infty}^{1-\frac{1}{r}}\quad \hbox{for any}\quad r\in [1,\infty).
$$

We define $A_{N}(i,j):=N^{-1}\theta_{ij}$ for $i,j=1,\dots,N$, as well as 
$Q_{N}:=(I-\Lambda A_{N})^{-1}$ on the event on which $I-\Lambda A_N$ is invertible.

\vip

For $1\leq K \leq N$, we introduce the $N$-dimensional vector $\boldsymbol{1}_K$ defined by
$\boldsymbol{1}_K(i)=\indiq_{\{1\leq i\leq K\}}$ for $i=1,\dots,N$, and the $N\times N$-matrix $I_K$  
defined by $I_K(i,j)={\bf 1}_{\{i=j\leq K\}}$.

\vip

We assume here that $\Lambda p\in\ (0,1)$ and we set 
$a=\frac{1+\Lambda p}{2}\in\ (0,1).$ Next, we introduce  the events
\begin{eqnarray*}
&\Omega_{N}^{1}:=\Big\{\Lambda |||A_{N}|||_{r}\le a ,\ \hbox{for all } \  r\in[1,\infty]\Big\},\quad\\ &\mathcal{F}_{N}^{K,1}:=\Big\{\Lambda |||I_{K}A_{N}|||_{r}\le \Big(\frac{K}{N}\Big)^{\frac{1}{r}}a, \hbox{for}\  \hbox{all}\  r\in [1,\infty)\Big\},\\
&\mathcal{F}_{N}^{K,2}:=\Big\{\Lambda |||A_{N}I_{K}|||_{r}\le \Big(\frac{K}{N}\Big)^{\frac{1}{r}}a,\ \hbox{for}\   \hbox{all}\  r\in [1,\infty)\Big\},\\
&\Omega^{1}_{N,K}:= \Omega^1_N \cap\mathcal{F}_{N}^{K,1}, \quad \Omega^{2}_{N,K}:= \Omega^1_N \cap\mathcal{F}_{N}^{K,2},
\quad  \Omega_{N,K}=\Omega^{1}_{N,K}\cap\Omega^{2}_{N,K}.
\end{eqnarray*}

Recall that $c_{p,\Lambda}=(1-\Lambda p)^2/(2\Lambda^2).$
\begin{lemma}\label{ONK}
Assume that $\Lambda p<1$. It holds that $$P(\Omega_{N,K})\ge 1-CNe^{-c_{p,\Lambda}K}$$ for some constants $C>0$.
\end{lemma}

\begin{proof}
On $\Omega_{N,K}^{1}$, we have 
$$
N|||I_{K}A_{N}|||_{1}=\sup_{j=1,...,N}\sum_{i=1}^{K}\theta_{ij}=\max\{X_{1}^{N,K},... ,X_{N}^{N,K}\},
$$ 
where $X_{i}^{N,K}=\sum_{j=1}^{K}\theta_{ij}$ for $i=1,...,N$ are i.i.d and Binomial$(K,p)$-distributed.
So, 
\begin{eqnarray*}
P\Big(\Lambda\frac{N}{K}|||I_{K}A_{N}|||_{1}\ge a\Big)&=&P\Big(\max\{X_{1}^{N,K},...X_{N}^{N,K}\}\ge\frac{Ka}{\Lambda}\Big)  \le NP\Big(X_{1}^{N,K}\ge\frac{Ka}{\Lambda}\Big)\\
&\le& NP\Big(|X_{1}^{N,K}-Kp|\ge K\Big(\frac{a}{\Lambda}-p\Big)\Big)\le 2Ne^{-2K(\frac{a}{\Lambda}-p)^{2}}=2Ne^{-c_{p,\Lambda}K}.
\end{eqnarray*}
The last equality follows from Hoeffding inequality.
On the event $\Omega_{N}^{1}\cap \{\Lambda\frac{N}{K}|||I_{K}A_{N}|||_{1}\le a\},$ 
we have 
$$
|||I_{K}A_{N}|||_{r}\le|||I_{K}A_{N}|||_{1}^{\frac{1}{r}}\|I_{K}A_{N}|||_{\infty}^{1-\frac{1}{r}}\le|||I_{K}A_{N}|||_{1}^{\frac{1}{r}}||A_{N}|||_{\infty}^{1-\frac{1}{r}}\le\Big(\frac{a}{\Lambda}\frac{K}{N}\Big)^{\frac{1}{r}}\Big(\frac{a}{\Lambda}\Big)^{1-\frac{1}{r}}=\frac{a}{\Lambda}\Big(\frac{K}{N}\Big)^{\frac{1}{r}}.
$$
We conclude that $\Omega^{1}_{N,K}=\Omega_{N}^{1}\cap\{\Lambda(\frac{N}{K})|||I_{K}A_{N}|||_{1}\le a\}$. And from the proof of \cite[Lemma 13]{A},  we  find that  that $P(\Omega^{1}_{N})\ge 1-CNe^{-c_{p,\Lambda}N}.$ Hence  
$$
P(\Omega^{1}_{N,K})\ge P(\Omega^{1}_{N})+P\Big(\Lambda\frac{N}{K}|||I_{K}A_{N}|||_{1}\le a\Big)-1\ge 1-CNe^{-c_{p,\Lambda}K}.
$$
By the same way, we prove that $P(\Omega^{2}_{N,K})\ge 1-CNe^{-c_{p,\Lambda}K}.$  
Finally by the definition of $\Omega_{N,K}$, we have $P(\Omega_{N,K})\ge P(\Omega^{1}_{N,K})+P(\Omega^{2}_{N,K})-1\ge  1-CNe^{-c_{p,\Lambda}K}.$
\end{proof}

\smallskip

Next, we set 
$\boldsymbol{\ell}_N:=Q_N\boldsymbol{1}_N$, i.e. $\ell_{N}(i):=\sum_{j=1}^{N}Q_{N}(i,j)$, as well as 
$\bar{\ell}_N:=\frac{1}{N}\sum_{i=1}^N\ell_N(i),\ \bar{\ell}^K_N:=\frac{1}{K}\sum_{i=1}^K\ell_N(i)$.
We also set $c^K_{N}(j):=\sum_{i=1}^{K}Q_{N}(i,j),\
\bar{c}^K_N:=\frac{1}{N}\sum_{j=1}^Nc^K_N(j)$.

\vip

We let $\boldsymbol{L}_N:=A_N\boldsymbol{1}_N$, i.e. $L_{N}(i):=\sum_{j=1}^{N}A_{N}(i,j).$ We also let $\bar{L}_N:=
\frac{1}{N}\sum_{i=1}^NL_N(i),\ \bar{L}^K_N:=\frac{1}{K}\sum_{i=1}^KL_N(i)$ and
$\boldsymbol{C}_N:=A_N^*\boldsymbol{1}_N$, i.e. $C_{N}(j):=\sum_{i=1}^{N}A_{N}(i,j),\
\bar{C}_N:=\frac{1}{N}\sum_{j=1}^NC_N(j),\ \bar{C}^K_N:=\frac{1}{K}\sum_{j=1}^KC_N(i)$
and consider the event
\begin{align}\label{mA}
\mathcal{A}_{N}:=\{\|\boldsymbol{L}_{N}-p\boldsymbol{1}_{N}\|_{2}+\|\boldsymbol{C}_{N}-p\boldsymbol{1}_{N}\|_{2}\le N^{\frac{1}{4}}\}.
 \end{align}
We also set $x_{N}(i)=\ell_{N}(i)-\bar{\ell}_{N}$,  $\boldsymbol{x}_{N}=(x_N(i))_{i=1,\dots,N}$,
$X_{N}(i)=L_{N}(i)-\bar{L}_{N}$ and $\boldsymbol{X}_{N}=(X_N(i))_{i=1,\dots,N}.$
We finally put $X_{N}^{K}(i)=(L_{N}(i)-\bar{L}^{K}_{N})\indiq_{\{i\leq K\}}$ and $\boldsymbol{X}^{K}_{N}=
(X_N^K(i))_{i=1,\dots,N}=\boldsymbol{L}_{N}^K- \bar{L}^{K}_{N}\boldsymbol{1}_K$, as well as
$x_{N}^{K}(i)=(\ell_{N}(i)-\bar{\ell}^{K}_{N})\indiq_{\{i\leq K\}}$ and $\boldsymbol{x}^{K}_{N}=
(x_N^K(i))_{i=1,\dots,N}=\boldsymbol{\ell}_{N}^K- \bar{\ell}^{K}_{N}\boldsymbol{1}_K$.
Next, we are going to review some important results in \cite{A}.
\begin{lemma}\label{lo}
We assume that $\Lambda p<1.$ Then
$\Omega_{N,K}\subset\Omega_{N}^{1}\subset\{|||Q_{N}|||_{r}\le C, $ for all $r \in[1,\infty]\}\subset\{\sup_{i=1...N}\ell_{N}(i)\le C\}$, where $C=(1-a)^{-1}$. For any $\alpha>0$, there exists a constant $C_{\alpha}$ such that 
$$P(\mathcal{A}_{N})\ge 1-C_{\alpha}N^{-\alpha}.$$

\end{lemma}
\begin{proof}
See \cite[Notation 12 and Proposition 14, Step 1]{A}.
\end{proof}

\subsection{Some auxilliary processes}\label{aux}

We first introduce a family of martingales: for $i=1,\dots,N$, recalling \eqref{sssy},
$$
M_{t}^{i,N}=\int_{0}^{t}\int_{0}^{\infty}  \boldsymbol{1}_{\{z\le\lambda_{s}^{i,N}\}}\widetilde{\pi}^{i}(ds,dz),
$$
where $\widetilde{\pi}^{i}(ds,dz)=\pi^i(ds,dz)-dsdz$.
We also introduce the family of centered   processes $U_{t}^{i,N}=Z_{t}^{i,N}-\mathbb{E}_{\theta}[Z_{t}^{i,N}]$.

\vip

We  denote  by $\boldsymbol{Z}_{t}^{N}$ (resp. $\boldsymbol{U}_{t}^{N}$, $\boldsymbol{M}_{t}^{N}$) the $N$ dimensional
vector  with  coordinates $Z_{t}^{i,N}$ (resp. $U_{t}^{i,N}$, $M_{t}^{i,N}$) and
set 
$$
\boldsymbol{Z}_{t}^{N,K}=I_{K}\boldsymbol{Z}_{t}^{N},\quad \boldsymbol{U}_{t}^{N,K}=I_{K}\boldsymbol{U}_{t}^{N},
$$
as well as $\bar{Z}^{N,K}_{t}=K^{-1}\sum_{i=1}^{K}Z_{t}^{i,N}$ and $\bar{U}^{N,K}_{t}=K^{-1}\sum_{i=1}^{K}U_{t}^{i,N}$.
By \cite[Remark 10 and Lemma 11]{A}, we have the following  equalities:
\begin{align}
\label{ee1}&\mathbb{E}_{\theta}[\boldsymbol{Z}_{t}^{N,K}]=\mu\sum_{n\ge0}\Big[\int_{0}^{t}s\phi^{*n}(t-s)ds\Big]I_{K}A_{N}^{n}\boldsymbol{1}_{N},\\
\label{ee2}&\boldsymbol{U}_{t}^{N,K}=\sum_{n\ge0}\int_{0}^{t}\phi^{*n}(t-s)I_{K}A_{N}^{n}\boldsymbol{M}_{s}^{N}ds,\\
\label{ee3}&[M^{i,N},M^{j,N}]_{t}=\boldsymbol{1}_{\{i=j\}}Z_{t}^{i,N}.
\end{align}

We use the convention that $\phi^{*0}=\delta_0$, whence in particular
$\int_{0}^{t}s\phi^{*0}(t-s)ds=t$.

\begin{lemma}\label{Zt}
Assume $H(q)$ for some $q \ge 1$. There exists a constant $C$ such that

\vip

(i) for all $r$ in $[1,\infty]$, all  $t \ge0$, a.s.,
$$
\boldsymbol{1}_{\Omega_{N,K}}\|\mathbb{E}_{\theta}[\boldsymbol{Z}_{t}^{N,K}]\|_{r}\le CtK^{\frac{1}{r}}.
$$

(ii) For any $r\in[1,\infty]$, for all $t\geq s \geq 0$,
$$
\boldsymbol{1}_{\Omega_{N,K}}\|\mathbb{E}_{\theta}[\boldsymbol{Z}^{N,K}_{t}-\boldsymbol{Z}_{s}^{N,K}-\mu(t-s)\boldsymbol{\ell}_{N}^{K}]\|_{r}\le C(\min\{1,s^{1-q}\})K^{\frac{1}{r}}.
$$

(iii) For all $t\geq s+1 \geq 1$, on $\Omega_{N,K}$, we have a.s.,
$$
 \mathbb{E}_{\theta}[(\bar{U}^{N,K}_{t}-\bar{U}_{s}^{N,K})^4]\le \frac{C (t-s)^{2}}{K^{2}}\quad \hbox{and}\quad
 \mathbb{E}_{\theta}[(\bar{Z}^{N,K}_{t}-\bar{Z}_{s}^{N,K})^4]\le C (t-s)^{4}. 
$$
\end{lemma}

\begin{proof}
See \cite[Lemma 5.1]{D} for the proofs of $(i)$ and $(ii).$ For $(iii)$, we deduce from \eqref{ee2}
that
 $$
 \bar U^{N,K}_t = K^{-1} \sum_{n\geq 0} \intot \phi^{\star n}(t-s) \sum_{i=1}^K\sum_{j=1}^N A_N^n(i,j)M^{j,N}_sds.
 $$ 
We set $\phi(s)=0$ for $s\le 0$. Separating the cases $n=0$ and $n\geq 1$,  
using the Minkowski inequality, we see that on $\Omega_{N,K}$, we have
\begin{align*}
&\mathbb{E}_{\theta}[(\bar{U}^{N,K}_{t}-\bar{U}_{s}^{N,K})^4]^\frac{1}{4}\\
\leq& \mathbb{E}_{\theta}[(\bar{M}^{N,K}_{t}-\bar{M}_{s}^{N,K})^4]^\frac{1}{4}\\
&+\frac{1}{K} \sum_{n\geq 1} \int_0^{\infty} \Big(\phi^{\star n}(t-u)- \phi^{\star n}(s-u)\Big)
\Et\Big[\Big(\sum_{i=1}^K\sum_{j=1}^N A_N^n(i,j)M^{j,N}_u\Big)^4\Big]^\frac{1}{4} du.
\end{align*}
By \cite[Lemma 16 (iii)]{A}, we already know that, on $\Omega_{N,K}$, 
$\max_{i=1,\dots,N} \Et [ (Z^{i,N}_{t}-Z^{i,N}_s)^2]\le C(t-s)^2.$
For the first term ($n=0$), we use (\ref{ee3}) and Burkholder's inequality:
\begin{align*}
\mathbb{E}_{\theta}[(\bar{M}^{N,K}_{t}-\bar{M}_{s}^{N,K})^4]=&\frac{1}{K^4}\mathbb{E}_{\theta}\Big[\Big(\sum_{i=1}^K(M^{i,N}_{t}-M_{s}^{i,N})\Big)^4\Big]\\
\le& \frac{C}{K^4}\mathbb{E}_{\theta}\Big[\Big(\sum_{i=1}^K(Z^{i,N}_{t}-Z_{s}^{i,N})\Big)^2\Big]\\
\le& \frac{C(t-s)^2}{K^2}
\end{align*}
For the second term ($n\ge 1$), we use again (\ref{ee3}) and by Burkholder's inequality and we get
\begin{align*}
    \Et\Big[\Big(\sum_{i=1}^K\sum_{j=1}^N A_N^n(i,j)M^{j,N}_u\Big)^4\Big]\le&  C\Et\Big[\Big([\sum_{i=1}^K\sum_{j=1}^N A_N^n(i,j)M^{j,N},\sum_{i=1}^K\sum_{j=1}^N A_N^n(i,j)M^{j,N}]_u\Big)^2\Big] \\
    \le& C\Et\Big[\Big(\sum_{j=1}^N\Big(\sum_{i=1}^K\ A_N^n(i,j)\Big)^2Z^{j,N}_u\Big)^2\Big]\\
    \le& C\Et\Big[\Big(\sum_{j=1}^N|||I_KA^n_N|||_1^2Z^{j,N}_u\Big)^2\Big]\\
     \le& C\Et\Big[\Big(\sum_{j=1}^N|||I_KA_N|||_1^2|||A_N|||_1^{2(n-1)}Z^{j,N}_u\Big)^2\Big]\\
     \le& CN^2u^2|||I_KA_N|||_1^4|||A_N|||_1^{4(n-1)}.
\end{align*}
It implies that 
\begin{align*}
    &\frac{1}{K} \sum_{n\geq 1} \int_0^{\infty} \Big(\phi^{\star n}(t-u)- \phi^{\star n}(s-u)\Big)
\Et\Big[\Big(\sum_{i=1}^K\sum_{j=1}^N A_N^n(i,j)M^{j,N}_u\Big)^4\Big]^\frac{1}{4} du\\
\le& \frac C {K} \sum_{n\geq 1}|||I_KA_N|||_1 |||A_N|||_1^{n-1}\intot \sqrt{Nu} \Big(\phi^{\star n}(t-u)- \phi^{\star n}(s-u)\Big) du\\
\leq& \frac {C(t-s)^{1/2}} {N^{1/2}} \sum_{n\geq 0} \Lambda^n |||A_N|||_1^{n} \leq \frac{C (t-s)^{1/2}}{N^{1/2}}.
\end{align*}
We used first that for all $n\geq 1$, it holds that 
\begin{align*}
\intot \sqrt{u} (\phi^{\star n}(t-u)- \phi^{\star n}(s-u)) du=&\intot \sqrt{t-u} \phi^{\star n}(u) du - 
\int_0^s \sqrt{s-u}\phi^{\star n}(u) du\\
\leq & \int_0^s [\sqrt{t-u}-\sqrt{s-u}] \phi^{\star n}(u) du +\int_s^t \sqrt{t-u} \phi^{\star n}(u) du \\
\leq & 2 \sqrt{t-s} \int_0^\infty \phi^{\star n}(u) du \leq 2 \Lambda^n \sqrt{t-s}.
\end{align*}
We next used that on $\Omega_{N,K}$, we have $\Lambda |||A_N|||_1 \leq a<1$ and  
$\Lambda |||I_K A_N|||_1 \leq a K /N$.
This completes the first part of $(iii)$. 

\vip

For the second part, by \cite[Lemma 5.1 (ii)]{D}, we have 
$\Et[\bar{Z}^{N,K}_{t}-\bar{Z}_{s}^{N,K}]\le C(t-s)$ on $\Omega_{N,K}$, whence
$$
\mathbb{E}_{\theta}[(\bar{Z}^{N,K}_{t}-\bar{Z}_{s}^{N,K})^4]\le 4\Big\{\Et[\bar{Z}^{N,K}_{t}-\bar{Z}_{s}^{N,K}]^4+\mathbb{E}_{\theta}[(\bar{U}^{N,K}_{t}-\bar{U}_{s}^{N,K})^4]\Big\}\le C (t-s)^{4}
$$
as desired.
\end{proof}

\section{Some limit theorems for the random matrix in the subcritical case}\label{sec3}
\subsection{First estimator}\label{Fe}
As we will see, the first estimator $\e^{N,K}_t$ is closely linked to $\bar{\ell}_{N}^{K}$. For this last quantity,
we will only use the following easy inequality,
of which the proof can be found in  \cite[Lemma 3.9]{D}.

\begin{lemma}\label{ellp}
If $\Lambda p<1$, there is $C>0$ such that for all $1\leq K \leq N$,
$$
\mathbb{E}\Big[\boldsymbol{1}_{\Omega_{N,K}}\Big|\bar{\ell}_{N}^{K}-\frac{1}{1-\Lambda p}\Big|^{2}\Big]\le\frac{C}{NK}.
$$
\end{lemma}

\subsection{Second estimator}\label{Sece}
The second estimator $\cV_t^{N,K}$ is related to $\cV_\infty^{N,K}=\frac{N}{K}\|\bx^{K}_{N}\|_{2}^{2}$, 
which we now study.
\begin{theorem}\label{21}
Assume $\Lambda p<1$. Then, in distribution, as $(N,K)\to (\infty,\infty),$ in the regime $Ne^{-c_{p,\Lambda}K}\to 0$,
$$
\indiq_{\Omega_{N,K}} \sqrt K\Big( \cV_\infty^{N,K}-\frac{\Lambda^{2}p(1-p)}{(1-\Lambda p)^{2}}\Big)
\longrightarrow \mathcal{N}\Big(0,\Big(\Lambda^{2}\frac{p(1-p)}{(1-\Lambda p)^2}\Big)^{2}\Big).$$
\end{theorem}

The proof relies on four lemmas.

\begin{lemma}\label{trans}
Assume that $\Lambda p<1$. There is $C>0$ such that for all $1\leq K \leq N$,
$$
\mathbb{E}[||(I_{K}A_{N})^{T}\bX_{N}^{K})||_{2}^{2}]\le \frac{CK^2}{N^{3}} .
$$
\end{lemma}

\begin{proof} Recall that $\bX_{N}^{K}=\bL_N^K - \bar L^{K}_N \indiq_K$. By symmetry, we have
\begin{align*}
\mathbb{E}[\|(I_{K}A_{N})^{T}\bX_{N}^{K}\|_{2}^{2}]
=&\frac{K}{N^{2}}\mathbb{E}\Big[\Big(\sum_{j=1}^{K}\theta_{j1}(L_{N}(j)-\bar{L}_{N}^{K})\Big)^{2}\Big]\\
\le& \frac{2K}{N^{2}}\Big\{\mathbb{E}\Big[\Big(\sum_{j=1}^{K}\theta_{j1}(L_{N}(j)-p)\Big)^{2}\Big]+\mathbb{E}\Big[\Big(\sum_{j=1}^{K}\theta_{j1}(p-\bar{L}_{N}^{K})\Big)^{2}\Big]\Big\}.
\end{align*}
First, since $\theta_{j1} \leq 1$, we obviously have
$$
\frac{K}{N^{2}}\mathbb{E}\Big[\Big(\sum_{j=1}^{K}\theta_{j1}(p-\bar{L}_{N}^{K})\Big)^{2}\Big]\le \frac{K^{3}}{N^{2}}\mathbb{E}[(p-\bar{L}_{N}^{K})^{2}]\le \frac{CK^{2}}{N^{3}}.
$$
Next, 
\begin{align*}
    &\frac{K}{N^{2}}\mathbb{E}\Big[\Big(\sum_{j=1}^{K}\theta_{j1}(L_{N}(j)-p)\Big)^{2}\Big]\\
    \le& \frac{2K}{N^{4}}\Big\{\mathbb{E}\Big[\Big(\sum_{j=1}^{K}\sum_{i=2}^{N}\theta_{j1}(\theta_{ji}-p)\Big)^{2}\Big]
    +\mathbb{E}\Big[\Big(\sum_{j=1}^{K}\theta_{j1}(\theta_{j1}-p)\Big)^{2}\Big]\Big\}\\
    \le& \frac{4K}{N^{4}}\Big\{\mathbb{E}\Big[\Big(\sum_{j=1}^{K}\sum_{i=2}^{N}(\theta_{j1}-p)(\theta_{ji}-p)\Big)^{2}\Big]+p^{2}\mathbb{E}\Big[\Big(\sum_{j=1}^{K}\sum_{i=2}^{N}(\theta_{ji}-p)\Big)^{2}\Big]+\mathbb{E}\Big[\Big(\sum_{j=1}^{K}\theta_{j1}(\theta_{j1}-p)\Big)^{2}\Big]\Big\}.
\end{align*}
This is controled by $C K^2/N^3$ as desired, because
$\mathbb{E}[(\sum_{j=1}^{K}\sum_{i=2}^{N}(\theta_{j1}-p)(\theta_{ji}-p))^{2}]\le CKN$
(since the family $\{(\theta_{ji}-p), i=2,\dots,N,j=1,\dots,K\}$ is independant and centered),
because $\mathbb{E}[(\sum_{j=1}^{K}\sum_{i=2}^{N}(\theta_{ji}-p))^{2}]\le CNK$ (for similar reasons), and 
$\mathbb{E}[(\sum_{j=1}^{K}\theta_{j1}(\theta_{j1}-p))^{2}]\le CK^{2}$.
\end{proof}

\begin{lemma}\label{inverse M}
Assume that $0<p\le 1$. There is $C>0$ such that for all $1\leq K \leq N$,
$$
\mathbb{E}\Big[\Big|\Big(I_{K}A_{N}\bX_{N},\bX_{N}^{K}\Big)\Big|\Big]\le \frac{CK}{N^2}.
$$
\end{lemma}
\begin{proof}
By definition, we have 
\begin{align*}
 (I_{K}A_{N}\bX_{N},\bX_{N}^{K})=&\frac{1}{N}\sum_{i,j=1}^{K}(\theta_{ij}-p)X_{N}(j)X_{N}^{K}(i)\\
 =&\frac{1}{N}\Big[\sum_{i,j=1}^{K}(\theta_{ij}-p)(L_{N}(j)-p)X_{N}^{K}(i)+(p-\bar{L}_{N})\sum_{i,j=1}^{K}(\theta_{ij}-p)X_{N}^{K}(i)\Big]\\
 =&\frac{1}{N}\Big[\sum_{i,j=1}^{K}(\theta_{ij}-p)(L_{N}(j)-p)(L_{N}(i)-p)+(p-\bar{L}_{N}^{K})\sum_{i,j=1}^{K}(\theta_{ij}-p)(L_{N}(j)-p)\\
 &\qquad+(p-\bar{L}_{N})\sum_{i,j=1}^{K}(\theta_{ij}-p)(L_{N}(i)-p)+(p-\bar{L}_{N})(p-\bar{L}_{N}^{K})\sum_{i,j=1}^{K}(\theta_{ij}-p)\Big].
\end{align*}
We start with the first term:
\begin{align*}
    &\mathbb{E}\Big[\Big(\sum_{i,j=1}^{K}(\theta_{ij}-p)(L_{N}(j)-p)(L_{N}(i)-p)\Big)^{2}\Big]\\
    =&\frac{1}{N^{4}}\mathbb{E}\Big[\Big(\sum_{i,j=1}^{K}\sum_{m,n=1}^{N}
    (\theta_{ij}-p)(\theta_{jm}-p)(\theta_{in}-p)\Big)^{2}\Big]\\
   =&\frac{1}{N^{4}}\mathbb{E}\Big[\sum_{i,j,i',j'=1}^{K}\sum_{m,n,m',n'=1}^{N}(\theta_{ij}-p)(\theta_{jm}-p)(\theta_{in}-p)(\theta_{i'j'}-p)(\theta_{j'm'}-p)(\theta_{i'n'}-p)\Big] \le \frac{CK^{2}}{N^{2}}
\end{align*}
since the family $\{(\theta_{ij}-p),i,j=1,\dots,N\}$ is i.i.d., centered, and bounded.
For the second term, we write, using the Cauchy-Schwarz inequality,
\begin{align*}
\mathbb{E}\Big[\Big|(p-\bar{L}_{N}^{K})\sum_{i,j=1}^{K}(\theta_{ij}-p)(L_{N}(j)-p)\Big|\Big]
\le \frac{1}{N}\mathbb{E}[(p-\bar{L}_{N}^{K})^{2}]^\frac{1}{2}
\mathbb{E}\Big[\Big(\sum_{i,j=1}^{K}\sum_{k=1}^{N}(\theta_{ij}-p)(\theta_{jk}-p)\Big)^{2}\Big]^{\frac{1}{2}}.
\end{align*}
This is dominated by $\frac{\sqrt{K}}{N}$, because on the first hand, we have the equality 
$\mathbb{E}[(p-\bar{L}_{N}^{K})^{2}]=\frac{1}{N^2K^2}\mathbb{E}[(\sum_{i=1}^K\sum_{j=1}^N(\theta_{ij}-p))^2]
=\frac{\mathbb{E}[(\theta_{11}-p)^2]}{NK}\le\frac{C}{NK}$, and on the other hand,
\begin{align*}
    &\mathbb{E}\Big[\Big(\sum_{i,j=1}^{K}\sum_{k=1}^{N}(\theta_{ij}-p)(\theta_{jk}-p)\Big)^{2}\Big]\\
    =&\mathbb{E}\Big[\sum_{i,j,i',j'=1}^{K}\sum_{k,k'=1}^{N}(\theta_{ij}-p)(\theta_{i'j'}-p)(\theta_{jk}-p)(\theta_{j'k'}-p)\Big]\le CNK^2.
\end{align*}
For the third term,  using Cauchy-Schwarz inequality, we can write (by the previous discussion,
we have $\E[(p-\bar{L}_{N})^{2}]=\E[(p-\bar{L}_{N}^N)^{2}]\leq \frac C{N^2}$),
\begin{align*}
&\mathbb{E}\Big[\Big|(p-\bar{L}_{N})\sum_{i,j=1}^{K}(\theta_{ij}-p)(L_{N}(i)-p)\Big|\Big]\\
\le& \frac{1}{N}\mathbb{E}[(p-\bar{L}_{N})^{2}]^{\frac{1}{2}}\mathbb{E}\Big[\Big(\sum_{i,j=1}^{K}\sum_{k=1}^{N}(\theta_{ij}-p)(\theta_{ik}-p)\Big)^{2}\Big]^{\frac{1}{2}}\\
\le& \frac{1}N \sqrt\frac{C}{N^2} \mathbb{E}\Big[\sum_{i,j,i',j'=1}^{K}\sum_{k,k'=1}^{N}(\theta_{ij}-p)(\theta_{ik}-p)(\theta_{i'j'}-p)(\theta_{i'k'}-p)\Big]^{\frac{1}{2}}\\
\le&  \frac{1}N \sqrt\frac{C}{N^2} \sqrt{K^2 N+K^4}= C \frac K{N^{3/2}}+C \frac {K^2}{N^2}.
\end{align*}
Finally, we study the last term. We observe that 
$\mathbb{E}[(\sum_{i,j=1}^{K}(\theta_{ij}-p))^{2}]=\mathbb{E}[\sum_{i,j=1}^{K}(\theta_{ij}-p)^{2}]=CK^2$ 
and $\mathbb{E}[(p-\bar{L}_{N}^{K})^{4}]=\frac{1}{N^4K^4}\E[(\sum_{i=1}^K\sum_{j=1}^N(\theta_{ij}-p))^4]
\le\frac{C}{N^2K^2}$. Hence 
\begin{align*}
    &\mathbb{E}\Big[\Big|(p-\bar{L}_{N})(p-\bar{L}_{N}^{K})\sum_{i,j=1}^{K}(\theta_{ij}-p)\Big|\Big]\\
    \le& \mathbb{E}[(p-\bar{L}_{N})^{4}]^{\frac{1}{4}}\mathbb{E}[(p-\bar{L}_{N}^{K})^{4}]^{\frac{1}{4}}\mathbb{E}\Big[\Big(\sum_{i,j=1}^{K}(\theta_{ij}-p)\Big)^{2}\Big]^{\frac{1}{2}}\\
    \le& C \Big(\frac1 {N^4} \Big)^{1/4}\Big(\frac1 {N^2 K^2} \Big)^{1/4} \sqrt{K^2}=\frac{\sqrt K}{ N \sqrt N}.
\end{align*}
This completes the proof.
\end{proof}
Recalling the definition of $\mathcal{A}_N$ in (\ref{mA}) first, we have the following lemma.
\begin{lemma}\label{disx}
Assume $\Lambda p<1.$ There is $C>0$ such that for all $1\leq K \leq N$,
$$
\frac{N}{K}\mathbb{E}\Big[\boldsymbol{1}_{\Omega_{N,K}\cap \cA_{N}}\Big|\Big(\|\bx^{K}_{N}\|_{2}^{2}-(\Lambda\bar{\ell}_{N})^{2}\|\bX_{N}^{K}\|^{2}_{2}\Big)-\|\bx_{N}^{K}-\bar{\ell}_{N}\Lambda \bX_{N}^{K}\|^{2}_{2}\Big|\Big]\le \frac{C}{N}.
$$
\end{lemma}

\begin{proof}
We start from
\begin{align*}
&\mathbb{E}\Big[\boldsymbol{1}_{\Omega_{N,K}\cap \cA_{N}}\Big|\Big(\|\bx^{K}_{N}\|_{2}^{2}-(\Lambda\bar{\ell}_{N})^{2}\|\bX_{N}^{K}\|^{2}_{2}\Big)-\|\bx_{N}^{K}-\bar{\ell}_{N}\Lambda \bX_{N}^{K}\|^{2}_{2}\Big|\Big]\\
&=2\mathbb{E}\Big[\boldsymbol{1}_{\Omega_{N,K}\cap \cA_{N}}\Big|\Lambda\bar{\ell}_{N}(\bX_{N}^{K},\Lambda\bar{\ell}_{N}\bX_{N}^{K}-\bx^{K}_{N})\Big|\Big]
\end{align*}
By \cite[Lemma 3.11]{D}, we already know that
$$\boldsymbol{x}_{N}^{K}-\Lambda\bar{\ell}_{N}\boldsymbol{X}_{N}^{K}
=\Lambda I_{K}A_{N}(\boldsymbol{x}_{N}-\Lambda\bar{\ell}_{N}\boldsymbol{X}_{N})-\frac{\Lambda}{K}(I_{K}A_{N}\boldsymbol{x}_{N}, \boldsymbol{1}_{K})\boldsymbol{1}_{K}+\bar{\ell}_{N}\Lambda^{2} I_{K}A_{N}\boldsymbol{X}_{N}.
$$
Since by definition $(\boldsymbol{1_{K}},\bX_{N}^{K})=0$, we conclude that
\begin{align*}
(\bx_{N}^{K}-\Lambda\bar{\ell}_{N}\bX_{N}^{K},\bX_{N}^{K})=&\Lambda (I_{K}A_{N}(\bx_{N}-\Lambda \bar{\ell}_{N}\bX_{N}),\bX_{N}^{K})
+\Lambda^2 \bar{\ell}_{N}(I_{K}A_{N}\bX_{N},\bX_{N}^{K})= e_{N,K,1}+e_{N,K,2},
\end{align*}
the last equality standing for a definition.
By Lemma \ref{trans}, and \cite[Lemma 4.11]{D}, we have:
\begin{align*}
\mathbb{E}[\boldsymbol{1}_{\Omega_{N,K}\cap \cA_{N}}e_{N,K,1} ]  =&\mathbb{E}\Big[\boldsymbol{1}_{\Omega_{N,K}\cap \cA_{N}}\Lambda \Big((\bx_{N}-\Lambda \bar{\ell}_{N}\bX_{N}),(I_{K}A_{N})^{T}\bX_{N}^{K}\Big)\Big]\\
    \le& \Lambda \mathbb{E}\Big[\boldsymbol{1}_{\Omega_{N,K}\cap \cA_{N}}\|\bx_{N}-\Lambda \bar{\ell}_{N}\bX_{N}\|^2_{2}\Big]^\frac{1}{2}\E\Big[\|(I_{K}A_{N})^{T}\bX_{N}^{K}\|^2_{2}\Big]^\frac{1}{2}\le \frac{CK}{N^2}
\end{align*}
Next, by Lemma \ref{lo}, we know that $\bar{\ell}_N$ is bounded in $\Omega_{N,K}$, and 
by Lemma  \ref{inverse M}, we conclude that
\begin{align*}
\mathbb{E}[\boldsymbol{1}_{\Omega_{N,K}}e_{2,N,K}]\leq \frac{CK}{N^2}
\end{align*}
as desired.
\end{proof}

\begin{lemma}\label{simple}
Assume $\Lambda p<1.$ Then, in distribution, as $(N,K)\to (\infty,\infty)$, in  regime $Ne^{-c_{p,\Lambda}K}\to 0$,
$$
\omg\sqrt{K}\Big[\frac{N}{K}(\bar{\ell}_{N}\|\bX_{N}^{K}\|_{2})^{2}-\frac{p(1-p)}{(1-\Lambda p)^{2}}\Big]\stackrel{d}{\longrightarrow} \mathcal{N}\Big(0,\Big[\frac{p(1-p)}{(1-\Lambda p)^2}\Big]^2\Big).
$$
\end{lemma}

\begin{proof}
Recall   \cite[Proposition 14]{A}: we already have $\E[\indiq_{\Omega_N^1} |\bar\ell_N - \frac{1}{1-\Lambda p} |^2] \leq \frac C {N^2}$. By Lemma \ref{lo}, we also know that $\bar{\ell}_N$  is bounded by some constant $C$ 
on  $\Omega_{N,K}.$ Also, it is easily checked, see e.g. \cite[Equation (7)]{D}, that
$\mathbb{E}[\frac{N^2}{K^2} \|\boldsymbol{X}_{N}^{K}\|_{2}^{4}]\le C$. All in all,
\begin{align*}
&\sqrt{K}\frac{N}{K}\E\Big[\omg\Big|(\bar{\ell}_{N})^{2}-\Big(\frac{1}{1-\Lambda p}\Big)^{2}\Big|\|\bX_{N}^{K}\|_{2}^{2}\Big]\\
\le& \sqrt{K}\mathbb{E}\Big[\frac{N^2}{K^2} \|\boldsymbol{X}_{N}^{K}\|_{2}^{4}\Big]^\frac{1}{2}\E\Big[\omg\Big|\bar{\ell}_{N}-\frac{1}{1-\Lambda p}\Big|^{2}\Big]^\frac{1}{2}\le\frac{C\sqrt{K}}{N} \leq \frac{C}{\sqrt{N}},
\end{align*}
whence it suffices to prove that 
$\omg\sqrt{K}[\frac{N}{K}(\|\bX_{N}^{K}\|_{2})^{2}-p(1-p)]\stackrel{d}{\longrightarrow} \mathcal{N}(0,p^2(1-p)^2).$
We recall that $\omg$ tends to $1$ in probability. Hence it suffices to verify that
$\sqrt{K}[\frac{N}{K}(\|\bX_{N}^{K}\|_{2})^{2}-p(1-p)]\stackrel{d}{\longrightarrow} \mathcal{N}(0,p^2(1-p)^2).$
But
$$
\|\bX_{N}^{K}\|_{2}^{2}=\sum_{i=1}^{K}(L_{N}(i)-\bar{L}^{K}_{N})^{2}=\sum_{i=1}^{K}(L_{N}(i)-p)^{2}-K(p-\bar{L}_{N}^{K})^{2}.
$$
As seen in the proof of Lemma \ref{inverse M}, we have  $\mathbb{E}[(p-\bar{L}_{N}^{K})^{2}]\le \frac{C}{NK},$ 
so that $\sqrt K \frac NK \mathbb{E}[K(p-\bar{L}_{N}^{K})^{2}]\le \frac{C}{\sqrt K}$. Hence, our goal is to 
verify that 
$$
\xi_{N,K}=\sqrt{K}\Big[\frac{N}{K}(\sum_{i=1}^{K}(L_{N}(i)-p)^{2}-p(1-p)\Big]\stackrel{d}{\longrightarrow} 
\mathcal{N}(0,p^2(1-p)^2).
$$
Recalling that $L_{N}(i)=N^{-1}\sum_{j=1}^N \theta_{ij}$, we can check that
\begin{align*}
\xi_{N,K}= \frac 1{N\sqrt K}\sum_{i=1}^{K}\sum_{j=1}^{N}[(\theta_{ij}-p)^{2}-p(1-p)]
+ \frac 1{N\sqrt K} \sum_{i=1}^{K}\sum_{j=1}^N \sum_{j'=1,j'\neq j}^{N}(\theta_{ij}-p)(\theta_{ij'}-p).
\end{align*}
The first term tends to $0$ in probability,
because by the central limit theorem, we have convergence in distribution of 
$\frac 1{\sqrt {NK}}\sum_{i=1}^{K}\sum_{j=1}^{N}[(\theta_{ij}-p)^{2}-p(1-p)]$.
And, using the central limit theorem again, we find that
\begin{align*}
    \frac{1}{N\sqrt{K}}\sum_{i=1}^{K}\sum_{j=1}^N \sum_{j'=1,j'\neq j}^{N}(\theta_{ij}-p)(\theta_{ij'}-p)\stackrel{d}{\longrightarrow} \mathcal{N}(0,p^{2}(1-p)^{2}).
\end{align*}
The proof is complete.
\end{proof}

Finally, we give the

\begin{proof}[Proof of Theorem \ref{21}]
Recalling that $\cV_\infty^{N,K}=\frac{N}{K}\|\bx^{K}_{N}\|_{2}^{2}$, we write
\begin{align*}
\sqrt{K}\Big(\cV_\infty^{N,K}\!-\!\frac{\Lambda^{2}p(1-p)}{(1-\Lambda p)^{2}}\Big)=\frac{N}{\sqrt{K}}\Big(\|\bx^{K}_{N}\|_{2}^{2}-(\Lambda\bar{\ell}_{N})^{2}\|\bX_{N}^{K}\|_{2}^{2}\Big)
+\frac{N(\Lambda\bar{\ell}_{N})^{2}}{\sqrt{K}}\|\bX_{N}^{K}\|_{2}^{2}
-\frac{\Lambda^{2}p(1-p)\sqrt{K}}{(1-\Lambda p)^{2}}.
\end{align*}
By Lemma \ref{simple}, it suffices to 
check that
$\zeta_{N,K}=\indiq_{\Omega_{N,K}}\frac{N}{\sqrt{K}}\Big(\|\bx^{K}_{N}\|_{2}^{2}-(\Lambda\bar{\ell}_{N})^{2}\|\bX_{N}^{K}\|_{2}^{2}\Big) \to 0$ in probability. Since moreover $\indiq_{\cA_N} \to 1$ a.s. by Lemma \ref{lo}, we only have to verify that $\indiq_{\cA_N}\zeta_{N,K} \to 0$ in probability. We write
\begin{align*}
\E[\indiq_{\cA_N}\zeta_{N,K}]\leq& \frac{N}{\sqrt{K}}\E[\indiq_{\Omega_{N,K}\cap \cA_N} \|\bx^{K}_{N}- \Lambda\bar{\ell}_{N}\bX_{N}^{K}\|_{2}^{2}]\\
&+ \frac{N}{\sqrt{K}}\mathbb{E}\Big[\boldsymbol{1}_{\Omega_{N,K}\cap \cA_{N}}\Big|\Big(\|\bx^{K}_{N}\|_{2}^{2}-(\Lambda\bar{\ell}_{N})^{2}\|\bX_{N}^{K}\|^{2}_{2}\Big)-\|\bx_{N}^{K}-\bar{\ell}_{N}\Lambda \bX_{N}^{K}\|^{2}_{2}\Big|\Big].
\end{align*}
By \cite[Lemma 4.11]{D}, the first term is bounded by $C/\sqrt K$. By Lemma \ref{disx}, the second term is bounded by $C \sqrt K / N$.
\end{proof}

\subsection{Third estimator}\label{Te}
The third estimator $\cX^{N,K}_{t,\Delta_t}$ is closely related to 
\begin{align*}
   \cX^{N,K}_{\infty,\infty}:=\cW^{N,K}_{\infty,\infty}-\frac{(N-K)\mu}{K}\bar{\ell}_N^K,\quad \textit{where}\quad  \cW^{N,K}_{\infty,\infty}:=\mu\frac{N}{K^{2}}\sum_{j=1}^{N}\Big(\sum_{i=1}^{K}Q_{N}(i,j)\Big)^{2}\ell_{N}(j). 
\end{align*}
We thus study the convergence of  $\cX^{N,K}_{\infty,\infty}$. The following easy estimate will be sufficient for our
task.

\begin{lemma}\label{W}
When $(N,K)$ tends to $(\infty,\infty)$, with $K\le N$,  we have 
$$
\boldsymbol{1}_{\Omega_{N,K}}\sqrt{K}\Big(\cX^{N,K}_{\infty,\infty}-\frac{\mu}{(1-\Lambda p)^{3}}\Big){\longrightarrow}0
$$
in probability.
\end{lemma}

\begin{proof}
By \cite[lemma 4.19]{D}, we have:

$$
\sqrt{K}\mathbb{E}\Big[\boldsymbol{1}_{\Omega_{N,K}}\Big|\cX^{N,K}_{\infty,\infty}-\frac{\mu}{(1-\Lambda p)^{3}}\Big|\Big]\le \frac{C}{\sqrt{K}}.
$$
which includes the result of the statement.
\end{proof}

We will also need the following estimate, asserting that, setting $A^{N,K}_{\infty,\infty}:=\sum_{i=1}^{N}\Big(c_{N}^{K}(i)\Big)^{2}\ell_{N}(i)$, it holds true that
$\frac{(A^{N,K}_{\infty,\infty})^{2}}{2K^{2}}$ is close to $\frac{1}{2}(\frac{N-K}{N(1-\Lambda p)}+\frac{K}{N(1-\Lambda p)^3})^2$.

\begin{lemma}\label{Agamma}
When $(N,K)$ tends to $(\infty,\infty)$, with $K\le N$ and in the regime where  
$\lim_{N,K\to\infty} \frac{K}{N}=\gamma\in [0,1]$, we have 
$$
\lim\frac{A^{N,K}_{\infty,\infty}}{K} = \frac{1-\gamma}{(1-\Lambda p)}+\frac{\gamma}{(1-\Lambda p)^3}
$$
in probability.
\end{lemma}

\begin{proof}
We have $\cW^{N,K}_{\infty,\infty}=(\mu N / K^2) A^{N,K}_{\infty,\infty}$, so that 
$$
\frac{A^{N,K}_{\infty,\infty}}K=  \frac{K}{\mu N} \cX^{N,K}_{\infty,\infty} + \frac{N-K}{N}\bar{\ell}_N^K.
$$
Since $\bar{\ell}_N^K\to \frac{1}{1-\Lambda p}$ by \cite[Lemma 4.9]{D} and since $\cX^{N,K}_{\infty,\infty}\to\frac{\mu}{(1-\Lambda p)^{3}}$
by Lemma \ref{W}, the conclusion follows.
\end{proof}

\section{The limit theorems for the first and the second estimators}\label{first theorem}

Since the dominating error term cannot come from the first estimator, we only need to recall the following result,
which is an immediate consequence of \cite[Lemma 6.3]{D}.

\begin{lemma}\label{barell}
Assume $H(q)$ for some $q\geq 1$, in the regime $\frac{K}{t^{2q}} \to 0,$ we have:
$$
\lim_{t,N,K\to\infty}\boldsymbol{1}_{\Omega_{N,K}}\sqrt{K}\mathbb{E}_{\theta}\Big[\Big|\varepsilon_{t}^{N,K}-\mu\bar{\ell}_{N}^{K}\Big|\Big]=0
$$
almost surely.
\end{lemma}

Recall the definition of $\cV_t^{N,K}$, see Section \ref{TRISC}. The main result of this section is
the following limit theorem.

\begin{theorem}\label{VVNK}
Assume $H(q)$ for some $q> 1$. When $(N,K,t)\to (\infty,\infty,\infty)$ 
in the regime where
$\frac{t\sqrt{K}}{N}(\frac{N}{t^q}+\sqrt{\frac{N}{Kt}})+Ne^{-c_{p,\Lambda}K}\to 0$ we have
$$
\omg\frac{t\sqrt{K}}{N}(\cV_{t}^{N,K}-\cV_\infty^{N,K})\stackrel{d}{\longrightarrow}\mathcal{N}
\Big(0,\frac{2\mu^2}{(1-\Lambda p)^2}\Big).
$$
\end{theorem}

We split $\mathcal{V}_{t}^{N,K}-\mathcal{V}_{\infty}^{N,K}= \Delta_{t}^{N,K,1}+\Delta_{t}^{N,K,2}+\Delta_{t}^{N,K,3},$
where 
\begin{align*}
\Delta_{t}^{N,K,1}&=\frac{N}{K}\Big\{\sum_{i=1}^{K}\Big[\frac{Z_{2t}^{i,N}-Z_{t}^{i,N}}t-\varepsilon_{t}^{N,K}\Big]^{2}
-\sum_{i=1}^{K}\Big[\frac{Z_{2t}^{i,N}-Z_{t}^{i,N}}t-\mu\bar{\ell}_{N}^{K}\Big]^{2}\Big\}, \\
\Delta_{t}^{N,K,2}&=\frac{N}{K}\Big\{\sum_{i=1}^{K}\Big[\frac{Z_{2t}^{i,N}-Z_{t}^{i,N}}t-\mu\ell_{N}(i)\Big]^{2}-\frac Kt
\varepsilon_{t}^{N,K}\Big\},\\
\Delta_{t}^{N,K,3}&=2\frac{N}{K}\sum_{i=1}^{K}\Big[\frac{Z_{2t}^{i,N}-Z_{t}^{i,N}}t-\mu\ell_{N}(i)\Big]\Big[\mu\ell_{N}(i)-\mu\bar{\ell}_{N}^{K}\Big].
\end{align*}
We also write 
$\Delta_{t}^{N,K,2}= \Delta_{t}^{N,K,21}+\Delta_{t}^{N,K,22}+\Delta_{t}^{N,K,23},$
where
\begin{align*}
\Delta_{t}^{N,K,21}&=\frac{N}{K}\Big\{\sum_{i=1}^{K}\Big[\frac{Z_{2t}^{i,N}-Z_{t}^{i,N}}t-\frac{\mathbb{E}_{\theta}[Z_{2t}^{i,N}-Z_{t}^{i,N}]}t\Big]^{2}-\frac Kt\varepsilon_{t}^{N,K}\Big\},\\
\Delta_{t}^{N,K,22}&=\frac{N}{K}\sum_{i=1}^{K}\Big\{\frac{\mathbb{E_{\theta}}[Z_{2t}^{i,N}-Z_{t}^{i,N}]}t-
\mu\ell_{N}(i)\Big\}^{2},\\
\Delta_{t}^{N,K,23}&=2\frac{N}{K}\sum_{i=1}^{K}\Big[\frac{Z_{2t}^{i,N}-Z_{t}^{i,N}}t-
\frac{\mathbb{E}_{\theta}(Z_{2t}^{i,N}-Z_{t}^{i,N})}t\Big]\Big[\frac{\mathbb{E}_{\theta}(Z_{2t}^{i,N}-Z_{t}^{i,N})}t-
\mu\ell_{N}(i)\Big].
\end{align*}
We next write  $\Delta_{t}^{N,K,21}= \Delta_{t}^{N,K,211}+\Delta_{t}^{N,K,212}+\Delta_{t}^{N,K,213}$, where
\begin{align*}
\Delta_{t}^{N,K,211}&= \frac{N}{K}\sum_{i=1}^{K}\Big\{\frac{(U_{2t}^{i,N}-U_{t}^{i,N})^{2}}{t^{2}}-
\frac{\mathbb{E_{\theta}}[(U_{2t}^{i,N}-U_{t}^{i,N})^{2}]}{t^{2}}\Big\},\\
 \Delta_{t}^{N,K,212}&= \frac{N}{K}\Big\{\sum_{i=1}^{K}\frac{\mathbb{E_{\theta}}[(U_{2t}^{i,N}-U_{t}^{i,N})^{2}]}{t^{2}}
-\mathbb{E_{\theta}}[\varepsilon_{t}^{N,K}]\frac Kt\Big\},\\
 \Delta_{t}^{N,K,213}&= \frac{N}{K}\Big\{\mathbb{E_{\theta}}[\varepsilon_{t}^{N,K}]\frac Kt- \varepsilon_{t}^{N,K}\frac Kt\Big\}.
\end{align*}
At last, we write $\Delta_{t}^{N,K,3}= \Delta_{t}^{N,K,31}+\Delta_{t}^{N,K,32}$, where
\begin{align*}
\Delta_{t}^{N,K,31}&=2\frac{N}{K}\sum_{i=1}^{K}\Big[\frac{Z_{2t}^{i,N}-Z_{t}^{i,N}}t-
\frac{\mathbb{E_{\theta}}[Z_{2t}^{i,N}-Z_{t}^{i,N}]}t\Big]\Big[\mu\ell_{N}(i)-\mu\bar{\ell}_{N}^{K}\Big],\\
 \Delta_{t}^{N,K,32}&=2\sum_{i=1}^{K}\Big[\frac{\mathbb{E_{\theta}}[Z_{2t}^{i,N}-Z_{t}^{i,N}]}t-\mu\ell_{N}(i)\Big]\Big[\mu\ell_{N}(i)-\mu\bar{\ell}_{N}^{K}\Big].
\end{align*}

We next summarizes some estimates of \cite{D}.

\begin{lemma}\label{12345}
Assume $H(q)$ for some $q\ge 1$. Then, on the set $\Omega_{N,K}$, for all $t\ge 1$, a.s.,

\vip

(i) $\mathbb{E}_{\theta}[|\Delta_{t}^{N,K,1}|]\le C(Nt^{-2q}+NK^{-1}t^{-1}),$
\vip

(ii) $\mathbb{E}_{\theta}[|\Delta_{t}^{N,K,22}|]\le CN t^{-2q},$
\vip

(iii) $\mathbb{E}_{\theta}[|\Delta_{t}^{N,K,23}|]\le CN t^{-q},$
\vip

(iv) $\mathbb{E}_{\theta}[|\Delta_{t}^{N,K,213}|]\le CNK^{-\frac{1}{2}}t^{-\frac{3}{2}},$
\vip

(v) $\mathbb{E}_{\theta}[|\Delta_{t}^{N,K,32}|]\le CN t^{-q}.$
\vip

(vi) $\mathbb{E}_{\theta}[|\Delta_{t}^{N,K,212}|]\le C t^{-1}.$
\vip

(vii) $\mathbb{E}[\boldsymbol{1}_{\Omega_{N,K}\cap \cA_{N}}|\Delta_{t}^{N,K,31}|]\le C N^{1/2}K^{-1/2}t^{-1/2}.$
\end{lemma}

\begin{proof}
Points (i)-(v) can be found in \cite[Lemma 7.2]{D}. 
For point (vi), see \cite[Lemma 7.3]{D}. Finally, for (vii), we first use \cite[Lemma 7.5]{D} which 
tells us that 
$$
\boldsymbol{1}_{\Omega_{N,K}\cap \cA_{N}}\mathbb{E}_\theta[|\Delta_{t}^{N,K,31}|] \leq C 
\frac{N}{K\sqrt t} \Big[\sum_{i=1}^K ( \ell_N(i)-\bar\ell^K_N)^2\Big]^{1/2}=
 C \frac{N}{K\sqrt t} ||\bx_N^K ||_2.
$$
To conclude, we use \cite[Lemmas 4.14,\ 7.5]{D} which implies that 
$\E[\boldsymbol{1}_{\Omega_{N,K}\cap \cA_{N}} ||\bx_N^K ||_2] \leq C K^{1/2}N^{-1/2}$.
\end{proof}

One immediately deduce the following result
(recall that $\Delta_{t}^{N,K,3} = \Delta_{t}^{N,K,31}+\Delta_{t}^{N,K,32}$).

\begin{cor}\label{Delta1}
Assume $H(q)$ for some $q> 1$.
When $(N,K,t)\to (\infty,\infty,\infty)$ in the regime 
$\frac{t\sqrt{K}}{N}(\frac{N}{t^q}+\sqrt{\frac{N}{Kt}})\to 0$, 
$$
\lim \frac{t\sqrt{K}}{N}\mathbb{E}\Big[\omg \Big|\Delta_{t}^{N,K,1}+\Delta_{t}^{N,K,212}+\Delta_{t}^{N,K,213}+\Delta_{t}^{N,K,22}+\Delta_{t}^{N,K,23}+\Delta_{t}^{N,K,3}\Big| \Big]= 0.
$$
\end{cor}

Next, we study the limit behaviour of the intensity $\lambda_t^{i,N}$, recall $(\ref{sssy})$.
 
\begin{lemma}\label{lambar}
Assume $H(q)$ for some $q\geq 1$. Then on the event $\Omega_{N,K},$ we have
$$
\sup_{t\in \R_+}\max_{i=1,...,N}\mathbb{E}_\theta[\lambda_t^{i,N}]\le \frac{\mu}{1-a}
$$
\end{lemma}

\begin{proof}
We directly find, observing that $\Et[Z^{i,N}_t]=\int_0^t \Et[\lambda^{i,N}_s] ds$, that
$$
\max_{i=1,...,N}\mathbb{E}_\theta[\lambda_t^{i,N}]=\mu +\max_{i=1,...,N}\Big\{\sum_{j=1}^NA_N(i,j)\int_0^t\phi(t-s)\mathbb{E}_\theta[\lambda_s^{j,N}] ds\Big\}.
$$
We define $a_N(t)=\sup_{0\le s\le t}\max_{i=1,...,N}\mathbb{E}_\theta[\lambda_s^{i,N}].$ 
By definition, on $\Omega_{N,K}$ we have the bound
$\Lambda\max_{i=1,...,N}\{\sum_{j=1}^NA_N(i,j)\}\le a<1$, whence, since $\Lambda=\int_0^\infty \phi(s)ds$,
$$
a_N(t)\le \mu +a_N(t)a,
$$
which completes the result.
\end{proof}

\begin{lemma}\label{intensity}
Assume $H(q)$ for some $q\geq 1$. On $\Omega_{N,K}$, for all $t\geq 1$, we a.s. have
$$\frac{1}{K}\sum_{i=1}^{K}\mathbb{E}_{\theta}\Big[\Big(\lambda_{t}^{i,N}-\mu\ell_{N}(i)\Big)^{2}\Big]^{\frac{1}{2}}\le \frac{C}{t^{q}}+\frac{C}{\sqrt{N}}.$$
\end{lemma}

\begin{proof}
By definition, we have $\bl_{N}=Q_{N}\boldsymbol{1}_{N}=(I-\Lambda A_N)^{-1}\indiq_N$, so that $\bl_{N}=\indiq_{N}+\Lambda A_{N}\bl_{N}.$ So, writing $\Lambda=\int_{0}^{t}\phi(t-s)ds + \int_t^\infty \phi(s)ds$, we find
$$\lambda_{t}^{i,N}-\mu\ell_{N}(i)=\frac{1}{N}\sum_{j=1}^{N}\theta_{ij}\Big(\int_{0}^{t}\phi(t-s)dZ_{s}^{j}-\mu\ell_N(j)\int_{0}^{t}\phi(t-s)ds\Big)-\frac{\mu}{N}\sum_{j=1}^{N}\theta_{ij}\ell_N(j)\int_{t}^{\infty}\phi(s)ds.$$
This implies, on $\Omega_{N,K},$ that
\begin{align*}
\mathbb{E}_{\theta}\Big[\Big(\lambda_{t}^{i,N}-\mu\ell_{N}(i)\Big)^{2}\Big]^{\frac{1}{2}}
\le& \mathbb{E}_{\theta}\Big[\Big(\frac{1}{N}\sum_{j=1}^{N}\theta_{ij}\Big(\int_{0}^{t}\phi(t-s)dZ_{s}^{j}-\mu\ell_{N}(j)\int_{0}^{t}\phi(t-s)ds\Big)\Big)^{2}\Big]^{\frac{1}{2}}\\
&\qquad+\mu\mathbb{E}_{\theta}\Big[\Big(\frac{1}{N}\sum_{j=1}^{N}\theta_{ij}\ell_N(j)\int_{t}^{\infty}\phi(s)ds\Big)^{2}\Big]^{\frac{1}{2}}.
\end{align*}
Recalling that $M^{i,N}_t=Z^{i,N}_t-\int_0^t \lambda^{i,N}_s ds$ and using that $\ell_N(j)$ is
uniformly bounded on $\Omega_{N,K},$
\begin{align*}
\mathbb{E}_{\theta}\Big[\Big(\lambda_{t}^{i,N}-\mu\ell_{N}(i)\Big)^{2}\Big]^{\frac{1}{2}}\le& \mathbb{E}_{\theta}\Big[\Big(\frac{1}{N}\sum_{j=1}^{N}\theta_{ij}\int_{0}^{t}\phi(t-s)dM_{s}^{j}\Big)^{2}\Big]^{\frac{1}{2}}
\\&+\mathbb{E}_{\theta}\Big[\Big(\frac{1}{N}\sum_{j=1}^{N}\theta_{ij}\int_{0}^{t}\phi(t-s)|\lambda_{s}^{j,N}-\mu\ell(j)_{N}|ds\Big)^{2}\Big]^{\frac{1}{2}}
+C\int_{t}^{\infty}\phi(s)ds.
\end{align*}
By Lemma \ref{lambar} and assumption $H(q)$, using \eqref{ee3},
\begin{align*}
\mathbb{E}_{\theta}\Big[\Big(\frac{1}{N}\sum_{j=1}^{N}\theta_{ij}\int_{0}^{t}\phi(t-s)dM_{s}^{j}\Big)^{2}\Big]^{\frac{1}{2}}
=&\frac{1}{N}\mathbb{E}_{\theta}\Big[\sum_{j=1}^{N}\int_{0}^{t}\Big(\theta_{ij}\phi(t-s)\Big)^2dZ_{s}^{j}\Big]^{\frac{1}{2}}\\
=&\frac{1}{N}\mathbb{E}_{\theta}\Big[\sum_{j=1}^{N}\int_{0}^{t}\Big(\theta_{ij}\phi(t-s)\Big)^2\lambda_{s}^{j}ds\Big]^{\frac{1}{2}}\\
\le& \frac{1}{\sqrt{N}}\Big[\int_{0}^{t}\Big(\phi(t-s)\Big)^2\max_{j=1,...,N}\mathbb{E}_{\theta}[\lambda_{s}^{j}]ds\Big]^{\frac{1}{2}}
\\\le& \frac{C}{\sqrt{N}}\Big(\int_0^t(\phi(t-s))^2ds\Big)^\frac{1}{2}\le \frac{C}{\sqrt{N}}.
\end{align*}
Defining $F^{K,N}_{t}:=\frac{1}{K}\sum_{i=1}^{K}\mathbb{E}_{\theta}[(\lambda_{t}^{i,N}-\mu\ell_{N}(i))^{2}]^{\frac{1}{2}},$
we thus have, by Minkowski's inequality,
\begin{align*}
F^{K,N}_{t}\le&  \frac{1}{KN}\sum_{j=1}^{N}\sum_{i=1}^{K}\theta_{ij}\int_{0}^{t}\phi(t-s)\mathbb{E}_\theta\Big[\Big|\lambda_{s}^{j,N}-\mu\ell_N(j)\Big|^{2}\Big]^{\frac{1}{2}}ds+C\int_{t}^{\infty}\phi(s)ds+\frac{C}{\sqrt{N}}\\
\le& \int_{0}^{t}\frac{N}{K}|||I_{K}A_{N}|||_{1}\phi(t-s)F^{N,N}_{s}ds+C\int_{t}^{\infty}\phi(s)ds+\frac{C}{\sqrt{N}}\\
\le& \int_{0}^{t}\frac{a}{\Lambda}\phi(t-s)F^{N,N}_{s}ds+C\int_{t}^{\infty}\phi(s)ds+\frac{C}{\sqrt{N}},
\end{align*}
because $N |||I_{K}A_{N}|||_{1}=\max_{j=1,\dots,N}\sum_{i=1}^K \theta_{ij}$ and $\frac{N}{K}|||I_KA_{N}|||_{1}\leq a/\Lambda $ on $\Omega_{N,K}.$
\vip
Defining $g_N(t):=C\int_{t}^{\infty}\phi(s)ds+\frac{C}{\sqrt{N}}$, we conclude that on $\Omega_{N,K}$, 
for all $K=1,\dots,N$,
\begin{equation}\label{ggg}
F^{K,N}_t\leq \int_{0}^{t}\frac{a}{\Lambda}\phi(t-s)F^{N,N}_{s}ds+g_N(t).
\end{equation}
Since 
$\int_{0}^{\infty}(1+s^{q})\phi(s)ds<\infty$ from $H(q)$, we have $g_N(t)\leq C(\frac{1}{t^{q}}\land 1)+C N^{-1/2}$.
Moreover, by Lemma \ref{lambar}, $F^{N,N}_{t}\le C$,
so that $\int_0^{t}(\frac{a}{\Lambda})^n\phi^{*n}(t-s)F^{N,N}_{s}ds\leq C a^n \to 0$ as $n\to\infty.$ 
Hence, iterating \eqref{ggg} (using it once with some fixed $K\in \{1,\dots,N\}$ and then always with $K=N$), 
one concludes that on  $\Omega_{N,K}$, 
\begin{align*}
F^{N,N}_t\leq& \sum_{n\ge 1}\int_0^{t}\Big( \frac{a}{\Lambda}\Big)^n\phi^{*n}(t-s)g_N(s)ds+g_N(t)\\
\le& \sum_{n\ge 1}\int_0^{\frac{t}{2}}\Big(\frac{a}{\Lambda}\Big)^n\phi^{*n}(t-s)g_N(s)ds+\sum_{n\ge 1}\int_{\frac{t}{2}}^{t}\Big(\frac{a}{\Lambda}\Big)^n\phi^{*n}(t-s)g_N(s)ds+g_N(t)\\
\le& C \sum_{n\ge 1}\int_{\frac{t}{2}}^t\Big(\frac{a}{\Lambda}\Big)^n\phi^{*n}(s)ds+g_N\Big(\frac{t}{2}\Big)\sum_{n\ge 1}\int_{0}^{\infty}\Big(\frac{a}{\Lambda}\Big)^n\phi^{*n}(s)ds+g_N(t),
\end{align*}
because $g_N$ is non-increasing and bounded. Recalling that 
$\int_0^\infty \phi^{*n}(s)ds =\Lambda^n$ and, see  \cite[Proof of Lemma 15-(ii)]{A}, that 
$$
\int_r^\infty \phi^{\star n}(u)du \leq C n^q\Lambda^n  r^{-q},
$$ 
we conclude that (since $a \in (0,1)$)
\begin{align*}
F^{N,N}_t
\le& C \Big(\frac{t}{2}\Big)^{-q}\sum_{n\ge 1}n^qa^n+ g_N\Big(\frac{t}{2}\Big)\frac{a}{1-a}+g_N(t)
\le \frac{C}{t^q}+\frac{C}{\sqrt{N}}.
\end{align*}
This completes the proof.
\end{proof}

\begin{lemma}\label{UZ4}
For all $t\geq s+1 \geq 1$, on $\Omega_{N,K}$, we have a.s.,
$$
 \max_{i=1,\dots,N}\mathbb{E}_{\theta}[(U^{i,N}_{t}-U_{s}^{i,N})^4]\le (t-s)^{2}\quad \hbox{and}\quad
 \max_{i=1,\dots,N}\mathbb{E}_{\theta}[(Z^{i,N}_{t}-Z_{s}^{i,N})^4]\le C (t-s)^{4}. 
$$
\end{lemma}

\begin{proof}
Recalling \eqref{ee2}, we may write 
$$
U^{i,N}_t =  \sum_{n\geq 0} \intot \phi^{\star n}(t-s) \sum_{j=1}^N A_N^n(i,j)M^{j,N}_sds.
$$ 
Hence, by the Minkowski inequality, we see that on $\Omega_{N,K}$, we have
\begin{align*}
\mathbb{E}_{\theta}[(U^{i,N}_{t}-U_{s}^{i,N})^4]^\frac{1}{4}
\leq& \mathbb{E}_{\theta}[(M^{i,N}_{t}-M_{s}^{i,N})^4]^\frac{1}{4}\\
&+ \sum_{n\geq 1} \int_0^{t} \Big(\phi^{\star n}(t-u)- \phi^{\star n}(s-u)\Big)
\Et\Big[\Big(\sum_{j=1}^N A_N^n(i,j)M^{j,N}_u\Big)^4\Big]^\frac{1}{4} du.
\end{align*}
By \cite[Lemma 16 (iii)]{A}, we already know that, on $\Omega_{N,K}$, 
$\max_{i=1,\dots,N} \Et [ (Z^{i,N}_{t}-Z^{i,N}_s)^2]\le C(t-s)^2.$
For the first term ($n=0$), we use (\ref{ee3}) and Burkholder's inequality:
\begin{align*}
\mathbb{E}_{\theta}[(M^{i,N}_{t}-M_{s}^{i,N})^4]
\le& C\mathbb{E}_{\theta}\Big[(Z^{i,N}_{t}-Z_{s}^{i,N})^2\Big]
\le C (t-s)^2.
\end{align*}
For the second term ($n\ge 1$), we use again (\ref{ee3}) and Burkholder's inequality and we get
\begin{align*}
    \Et\Big[\Big(\sum_{j=1}^N A_N^n(i,j)M^{j,N}_u\Big)^4\Big]\le&  C\Et\Big[\Big([\sum_{j=1}^N A_N^n(i,j)M^{j,N},\sum_{j=1}^N A_N^n(i,j)M^{j,N}]_u\Big)^2\Big] \\
    \le& C\Et\Big[\Big(\sum_{j=1}^N( A_N^n(i,j))^2Z^{j,N}_u\Big)^2\Big]\\
=&C \sum_{j,j'=1}^N (A_N^n(i,j))^2(A_N^n(i,j'))^2 \Et[ Z^{j,N}_u Z^{j',N}_u]\\
\le& C\Big(\sum_{j=1}^N( A_N^n(i,j))^2\Big)^2u^2\\
    \le& C\Big(\sum_{j=1}^NA_N^n(i,j)\Big)^4u^2
    \le C|||A_{N}|||^{4n}_{1}u^2.
\end{align*}
It implies that 
\begin{align*}
    & \sum_{n\geq 1} \int_0^{\infty} \Big(\phi^{\star n}(t-u)- \phi^{\star n}(s-u)\Big)
\Et\Big[\Big(\sum_{i=1}^K\sum_{j=1}^N A_N^n(i,j)M^{j,N}_u\Big)^4\Big]^\frac{1}{4} du\\
\le& C \sum_{n\geq 1} |||A_N|||_1^{n}\intot \sqrt{u} \Big(\phi^{\star n}(t-u)- \phi^{\star n}(s-u)\Big) du\\
\leq& C(t-s)^{1/2} \sum_{n\geq 1} \Lambda^n |||A_N|||_1^{n} \leq (t-s)^{1/2}
\end{align*}
since we showed that $\intot \sqrt{u} (\phi^{\star n}(t-u)- \phi^{\star n}(s-u)) du\le 2 \Lambda^n \sqrt{t-s}$
in the proof of Lemma \ref{Zt} and since, as usual, on $\Omega_{N,K}$, we have
$\Lambda |||A_N|||_1 < a <1$. This completes the first part of this Lemma. 

\vip

For the second part, we recall from Lemma \ref{Zt} (ii) with $K=N$ and $r=\infty$ that we have
$\max_{i=1,\dots,N}\Et[Z^{i,N}_{t}-Z^{i,N}_{s}]\le C(t-s)$ on $\Omega_{N,N}\supset \Omega_{N,K}$, whence
$$
\max_{i=1,\dots,N}\Et[(Z^{i,N}_{t}-Z^{i,N}_{s})^4]\le 8\Big\{\max_{i=1,\dots,N}\Et[Z^{i,N}_{t}-Z^{i,N}_{s}]^4
+\max_{i=1,\dots,N}\Et[(U^{i,N}_{t}-U^{i,N}_{s})^4]\Big\}\le C (t-s)^{4}
$$
as desired.
\end{proof}

\begin{lemma}\label{hard}
Assume $H(q)$ for some $q\geq 1$. As $(t,N,K)\to (\infty,\infty,\infty),$ in the regime $Ne^{-c_{p,\Lambda}K}\to 0,$
$$\indiq_{\Omega_{N,K}} \frac{t\sqrt K}N \Delta^{N,K,211}_t\!=\!
\indiq_{\Omega_{N,K}}\frac{1}{t\sqrt{K}}\sum_{i=1}^{K}\Big\{(U_{2t}^{i,N}-U_{t}^{i,N})^{2}
-\mathbb{E}_{\theta}[(U_{2t}^{i,N}-U_{t}^{i,N})^{2}]\Big\}\stackrel{d}{\longrightarrow}
\mathcal{N}\Big(0, \frac{2\mu^2}{(1-\Lambda p)^2}\Big).
$$ 
\end{lemma}

\begin{proof}
We work on the set $\Omega_{N,K}.$ Recalling \eqref{ee2}, we have 
$$U_{t}^{i,N}=\sum_{n\ge 0}\sum_{j=1}^{N}\int_{0}^{t}\phi^{*n}(t-s)A_{N}^{n}(i,j)M_{s}^{j,N}ds$$
and we write
$$
(U_{2t}^{i,N}-U_{t}^{i,N})^{2}=(M_{2t}^{i,N}-M_{t}^{i,N})^{2}+2T_{t}^{i,N}(M_{2t}^{i,N}-M_{t}^{i,N})+(T_{t}^{i,N})^{2},
$$
where 
$$T_{t}^{i,N}=\sum_{n\ge 1} \sum_{j=1}^{N}\int_{0}^{2t}\phi^{*n}(2t-s)A_{N}^{n}(i,j)M_{s}^{j,N}ds-\sum_{n\ge 1} \sum_{j=1}^{N}\int_{0}^{t}\phi^{*n}(t-s)A_{N}^{n}(i,j)M_{s}^{j,N}ds.$$
We treat these terms one by one and set $\phi(s)=0$ for $s\le 0$ as usual.

\vip

{\bf Step 1.} Here we verify that
$$
\lim \indiq_{\Omega_{N,K}}\frac1{t\sqrt K}\Et\Big[\sum_{i=1}^K \Big|(T^{i,N}_{t})^2- \Et[(T^{i,N}_{t})^2] \Big| \Big]=0.
$$
We will check that for all $i=1,\dots,K$, we have $\Et[(T^{i,N}_{t})^2]\leq C t /N$, which of course
suffices.
\vip
Setting  $\beta_n(s,t,r)=\phi^{\star n}(t-r)-\phi^{\star n}(s-r)$, we may write
\begin{equation}\label{frf}
T^{i,N}_{t}=\sum_{n\geq 1} \int_{0}^{2t}\beta_n(t,2t,u)\sum_{j=1}^{N} A_{N}^{n}(i,j) M^{j,N}_{u} du.
\end{equation}
Hence 
\begin{align*}
\mathbb{E}_{\theta}[(T^{i,N}_{t})^2]= \sum_{m,n\geq 1} \int_{0}^{2t}\int_{0}^{2t}\beta_m(t,2t,u)\beta_n(t,2t,v)\sum_{j,k=1}^{N} A_{N}^{m}(i,j) A_{N}^{n}(i,k)
\mathbb{E}_{\theta}[M^{j,N}_{u} M^{k,N}_{v}] dvdu.
\end{align*}
It is obvious that $\int_{0}^{2t}\beta_n(t,2t,u)\le 2\Lambda^n$ for any $ n\ge 0.$
Using \eqref{ee3} and that $M^{j,N}$ and $M^{k,N}$ are martingales, 
we see that $\Et[M^{j,N}_u M^{k,N}_v]=\indiq_{\{j=k\}}\Et[Z^{j,N}_{u\land v}]\leq C (u\land v)$ (on
$\Omega_{N,K}$, due to Lemma \ref{Zt}-(i) with $r=\infty$), whence
\begin{align*}
\mathbb{E}_{\theta}[(T^{i,N}_{t})^2]=& \sum_{m,n\geq 1} \int_{0}^{2t}\int_{0}^{2t}\beta_m(t,2t,u)\beta_n(t,2t,v)\sum_{j,k=1}^{N} A_{N}^{m}(i,j) A_{N}^{n}(i,k)
\mathbb{E}_{\theta}[M^{j,N}_{u} M^{k,N}_{v}] dvdu\\
\le& Ct \sum_{m,n\geq 1} \int_{0}^{2t}\int_{0}^{2t}\beta_m(t,2t,u)\beta_n(t,2t,v)dvdu\sum_{j=1}^{N}
A_{N}^{m}(i,j) A_{N}^{n}(i,j)\\ 
\le& Ct \sum_{m,n\geq 1} \Lambda^{m+n}\sum_{j=1}^{N} A_{N}^{m}(i,j) A_{N}^{n}(i,j) \\ 
\le& Ct\sum_{j=1}^{N}(Q_{N}(i,j)-\boldsymbol{1}_{\{i=j\}})^{2}\le \frac{Ct}{N} 
\end{align*}
The reason of the last step comes from the fact that 
by \cite[Equation (8)]{A},
on $\Omega_{N}^1$, we have
$\indiq_{\{i=j\}} \leq Q_N(i,j)\leq \indiq_{\{i=j\}} + \Lambda C N^{-1}.$

\vip

{\bf Step 2.} Here we verify that
$$
\lim \indiq_{\Omega_{N,K}}\frac1{t\sqrt K}\Et\Big[\Big| \sum_{i=1}^K \Big(T^{i,N}_{t}(M^{i,N}_{2t}-M^{i,N}_{t})- \Et[
T^{i,N}_{t}(M^{i,N}_{2t}-M^{i,N}_{t})] \Big)\Big| \Big]=0.
$$
Actually, this will follows from the estimate (on $\Omega_{N,K}$)
$$
x:=\mathbb{V}ar_{\theta}\Big[\sum_{i=1}^{K}(T_{t}^{i,N}(M_{2t}^{i,N}-M_{t}^{i,N}))\Big] \leq C \frac{K t^2}N
$$
that we now verify.
We start from
\begin{align*}
x=&\mathbb{E}_{\theta}\Big[\sum_{i,j=1}^{K}\Big(T_{t}^{i,N}(M_{2t}^{i,N}-M_{t}^{i,N})) -
\mathbb{E}_{\theta}[T_{t}^{i,N}(M_{2t}^{i,N}-M_{t}^{i,N}))]\Big) \\
& \hskip3cm 
\Big(T_{t}^{j,N}(M_{2t}^{j,N}-M_{t}^{j,N})) -
\mathbb{E}_{\theta}[T_{t}^{j,N}(M_{2t}^{j,N}-M_{t}^{j,N}))]\Big)\Big].
\end{align*}
Recalling \eqref{frf} and setting $\alpha_{N}(u,t,i,j)=\sum_{n\ge 1}\beta_n(t,2t,u)A_{N}^{n}(i,j),$ 
\begin{align*}
x&\le \sum_{i,j=1}^{K}\int_{0}^{2t}\int_{0}^{2t}\sum_{k,m=1}^{N}|\alpha_{N}(s,t,i,k)\alpha_{N}(u,t,j,m)|\\
&\hskip3cm|\mathbb{C}ov_{\theta}[(M_{2t}^{i,N}-M_{t}^{i,N})M_{s}^{k,N},(M_{2t}^{j,N}-M_{t}^{j,N})M_{u}^{m,N}]|dsdu.
\end{align*}
But
\begin{align*}
\int_{0}^{2t}|\alpha_{N}(s,t,i,k)| ds \leq \sum_{n\geq 1} A_N^n(i,k)  \int_{0}^{2t}|\beta_n(t,2t,s)|ds 
\leq 2 \sum_{n\geq 1} A_N^n(i,k) \Lambda^n \leq 2 (Q_N(i,k)-\indiq_{\{i=k\}})
\end{align*}
which is bounded by $C/N$, as seen a few lines above.

\vip

And by \cite[Lemma 22]{A}, we already know that, for $s$ and $u$ in $[0,2t]$, still on 
$\Omega_{N,K}$,
$$
|\mathbb{C}ov_{\theta}[(M_{2t}^{i,N}-M_{t}^{i,N})M_{s}^{k,N},(M_{2t}^{j,N}-M_{t}^{j,N})M_{u}^{m,N}]|
\le C(\indiq_{\#\{k,i,j,m\}=3}N^{-2}t+\indiq_{\#\{k,i,j,m\}\le 2} t^2).
$$
Hence we conclude that
$$
x\leq \frac C {N^2} \sum_{i,j=1}^K \sum_{k,m=1}^N (\indiq_{\#\{k,i,j,m\}=3}N^{-2}t+\indiq_{\#\{k,i,j,m\}\le 2} t^2)
\leq \frac C {N^2} \Big(N^2 K  \times N^{-2}t + N K \times t^2\Big),
$$
which is bounded by $C K t^2 / N$ as desired.

\vip

{\bf Step 3.}
It only remains to show that
$$
\indiq_{\Omega_{N,K}}\frac{1}{t\sqrt{K}}\Big[\sum_{i=1}^{K}(M_{2t}^{i,N}-M_{t}^{i,N})^{2}-\sum_{i=1}^{K}\mathbb{E}_{\theta}[(M_{2t}^{i,N}-M_{t}^{i,N})^{2}]\Big]
$$ converges to some Gaussian random variable with variance $2\mu^2/(1-\Lambda p)^2$.
By Ito's formula, we have 
$$(M_{2t}^{i,N}-M_{t}^{i,N})^{2}=2\int_{t}^{2t}(M_{s-}^{i,N}-M_{t}^{i,N})dM_{s}^{i,N}+Z_{2t}^{i,N}-Z_{t}^{i,N}.$$
By \cite[Lemma 6.2-(ii)]{D}, we know that $\boldsymbol{1}_{\Omega_{N,K}}\mathbb{E}_{\theta}[|\bar{U}_{t}^{N,K}|^{2}]\le\frac{Ct}{K}$. This directly implies that $\boldsymbol{1}_{\Omega_{N,K}}\frac{1}{t\sqrt{K}}\sum_{i=1}^{K}\{(Z_{2t}^{i,N}-Z_{t}^{i,N})-\mathbb{E}_\theta[Z_{2t}^{i,N}-Z_{t}^{i,N}]\}=\boldsymbol{1}_{\Omega_{N,K}} \frac{\sqrt K}t [\bar{U}_{2t}^{N,K}-
\bar{U}_{t}^{N,K}|\to 0$.

\vip

We introduce $N_{u}^{t,i,N}:=\int_{t}^{t+\sqrt{u}t}(M_{s-}^{i,N}-M_{t}^{i,N})dM_{s}^{i,N}$.
We observe that for $t\geq 0$ fixed, $(N_{u}^{t,i,N})_{u \in [0,1]}$ is a martingale 
in the filtration $\mathcal{F}^{N}_{t+\sqrt{u}t}$. We will prove that, as $(t,N,K)\to (\infty,\infty,\infty)$,
\begin{equation}\label{ababa}
\Big(\frac{1}{t\sqrt{K}} \sum_{i=1}^{K} N_{u}^{t,i,N} \Big)_{u\in [0,1]}
\stackrel{(d)}\to \Big(\frac{\mu}{\sqrt{2}(1-\Lambda p)} B_u \Big)_{u\in[0,1]},
\end{equation}
where $(B_u)_{u\in [0,1]}$ is a Brownian motion. Since we have
$\frac{1}{t\sqrt{K}}\sum_{i=1}^{K}\int_{t}^{2t}(M_{s-}^{i,N}-M_{t}^{i,N})dM_{s}^{i,N}= 
\frac{1}{t\sqrt{K}}\sum_{i=1}^{K}  N_{1}^{t,i,N}$, this will complete the proof.
To prove \eqref{ababa}, by Jacod-Shiryaev \cite[Theorem VIII-3-8]{B},
it suffices to verify that, as 
$(t,N,K)\to (\infty,\infty,\infty)$,
\vip

(a) $[\frac{1}{t\sqrt{K}} \sum_{i=1}^{K} N_{u}^{t,i,N}]_u \to \frac{\mu^2}{2(1-\Lambda p)^2}u$ in probability, for all $u\in [0,1]$ fixed,

\vip

(b) $\sup_{u \in [0,1]} \frac{1}{t\sqrt{K}} \sum_{i=1}^{K} |N_{u}^{t,i,N}-N_{u-}^{t,i,N}| \to 0$ in probability.

\vip

Point (b) is not difficult: using that the Poisson measures are independant in \eqref{sssy}
and that the jumps of $M^{i,N}$ are always equal to $1$, one has, using \eqref{ee3},
\begin{align*}
\frac{1}{t\sqrt{K}}\mathbb{E}\Big[\omg\sup_{u\in[0,1]} \sum_{i=1}^{K} |N_{u}^{t,i,N}-N_{u-}^{t,i,N}| \Big]
\le& \frac{C}{t\sqrt{K}}\mathbb{E}\Big[\omg\sup_{u\in[0,1]}\max_{i=1,...,K}\Big|M_{t+t\sqrt{u}}^{i,N}-M_{t}^{i,N}\Big|\Big]\\
\le& \frac{C}{t\sqrt{K}}\mathbb{E}\Big[\omg\sup_{u\in[0,1]}\Big|\sum_{i=1}^K(M_{t+t\sqrt{u}}^{i,N}-M_{t}^{i,N})^{2}\Big|^\frac{1}{2}\Big]\\
\le& \frac{C}{t\sqrt{K}}\mathbb{E}\Big[\omg\sup_{u\in[0,1]}\sum_{i=1}^K(M_{t+t\sqrt{u}}^{i,N}-M_{t}^{i,N})^{2}\Big]^\frac{1}{2}\\
\le& \frac{C}{t\sqrt{K}}\mathbb{E}\Big[\omg\Big|\sum_{i=1}^K(Z_{2t}^{i,N}-Z_{t}^{i,N})\Big|\Big]^\frac{1}{2}
\le \frac{C}{\sqrt{t}}.
\end{align*}

Concerning point (a), we fix $u$ and write,
\begin{align*}
\Big[\frac{1}{t\sqrt{K}}\sum_{i=1}^{K}N_{.}^{t,i,N}\Big]_{u}=& \frac{1}{t^2 K}
\sum_{i=1}^{K}\int_{t}^{t+\sqrt{u}t}(M_{s-}^{i,N}-M_{t}^{i,N})^{2}dZ_{s}^{i,N} = I^1_{t,N,K,u}+I^2_{t,N,K,u}+I^3_{t,N,K,u},
\end{align*}
where, recalling that $Z^{i,N}_t=M^{i,N}_t+\int_0^t \lambda^{i,N}_s ds$,
\begin{align*}
I^1_{t,N,K,u}=&\frac{1}{t^2 K}\sum_{i=1}^{K}\int_{t}^{t+\sqrt{u}t}(M_{s-}^{i,N}-M_{t}^{i,N})^{2}dM_{s}^{i,N},\\
I^2_{t,N,K,u}=&\frac{1}{t^2 K}\sum_{i=1}^{K}\int_{t}^{t+\sqrt{u}t}(M_{s}^{i,N}-M_{t}^{i,N})^{2}(\lambda_{s}^{i,N}-\mu \ell_{N}(i))ds,\\
I^3_{t,N,K,u}=&\frac{1}{t^2 K} \sum_{i=1}^{K}\mu\ell_{N}(i)\int_{t}^{t+\sqrt{u}t}(M_{s}^{i,N}-M_{t}^{i,N})^{2}ds.
\end{align*}

{\bf Step 3.1.}
Here we verify that $\E[\omg I^1_{t,N,K,u}] \to 0$.  By \eqref{ee3}, we have
\begin{align*}
\Et[&(I^1_{t,N,K,u})^2]=\frac{1}{K^{2}t^{4}}\sum_{i=1}^{K}\mathbb{E}_{\theta}\Big[\int_{t}^{t+\sqrt{u}t}(M_{s-}^{i,N}-M_{t}^{i,N})^{4}dZ_{s}^{i,N}\Big]\\
=&\frac{1}{K^{2}t^{4}}\sum_{i=1}^{K}\mathbb{E}_{\theta}\Big[\int_{t}^{t+\sqrt{u}t}(M_{s}^{i,N}-M_{t}^{i,N})^{4}\lambda_{s}^{i,N}ds\Big]\\
\le& \frac{1}{K^{2}t^{4}}\sum_{i=1}^{K}\int_{t}^{t+\sqrt{u}t}\Big\{\mathbb{E}_{\theta}[(M_{s}^{i,N}-M_{t}^{i,N})^{4}|\lambda_{s}^{i,N}-\mu \ell_{N}(i)|]+\mu\mathbb{E}_{\theta}[(M_{s}^{i,N}-M_{t}^{i,N})^{4}]|\ell_{N}(i)|\Big\}ds.
\end{align*}
Hence, using the Cauchy-Schwarz and Burkholder inequalities,
\begin{align*}
&\Et[(I^1_{t,N,K,u})^2]\\
\le& \frac{1}{K^{2}t^{4}}\sum_{i=1}^{K}\int_{t}^{t+\sqrt{u}t}\!\!\!\Big\{\mathbb{E}_{\theta}[(M_{s}^{i,N}-M_{t}^{i,N})^{8}]^\frac{1}{2}\mathbb{E}_{\theta}[|\lambda_{s}^{i,N}-\mu\ell_{N}(i)|^2]^\frac{1}{2}+\mu\mathbb{E}_{\theta}[(M_{s}^{i,N}-M_{t}^{i,N})^{4}]|\ell_{N}(i)|\Big\}ds\\
\le& \frac{C}{K^{2}t^{4}}\sum_{i=1}^{K}\int_{t}^{t+\sqrt{u}t}\!\!\!\Big\{\mathbb{E}_{\theta}[(Z_{s}^{i,N}-Z_{t}^{i,N})^{4}]^\frac{1}{2}\mathbb{E}_{\theta}[|\lambda_{s}^{i,N}-\mu\ell_{N}(i)|^2]^\frac{1}{2}+\mu\mathbb{E}_{\theta}[(M_{s}^{i,N}-M_{t}^{i,N})^{4}]|\ell_{N}(i)|\Big\}ds.
\end{align*}
By Lemma \ref{UZ4}, we know that on
$\Omega_{N,K}$, for all $s\geq t$, we have
$\max_{i=1,...,N}\mathbb{E}_\theta[(M_{s}^{i,N}-M_{t}^{i,N})^{4}]\le C(t-s)^2$, as well as  
$\max_{i=1,...,N}\mathbb{E}_\theta[(Z_{s}^{i,N}-Z_{t}^{i,N})^{4}]\le C(t-s)^4$.
Hence, recalling that $\ell_N$ is bounded on $\Omega_{N,K}$,
\begin{align*}
\Et[(I^1_{t,N,K,u})^2] \le& \frac{C}{K^{2}t^{2}}\sum_{i=1}^{K}\int_{t}^{t+\sqrt{u}t}
\Big( 1+ \mathbb{E}_{\theta}[|\lambda_{s}^{i,N}-\mu \ell_{N}(i)|^2]^\frac{1}{2}\Big)  ds
\leq \frac C{Kt}\Big(\frac 1{t^q}+ \frac 1 {\sqrt N} \Big),
\end{align*}
which ends the step. We used Lemma \ref{intensity} for the last inequality.

\vip

{\bf Step 3.2.} Similarly, one verifies that, on $\Omega_{N,K}$, 
\begin{align*}
\Et[|I^2_{t,N,K,u}|]
\le &\frac{1}{Kt^{2}}\sum_{i=1}^{K}\int_{t}^{t+\sqrt{u}t}\mathbb{E}_{\theta}[(M_{s}^{i,N}-M_{t}^{i,N})^{4}]^{\frac{1}{2}}\mathbb{E}_{\theta}[|\lambda_{s}^{i,N}-\mu\ell_{N}(i)|^{2}]^{\frac{1}{2}}\\
\le & \frac{C }{Kt} \sum_{i=1}^{K} \int_t^{2t} \mathbb{E}_{\theta}[|\lambda_{s}^{i,N}-\mu\ell_{N}(i)|^{2}]^{\frac{1}{2}}
ds \le \frac{C}{t^{q}}+\frac{C}{\sqrt{N}}.
\end{align*}

{\bf Step 3.3.} Finally, we have to prove that $I^3_{t,N,K,u} \to \mu^2u/[2(1-\Lambda p)^2]$ in probability
as $(t,N,K)\to(\infty,\infty,\infty)$.
Using the It\^o formula and \eqref{ee3}, we write
\begin{align*}
&(M_{s}^{i,N}-M_{t}^{i,N})^{2}\\
=&2\int_{t}^{s}(M_{r-}^{i,N}-M_{t}^{i,N})dM_{r}^{i,N}+Z_{s}^{i,N}-Z_{t}^{i,N}\\
=&2\int_{t}^{s}(M_{r-}^{i,N}-M_{t}^{i,N})dM_{r}^{i,N}+U_{s}^{i,N}-U_{t}^{i,N}+\Et[Z_{s}^{i,N}-Z_{t}^{i,N}-\mu(s-t)\ell_{N}(i)]+\mu(s-t)\ell_{N}(i).
\end{align*}
and $I^3_{t,N,K,u}=I^{3,1}_{t,N,K,u}+I^{3,2}_{t,N,K,u}+I^{3,3}_{t,N,K,u}+I^{3,4}_{t,N,K,u}$, where 
\begin{align*}
I^{3,1}_{t,N,K,u}=&\frac{2}{t^2 K} \sum_{i=1}^{K} \mu\ell_{N}(i)\int_{t}^{t+\sqrt{u}t}
\int_{t}^{s}(M_{r-}^{i,N}-M_{t}^{i,N})dM_r^{i,N}ds,\\
I^{3,2}_{t,N,K,u}=&\frac{1}{t^2 K} \sum_{i=1}^{K} \mu\ell_{N}(i)\int_{t}^{t+\sqrt{u}t}(U_{s}^{i,N}-U_{t}^{i,N}) ds,\\
I^{3,3}_{t,N,K,u}=&\frac{1}{t^2 K} \sum_{i=1}^{K} \mu\ell_{N}(i)\int_{t}^{t+\sqrt{u}t}\Et[Z_{s}^{i,N}-Z_{t}^{i,N}-\mu(s-t)\ell_{N}(i)] ds,\\
I^{3,4}_{t,N,K,u}=&\frac{1}{t^2 K} \sum_{i=1}^{K} \mu^2 (\ell_{N}(i))^2 \times \frac{u t^2}2= \frac{\mu^2 u}{2K}
\sum_{i=1}^{K} (\ell_{N}(i))^2.
\end{align*}
First,
$$
2I^{3,4}_{t,N,K,u}= \mu^2 u (\bar \ell_N^K)^2 + \frac{\mu^2u}K \sum_{i=1}^{K} (\ell_{N}(i)-\bar\ell_N^K)^2
=  \mu^2 u(\bar \ell_N^K)^2 + \frac{\mu^2u}K ||\bx_N^K||_2^2=\mu^2u (\bar \ell_N^K)^2 + \frac{\mu^2u}N \cV^{N,K}_\infty
$$
and we immediately deduce from Lemmas \ref{ellp} and \ref{21} that $I^{3,4}_{t,N,K,u}$ tends to 
$\mu^2u/[2(1-\Lambda p)^2]$. 

\vip

For the second term, we recall 
\eqref{ee2} and we write for $s\ge t,$
$$
U^{i,N}_s-U^{i,N}_t =  \sum_{n\geq 0} \int_0^s (\phi^{\star n}(s-u)-\phi^{\star n}(t-u)) \sum_{j=1}^N A_N^n(i,j)M^{j,N}_udu,
$$ 
so that, by Minkowski's inequality and separating as usual the terms $n=0$ and $n\geq 1$,
\begin{align*}
\Et[|I^{3,2}_{t,N,K,u}|^2]^\frac{1}{2}
\le&\frac{C}{t^2 K}  \int_{t}^{t+\sqrt{u}t}\Et\Big[\Big(\sum_{i=1}^{K}\ell_{N}(i)(U_{s}^{i,N}-U_{t}^{i,N})\Big)^2\Big]^\frac{1}{2} ds\\
 \leq& \frac{C}{t^2 K}  \int_{t}^{t+\sqrt{u}t}\Big\{\Et\Big[\Big(\sum_{i=1}^{K}\ell_{N}(i)(M_{s}^{i,N}-M_{t}^{i,N})\Big)^2\Big]^\frac{1}{2} \\
&\hskip0.3cm+ \sum_{n\geq 1} \int_0^s (\phi^{\star n}(s-r)-\phi^{\star n}(t-r))
\Et\Big[\Big(\sum_{i=1}^K\sum_{j=1}^N \ell_{N}(i) A_N^n(i,j)M^{j,N}_r\Big)^2\Big]^\frac{1}{2} dr\Big\}ds.
\end{align*}
By (\ref{ee3}), we see that on $\Omega_{N,K}$, for all $t\le s\le 2t,$ we have
\begin{align*}
 \Et\Big[\Big(\sum_{i=1}^{K}\ell_{N}(i)(M_{s}^{i,N}-M_{t}^{i,N})\Big)^2\Big]^\frac{1}{2}
=& \Et\Big[\sum_{i=1}^{K}(\ell_{N}(i))^2(Z_{s}^{i,N}-Z_{t}^{i,N})\Big]^\frac{1}{2}\\
=& \Big\{\sum_{i=1}^{K}(\ell_{N}(i))^2\Et\Big[Z_{s}^{i,N}-Z_{t}^{i,N}\Big]\Big\}^\frac{1}{2}\\
\le& C\sqrt{Kt}
\end{align*}
by Lemma \ref{Zt}-(i) with $r=\infty$ and since $\ell_N(i)$ is bounded on $\Omega_{N,K}$.
Next, for $n \geq 1$,
\begin{align*}
    \Et\Big[\Big(\sum_{i=1}^K\sum_{j=1}^N \ell_{N}(i) A_N^n(i,j)M^{j,N}_r\Big)^2\Big]=&\sum_{j=1}^N\Big(\sum_{i=1}^K\ell_{N}(i) A_N^n(i,j)\Big)^2\Et[Z^{j,N}_r]\\
    \le& C\sum_{j=1}^N\Big(\sum_{i=1}^K A_N^n(i,j)\Big)^2\Et[Z^{j,N}_r]\\
     \le& C\sum_{j=1}^N|||I_KA^n_N|||^2_1\Et[Z^{j,N}_r]\\
\le& \frac{CK^2}{N}|||A_N|||^{2n-2}_1r
\end{align*}
because $|||I_KA_N|||_1 \leq C K/N$ on $\Omega_{N,K}$ and by and   Lemma \ref{Zt}-(i) again.
So, for all $u\in [0,1]$ (recall that $\int_0^\infty \phi^{*n}(u)du=\Lambda^n$),
\begin{align*}
\Et[|I^{3,2}_{t,N,K,u}|^2]^\frac{1}{2}
 \leq& \frac{C}{t^2 K}  \int_{t}^{t+\sqrt{u}t}\Big\{\sqrt{Kt}+\sum_{n\geq 1} \int_0^s (\phi^{\star n}(s-r)-\phi^{\star n}(t-r))\frac{K}{\sqrt{N}}|||A_N|||^{n-1}_1\sqrt{r} dr\Big\}ds\\
 \le& \frac{C}{t^2 K} \int_{t}^{t+\sqrt{u}t}\Big\{\sqrt{Kt}+\sum_{n\geq 1} \sqrt{s} \frac{K}{\sqrt{N}}\Lambda^n|||A_N|||^{n-1}_1\Big\}ds\le \frac{C}{\sqrt{Kt}}.
\end{align*}
We finally used that $\Lambda |||A_N|||_1 \leq a<1$ on $\Omega_{N,K}$.

\vip

By Minkowski's inequality and   Lemma \ref{Zt}-(ii) with $r=1$, we find, on $\Omega_{N,K},$
\begin{align*}
\Et[|I^{3,3}_{t,N,K,u}|]
\le&\frac{1}{Kt^{2}}\sum_{i=1}^{K}\int_{t}^{t+\sqrt{u}t}|\ell_N(i)|\Big|\mathbb{E}_{\theta}[Z_{s}^{i,N}-Z_{t}^{i,N}-\mu(s-t)\ell_{N}(i)]\Big|ds\\
\le& \frac{C}{Kt^{2}}\sum_{i=1}^{K}\int_{t}^{t+\sqrt{u}t}\Big|\mathbb{E}_{\theta}[Z_{s}^{i,N}-Z_{t}^{i,N}-\mu(s-t)\ell_{N}(i)]\Big|ds\le \frac{C}{t^q}.
\end{align*}

Finally, we set $\mathbb{N}_u^{t,i,N}:=M_{t+ut}^{i,N}-M_{t}^{i,N}$. 
Then  $\mathbb{N}_u^{t,i,N}$ is a martingale for the filtration  $\mathcal{F}^{N}_{t+ut}$  
with parameter $0\le u\le 1$ and we have, by \eqref{ee3},
$$[\mathbb{N}_.^{t,i,N},\mathbb{N}_.^{t,j,N}]_u=\indiq_{\{i=j\}}(Z^{i,N}_{t+ut}-Z^{i,N}_{t}).$$
Then we have, on $\Omega_{N,K}$, using the change variables $s=t+at$,
\begin{align*}
 \Et[(I^{3,1}_{t,N,K,u})^2]
\! =&\frac{2}{K^{2}t^{2}}\mathbb{E}_{\theta}\Big[\Big(\sum_{i=1}^{K}\mu\ell_{N}(i)\int_{0}^{\sqrt{u}}\int_{t}^{t+at}(M_{r-}^{i,N}-M_{t}^{i,N})dM_{r}^{i,N}da\Big)^{2}\Big]\\
 =&\frac{1}{K^{2}t^{2}}\mathbb{E}_{\theta}\Big[\Big(\sum_{i=1}^{K}\mu\ell_{N}(i)\int_{0}^{\sqrt{u}}\int_{0}^{a}\mathbb{N}_{b-}^{t,i,N}d\mathbb{N}_b^{t,i,N}da\Big)^{2}\Big]\\
    =& \frac{\mu^2}{K^{2}t^{2}}\sum_{i=1}^{K}\sum_{i'=1}^{K}\ell_N(i)\ell_N(i')\int_{0}^{\sqrt{u}}\!\int_{0}^{\sqrt{u}}\!\mathbb{E}_{\theta}\Big[\int_{0}^{a}\mathbb{N}_{b-}^{t.i,N}d\mathbb{N}_{b}^{t,i,N}\int_{0}^{a'}\mathbb{N}_{b'-}^{t,i',N}d\mathbb{N}_{b'}^{t,i',N}\Big]dada'\\
    \le&\frac{C}{K^{2}t^{2}}\sum_{i=1}^{K}\int_{0}^{1}\int_{0}^{1}\mathbb{E}_{\theta}\Big[\Big(\int_{0}^{a\land a'}\mathbb{N}_{b-}^{t.i,N}d\mathbb{N}_{b}^{t,i,N}\Big)^2\Big]dada'\\
    \le&\frac{C}{K^{2}t^{2}}\sum_{i=1}^{K}\mathbb{E}_{\theta}\Big[\int_{0}^{1}(\mathbb{N}_{b-}^{t.i,N})^2dZ_{t+bt}^{i,N}\Big]\\
    \le&\frac{C}{K^{2}t^{2}}\sum_{i=1}^{K}\mathbb{E}_{\theta}\Big[\Big(\sup_{0\le b\le 1}(\mathbb{N}_{b}^{t.i,N})^2\Big)Z_{2t}^{i,N}\Big]
.
\end{align*}
Hence, using the Cauchy-Schwarz and Burkholder inequalities,
$$ 
\Et[(I^{3,1}_{t,N,K,u})^2]\leq
\frac{C}{K^{2}t^2}\sum_{i=1}^{K}\mathbb{E}_{\theta}\Big[\sup_{0\le b\le 1}(\mathbb{N}_{b}^{t.i,N})^4\Big]^\frac{1}{2}
\Et[(Z_{2t}^{i,N})^2]^\frac{1}{2}\le \frac{C}{K^{2}t^2}\sum_{i=1}^{K}
\mathbb{E}_{\theta}[(Z_{2t}^{i,N})^2]\le \frac{C}{K}
$$
by Lemma \ref{UZ4}.
\end{proof}

Finally, we can give the 

\begin{proof}[Proof of Theorem \ref{VVNK}]
Recall that we work when $(N,K,t)\to (\infty,\infty,\infty)$ in the regime
where $\frac{t\sqrt{K}}{N}(\frac{N}{t^q}+\sqrt{\frac{N}{Kt}})+Ne^{-c_{p,\Lambda}K}\to 0$.
In the beginning of the section, we have written
$$
\cV^{N,K}_t-\cV^{N,K}_\infty=\Delta^{N,K,1}_t+\Delta^{N,K,211}_t+\Delta^{N,K,212}_t+\Delta^{N,K,213}_t+\Delta^{N,K,22}_t+
\Delta^{N,K,23}_t+\Delta^{N,K,3}_t.
$$
We have seen in Corollary \ref{Delta1} that the terms $1,212,213,22,23,3$, when multiplied by
$t\sqrt K / N$, all tend to $0$, while Lemma \ref{hard} tells us that
$$
\indiq_{\Omega_{N,K}} \frac{t\sqrt K}N \Delta^{N,K,211}_t\stackrel{d}{\longrightarrow}
\mathcal{N}\Big(0, \frac{2\mu^2}{(1-\Lambda p)^2}\Big),
$$
which completes the proof.
\end{proof}

\section{Some limit theorems for the third estimator}\label{sec5}
The aim of this section is to prove the following theorem.
\begin{theorem}\label{corX}
Assume $H(q)$ for some $q> 3$, $K\le N$ and $\lim_{N,K\to\infty} \frac{K}{N}=\gamma\le 1,$ $\Delta_t= t/(2 \lfloor t^{1-4/(q+1)}\rfloor) \sim t^{4/(q+1)}/2$ (for $t$ large). If  $(N,K,t)\to (\infty,\infty,\infty)$
and $\frac 1{\sqrt K} + \frac NK \sqrt{\frac{\Delta_t}t}+ \frac{N}{t\sqrt K}+Ne^{-c_{p,\Lambda}K} \to 0$,
$$
\lim_{N,K,t\to +\infty}\omg\frac{K}{N}\sqrt{\frac{t}{\Delta_t}}\Big(\mathcal{X}_{\Delta_t,t}^{N,K}-\cX^{N,K}_{\infty,\infty}\Big)\longrightarrow \mathcal{N}\Big(0,\frac{3}{2}\Big(\frac{1-\gamma}{(1-\Lambda p)}+\frac{\gamma}{(1-\Lambda p)^3}\Big)^2\Big)
$$

\end{theorem}

Recall the definition in section \ref{TRISC} and we define:
\begin{align*}
&\mathcal{X}_{\Delta,t}^{N,K}-\mathcal{X}_{\infty,\infty}^{N,K}\\
=& (\mathcal{W}_{\Delta,t}^{N,K}-\mathcal{W}_{\infty,\infty}^{N,K})+\frac{N-K}{K}\Big(\varepsilon_t^{N,K}-\mu\bar{\ell}_N^K\Big)\\
=& D_{\Delta,t}^{N,K,1}+2D_{2\Delta,t}^{N,K,1}+D_{\Delta,t}^{N,K,2}+2D_{2\Delta,t}^{N,K,2}+D_{\Delta,t}^{N,K,3}+2D_{2\Delta,t}^{N,K,3}+D_{\Delta,t}^{N,K,4}+\frac{N-K}{K}\Big(\varepsilon_t^{N,K}-\mu \bar{\ell}_N^K\Big),
\end{align*}
where
\begin{align*}
D_{\Delta,t}^{N,K,1}=&\frac{N}{t}\Big\{\sum_{a=\frac{t}{\Delta}+1}^{\frac{2t}{\Delta}}\Big(\bar{Z}_{a\Delta}^{N,K}-\bar{Z}_{(a-1)\Delta}^{N,K}-\Delta\varepsilon_{t}^{N,K}\Big)^{2}-\sum_{a=\frac{t}{\Delta}+1}^{\frac{2t}{\Delta}}\Big(\bar{Z}_{a\Delta}^{N,K}-\bar{Z}_{(a-1)\Delta}^{N,K}-\Delta \mu\bar{\ell}_{N}^{K}\Big)^{2}\Big\}, \\
D_{\Delta,t}^{N,K,2}=&\frac{N}{t}\Big\{\sum_{a=\frac{t}{\Delta}+1}^{\frac{2t}{\Delta}}\Big(\bar{Z}_{a\Delta}^{N,K}-\bar{Z}_{(a-1)\Delta}^{N,K}-\Delta\mu\bar{\ell}_K^{K}\Big)^{2}\\
&-\sum_{a=\frac{t}{\Delta}+1}^{\frac{2t}{\Delta}}\Big(\bar{Z}_{a\Delta}^{N,K}-\bar{Z}_{(a-1)\Delta}^{N,K}-\mathbb{E}_{\theta}[\bar{Z}_{a\Delta}^{N,K}-\bar{Z}_{(a-1)\Delta}^{N,K}]\Big)^{2}\Big\},\\
D_{\Delta,t}^{N,K,3}=&\frac{N}{t}\Big\{\sum_{a=\frac{t}{\Delta}+1}^{\frac{2t}{\Delta}}\Big(\bar{Z}_{a\Delta}^{N,K}-\bar{Z}_{(a-1)\Delta}^{N,K}-\mathbb{E}_{\theta}[\bar{Z}_{a\Delta}^{N,K}-\bar{Z}_{(a-1)\Delta}^{N,K}]\Big)^{2}\\
&-\mathbb{E}_{\theta}\Big[\sum_{a=\frac{t}{\Delta}+1}^{\frac{2t}{\Delta}}\Big(\bar{Z}_{a\Delta}^{N,K}-\bar{Z}_{(a-1)\Delta}^{N,K}-\mathbb{E}_{\theta}[\bar{Z}_{a\Delta}^{N,K}-\bar{Z}_{(a-1)\Delta}^{N,K}]\Big)^{2}\Big]\Big\},\\
\end{align*}
and finally
\begin{align*}
D_{\Delta,t}^{N,K,4}=&\Big\{\frac{2N}{t}\mathbb{E}_{\theta}\Big[\sum_{a=\frac{t}{2\Delta}+1}^{\frac{t}{\Delta}}\Big(\bar{Z}_{2a\Delta}^{N,K}-\bar{Z}_{2(a-1)\Delta}^{N,K}-\mathbb{E}_{\theta}[\bar{Z}_{2a\Delta}^{N,K}-\bar{Z}_{2(a-1)\Delta}^{N,K}]\Big)^{2}\Big]\\&-\frac{N}{t}\mathbb{E}_{\theta}\Big[\sum_{a=\frac{t}{\Delta}+1}^{\frac{2t}{\Delta}}\Big(\bar{Z}_{a\Delta}^{N,K}-\bar{Z}_{(a-1)\Delta}^{N,K}
-\mathbb{E}_{\theta}[\bar{Z}_{a\Delta}^{N,K}-\bar{Z}_{(a-1)\Delta}^{N,K}]\Big)^{2}\Big]-\mathcal{W}^{N,K}_{\infty,\infty}\Big\}.
\end{align*}

\subsection{Some small terms of the estimator}
First, we are going to prove  the terms $D_{\Delta,t}^{N,K,1},$ $D_{\Delta,t}^{N,K,2},$ $D_{\Delta,t}^{N,K,4}$ are small.
\begin{lemma}
Assume $H(q)$ for some $q\geq 1$. Then for all $t\geq 4$ and all $\Delta \in [1,t/4]$
such that $t/(2\Delta)$ is a positive integer,

(i) $\mathbb{E}[\boldsymbol{1}_{\Omega_{N,K}}|D_{\Delta,t}^{N,K,1}|]\le C\Delta\Big(\frac{N}{t^{2q}}+\frac{N}{Kt}\Big),$

(ii) $\mathbb{E}[\boldsymbol{1}_{\Omega_{N,K}}|D_{\Delta,t}^{N,K,2}|]\le C\frac{N}{t^{q-1}},$

(iii) $\mathbb{E}[\boldsymbol{1}_{\Omega_{N,K}}|D_{\Delta,t}^{N,K,4}|]\le C\frac{Nt}{K\Delta^{1+q}},$

(iv) $\frac{N}{K}\mathbb{E}[\boldsymbol{1}_{\Omega_{N,K}}|\varepsilon_t^{N,K}-\mu\bar{\ell}_N^K|]\le\frac{CN}{Kt^q}+\frac{CN}{K\sqrt{tK}}$.
\end{lemma}

\begin{proof}
The results follow easily from \cite[Lemmas 6.3, 8.2, 8.3 and 8.5]{D}.
\end{proof}
We then deduce the following corollary.
\begin{cor}\label{D124}
Assume $H(q)$ for some $q> 3$. If we choose $\Delta_t= t/(2 \lfloor t^{1-4/(q+1)}\rfloor) \sim t^{4/(q+1)}/2$ (for $t$ large) then, in the regime $\frac 1{\sqrt K} + \frac NK \sqrt{\frac{\Delta_t}t}+ \frac{N}{t\sqrt K}+Ne^{-c_{p,\Lambda}K} \to 0$, we have the convergence in probability
$$
\lim_{N,K,t\to +\infty}\boldsymbol{1}_{\Omega_{N,K}}\frac{K}{N}\sqrt{\frac{t}{\Delta_t}}\Big\{|D_{\Delta_t,t}^{N,K,1}|+|D_{\Delta_t,t}^{N,K,2}|+|D_{\Delta_t,t}^{N,K,4}|+\frac{N}{K}|\varepsilon_t^{N,K}-\mu\bar{\ell}_N^K|\Big\}=0 .
$$
\end{cor}
 
 Next, we consider the term $D_{\Delta,t}^{N,K,3}.$
For $0\le v\le 1,$ we define:
\begin{align}\label{mathbX}
   \mathbb{X}^{N,K}_{\Delta,t,v}:=\sum_{a=[\frac{vt}{\Delta}]+1}^{[\frac{2vt}{\Delta}]}\Big\{(\mathcal{Y}_{(a-1)\Delta,a\Delta}^{N,K})^2-\mathbb{E}_\theta[(\mathcal{Y}_{(a-1)\Delta,a\Delta}^{N,K})^2]\Big\},
\end{align}
where
\begin{align}
\label{Ya1}\mathcal{Y}_{(a-1)\Delta,a\Delta}^{N,K}:=\frac{1}{K}\sum_{j=1}^{N}c_{N}^{K}(j)(M^{j,N}_{a\Delta}-M^{j,N}_{(a-1)\Delta}).
\end{align}
\begin{lemma}\label{D3}
Assume $H(q)$ for some $q\ge 2$, then we have
\begin{align*}
\frac{K}{N}\sqrt{\frac{t}{\Delta}}\mathbb{E}\Big[\indiq_{\Omega_{N,K}}\Big|D_{\Delta,t}^{N,K,3}-\frac{N}{t}\mathbb{X}_{\Delta,t,1}^{N,K}\Big|\Big]
\le \frac{CK}{N\Delta}+\frac{CKt^{\frac{3}{4}}}{\Delta^{1+\frac{q}{2}}\sqrt{N}}+\frac{C\sqrt{K}}{\sqrt{N\Delta}}+\frac{Ct^\frac{3}{4}}{\Delta^{1+\frac{q}{2}}}.
\end{align*}
\end{lemma}

Before the proof, we need some preparations. 
 For $a\in \{t/(2\Delta)+1,...,2t/\Delta\}$, we write that  $U^{i,N}_{a\Delta}- U^{i,N}_{(a-1)\Delta}=\Gamma_{(a-1)\Delta,a\Delta}^{i,N}+X_{(a-1)\Delta,a\Delta}^{i,N}$, where
 \begin{align}
\label{defGamma}
\Gamma_{(a-1)\Delta,a\Delta}^{i,N}:=&\sum_{n\ge 1}\Big\{\int_{0}^{a\Delta}\phi^{*n}(a\Delta-s)\sum_{j=1}^{N}A_{N}^{n}(i,j)[M^{j,N}_{s}-M^{j,N}_{a\Delta}]ds\\
&\qquad -\int_{0}^{(a-1)\Delta}\phi^{*n}((a-1)\Delta-s)\sum_{j=1}^{N}A_{N}^{n}(i,j)[M^{j,N}_{s}-M^{j,N}_{(a-1)\Delta}]ds\Big\},\notag\\
\label{defXa}
X_{(a-1)\Delta,a\Delta}^{i,N}:=&\sum_{n\ge 0}\sum_{j=1}^{N}\Big\{\int_{0}^{a\Delta}\phi^{*n}(a\Delta-s)ds A_{N}^{n}(i,j)M^{j,N}_{a\Delta}\\
&\qquad -\int_{0}^{(a-1)\Delta}\phi^{*n}((a-1)\Delta-s)ds A_{N}^{n}(i,j)M^{j,N}_{(a-1)\Delta}\Big\}.\notag
 \end{align}
And we define the mean values as
\begin{align*}
    \bar{\Gamma}_{(a-1)\Delta,a\Delta}^{N,K}
    :=\frac{1}{K} \sum_{i=1}^K \Gamma_{(a-1)\Delta,a\Delta}^{i,N},
    %\sum_{n\ge 1}\sum_{i=1}^{K}\Big\{\int_{0}^{a\Delta}\phi^{*n}(a\Delta-s)\sum_{j=1}^{N}A_{N}^{n}(i,j)[M^{j,N}_{s}-M^{j,N}_{a\Delta}]ds\\
%-\int_{0}^{(a-1)\Delta}\phi^{*n}((a-1)\Delta-s)\sum_{j=1}^{N}A_{N}^{n}(i,j)[M^{j,N}_{s}-M^{j,N}_{(a-1)\Delta}]ds\Big\},\\
\quad \bar{X}_{(a-1)\Delta,a\Delta}^{N,K}:= \frac{1}{K} \sum_{i=1}^K X_{(a-1)\Delta,a\Delta}^{i,N}.
%\sum_{n\ge 0}\sum_{i=1}^{K}\sum_{j=1}^{N}\Big\{\int_{0}^{a\Delta}\phi^{*n}(a\Delta-s)ds A_{N}^{n}(i,j)M^{j,N}_{a\Delta}\\-\int_{0}^{(a-1)\Delta}\phi^{*n}((a-1)\Delta-s)ds A_{N}^{n}(i,j)M^{j,N}_{(a-1)\Delta}\Big\}
\end{align*}
First, we consider the term $\bar{\Gamma}_{(a-1)\Delta,a\Delta}^{N,K}.$
Then we can write
$$\bar{\Gamma}_{(a-1)\Delta,a\Delta}^{N,K}=C_{a\Delta}^{N,K}+B_{a\Delta}^{N,K}-C_{(a-1)\Delta}^{N,K}-B_{(a-1)\Delta}^{N,K},$$
where
\begin{align}
\label{Cdelta}C_{a\Delta}^{N,K}:=&\frac{1}{K}\sum_{i=1}^{K}\sum_{j=1}^{N}\sum_{n\ge 1}\int^{\Delta}_{0}\phi^{*n}(s)A_{N}^{n}(i,j)(M_{(a\Delta-s)}^{j,N}-M^{j,N}_{a\Delta})ds,\\
\label{Bdelta}B_{a\Delta}^{N,K}:=&\frac{1}{K}\sum_{i=1}^{K}\sum_{j=1}^{N}\sum_{n\ge 1}\int^{a\Delta}_{\Delta}\phi^{*n}(s)A_{N}^{n}(i,j)(M^{j,N}_{(a\Delta-s)}-M^{j,N}_{a\Delta})ds.
\end{align} 

\begin{lemma}\label{B}
 Assume $H(q)$ for some $q\ge 1$. On the set $\Omega_{N,K}$, for $a\in \{t/(2\Delta)+1,...,2t/\Delta\}$, we a.s. 
have
$$\mathbb{E}_{\theta}[(B_{a\Delta}^{N,K})^{2}]\le \frac{C}{N}\Delta^{1-2q}.$$
\end{lemma}

\begin{proof}
We work on  the set $\Omega_{N,K}$.
Recall (\ref{ee3}). By \cite[Lemma 16.(iii)]{A} and Cauchy-Schwarz inequality we have: 
\begin{align*}
\mathbb{E}_{\theta}[(M^{j,N}_{(a\Delta-s)}-M^{j,N}_{a\Delta})(M_{(a\Delta-s')}^{j',N}-M_{a\Delta}^{j',N})]&\le \indiq_{j=j'}\mathbb{E}_{\theta}[Z^{j,N}_{a\Delta}-Z^{j,N}_{(a\Delta-s)}]^\frac{1}{2}\mathbb{E}_{\theta}[Z^{j,N}_{a\Delta}-Z^{j,N}_{(a\Delta-s')}]^\frac{1}{2}\\&\le \indiq_{j=j'}\sqrt{ss'}.
\end{align*}
From Lemma \ref{phisq}, we already have
$
\int_r^\infty \sqrt{u}\phi^{\star n}(u)du \leq C \Lambda^{n}n^qr^{\frac{1}{2}-q}.
$ Hence,
\begin{align*}
&\mathbb{E}_{\theta}[(B_{a\Delta}^{N,K})^{2}]= \frac{1}{K^{2}}\sum_{i,i'=1}^{K}\sum_{j,j'=1}^{N}\sum_{n,m\ge 1}\int^{a\Delta}_{\Delta}\int^{a\Delta}_{\Delta}\phi^{*n}(s)\phi^{*m}(s')A_{N}^{n}(i,j)A_{N}^{m}(i',j')\\
&\mathbb{E}_{\theta}[(M^{j,N}_{(a\Delta-s)}-M^{j,N}_{a\Delta})(M_{(a\Delta-s')}^{j',N}-M_{a\Delta}^{j',N})]dsds'\\
&\le \frac{CN}{K^2}\Big(\sum_{n,m\ge 1} \Lambda^{n+m}|||I_KA_{N}|||^{2}_1|||A_{N}|||^{n+m-2}_{1}\int_{\Delta}^{\infty}\int_{\Delta}^{\infty}\sqrt{ss'}\phi^{*n}(s)\phi^{*m}(s')dsds'\Big)\\
&\le \frac{C}{N}\Big(\sum_{n\ge 1} n^q\Lambda^{n}|||A_{N}|||^{n-1}_{1}\Big)^{2}(\Delta)^{1-2q}\le \frac{C}{N}\Delta^{1-2q}
\end{align*}
\end{proof}

\begin{lemma}\label{C}
 Assume $H(q)$ for some $q\ge 1$. On the set $\Omega_{N,K}$, for $a\in \{t/(2\Delta)+1,...,2t/\Delta\}$, we a.s.
have
$$\mathbb{E}_{\theta}[(C_{a\Delta}^{N,K})^{4}]\le \frac{C}{N^{2}} .$$
\end{lemma}

\begin{proof}
We write 
$$
C_{a\Delta}^{N,K}=\sum_{n\ge 1}\int^{\Delta}_{0}\phi^{*n}(s)O_{s,s,a\Delta}^{N,K,n}
$$
where for $r\ge 0$ and $0\le s\le a\Delta,$
$$
O_{r,s,a\Delta}^{N,K,n}=\frac{1}{K}\sum_{i=1}^{K}\sum_{j=1}^{N}A_{N}^{n}(i,j)(M_{(a\Delta-s)}^{j,N}-M^{j,N}_{a\Delta-s+r}).
$$
When, we fix $s$, $\{O_{r,s,a\Delta}^{N,K,n}\}_{r\ge 0}$ is a family of martingale for the filtration $(\mathcal{F}_{a\Delta-s+r})_{r\ge 0}$.
By (\ref{ee3}), we have $[M^{i,N},M^{j,N}]_{t}=\boldsymbol{1}_{\{i=j\}}Z_{t}^{i,N}$. Hence, 
for $n\ge 1$, on $\Omega_{N,K},$
\begin{align*}
[O_{.,s,a\Delta}^{N,K,n},O_{.,s,a\Delta}^{N,K,n}]_{r}=&\frac{1}{K^{2}}\sum_{j=1}^{N}\Big(\sum_{i=1}^{K}A^{n}_{N}(i,j)\Big)^{2}(Z_{a\Delta-s+r}^{j,N}-Z_{a\Delta-s}^{j,N})
\\\le& \frac{N}{K^{2}}|||I_{K}A^n_{N}|||^{2}_{1}(\bar{Z}_{a\Delta-s+r}^{N}-\bar{Z}_{a\Delta-s}^{N})\\
\le& \frac{N}{K^{2}}|||I_{K}A_{N}|||^{2}_{1}|||A_{N}|||^{2n-2}_{1}(\bar{Z}_{a\Delta-s+r}^{N}-\bar{Z}_{a\Delta-s}^{N})
\\\le& \frac{1}{N}|||A_{N}|||^{2n-2}_{1}(\bar{Z}_{a\Delta-s+r}^{N}-\bar{Z}_{a\Delta-s}^{N})
\end{align*}
Since $|||I_{K}A_{N}|||^{2}_{1}\le \frac{K^2}{N^2}$ on $\Omega_{N,K}$, by Burkholder-Davis-Gundy inequality, we have on $\Omega_{N,K}$,
\begin{align*}
    \mathbb{E}_\theta[(O_{r,s,a\Delta}^{N,K,n})^4]\le 4\mathbb{E}_\theta\Big[\Big([O_{,s,a\Delta}^{N,K,n},O_{,s.a\Delta}^{N,K,n}]_{r}\Big)^2\Big]\le \frac{C|||A_{N}|||^{4n-4}_{1}}{N^2}\mathbb{E}_\theta[(\bar{Z}_{a\Delta-s+r}^{N}-\bar{Z}_{a\Delta-s}^{N})^2].
\end{align*}
From \cite [lemma 16 (iii)]{A}, we already have $\sup_{i=1,\dots,N}\Et[(Z^{i,N}_t-Z^{i,N}_s)^2] \leq C (t-s)^2.$
Recalling the second part of Lemma \ref{phisq}, by Minkowski inequality, we deduce:
\begin{align*}
\mathbb{E}_{\theta}[(C_{a\Delta}^{N,K})^{4}]^\frac{1}{4}\le& \sum_{n\ge 1}\int^{\Delta}_{0}\phi^{*n}(s)\mathbb{E}_\theta[(O_{s,s,a\Delta}^{N,K,n})^4]^\frac{1}{4}ds\\
\le& \frac{1}{\sqrt{N}}\sum_{n\ge 0}|||A_{N}|||^{n}_{1}\int^{\Delta}_{0}\sqrt{s}\phi^{*(n+1)}(s)ds\\
\le& \frac{1}{\sqrt{N}}\sum_{n\ge 0}\sqrt{n+1}\,\Lambda^{n+1}|||A_{N}|||^{n}_{1}ds\le \frac{C}{\sqrt{N}}.
\end{align*}
This completes the proof.
\end{proof}

\begin{lemma}\label{covc}
 Assume $H(q)$ for some $q\ge 1$ and $a, b\in\{\frac{t}{\Delta}+1,...,\frac{2t}{\Delta}\}.$ Then a.s. on the set $\Omega_{N,K}$
$$\mathbb{C}ov_{\theta}[(C_{a\Delta}^{N,K}-C_{(a-1)\Delta}^{N,K})^{2},(C_{b\Delta}^{N,K}-C_{(b-1)\Delta}^{N,K})^{2}]\le \frac{C\sqrt{t}}{N\Delta^{q-1}},\ |a-b|\ge 4.$$
\end{lemma}

\begin{proof}
Because
\begin{align*}
&\mathbb{C}ov_{\theta}[(C_{a\Delta}^{N,K}-C_{(a-1)\Delta}^{N,K})^{2},(C_{b\Delta}^{N,K}-C_{(b-1)\Delta}^{N,K})^{2}]\\
=&\frac{1}{K^{4}}\sum_{i,k,i',k'=1}^{K}\sum_{j,l,j',l'=1}^{N}\sum_{m,n,m',n'\ge 1}
\int_{0}^{\Delta}\int_{0}^{\Delta}\int_{0}^{\Delta}\int_{0}^{\Delta}
\phi^{*n}(s)\phi^{*m}(t)\phi^{*n'}(s')\phi^{*m'}(s') \\
&\hskip2cm \times A_{N}^{n}(i,j)A_{N}^{m}(k,l)A_{N}^{n'}(i',j')A_{N}^{m'}(k',l')\\
&\hskip2cm \times\mathbb{C}ov_{\theta}[(M^{j,N}_{(a\Delta-s)}-M^{j,N}_{a\Delta}-M^{j,N}_{((a-1)\Delta-s)}+M^{j,N}_{(a-1)\Delta})\\
&\hskip4cm \times (M^{j',N}_{(a\Delta-s')}-M^{j',N}_{a\Delta}-M^{j',N}_{((a-1)\Delta-s')}+M^{j',N}_{(a-1)\Delta}),\\
&\hskip3.3cm  (M^{l,N}_{(b\Delta-r)}-M^{l,N}_{b\Delta}-M^{l,N}_{((b-1)\Delta-r)}+M^{l,N}_{(b-1)\Delta})\\
&\hskip4cm \times (M^{l',N}_{(b\Delta-r')}-M^{l',N}_{b\Delta}+(M^{l',N}_{((b-1)\Delta-r')}-M^{l',N}_{(b-1)\Delta})]dsdrds'dr'.
\end{align*}
We define $\zeta^{j,N}_{a\Delta,s}:=M^{j,N}_{(a\Delta-s)}-M^{j,N}_{a\Delta}$ for $0\le s\le \Delta.$ Then we can rewrite that:
\begin{align*}
&\mathbb{C}ov_{\theta}[(M^{j,N}_{(a\Delta-s)}-M^{j,N}_{a\Delta}-M^{j,N}_{((a-1)\Delta-s)}+M^{j,N}_{(a-1)\Delta})\\
&\hskip4cm \times (M^{j',N}_{(a\Delta-s')}-M^{j',N}_{a\Delta}-M^{j',N}_{((a-1)\Delta-s')}+M^{j',N}_{(a-1)\Delta}),\\
&\hskip1cm (M^{l,N}_{(b\Delta-r)}-M^{l,N}_{b\Delta}-M^{l,N}_{((b-1)\Delta-r)}+M^{l,N}_{(b-1)\Delta})\\
&\hskip4cm \times (M^{l',N}_{(b\Delta-r')}-M^{l',N}_{b\Delta}+(M^{l',N}_{((b-1)\Delta-r')}-M^{l',N}_{(b-1)\Delta})]\\
=&\mathbb{C}ov_{\theta}[(\zeta^{j,N}_{a\Delta,s}-\zeta^{j,N}_{(a-1)\Delta,s})(\zeta^{j',N}_{a\Delta,s'}-\zeta^{j',N}_{(a-1)\Delta,s'}),
(\zeta^{l,N}_{b\Delta,r}-\zeta^{l,N}_{(b-1)\Delta,r})(\zeta^{l',N}_{b\Delta,r'}-\zeta^{l',N}_{(b-1)\Delta,r'})].
\end{align*}
Because $0\le s,s',r,r'\le \Delta,$ we have: 
$$
\Et[\zeta^{j,N}_{(a-1)\Delta,s}\zeta^{j',N}_{a\Delta,s'}]=\Et[\zeta^{j,N}_{a\Delta,s}\zeta^{j',N}_{(a-1)\Delta,s'}]=\Et[\zeta^{l,N}_{(b-1)\Delta,r}\zeta^{l',N}_{b\Delta,r'}]=\Et[\zeta^{l,N}_{b\Delta,r}\zeta^{l',N}_{(b-1)\Delta,r'}]=0
$$ 
Without loss and generality, we assume $a-b\ge 4$ and $s\le s',$ first we notice that
\begin{align*}
    &\mathbb{C}ov_{\theta}[\zeta^{j,N}_{a\Delta,s}\zeta^{j',N}_{(a-1)\Delta,s'},(\zeta^{l,N}_{b\Delta,r}-\zeta^{l,N}_{(b-1)\Delta,r})(\zeta^{l',N}_{b\Delta,r'}-\zeta^{l',N}_{(b-1)\Delta,r'})]\\
    =&\Et\Big[\zeta^{j,N}_{a\Delta,s}|\mathcal{F}_{(a-1)\Delta}\Big]\zeta^{j',N}_{(a-1)\Delta,s'}\\
    \times&\Big((\zeta^{l,N}_{b\Delta,r}-\zeta^{l,N}_{(b-1)\Delta,r})(\zeta^{l',N}_{b\Delta,r'}-\zeta^{l',N}_{(b-1)\Delta,r'})-\Et\Big[(\zeta^{l,N}_{b\Delta,r}-\zeta^{l,N}_{(b-1)\Delta,r})(\zeta^{l',N}_{b\Delta,r'}-\zeta^{l',N}_{(b-1)\Delta,r'})\Big]\Big)=0.
\end{align*}
And by the same reason we conclude that
$$
\mathbb{C}ov_{\theta}[\zeta^{j,N}_{(a-1)\Delta,s}\zeta^{j',N}_{a\Delta,s'},(\zeta^{l,N}_{b\Delta,r}-\zeta^{l,N}_{(b-1)\Delta,r})(\zeta^{l',N}_{b\Delta,r'}-\zeta^{l',N}_{(b-1)\Delta,r'})]=0.
$$
When $j\ne j'$, the covariance vanishes because 
$$\Et[(\zeta^{j,N}_{a\Delta,s}-\zeta^{j,N}_{(a-1)\Delta,s})(\zeta^{j',N}_{a\Delta,s'}-\zeta^{j',N}_{(a-1)\Delta,s'})|\mathcal{F}_{b\Delta}]=0.$$ 
Next we assume that $j=j',$ we have 
\begin{align*}
K:=& \mathbb{C}ov_{\theta}[(\zeta^{j,N}_{a\Delta,s}-\zeta^{j,N}_{(a-1)\Delta,s})(\zeta^{j',N}_{a\Delta,s'}-\zeta^{j',N}_{(a-1)\Delta,s'}),
(\zeta^{l,N}_{b\Delta,r}-\zeta^{l,N}_{(b-1)\Delta,r})(\zeta^{l',N}_{b\Delta,r'}-\zeta^{l',N}_{(b-1)\Delta,r'})]\\
=& \mathbb{C}ov_{\theta}[(\zeta^{j,N}_{a\Delta,s}\zeta^{j',N}_{a\Delta,s'}+\zeta^{j,N}_{(a-1)\Delta,s}\zeta^{j',N}_{(a-1)\Delta,s'}),
(\zeta^{l,N}_{b\Delta,r}-\zeta^{l,N}_{(b-1)\Delta,r})(\zeta^{l',N}_{b\Delta,r'}-\zeta^{l',N}_{(b-1)\Delta,r'})].\\
\end{align*}
Since $\Et[\zeta^{j,N}_{a\Delta,s}\zeta^{j,N}_{a\Delta,s'}|\mathcal{F}_{a\Delta-s'}]=\Et[(M^{j,N}_{a\Delta})^2-(M^{j,N}_{a\Delta-s})^2|\mathcal{F}_{a\Delta-s'}]$, writung as usual 
$(M^{j,N}_{a\Delta})^2-(M^{j,N}_{a\Delta-s})^2=2\int^{a\Delta}_{a\Delta-s}M^{j,N}_{\tau-}
dM^{j,N}_\tau +Z^{j,N}_{a\Delta}-Z^{j,N}_{a\Delta-s}$, we find that
\begin{align*}
    K=&\mathbb{C}ov_{\theta}[Z^{j,N}_{a\Delta}-Z^{j,N}_{a\Delta-s}+Z^{j,N}_{(a-1)\Delta}-Z^{j,N}_{(a-1)\Delta-s},
(\zeta^{l,N}_{b\Delta,r}-\zeta^{l,N}_{(b-1)\Delta,r})(\zeta^{l',N}_{b\Delta,r'}-\zeta^{l',N}_{(b-1)\Delta,r'})]\\
=&\mathbb{C}ov_{\theta}[U^{j,N}_{a\Delta}-U^{j,N}_{a\Delta-s}+U^{j,N}_{(a-1)\Delta}-U^{j,N}_{(a-1)\Delta-s},
(\zeta^{l,N}_{b\Delta,r}-\zeta^{l,N}_{(b-1)\Delta,r})(\zeta^{l',N}_{b\Delta,r'}-\zeta^{l',N}_{(b-1)\Delta,r'})].\\
\end{align*}
Recalling that $\beta_n(x,z,r)=\phi^{\star n}(z-r)-\phi^{\star n}(x-r)$,
we can write 
$$
U^{j,N}_{a\Delta}-U^{j,N}_{a\Delta-s}=\sum_{n\geq 0} \int_0^{a\Delta} \beta_n(a\Delta-s,a\Delta,r) \sum_{j=1}^N A_N^n(i,j) M^{j,N}_rdr=R^{j,N}_{a\Delta,a\Delta-s}+T^{j,N}_{a\Delta,a\Delta-s},
$$
where
\begin{align*}
R^{j,N}_{a\Delta,a\Delta-s}=&  \sum_{n\geq 0} \int_{(a-1)\Delta-s}^{a\Delta} \beta_n(x,z,r) \sum_{j=1}^N A_N^n(i,j) 
(M^{j,N}_r  -M^{j,N}_{(a-1)\Delta-s})dr,\\
T^{j,N}_{a\Delta,a\Delta-s}=&  \sum_{n\geq 0} \Big(\int_{(a-1)\Delta-s}^{a\Delta} \beta_n(x,z,r) dr \Big)\sum_{j=1}^N A_N^n(i,j) 
M^{j,N}_{(a-1)\Delta-s} \\ 
&+ \sum_{n\geq 0} \int_0^{(a-1)\Delta-s}\beta_n(x,z,r) \sum_{j=1}^N A_N^n(i,j) M^{j,N}_rdr.
\end{align*}
The conditional expectation of $R^{j,N}_{a\Delta,a\Delta-s}$ knowing $\mathcal{F}_{b\Delta}$ vanishes.
Hence
\begin{align*}
    K=\mathbb{C}ov_{\theta}[T^{j,N}_{a\Delta,a\Delta-s}+T^{j,N}_{(a-1)\Delta,(a-1)\Delta-s},
(\zeta^{l,N}_{b\Delta,r}-\zeta^{l,N}_{(b-1)\Delta,r})(\zeta^{l',N}_{b\Delta,r'}-\zeta^{l',N}_{(b-1)\Delta,r'})].\\
\end{align*}
Recall \cite[proof of Lemma 30, Step 1]{A} (and notice that $T^{j,N}_{a\Delta,a\Delta-s}$ 
is exactly the $X^{j,N}_{a\Delta-s,a\Delta}$ in \cite{A}), we have 
$\sup_{i=1,\dots,N} \Et[(T^{j,N}_{a\Delta,a\Delta-s})^4] \leq C t^2 \Delta^{-4q}.$ Since $r\le \Delta$,
\begin{align*}
    \Et[(\zeta^{l,N}_{b\Delta,r}-\zeta^{l,N}_{(b-1)\Delta,r})^4]^\frac{1}{4}&\le \Et[(M^{l,N}_{(b\Delta-r)}-M^{l,N}_{b\Delta})^4]^\frac{1}{4}+\Et[(M^{l,N}_{((b-1)\Delta-r)}-M^{l,N}_{(b-1)\Delta})^4]^\frac{1}{4}\\
    &\le C\sqrt{\Delta},
\end{align*}
whence 
\begin{align*}
|K|\le&\{\Et[(T^{j,N}_{a\Delta,a\Delta-s})^2]^\frac{1}{2}+\Et[(T^{j,N}_{(a-1)\Delta,(a-1)\Delta-s})^2]^\frac{1}{2}\}\\
&\qquad \times 
\Et[(\zeta^{l,N}_{b\Delta,r}-\zeta^{l,N}_{(b-1)\Delta,r})^4]^\frac{1}{4}\Et[(\zeta^{l',N}_{b\Delta,r'}-\zeta^{l',N}_{(b-1)\Delta,r'})^4]^\frac{1}{4}\\
\le&  C t^{1/2} \Delta^{-q} \Delta.
\end{align*}
Moreover, by symmetry, we conclude that, when $|a-b|\ge 4$,
\begin{align*}
&\mathbb{C}ov_{\theta}[(\zeta^{j,N}_{a\Delta,s}-\zeta^{j,N}_{(a-1)\Delta,s})(\zeta^{j',N}_{a\Delta,s'}-\zeta^{j',N}_{(a-1)\Delta,s'}),(\zeta^{l,N}_{b\Delta,r}-\zeta^{l,N}_{(b-1)\Delta,r})(\zeta^{l',N}_{b\Delta,r'}-\zeta^{l',N}_{(b-1)\Delta,r'})]\\
\le& C\Big(\boldsymbol{1}_{l=l'}+\boldsymbol{1}_{j=j'}\Big)\sqrt{t}\Delta^{1-q}.
\end{align*}
Recall the definition of $\Omega_{N,K}$, we have $|||I_KA^n_{N}|||_1\le |||I_KA_{N}|||_1|||A_{N}|||^{n-1}_1\le \frac{CK}{N}|||A_{N}|||^{n-1}_1$, which gives us 
\begin{align*}
   &\mathbb{C}ov_{\theta}[(C_{a\Delta}^{N,K}-C_{(a-1)\Delta}^{N,K})^{2},(C_{b\Delta}^{N,K}-C_{(b-1)\Delta}^{N,K})^{2}]\\
   \le& \frac{C\sqrt{t}\Delta^{1-q}}{K^{4}}\sum_{i,k,i',k'=1}^{K}\sum_{j,l,j',l'=1}^{N}\sum_{m,n,m',n'\ge 1}\\
&\qquad\Lambda^{n+m+n'+m'}
A_{N}^{n}(i,j)A_{N}^{m}(k,l)A_{N}^{n'}(i',j')A_{N}^{m'}(k',l')\Big(\boldsymbol{1}_{l=l'}+\boldsymbol{1}_{j=j'}\Big)\\
\le& \frac{C\sqrt{t}\Delta^{1-q}}{K^{4}}N^3\sum_{m,n,m',n'\ge 1}
\Lambda^{n+m+n'+m'}
|||I_KA^n_{N}|||_1|||I_KA^m_{N}|||_1|||I_KA^{n'}_{N}|||_1|||I_KA^{m'}_{N}|||_1
\\
   \le& \sum_{n,m,n',m'\ge 1}\frac{1}{K^{4}}N^{3}\Big(\frac{K}{N}\Big)^{4} \Lambda^4 (\Lambda|||A_{N}|||_1)^{n+m+n'+m'-4}
    \frac{C\sqrt{t}}{\Delta^{q-1}}
    \le \frac{C\sqrt{t}}{N\Delta^{q-1}}.
\end{align*}
The proof is finished.
\end{proof}

\begin{lemma}\label{Gamma}
 Assume $H(q)$ for some $q\ge 1$. Then a.s. on the set $\Omega_{N,K}$
\begin{align*}
&\frac{K}{N}\sqrt{\frac{t}{\Delta}}\frac{N}{t}\mathbb{E}_{\theta}\Big[\Big|\sum_{a=\frac{t}{\Delta}+1}^{\frac{2t}{\Delta}}\Big\{(\bar{\Gamma}^{N,K}_{(a-1)\Delta,a\Delta})^{2}-\mathbb{E}_{\theta}[(\bar{\Gamma}^{N,K}_{(a-1)\Delta,a\Delta})^{2}]\Big\}\Big|\Big]
\le \frac{CK\sqrt{t}}{N\Delta^{(q+1)}}+\frac{CK}{N\Delta}+\frac{CKt^{\frac{3}{4}}}{\Delta^{(1+\frac{q}{2})}\sqrt{N}}.
\end{align*}
\end{lemma}
\begin{proof}
We start from 
$$(\bar{\Gamma}^{N,K}_{(a-1)\Delta,a\Delta})^{2}= (C_{a\Delta}^{N,K}-C_{(a-1)\Delta}^{N,K})^{2}+2(B_{a\Delta}^{N,K}-B_{(a-1)\Delta}^{N,K})(C_{a\Delta}^{N,K}-C_{(a-1)\Delta}^{N,K})+(B_{a\Delta}^{N,K}-B_{(a-1)\Delta}^{N,K})^{2}.$$
By Lemma \ref{B}, we have 
\begin{align*}
&\mathbb{E}_{\theta}\Big[\Big|\sum_{a=\frac{t}{\Delta}+1}^{\frac{2t}{\Delta}}(B_{a\Delta}^{N,K}-B_{(a-1)\Delta}^{N,K})^{2}-\mathbb{E}_{\theta}\Big[\sum_{a=\frac{t}{\Delta}+1}^{\frac{2t}{\Delta}}(B_{a\Delta}^{N,K}-B_{(a-1)\Delta}^{N,K})^{2}\Big]\Big|\Big]\\
\le& 2\mathbb{E}_{\theta}\Big[\sum_{a=\frac{t}{\Delta}+1}^{\frac{2t}{\Delta}}(B_{a\Delta}^{N,K}-B_{(a-1)\Delta}^{N,K})^{2}\Big]\le \frac{Ct}{N\Delta^{2q}}.
\end{align*}
And by lemma \ref{C} and lemma \ref{covc}
\begin{align*}
 &\mathbb{E}_{\theta}\Big[\Big|\sum_{a=\frac{t}{\Delta}+1}^{\frac{2t}{\Delta}}(C_{a\Delta}^{N,K}-C_{(a-1)\Delta}^{N,K})^{2}-\mathbb{E}_{\theta}[\sum_{a=\frac{t}{\Delta}+1}^{\frac{2t}{\Delta}}(C_{a\Delta}^{N,K}-C_{(a-1)\Delta}^{N,K})^{2}]\Big|^2\Big]\\
=&\mathbb{V}ar_{\theta}\Big[\sum_{a=\frac{vt}{\Delta}+1}^{\frac{2vt}{\Delta}}(C_{a\Delta}^{N,K}-C_{(a-1)\Delta}^{N,K})^{2}\Big]\\
\le& \sum_{\substack{t/\Delta+1\le a,b \le 2t/\Delta\\ |a-b|\le 3}}
\mathbb{E}_\theta\Big[(C_{a\Delta}^{N,K}-C_{(a-1)\Delta}^{N,K})^{4}\Big]^\frac{1}{2}\mathbb{E}_\theta\Big[(C_{b\Delta}^{N,K}-C_{(b-1)\Delta}^{N,K})^{4}\Big]^\frac{1}{2}\\
& + \sum_{\substack{t/\Delta+1\le a,b \le 2t/\Delta\\ |a-b|\ge 4}}
\mathbb{C}ov_\theta\Big[(C_{a\Delta}^{N,K}-C_{(a-1)\Delta}^{N,K})^{2},(C_{b\Delta}^{N,K}-C_{(b-1)\Delta}^{N,K})^{2}\Big]\\
\le& C\Big[\frac{t}{\Delta}\frac{1}{N^{2}}+\frac{t^{\frac{5}{2}}}{\Delta^{q+1}N}\Big].
\end{align*}
Moreover we have, by Lemmas \ref{B} and \ref{C},
\begin{align*}
   &\mathbb{E}_{\theta}\Big[\Big|\sum_{a=\frac{t}{\Delta}+1}^{\frac{2t}{\Delta}}(B_{a\Delta}^{N,K}-B_{(a-1)\Delta}^{N,K})(C_{a\Delta}^{N,K}-C_{(a-1)\Delta}^{N,K})
   -\mathbb{E}_{\theta}[(B_{a\Delta}^{N,K}-B_{(a-1)\Delta}^{N,K})(C_{a\Delta}^{N,K}-C_{(a-1)\Delta}^{N,K})]\Big|\Big]\\
\le& 4\sum_{a=\frac{t}{\Delta}+1}^{\frac{2t}{\Delta}}\Big\{\mathbb{E}_{\theta}\Big[\Big|B_{a\Delta}^{N,K}C_{a\Delta}^{N,K}\Big|\Big]+\mathbb{E}_{\theta}\Big[\Big|B_{(a-1)\Delta}^{N,K}C_{a\Delta}^{N,K}\Big|\Big]+\mathbb{E}_{\theta}\Big[\Big|B_{a\Delta}^{N,K}C_{(a-1)\Delta}^{N,K}\Big|\Big]\\
& \hskip5cm +\mathbb{E}_{\theta}\Big[\Big|B_{(a-1)\Delta}^{N,K}C_{(a-1)\Delta}^{N,K}\Big|\Big]\Big\}\\
   \le& 4\sum_{a=\frac{t}{\Delta}+1}^{\frac{2t}{\Delta}}\Big\{\mathbb{E}_{\theta}\Big[\Big|B_{a\Delta}^{N,K}\Big|^{2}\Big]^{\frac{1}{2}}+\Big|B_{(a-1)\Delta}^{N,K}\Big|^{2}\Big]^{\frac{1}{2}}\Big\}\Big\{\mathbb{E}_{\theta}\Big[\Big|C_{a\Delta}^{N,K}\Big|^{2}\Big]^{\frac{1}{2}}+\mathbb{E}_{\theta}\Big[\Big|C_{(a-1)\Delta}^{N,K}\Big|^{2}\Big]^{\frac{1}{2}}\Big\}\\
\le& \frac{Ct}{\Delta^{q+\frac{1}{2}}}\frac{1}{N}
\end{align*}
Overall, we have:
\begin{align*}
    &\frac{K}{N}\sqrt{\frac{t}{\Delta}}\frac{N}{t}\mathbb{E}_{\theta}\Big[\Big|\sum_{a=\frac{t}{\Delta}+1}^{\frac{2t}{\Delta}}\Big\{(\bar{\Gamma}^{N,K}_{(a-1)\Delta,a\Delta})^{2}-\mathbb{E}_{\theta}[(\bar{\Gamma}^{N,K}_{(a-1)\Delta,a\Delta})^{2}]\Big\}\Big|\Big]\\
\le& \frac{K}{\sqrt{t\Delta}}\Big\{\mathbb{E}_{\theta}\Big[\Big|\sum_{a=\frac{t}{\Delta}+1}^{\frac{2t}{\Delta}}(B_{a\Delta}^{N,K}-B_{(a-1)\Delta}^{N,K})^{2}-\mathbb{E}_{\theta}\Big[\sum_{a=\frac{t}{\Delta}+1}^{\frac{2t}{\Delta}}(B_{a\Delta}^{N,K}-B_{(a-1)\Delta}^{N,K})^{2}\Big]\Big|\Big]\\
    &+\mathbb{E}_{\theta}\Big[\Big|\sum_{a=\frac{t}{\Delta}+1}^{\frac{2t}{\Delta}}(C_{a\Delta}^{N,K}-C_{(a-1)\Delta}^{N,K})^{2}-\mathbb{E}_{\theta}\Big[\sum_{a=\frac{t}{\Delta}+1}^{\frac{2t}{\Delta}}(C_{a\Delta}^{N,K}-C_{(a-1)\Delta}^{N,K})^{2}\Big]\Big|\Big]\\
    &+2\mathbb{E}_{\theta}\Big[\Big|\sum_{a=\frac{t}{\Delta}+1}^{\frac{2t}{\Delta}}(B_{a\Delta}^{N,K}-B_{(a-1)\Delta}^{N,K})(C_{a\Delta}^{N,K}-C_{(a-1)\Delta}^{N,K})\\
&    \hskip5cm -\mathbb{E}_{\theta}\Big[(B_{a\Delta}^{N,K}-B_{(a-1)\Delta}^{N,K})(C_{a\Delta}^{N,K}-C_{(a-1)\Delta}^{N,K})\Big]\Big|\Big]\Big\}\\
    \le&  \frac{K}{\sqrt{t\Delta}}\Big\{\frac{1}{N}\frac{Ct}{\Delta^{2q}}+\Big[\frac{t}{\Delta}\frac{1}{N^{2}}+\frac{t^{\frac{5}{2}}}{\Delta^{q+1}N}\Big]^{\frac{1}{2}}+\frac{1}{N}\frac{Ct}{\Delta^{q+\frac{1}{2}}}\Big\}
    \le C\Big\{\frac{K}{N\Delta}+\frac{Kt^{\frac{3}{4}}}{\Delta^{(1+\frac{q}{2})}\sqrt{N}}+\frac{K\sqrt{t}}{N\Delta^{(q+1)}}\Big\}
\end{align*}
The proof is finished.
\end{proof}

Recall that $c^K_{N}(j):=\sum_{i=1}^{K}Q_{N}(i,j).$ Next, we will prove that
$\bar{X}_{(a-1)\Delta,a\Delta}^{N,K}$ is close to $\mathcal{Y}_{(a-1)\Delta,a\Delta}^{N,K},$

\begin{lemma}\label{X}
Assume $H(q)$ for some $q\ge 2$.  Then a.s.\ on the set $\Omega_{N,K}$, one has
\begin{gather*}
\mathbb{E}_{\theta}[(\mathcal{Y}_{(a-1)\Delta,a\Delta}^{N,K}-\bar{X}_{(a-1)\Delta,a\Delta}^{N,K})^{2}] \le \frac{C}{N}\Big[\frac{1}{(a\Delta)^{2q-1}}+\frac{1}{((a-1)\Delta)^{2q-1}}\Big]\ ,\\
\frac{K}{\sqrt{t\Delta}}\mathbb{E}_{\theta}\Big[\sum_{a=\frac{t}{\Delta}+1}^{\frac{2t}{\Delta}}\Big|\mathcal{Y}_{(a-1)\Delta,a\Delta}^{N,K}-\bar{X}_{(a-1)\Delta,a\Delta}^{N,K}\Big|^{2}\Big]\le \frac{CK}{N\sqrt{t}\Delta^{2q-\frac{1}{2}}}. 
\end{gather*}
\end{lemma}

\begin{proof}
By Lemma \ref{phi}, 
on the set $\Omega_{N,K}$, we have:
\begin{align*}
&\boldsymbol{1}_{\Omega_{N,K}}\mathbb{E}_{\theta}[(\mathcal{Y}_{(a-1)\Delta,a\Delta}^{N,K}-\bar{X}_{(a-1)\Delta,a\Delta}^{N,K})^{2}]\\
\le& 
\frac{1}{K^2}\Et\Big[\Big|\sum_{j=1}^{N}\Big\{\sum_{n\ge 0}\sum_{i=1}^{K}\Big(\int_{0}^{a\Delta}\phi^{*n}(a\Delta-s)ds-\Lambda^{n}\Big) A_{N}^{n}(i,j)\Big\}M^{j,N}_{a\Delta}\Big|^2\Big]\\
&+\frac{1}{K^2}\Et\Big[\Big|\sum_{j=1}^{N}\Big\{\sum_{n\ge 0}\sum_{i=1}^{K}\Big(\int_{0}^{(a-1)\Delta}\phi^{*n}((a-1)\Delta-s)ds-\Lambda^{n}\Big) A_{N}^{n}(i,j)\Big\}M^{j,N}_{(a-1)\Delta}\Big|^2\Big]\\
=& 
\frac{1}{K^2}\sum_{j=1}^{N}\Big\{\sum_{n\ge 0}\sum_{i=1}^{K}\Big(\int_{0}^{a\Delta}\phi^{*n}(a\Delta-s)ds-\Lambda^{n}\Big) A_{N}^{n}(i,j)\Big\}^2\Et[Z^{j,N}_{a\Delta}]\\
&+\frac{1}{K^2}\sum_{j=1}^{N}\Big\{\sum_{n\ge 0}\sum_{i=1}^{K}\Big(\int_{0}^{(a-1)\Delta}\phi^{*n}((a-1)\Delta-s)ds-\Lambda^{n}\Big) A_{N}^{n}(i,j)\Big\}^2\Et[Z^{j,N}_{(a-1)\Delta}]\\
\le& 
\frac{1}{K^2}\sum_{j=1}^{N}\Big\{\sum_{n\ge 1}n\int_{(a\Delta)/n}^{\infty}\phi(s)ds\Lambda^{n-1}|||I_{K}A^{n}_{N}|||_1\Big)\Big\}^2\Et[Z^{j,N}_{a\Delta}]\\
&+\frac{1}{K^2}\sum_{j=1}^{N}\Big\{\sum_{n\ge 1}n\int_{(a-1)\Delta/n}^{\infty}\phi(s)ds\Lambda^{n-1}|||I_{K}A^{n}_{N}|||_1\Big\}^2\Et[Z^{j,N}_{(a-1)\Delta}]\\
&\le \frac{1}{N^2}\sum_{j=1}^{N}\Big\{\sum_{n\ge 1}n\int_{(a\Delta)/n}^{\infty}\phi(s)ds\Lambda^{n-1}|||A_{N}|||^{n-1}_1\Big)\Big\}^2\Et[Z^{j,N}_{a\Delta}]\\
&+\frac{1}{N^2}\sum_{j=1}^{N}\Big\{\sum_{n\ge 1}n\int_{(a-1)\Delta/n}^{\infty}\phi(s)ds\Lambda^{n-1}|||A_{N}|||^{n-1}_1\Big\}^2\Et[Z^{j,N}_{(a-1)\Delta}]\\
\le& \frac{1}{N^2}\sum_{j=1}^{N}\Big\{\sum_{n\ge 1}n^{1+q}(a\Delta)^{-q}\int_{0}^{\infty}s^q\phi(s)ds\Lambda^{n-1}|||A_{N}|||^{n-1}_1\Big)\Big\}^2\Et[Z^{j,N}_{a\Delta}]\\
&+\frac{1}{N^2}\sum_{j=1}^{N}\Big\{\sum_{n\ge 1}n^{1+q}[(a-1)\Delta]^{-q}\int_{0}^{\infty}s^q\phi(s)ds\Lambda^{n-1}|||A_{N}|||^{n-1}_1\Big\}^2\Et[Z^{j,N}_{(a-1)\Delta}]\\
&\le \frac{1}{N^2(a\Delta)^{2q}}\mathbb{E}_{\theta}\Big[\sum_{j=1}^{N}Z^{j,N}_{a\Delta}\Big]+\frac{1}{N^2((a-1)\Delta)^{2q}}\mathbb{E}_{\theta}\Big[\sum_{j=1}^{N}Z^{j,N}_{(a-1)\Delta}\Big]\\
    &\le \frac{C}{N}\Big[\frac{1}{(a\Delta)^{2q-1}}+\frac{1}{((a-1)\Delta)^{2q-1}}\Big].
\end{align*}
For $q\ge 2$, we always have $\sum_{a=1}^{\infty}a^{1-2q}<+\infty,$ which concludes the result. 
\end{proof}

\begin{lemma}\label{mathY}
Assume $H(q)$ for some $q\ge 2$. Then a.s. on the set $\Omega_{N,K}$,

$(i)$ $\mathbb{E}\Big[\indiq_{\Omega_{N,K}}\Big|\mathcal{Y}_{(a-1)\Delta,a\Delta}^{N,K}\Big|^{2}\Big]\le \frac{C\Delta}{K}$,\\
 
$(ii)$ $\frac{K}{\sqrt{\Delta t}}\mathbb{E}\Big[\boldsymbol{1}_{\Omega_{N,K}}\sum_{a=\frac{t}{\Delta}}^{\frac{2t}{\Delta}}\Big|\mathcal{Y}_{(a-1)\Delta,a\Delta}^{N,K}\Big|\Big|\mathcal{Y}_{(a-1)\Delta,a\Delta}^{N,K}-\bar{X}_{(a-1)\Delta,a\Delta}^{N,K}\Big|\Big]\le \frac{C\sqrt{K}}{\Delta^{q-\frac{1}{2}}\sqrt{Nt}}$.
\end{lemma}

\begin{proof}
By \cite[Lemma 4.19]{D} and the definition of $\mathcal{X}^{N,K}_{\infty,\infty}$, we have on $\Omega_{N,K}$, 
\begin{align*}
\Big|\mathbb{E}\Big[\boldsymbol{1}_{\Omega_{N,K}}\sum_{j=1}^N\Big(c_{N}^{K}(j)\Big)^{2}\ell_{N}(j)\Big]\Big|\le \frac{CK^2}{N}\mathbb{E}\Big[\boldsymbol{1}_{\Omega_{N,K}}\Big\{\Big|\frac{\mu(N-K)}{K}\bar{\ell}_N^K\Big|+\Big|\frac{\mu}{(1-\Lambda p)^{3}}\Big|+\frac{1}{K}\Big\}\Big]\le CK.
\end{align*}
Recalling \cite[Lemma 16 (ii)]{A}, we already know that
$$
\sup_{j=1,...,N}\Big|\mathbb{E}_\theta[Z^{j,N}_{a\Delta}-Z^{j,N}_{(a-1)\Delta}-\Delta\ell_N(j)]\Big|\le C\Delta^{1-q}.
$$
From \cite[Lemma 4.16]{D}, we easily conclude that
$$
\mathbb{E}\Big[\boldsymbol{1}_{\Omega_{N,K}}\|\boldsymbol{c}_{N}^{K}\|^{2}_{2}\Big]\le 2\mathbb{E}\Big[\boldsymbol{1}_{\Omega_{N,K}}\|\boldsymbol{t}_{N}^{K}\|^{2}_{2}+\|\bar{c}_{N}^{K}\boldsymbol{1}^{T}_{N}-\boldsymbol{1}^{T}_{K}+\frac{K}{N}\boldsymbol{1}^{T}_{N}\|^2\Big]\le \frac{CK^2}{N^2}(1+N)+CK\le CK.
$$
Recalling the definition of $\mathcal{Y}_{(a-1)\Delta,a\Delta}^{N,K}$ and (\ref{ee3}), we have
\begin{align*}
    &\mathbb{E}\Big[\indiq_{\Omega_{N,K}}\Big|\mathcal{Y}_{(a-1)\Delta,a\Delta}^{N,K}\Big|^{2}\Big]\\
    =& \frac{1}{K^2}\mathbb{E}\Big[\indiq_{\Omega_{N,K}}\Big\{\sum_{j=1}^N\Big(c_N^K(j)\Big)^2\mathbb{E}_\theta[Z^{j,N}_{a\Delta}-Z^{j,N}_{(a-1)\Delta}-\Delta\ell_N(j)]+\Delta\sum_{j=1}^{N}\Big(c_{N}^{K}(j)\Big)^{2}\ell_{N}(j)\Big\}\Big]\\
    \le& \frac{C\Delta}{K}
\end{align*}
which completes the proof of $(i).$
From Lemma \ref{X}, we conclude that
\begin{align*}
    &\frac{K}{\sqrt{\Delta t}}\mathbb{E}\Big[\boldsymbol{1}_{\Omega_{N,K}}\sum_{a=\frac{t}{\Delta}+1}^{\frac{2t}{\Delta}}\Big|\mathcal{Y}_{(a-1)\Delta,a\Delta}^{N,K}\Big|\Big|\mathcal{Y}_{(a-1)\Delta,a\Delta}^{N,K}-\bar{X}_{(a-1)\Delta,a\Delta}^{N,K}\Big|\Big]\\
    \le& \frac{K}{\sqrt{\Delta t}}\sum_{a=\frac{t}{\Delta}+1}^{\frac{2t}{\Delta}}\mathbb{E}\Big[\boldsymbol{1}_{\Omega_{N,K}}\Big|\mathcal{Y}_{(a-1)\Delta,a\Delta}^{N,K}\Big|^{2}\Big]^\frac{1}{2}\mathbb{E}\Big[\boldsymbol{1}_{\Omega_{N,K}}\Big|\mathcal{Y}_{(a-1)\Delta,a\Delta}^{N,K}-\bar{X}_{(a-1)\Delta,a\Delta}^{N,K}\Big|^{2}\Big]^\frac{1}{2}\\
    \le& \frac{C\sqrt{K}}{\Delta^{q-\frac{1}{2}}\sqrt{Nt}}\sum_{a=\frac{t}{\Delta}}^{\frac{2t}{\Delta}}a^{\frac{1}{2}-q} \le \frac{C\sqrt{K}}{\Delta^{q-\frac{1}{2}}\sqrt{Nt}}.
\end{align*}
For the last step, we used that since $q\ge 2$, we always have $\sum_{a=1}^{\infty}a^{\frac{1}{2}-q}<+\infty.$
\end{proof}

\begin{lemma}\label{XG}
 Assume $H(q)$ for some $q\ge 1$. Then 
\begin{align*}
&\mathbb{E}\Big[\boldsymbol{1}_{\Omega_{N,K}}\frac{K}{N}\sqrt{\frac{t}{\Delta}}\frac{N}{t}\Big|\sum_{a=\frac{t}{\Delta}+1}^{\frac{2t}{\Delta}}\bar{\Gamma}_{(a-1)\Delta,a\Delta}^{N,K}\bar{X}_{(a-1)\Delta,a\Delta}^{N,K}-\mathbb{E}_{\theta}[\sum_{a=\frac{t}{\Delta}+1}^{\frac{2t}{\Delta}}\bar{\Gamma}_{(a-1)\Delta,a\Delta}^{N,K}\bar{X}_{(a-1)\Delta,a\Delta}^{N,K}]\Big|\Big]
\\
\le& \frac{CK}{N\Delta^q\sqrt{t}}+\frac{C\sqrt{tK}}{\Delta^{q+\frac{1}{2}}\sqrt{N}}+\frac{C\sqrt{K}}{\sqrt{N\Delta}}+\frac{Ct^\frac{3}{4}}{\Delta^{1+\frac{q}{2}}}.
\end{align*}
\end{lemma}

\begin{proof}
Recalling \eqref{Cdelta} and \eqref{Bdelta}, we write
\begin{align*}
\bar{\Gamma}_{(a-1)\Delta,a\Delta}^{N,K}\bar{X}_{(a-1)\Delta,a\Delta}^{N,K}=\bar{\Gamma}_{(a-1)\Delta,a\Delta}^{N,K}(\bar{X}_{(a-1)\Delta,a\Delta}^{N,K}-\mathcal{Y}_{(a-1)\Delta,a\Delta}^{N,K})+\mathcal{Y}_{(a-1)\Delta,a\Delta}^{N,K}\bar{\Gamma}_{(a-1)\Delta,a\Delta}^{N,K}\\
=\bar{\Gamma}_{(a-1)\Delta,a\Delta}^{N,K}(\bar{X}_{(a-1)\Delta,a\Delta}^{N,K}-\mathcal{Y}_{(a-1)\Delta,a\Delta}^{N,K})+\mathcal{Y}_{(a-1)\Delta,a\Delta}^{N,K}(C_{a\Delta}^{N,K}+B_{a\Delta}^{N,K}-C_{(a-1)\Delta}^{N,K}-B_{(a-1)\Delta}^{N,K}).
\end{align*}
From Lemmas \ref{X}, \ref{B} and  \ref{C}, we get
\begin{align*}
    &\mathbb{E}\Big[\boldsymbol{1}_{\Omega_{N,K}}\Big|\bar{\Gamma}_{(a-1)\Delta,a\Delta}^{N,K}\Big(\bar{X}_{(a-1)\Delta,a\Delta}^{N,K}-\mathcal{Y}_{(a-1)\Delta,a\Delta}^{N,K}\Big)\Big|\Big]^{2}\\
    \le&\mathbb{E}\Big[\boldsymbol{1}_{\Omega_{N,K}}\Big|\mathcal{Y}_{(a-1)\Delta,a\Delta}^{N,K}-\bar{X}_{(a-1)\Delta,a\Delta}^{N,K}\Big|^{2}\Big]\mathbb{E}\Big[\boldsymbol{1}_{\Omega_{N,K}}(\bar{\Gamma}_{(a-1)\Delta,a\Delta}^{N,K})^{2}\Big]\\
 =&\mathbb{E}\Big[\boldsymbol{1}_{\Omega_{N,K}}\Big|\mathcal{Y}_{(a-1)\Delta,a\Delta}^{N,K}-\bar{X}_{(a-1)\Delta,a\Delta}^{N,K}\Big|^{2}\Big]\mathbb{E}\Big[\boldsymbol{1}_{\Omega_{N,K}}\Big(C_{a\Delta}^{N,K}+B_{a\Delta}^{N,K}-C_{(a-1)\Delta}^{N,K}-B_{(a-1)\Delta}^{N,K}\Big)^{2}\Big]\\
    \le& 4\mathbb{E}\Big[\boldsymbol{1}_{\Omega_{N,K}}\Big|\mathcal{Y}_{(a-1)\Delta,a\Delta}^{N,K}-\bar{X}_{(a-1)\Delta,a\Delta}^{N,K}\Big|^{2}\Big]\\
    &\mathbb{E}\Big[\boldsymbol{1}_{\Omega_{N,K}}\Big\{\Big(C_{a\Delta}^{N,K}\Big)^{2}+\Big(B_{a\Delta}^{N,K}\Big)^{2}+\Big(C_{(a-1)\Delta}^{N,K}\Big)^{2}+\Big(B_{(a-1)\Delta}^{N,K}\Big)^{2}\Big\}\Big]\\
    &\le  \frac{C}{N}\Big[(a\Delta)^{1-2q}+\Big((a-1)\Delta\Big)^{1-2q}\Big]\Big(\frac{1}{N}+\frac{1}{N}\Delta^{1-2q}\Big)\le \Big[(a\Delta)^{1-2q}+\Big((a-1)\Delta\Big)^{1-2q}\Big]\frac{C}{N^2}.
\end{align*}
\\
And by Lemmas \ref{B} and \ref{mathY}, we have:
\begin{align*}
    &\indiq_{\Omega_{N,K}}\mathbb{E}_{\theta}\Big[\Big|\mathcal{Y}_{(a-1)\Delta,a\Delta}^{N,K}(B_{a\Delta}^{N,K}-B_{(a-1)\Delta}^{N,K})\Big|\Big]^{2}\\
    \le& \indiq_{\Omega_{N,K}}\mathbb{E}_{\theta}\Big[\Big|\mathcal{Y}_{(a-1)\Delta,a\Delta}^{N,K}\Big|^{2}\Big]\mathbb{E}_{\theta}\Big[\Big(B_{a\Delta}^{N,K}-B_{(a-1)\Delta}^{N,K}\Big)^{2}\Big]
    \le  \frac{C}{NK\Delta^{2q-2}}.
\end{align*}
Next, we consider the term $\mathcal{Y}_{(a-1)\Delta,a\Delta}^{N,K}(C_{a\Delta}^{N,K}-C_{(a-1)\Delta}^{N,K}).$ We can write
\begin{align*}
    \mathcal{Y}_{(a-1)\Delta,a\Delta}^{N,K}(C_{a\Delta}^{N,K}-C_{(a-1)\Delta}^{N,K})=\frac{1}{K^{2}}\sum_{i=1}^{K}\sum_{j,j'=1}^N\sum_{n\ge 1}\int_{0}^{\Delta}\phi^{*n}(s)A_{N}^{n}(i,j)c_{N}^{K}(j')\hskip2cm\\
    (M^{j,N}_{a\Delta-s}-M^{j,N}_{a\Delta}-M^{j,N}_{(a-1)\Delta-s}+M^{j,N}_{(a-1)\Delta})(M^{j',N}_{a\Delta}-M^{j',N}_{(a-1)\Delta}).
\end{align*}
We set for $1\le j,j',l,l'\le N$ and $a,\ b\in \{t/(2\Delta)+1,...,2t/\Delta\}$
\begin{align*}
\Upsilon_{a,b}(j,j',l,l'):=
   &\mathbb{C}ov_{\theta}[(M^{j,N}_{a\Delta-s}-M^{j,N}_{a\Delta}-M^{j,N}_{(a-1)\Delta-s}+M^{j,N}_{(a-1)\Delta})(M^{j',N}_{a\Delta}-M^{j',N}_{(a-1)\Delta}),\\
   &\hskip1cm(M^{l,N}_{b\Delta-s}-M^{l,N}_{b\Delta}-M^{l,N}_{(b-1)\Delta-s}+M^{l,N}_{(b-1)\Delta})(M^{l',N}_{b\Delta}-M^{l',N}_{(b-1)\Delta})]
\end{align*}
It is obvious that without any condition on $(a,b)$, we have that on $\Omega_{N,K}$ 
\begin{align*}
&|\Upsilon_{a,b}(j,j',l,l')|\\
\le& \Big\{\mathbb{E}_{\theta}\Big[\Big(M^{j,N}_{a\Delta}-M^{j,N}_{(a-1)\Delta}\Big)^4\Big]^\frac{1}{4}+\mathbb{E}_{\theta}\Big[\Big(M^{j,N}_{a\Delta-s}-M^{j,N}_{(a-1)\Delta-s}\Big)^4\Big]^\frac{1}{4}\Big\}\mathbb{E}_{\theta}\Big[\Big(M^{j',N}_{a\Delta}-M^{j',N}_{(a-1)\Delta}\Big)^4\Big]^\frac{1}{4}\\
&\Big\{\mathbb{E}_{\theta}\Big[\Big(M^{l,N}_{b\Delta}-M^{l,N}_{(b-1)\Delta}\Big)^4\Big]^\frac{1}{4}+\mathbb{E}_{\theta}\Big[\Big(M^{l,N}_{b\Delta-s}-M^{l,N}_{(b-1)\Delta-s}\Big)^4\Big]^\frac{1}{4}\Big\}\mathbb{E}_{\theta}\Big[\Big(M^{l',N}_{b\Delta}-M^{l',N}_{(b-1)\Delta}\Big)^4\Big]^\frac{1}{4}\\
\le& C\Delta^2 
\end{align*}
and $\boldsymbol{1}_{\#\{j,j',l,l'\}=4}|\Upsilon_{a,b}(j,j',l,l')|=0.$

Next, we consider the case when  $a-b\ge 4$.
Recalling that $\zeta^{j,N}_{a\Delta,s}:=M^{j,N}_{(a\Delta-s)}-M^{j,N}_{a\Delta}$ for $0\le s\le \Delta$,
\begin{align*}
    \Upsilon_{a,b}(j,j',l,l')=&\mathbb{C}ov_{\theta}[(\zeta^{j,N}_{a\Delta,s}-\zeta^{j,N}_{(a-1)\Delta,s})\zeta^{j',N}_{a\Delta,\Delta},(\zeta^{l,N}_{b\Delta,r}-\zeta^{l,N}_{(b-1)\Delta,r})\zeta^{l',N}_{b\Delta,\Delta}]\\
    =&\mathbb{C}ov_{\theta}[\zeta^{j,N}_{a\Delta,s}\zeta^{j',N}_{a\Delta,\Delta},(\zeta^{l,N}_{b\Delta,r}-\zeta^{l,N}_{(b-1)\Delta,r})\zeta^{l',N}_{b\Delta,\Delta}].
\end{align*}
Using the same strategy as the proof in Lemma \ref{covc}, we have
\begin{align*}
&|\mathbb{C}ov_{\theta}[\zeta^{j,N}_{a\Delta,s}\zeta^{j',N}_{a\Delta,\Delta},(\zeta^{l,N}_{b\Delta,r}-\zeta^{l,N}_{(b-1)\Delta,r})\zeta^{l',N}_{b\Delta,\Delta}]|\\
=&|\mathbb{C}ov_{\theta}[T^{j,N}_{a\Delta,(a-1)\Delta},
(\zeta^{l,N}_{b\Delta,r}-\zeta^{l,N}_{(b-1)\Delta,r})\zeta^{l',N}_{b\Delta,\Delta}]|\\
\le& \{\Et[(T^{j,N}_{a\Delta,(a-1)\Delta})^2]^\frac{1}{2}\}
\Et[(\zeta^{l,N}_{b\Delta,r}-\zeta^{l,N}_{(b-1)\Delta,r})^4]^\frac{1}{4}\Et[(\zeta^{l',N}_{b\Delta,\Delta})^4]^\frac{1}{4}\\
\le&  C t^{1/2} \Delta^{-q} \Delta.
\end{align*}
Hence, by symmetry, for $|a-b|\ge 4,$ we have $|\Upsilon_{a,b}(j,j',l,l')|\le C\Big(\boldsymbol{1}_{l=l'}+\boldsymbol{1}_{j=j'}\Big)\sqrt{t}\Delta^{1-q}.$
Hence, still for $|a-b|\ge 4,$
\begin{align*}
&\mathbb{E}\Big[\boldsymbol{1}_{\Omega_{N,K}}\Big|\mathbb{C}ov_{\theta}\Big[\cY_{(a-1)\Delta,a\Delta}^{N,K}(C_{a\Delta}^{N,K}-C_{(a-1)\Delta}^{N,K}),\cY_{(b-1)\Delta,b\Delta}^{N,K}(C_{b\Delta}^{N,K}-C_{(b-1)\Delta}^{N,K})\Big]\Big|\Big]\\
=&\frac{1}{K^{4}}\mathbb{E}\Big[\boldsymbol{1}_{\Omega_{N,K}}\Big|\sum_{i,i'=1}^{K}\sum_{l,l',j,j'=1}^{N}\sum_{n,n'\ge 1}\int_{0}^{\Delta}\int_{0}^{\Delta}
   \\ &\qquad\qquad\phi^{*n}(s)\phi^{*n'}(s')A_{N}^{n}(i,j)A_{N}^{n'}(i',l)c_{N}^{K}(j')c_{N}^{K}(l')\Upsilon_{a,b}(j,j',l,l')dsds'\Big|\Big]\\
    \le& \frac{t^{1/2}}{K^{4}\Delta^{q-1}}\mathbb{E}\Big[\boldsymbol{1}_{\Omega_{N,K}}\Big|\sum_{j=1}^N\Big(c_{N}^{K}(j)-1\Big)c_{N}^{K}(j)\Big|\Big|\sum_{l=1}^Nc_{N}^{K}(l)\Big|\sum_{n\ge 1}N\Lambda^n|||I_KA_N|||_1|||A_N|||^{n-1}_1\Big]\\
    &\le \frac{Ct^{1/2}}{K^{2}\Delta^{q-1}}.
\end{align*}
The last step follows from Lemma \ref{mathY} in which we have proved that 
$\mathbb{E}[\boldsymbol{1}_{\Omega_{N,K}}\|\boldsymbol{c}_{N}^{K}\|^{2}_{2}]\le CK$ and from the fact that
on the event $\Omega_{N,K}$, it holds that 
$|||I_KA_N|||_1\le \frac{K}{N},$ $|\sum_{l=1}^Nc_{N}^{K}(l)|=K|\bar{\ell}_N^K|\le CK.$ 

\vip

Next, when $|a-b|\le 4,$

\begin{align*}
    &\mathbb{E}\Big[\boldsymbol{1}_{\Omega_{N,K}}\Big|\mathbb{C}ov_{\theta}\Big[\cY_{(a-1)\Delta,a\Delta}^{N,K}(C_{a\Delta}^{N,K}-C_{(a-1)\Delta}^{N,K}),\cY_{(b-1)\Delta,b\Delta}^{N,K}(C_{b\Delta}^{N,K}-C_{(b-1)\Delta}^{N,K})\Big]\Big|\Big]\\
    \le& \mathbb{E}\Big[\indiq_{\Omega_{N,K}}\mathbb{E}_{\theta}\Big[\Big|\mathcal{Y}_{(a-1)\Delta,a\Delta}^{N,K}\Big|^{2}\Big]\mathbb{E}_{\theta}\Big[\Big(C_{a\Delta}^{N,K}-C_{(a-1)\Delta}^{N,K}\Big)^{2}\Big]\mathbb{E}_{\theta}\Big[\Big|\mathcal{Y}_{(b-1)\Delta,b\Delta}^{N,K}\Big|^{2}\Big]\mathbb{E}_{\theta}\Big[\Big(C_{b\Delta}^{N,K}-C_{(b-1)\Delta}^{N,K}\Big)^{2}\Big]\Big]\\
    \le& \frac{C\Delta}{NK}.
\end{align*}

Finally, 
\begin{align*}
\mathbb{E}\Big[\indiq_{\Omega_{N,K}}\mathbb{V}ar_{\theta}\Big[\sum_{a=\frac{t}{\Delta}+1}^{\frac{2t}{\Delta}}\mathcal{Y}_{(a-1)\Delta,a\Delta}^{N,K}(C_{a\Delta}^{N,K}-C_{(a-1)\Delta}^{N,K})\Big]\Big]\le \frac{Ct}{NK}+\frac{Ct^{5/2}}{K^{2}\Delta^{q+1}}.
\end{align*}
Overall we conclude that 
\begin{align*}
&\mathbb{E}\Big[\boldsymbol{1}_{\Omega_{N,K}}\frac{K}{N}\sqrt{\frac{t}{\Delta}}\frac{N}{t}\Big|\sum_{a=\frac{t}{\Delta}+1}^{\frac{2t}{\Delta}}\bar{\Gamma}_{(a-1)\Delta,a\Delta}^{N,K}\bar{X}_{(a-1)\Delta,a\Delta}^{N,K}-\mathbb{E}_{\theta}[\sum_{a=\frac{t}{\Delta}+1}^{\frac{2t}{\Delta}}\bar{\Gamma}_{(a-1)\Delta,a\Delta}^{N,K}\bar{X}_{(a-1)\Delta,a\Delta}^{N,K}]\Big|\Big]\\
\le& 
\frac{K}{\sqrt{\Delta t}}\Big\{\mathbb{E}\Big[\boldsymbol{1}_{\Omega_{N,K}}\sum_{a=\frac{t}{\Delta}+1}^{\frac{2t}{\Delta}}\Big(\Big|\bar{\Gamma}_{(a-1)\Delta,a\Delta}^{N,K}(\bar{X}_{(a-1)\Delta,a\Delta}^{N,K}-\cY_{(a-1)\Delta,a\Delta}^{N,K})\Big|\\
&+\Big|\mathcal{Y}_{(a-1)\Delta,a\Delta}^{N,K}(B_{a\Delta}^{N,K}-B_{(a-1)\Delta}^{N,K})\Big|\Big)\Big]
 +\mathbb{E}\Big[\boldsymbol{1}_{\Omega_{N,K}}\mathbb{V}ar_{\theta}\Big[\sum_{a=\frac{t}{\Delta}+1}^{\frac{2t}{\Delta}}\cY_{(a-1)\Delta,a\Delta}^{N,K}(C_{a\Delta}^{N,K}-C_{(a-1)\Delta}^{N,K})\Big]\Big]^{\frac{1}{2}}\Big\}\\
\le& \frac{CK}{N\Delta^q\sqrt{t}}+\frac{C\sqrt{tK}}{\Delta^{q+\frac{1}{2}}\sqrt{N}}+\frac{C\sqrt{K}}{\sqrt{N\Delta}}+\frac{Ct^\frac{3}{4}}{\Delta^{1+\frac{q}{2}}}.
\end{align*}
The proof is finished.
\end{proof}

Finally, we can give the proof of Lemma \ref{D3}.

\begin{proof}
Recalling \eqref{Ya1} and \eqref{mathbX}, as well as 
Lemmas \ref{Gamma}, \ref{XG}, \ref{mathY}-$(ii)$ and \ref{X},
\begin{align*}
&\frac{K}{N}\sqrt{\frac{t}{\Delta}}\mathbb{E}\Big[\indiq_{\Omega_{N,K}}\Big|D_{\Delta,t}^{N,K,3}-\frac{N}{t}\mathbb{X}_{\Delta,t,1}^{N,K}\Big|\Big]\\
\le& \frac{2K}{\sqrt{t\Delta}}\mathbb{E}\Big[\indiq_{\Omega_{N,K}}\Big(\Big|\sum_{a=\frac{t}{\Delta}+1}^{\frac{2t}{\Delta}}\Big\{(\bar{\Gamma}^{N,K}_{(a-1)\Delta,a\Delta})^{2}-\mathbb{E}_{\theta}[(\bar{\Gamma}^{N,K}_{(a-1)\Delta,a\Delta})^{2}]\Big\}\Big|\\
&\qquad+\Big|\sum_{a=\frac{t}{\Delta}+1}^{\frac{2t}{\Delta}}\bar{\Gamma}_{(a-1)\Delta,a\Delta}^{N,K}\bar{X}_{(a-1)\Delta,a\Delta}^{N,K}-\mathbb{E}_{\theta}[\sum_{a=\frac{t}{\Delta}+1}^{\frac{2t}{\Delta}}\bar{\Gamma}_{(a-1)\Delta,a\Delta}^{N,K}\bar{X}_{(a-1)\Delta,a\Delta}^{N,K}]\Big|\\
&\qquad+\sum_{a=\frac{t}{\Delta}+1}^{\frac{2t}{\Delta}}\Big|\mathcal{Y}_{(a-1)\Delta,a\Delta}^{N,K}\Big|\Big|\mathcal{Y}_{(a-1)\Delta,a\Delta}^{N,K}-\bar{X}_{(a-1)\Delta,a\Delta}^{N,K}\Big|+\sum_{a=\frac{t}{\Delta}+1}^{\frac{2t}{\Delta}}\Big|\mathcal{Y}_{(a-1)\Delta,a\Delta}^{N,K}-\bar{X}_{(a-1)\Delta,a\Delta}^{N,K}\Big|^{2}\Big)\Big]\\
\le& \frac{CK}{N\Delta}+\frac{CKt^{\frac{3}{4}}}{\Delta^{(1+\frac{q}{2})}\sqrt{N}}+\frac{CK\sqrt{t}}{N\Delta^{(q+1)}}+ \frac{CK}{N\Delta^q\sqrt{t}}+\frac{C\sqrt{tK}}{\Delta^{q+\frac{1}{2}}\sqrt{N}}\\
&\qquad+\frac{C\sqrt{K}}{\sqrt{N\Delta}}+\frac{Ct^\frac{3}{4}}{\Delta^{1+\frac{q}{2}}}+\frac{C\sqrt{K}}{\Delta^{q-\frac{1}{2}}\sqrt{Nt}}+\frac{CK}{N\sqrt{t}\Delta^{2q+\frac{1}{2}}}\\
\le& \frac{CKt^{\frac{3}{4}}}{\Delta^{1+\frac{q}{2}}\sqrt{N}}+\frac{C\sqrt{K}}{\sqrt{N\Delta}}+\frac{Ct^\frac{3}{4}}{\Delta^{1+\frac{q}{2}}},
\end{align*}
which completes the proof.
\end{proof}

\subsection{The convergence of $\mathbb{X}_{\Delta_t,t,v}^{N,K}$ }\label{PD3}
The aim of this subsection is to prove the following Lemma.
\begin{lemma}\label{D33}
Assume $H(q)$ for some $q> 3$. For $t\geq 1$, set $\Delta_t= t/(2 \lfloor t^{1-4/(q+1)}\rfloor) \sim t^{4/(q+1)}/2$
(as $t\to \infty$). In the limit $(N,K,t)\to (\infty,\infty,\infty)$ and in the regime where 
$\frac{K}{N} \to \gamma\le 1$ and where
$\frac 1{\sqrt K} + \frac NK \sqrt{\frac{\Delta_t}t}+ \frac{N}{t\sqrt K}+Ne^{-c_{p,\Lambda}K} \to 0$, it holds that
$$\left(\frac{K}{N}\sqrt{\frac{t}{\Delta_t}}\frac{N}{t}\mathbb{X}_{\Delta_t,t,v}^{N,K}\right)_{v\ge 0}\stackrel{d}{\longrightarrow}\frac{1}{\sqrt{2}}\Big(\frac{1-\gamma}{(1-\Lambda p)}+\frac{\gamma}{(1-\Lambda p)^3}\Big)(B_{2v}-B_v)_{v\ge 0}$$
for the Skorohod topology, where $B$ is a standard Brownian motion.
\end{lemma}

We start by applying the Ito formula to write
\begin{align}\label{KYQZ}
    (K\mathcal{Y}_{(a-1)\Delta,a\Delta}^{N,K})^2= Q_{a,N,K}+\sum_{j=1}^{N}\Big(c_{N}^{K}(j)\Big)^{2}\Big(Z_{a\Delta}^{j,N}-Z^{j,N}_{(a-1)\Delta}\Big)
    \end{align}
    where $Q_{a,N,K}=\int_{(a-1)\Delta}^{a\Delta}\sum_{j=1}^{N}c_{N}^{K}(j)(M^{j,N}_{s-}-M^{j,N}_{(a-1)\Delta})\sum_{j=1}^{N}c_{N}^{K}(j)dM^{j,N}_{s}.$
First, we verify that:
\begin{lemma}\label{cZtq}
Assume $H(q)$ for some $q\ge 1$, $0\le v\le 1,$
\begin{align*}
    \mathbb{E}_{\theta}\Big[\Big|\frac{1}{K\sqrt{\Delta t}}\sum_{a=[\frac{vt}{\Delta}]+1}^{[\frac{2vt}{\Delta}]}\sum_{j=1}^{N}\Big(c_{N}^{K}(j)\Big)^{2}\Big(Z_{a\Delta}^{j,N}-Z_{(a-1)\Delta}^{j,N}-\mathbb{E}_{\theta}[Z_{a\Delta}^{j,N}-Z_{(a-1)\Delta}^{j,N}]\Big)\Big|\Big]\le \frac{C}{\sqrt{\Delta t}}.
\end{align*}
\end{lemma}
\begin{proof}
\begin{align*}
&\frac{1}{K\sqrt{\Delta t}}\sum_{a=[\frac{vt}{\Delta}]+1}^{[\frac{2vt}{\Delta}]}\sum_{j=1}^{N}\Big(c_{N}^{K}(j)\Big)^{2}\Big(Z_{a\Delta}^{j,N}-Z_{(a-1)\Delta}^{j,N}-\mathbb{E}_{\theta}[Z_{a\Delta}^{j,N}-Z_{(a-1)\Delta}^{j,N}]\Big)\\
&=\frac{1}{K\sqrt{\Delta t}}\Big\{\sum_{j=1}^{N} \Big(c_{N}^{K}(j)\Big)^{2}\Big(Z_{2vt}^{j,N}-Z_{vt}^{j,N}-\mu vt\ell_{N}(j)\Big)\\
&\qquad+\sum_{j=1}^{N} \mathbb{E}_{\theta}\Big[\Big(c_{N}^{K}(j)\Big)^{2}\Big(\mu vt\ell_{N}(j)-Z_{2vt}^{j,N}+Z_{vt}^{j,N}\Big)\Big]\Big\}.
\end{align*}
Recalling \cite[(8)]{A}, we already know that $\indiq_{\{i=j\}} \leq Q_N(i,j)\leq \indiq_{\{i=j\}} + \Lambda C N^{-1}$
for all $i,j=1,\dots,N$ on $ \Omega_{N,K}\subset \Omega_N^1.$ So $|c^K_{N}(j)|=|\sum_{i=1}^{K}Q_{N}(i,j)|$ is bounded 
by some constant $C$ for $j=1,...,K$ and smaller than $\frac{CK}{N}$ for $K+1\le j\le N.$   
By \cite[Lemma 16-(ii)]{A}, we also know that
$$
\max_{j=1,...,N}\mathbb{E}_{\theta}\Big[ \Big|\Big(Z_{2vt}^{j,N}-Z_{vt}^{j,N}-\mu v t\ell_{N}(j)\Big)\Big|\Big]\le C.
$$
Hence
\begin{align*}
    &\mathbb{E}_{\theta}\Big[\frac{1}{K\sqrt{\Delta t}}\sum_{j=1}^{N} \Big|\Big(c_{N}^{K}(j)\Big)^{2}\Big(Z_{2vt}^{j,N}-Z_{vt}^{j,N}-\mu v t\ell_{N}(j)\Big)\Big|\Big]\le \frac{C}{\sqrt{\Delta t}}.
\end{align*}
So
\begin{align*}
    \mathbb{E}_{\theta}\Big[\Big|\frac{1}{K\sqrt{\Delta t}}\sum_{a=[\frac{vt}{\Delta}]+1}^{[\frac{2vt}{\Delta}]}\sum_{j=1}^{N}\Big(c_{N}^{K}(j)\Big)^{2}\Big(Z_{a\Delta}^{j,N}-Z_{(a-1)\Delta}^{j,N}-\mathbb{E}_{\theta}[Z_{a\Delta}^{j,N}-Z_{(a-1)\Delta}^{j,N}]\Big)\Big|\Big]\le\frac{C}{\sqrt{\Delta t}},
\end{align*}
which ends  the proof.
\end{proof}

We next define
$$\mathcal{L}^{t,\Delta}_{N,K}(u):=\frac{1}{K\sqrt{\Delta t}}\sum_{a=1}^{[\frac{t}{\Delta}u]}Q_{a,N,K},\quad
\text{ for $0\le u\le 2$}.$$
We notice $\mathbb{E}[Q_{a,N,K}|\mathcal{F}_{a-1}]=0.$ So $\mathcal{L}^{t,\Delta}_{N,K}(u)$ is a martingale for the filtration $\mathcal{F}_{[\frac{t}{\Delta}u]}.$ 
Recalling the equality (\ref{KYQZ}) and definition (\ref{mathbX}), as well as Lemma \ref{cZtq}, we 
conclude the following estimate.
\begin{cor}\label{LmathbX}
Assume $H(q)$ for some $q\ge 1$, then for $0\le v\le 1,$
$$
\mathbb{E}\Big[\omg\Big|\mathcal{L}^{t,\Delta}_{N,K}(2v)-\mathcal{L}^{t,\Delta}_{N,K}(v)-\frac{K}{\sqrt{\Delta t}}\mathbb{X}^{N,K}_{\Delta,t,v}\Big|\Big]\le \frac{C}{\sqrt{\Delta t}}.
$$
\end{cor}

Next we will prove the convergence of $\mathcal{L}^{t,\Delta}_{N,K}(u)$ to a Brownian motion. 

\begin{lemma}\label{4}
Assume $H(q)$ for some $q\ge 1$. Then a.s. on the set $\Omega_{N,K}$, for all $\Delta\ge 1$:
$$\mathbb{E}_{\theta}[(Q_{a,N,K})^{4}]\le   C(K\Delta)^4.$$
\end{lemma}

\begin{proof}
For $0\le u\le 1,$ we set 
\begin{align*}
q_{a,N,K}(u):=\int_{(a-1)\Delta}^{[(a-1)+u]\Delta}\sum_{j=1}^{N}c_{N}^{K}(j)(M^{j,N}_{s}-M^{j,N}_{(a-1)\Delta})d\Big(\sum_{j=1}^{N}c_{N}^{K}(j)M^{j,N}_{s}\Big).
\end{align*}
 It is obvious $q_{a,N,K}(1)=Q_{a,N,K}.$
And
\begin{align*}
[q_{a,N,K}(.),q_{a,N,K}(.)]_{u}=\int_{(a-1)\Delta}^{(a-1+u)\Delta}\Big(\sum_{j=1}^{N}c_{N}^{K}(j)(M_{s}^{j,N}-M_{(a-1)\Delta}^{j,N})\Big)^{2}\sum_{j=1}^{N}\Big(c_{N}^{K}(j)\Big)^{2}d Z_{s}^{j,N}.
\end{align*}
From \cite[(8)]{A}, we already have on the event $\Omega_{N,K}$,  
$\boldsymbol{1}_{\{i=j\}}\le Q_{N}(i,j)\le \boldsymbol{1}_{\{i=j\}}+\frac{C}{N}$ for all $i,j=1,...,N.$
Recalling that $c_{N}^{K}(i)=\sum_{j=1}^{K}Q_{N}(j,i)$, 
\begin{align*}
1\le c_{N}^{K}(i)\le 1+\frac{CK}{N}\hbox{  when } 1\le i\le K \hbox{ and }
0\le c_{N}^{K}(i)\le \frac{CK}{N} \hbox{ when } (K+1)\le i\le N.
\end{align*}
Then we conclude that 
$$
\sum_{j=1}^{N}\Big(c_{N}^{K}(j)\Big)^{2}dZ_{s}^{j,N}\le C\Big(Kd\bar{Z}^{N,K}_s+\frac{K^2}{N}d\bar{Z}^{N,N}_s\Big).
$$
On $\Omega_{N,K}$, we have
\begin{align*}
&\mathbb{E}_{\theta}[(q_{a,N,K}(u))^{4}]\le    4\mathbb{E}_{\theta}[([q_{a,N,K}(.),q_{a,N,K}(.)]_{u})^{2}]
 \\
=&4\mathbb{E}_{\theta}\Big[\Big(\int_{(a-1)\Delta}^{(a-1+u)\Delta}\Big(\sum_{j=1}^{N}c_{N}^{K}(j)(M_{s}^{j,N}-M_{(a-1)\Delta}^{j,N})\Big)^{2}d\sum_{j=1}^{N}\Big(c_{N}^{K}(j)\Big)^{2}Z_{s}^{j,N}\Big)^{2}\Big]\\
\le& 4\mathbb{E}_{\theta}\Big[\sup_{0\le s\le u\Delta}\Big(\sum_{j=1}^Nc_N^K(j)(M_{(a-1)\Delta+s}^{j,N}-M_{(a-1)\Delta}^{j,N})\Big)^4\Big(\sum_{j=1}^{N}\Big(c_{N}^{K}(j)\Big)^{2}(Z_{(a-1+u)\Delta}^{j,N}-Z_{(a-1)\Delta}^{j,N})\Big)^{2}\Big]\\
\le& 8\mathbb{E}_{\theta}\Big[\sup_{0\le s\le u\Delta}\Big(\sum_{j=1}^Nc_N^K(j)(M_{(a-1)\Delta+s}^{j,N}-M_{(a-1)\Delta}^{j,N})\Big)^8+\Big(\sum_{j=1}^{N}\Big(c_{N}^{K}(j)\Big)^{2}(Z_{(a-1+u)\Delta}^{j,N}-Z_{(a-1)\Delta}^{j,N})\Big)^{4}\Big]\\
\le& C\mathbb{E}_{\theta}\Big[\Big(\sum_{j=1}^{N}\Big(c_{N}^{K}(j)\Big)^{2}(Z_{(a-1+u)\Delta}^{j,N}-Z_{(a-1)\Delta}^{j,N})\Big)^{4}\Big]\\
\le& C\mathbb{E}_{\theta}\Big[\Big(K(\bar{Z}_{(a-1+u)\Delta}^{N,K}-\bar{Z}_{(a-1)\Delta}^{N,K})+\frac{K^2}{N}(\bar{Z}_{(a-1+u)\Delta}^{N,N}-\bar{Z}_{(a-1)\Delta}^{N,N})\Big)^4\Big]\le C(Ku\Delta)^4.
\end{align*}
We used the Burkholder-Davies-Gundy inequality in the fourth step as well as 
Lemma \ref{Zt}-(iii) in the last one.
\end{proof}
Next, we are going to prove that the jumps of $\mathcal{L}^{t,\Delta}_{N,K}(u)$ are not large.
\begin{lemma}\label{jum}
Assume $H(q)$ for some $q\ge 1$.
$$\boldsymbol{1}_{\Omega_{N,K}}\mathbb{E}_{\theta}\Big[\sup_{0\le u\le 2}\Big|\mathcal{L}^{t,\Delta}_{N,K}(u)-\mathcal{L}^{t,\Delta}_{N,K}(u-)\Big|\Big]\le C\Big(\frac{\Delta}{t}\Big)^{\frac{1}{4}}.$$
\end{lemma}

\begin{proof}
First, we notice $\mathcal{L}^{t,\Delta}_{N,K}(u)$ is a pure jump process. So
$$\mathcal{L}^{t,\Delta}_{N,K}(u)-\mathcal{L}^{t,\Delta}_{N,K}(u-)=\frac{1}{K\sqrt{\Delta t}}\Big([\frac{t}{\Delta}u]-[\frac{t}{\Delta}u-]\Big)Q_{[\frac{t}{\Delta}u],N,K}.$$
Then by Lemma \ref{4}, we have 
\begin{align*}
\mathbb{E}_{\theta}\Big[\sup_{0\le u\le 2}\Big|\mathcal{L}^{t,\Delta}_{N,K}(u)-\mathcal{L}^{t,\Delta}_{N,K}(u-)\Big|\Big]=&\frac{1}{K\sqrt{\Delta t}}\mathbb{E}_{\theta}\Big[\sup_{\{i=1...[\frac{2t}{\Delta}]\}}|Q_{[i\Delta],N,K}|\Big]\\
\le& \frac{1}{K\sqrt{\Delta t}}\mathbb{E}_{\theta}\Big[\Big(\sum_{i=1}^{[\frac{2t}{\Delta}]}|Q_{[i\Delta],N,K}|^{4}\Big)^{\frac{1}{4}}\Big]\\
\le& \frac{1}{K\sqrt{\Delta t}}\mathbb{E}_{\theta}\Big[\sum_{i=1}^{[\frac{2t}{\Delta}]}|Q_{[i\Delta],N,K}|^{4}\Big]^{\frac{1}{4}}\\
\le& C\Big(\frac{\Delta}{t}\Big)^{\frac{1}{4}}.
\end{align*}
The proof is complete.
\end{proof}

\begin{lemma}\label{constant}
Assume $H(q)$ for some $q\ge 1$. For all $t\ge \Delta$, it holds that
\begin{align*}
&\frac{1}{K^{2}\Delta t}\mathbb{E}\Big[\boldsymbol{1}_{\Omega_{N,K}}\Big|A^{N,K}_{\infty,\infty}\Big\{\sum_{j=1}^{N}\Big(c_{N}^{K}(j)\Big)^{2}\sum_{a=0}^{\frac{t}{\Delta}}\int_{a\Delta}^{(a+1)\Delta}\mathbb{E}_{\theta}[Z_{s}^{j,N}-Z^{j,N}_{a\Delta}(s)]ds\Big\}
-\frac{\Delta t}{2}\Big(A^{N,K}_{\infty,\infty}\Big)^{2}\Big|\Big]\\
\le& \frac{C}{\Delta^q}+\frac{C}{t},
\end{align*}
where $A^{N,K}_{\infty,\infty}:=\sum_{i=1}^{N}\Big(c_{N}^{K}(i)\Big)^{2}\ell_{N}(i).$
\end{lemma}

\begin{proof}
Recall that on $\Omega_{N,K}$, $ |\ell_N(i)|\le C$ for all $1\le i\le N.$
\begin{align*}
1\le c_{N}^{K}(i)\le 1+\frac{CK}{N} \hbox{ when } 1\le i\le K 
\hbox{ and } 0\le c_{N}^{K}(i)\le \frac{CK}{N} \hbox{ when } (K+1)\le i\le N.
\end{align*}
By \cite[Lemma 16 (ii) ]{A}, we have 
\begin{align*}
    &\frac{1}{K^{2}\Delta t}\mathbb{E}\Big[\boldsymbol{1}_{\Omega_{N,K}}\Big|A^{N,K}_{\infty,\infty}\Big\{\sum_{j=1}^{N}\Big(c_{N}^{K}(j)\Big)^{2}\sum_{a=0}^{\frac{t}{\Delta}}\int_{a\Delta}^{(a+1)\Delta}\mathbb{E}_{\theta}[Z_{s}^{j,N}-Z^{j,N}_{a\Delta}(s)]ds\Big\}
-\frac{\Delta t}{2}\Big(A^{N,K}_{\infty,\infty}\Big)^{2}\Big|\Big]\\
    =&\frac{1}{K^{2}\Delta t}\mathbb{E}\Big[\boldsymbol{1}_{\Omega_{N,K}}\Big|A^{N,K}_{\infty,\infty}\Big\{\sum_{j=1}^{N}\Big(c_{N}^{K}(j)\Big)^{2}\sum_{a=0}^{\frac{t}{\Delta}}\int_{a\Delta}^{(a+1)\Delta}\mathbb{E}_{\theta}\Big[Z_{s}^{j,N}-Z^{j,N}_{a\Delta}(s)-\mu(s-a\Delta)\ell_{N}(j)\Big]ds\Big\}\Big|\Big]\\
    \le& \frac{CK^{2}}{K^{2}\Delta t}\sum_{a=0}^{\frac{t}{\Delta}}\int_{a\Delta}^{(a+1)\Delta}(1\wedge(a\Delta)^{1-q})ds\le \frac{C}{\Delta^q}+\frac{C}{t}.
\end{align*}
This concludes the proof.
\end{proof}

\begin{lemma}\label{bro}
We assume $H(q)$ for some $q\ge 1$. For $0\le u\le 2,$
\begin{align*}
\mathbb{E}\Big[\boldsymbol{1}_{\Omega_{N,K}}\Big|[\mathcal{L}^{t,\Delta}_{N,K}(.),\mathcal{L}^{t,\Delta}_{N,K}(.)]_{u}-\frac{u(A^{N,K}_{\infty,\infty})^{2}}{2K^{2}}\Big|\Big]\le C\Big(\frac{1}{K\Delta}+\frac{1}{\sqrt{N}}+\Big(\frac{K\sqrt{t}}{\Delta^{q+1}}\Big)^\frac{1}{2}+\sqrt{\frac{\Delta}{t}}\Big)
\end{align*}
\end{lemma}

\begin{proof}
For $s\geq 0$, we introduce $\phi_{t,\Delta}(s)=a\Delta$, where $a$ is the unique integer such that
$a\Delta\le s< (a+1)\Delta$. Then we have
$$\mathcal{L}^{t,\Delta}_{N,K}(u)=\int_{0}^{tu}\sum_{j=1}^{N}c_{N}^{K}(j)(M_{s}^{j,N}-M_{\phi_{t,\Delta}(s)}^{j,N}) \sum_{i=1}^{N}dM_{s}^{j,N}. $$
So
\begin{align*}
[\mathcal{L}^{t,\Delta}_{N,K}(.),\mathcal{L}^{t,\Delta}_{N,K}(.)]_{u}
=&\frac{1}{K^{2}\Delta t}\int_{0}^{tu}\Big(\sum_{j=1}^{N}c_{N}^{K}(j)(M_{s}^{j,N}-M^{j,N}_{\phi_{t,\Delta}(s)})\Big)^{2}\sum_{i=1}^{N}\Big(c_{N}^{K}(i)\Big)^{2}dZ_{s}^{i,N}\\
=& \mathcal{A}^{u,1}_{N,K}+\mathcal{A}^{u,2}_{N,K}+\mathcal{A}^{u,3}_{N,K},
\end{align*}
where
\begin{align*}
    \mathcal{A}^{u,1}_{N,K}:=&\frac{1}{K^{2}\Delta t}\int_{0}^{tu}\Big(\sum_{j=1}^{N}c_{N}^{K}(j)(M_{s}^{j,N}-M^{j,N}_{\phi_{t,\Delta}(s)})\Big)^{2} \sum_{i=1}^{N}\Big(c_{N}^{K}(i)\Big)^{2}dM_{s}^{i,N},\\
    \mathcal{A}^{u,2}_{N,K}:=&\frac{1}{K^{2}\Delta t}\int_{0}^{tu}\Big(\sum_{j=1}^{N}c_{N}^{K}(j)(M_{s}^{j,N}-M_{\phi_{t,\Delta}(s)}^{j,N})\Big)^{2}\sum_{i=1}^{N}\Big(c_{N}^{K}(i)\Big)^{2}\Big(\lambda_{s}^{i,N}-\mu\ell_{N}(i)\Big)ds,\\
    \mathcal{A}^{u,3}_{N,K}:=&\Big[\mu\sum_{i=1}^{N}\Big(c_{N}^{K}(i)\Big)^{2}\ell_{N}(i)\Big]\frac{1}{K^{2}\Delta t}\int_{0}^{tu}\Big(\sum_{j=1}^{N}c_{N}^{K}(j)(M_{s}^{j,N}-M_{\phi_{t,\Delta}(s)}^{j,N})\Big)^{2}ds.
\end{align*}

First we give an upper-bound for $\mathcal{A}^{u,1}_{N,K}$. Recalling (\ref{ee3}) and that
\begin{align*}
1\le c_{N}^{K}(i)\le 1+\frac{CK}{N} \hbox{ when } 1\le i\le K \hbox{ and }
0\le c_{N}^{K}(i)\le \frac{CK}{N} \hbox{  when } (K+1)\le i\le N,
\end{align*}
and using Doob's inequality and Lemma \ref{UZ4}, we obtain
\begin{align*}
    &\mathbb{E}_{\theta}\Big[\Big(\mathcal{A}^{u,1}_{N,K}\Big)^2\Big]
    =\frac{1}{K^{4}(\Delta t)^{2}}\mathbb{E}_{\theta}\Big[\int_{0}^{tu}\Big(\sum_{j=1}^{N}c_{N}^{K}(j)(M_{s}^{j,N}-M^{j,N}_{\phi_{t,\Delta}(s)})\Big)^{4} \sum_{i=1}^{N}\Big(c_{N}^{K}(i)\Big)^{4}dZ_{s}^{i,N}\Big]\\
    \le& \frac{C}{K^{4}(\Delta t)^{2}}\mathbb{E}_{\theta}\Big[\max_{a=1,...,[\frac{ut}{\Delta}]+1}\sup_{0\le s\le u\Delta}\Big(\sum_{j=1}^Nc_N^K(j)(M_{(a-1)\Delta+s}^{j,N}-M_{(a-1)\Delta}^{j,N})\Big)^4\sum_{j=1}^{N}\Big(c_{N}^{K}(j)\Big)^{4}Z_{tu}^{j,N}\Big]\\
    \le& \frac{C}{K^{4}(\Delta t)^{2}}\mathbb{E}_{\theta}\Big[\max_{a=1,...,[\frac{ut}{\Delta}]+1}\sup_{0\le s\le u\Delta}\Big(\sum_{j=1}^Nc_N^K(j)(M_{(a-1)\Delta+s}^{j,N}-M_{(a-1)\Delta}^{j,N})\Big)^8+\Big(\sum_{j=1}^{N}\Big(c_{N}^{K}(j)\Big)^{4}Z_{tu}^{j,N}\Big)^2\Big]\\
    \le& \frac{C}{K^{4}(\Delta t)^{2}}\mathbb{E}_{\theta}\Big[\sum^{[\frac{ut}{\Delta}]+1}_{a=1}\sup_{0\le s\le u\Delta}\Big(\sum_{j=1}^Nc_N^K(j)(M_{(a-1)\Delta+s}^{j,N}-M_{(a-1)\Delta}^{j,N})\Big)^8\Big]+\frac{C}{K^2\Delta^2}\\
    \le& \frac{C}{K^{4}(\Delta t)^{2}}\mathbb{E}_{\theta}\Big[\sum^{[\frac{ut}{\Delta}]+1}_{a=1}\Big(\sum_{j=1}^N(c_N^K(j))^2(Z_{a\Delta}^{j,N}-Z_{(a-1)\Delta}^{j,N})\Big)^4\Big]+\frac{C}{K^2\Delta^2}\\
    \le& \frac{C\Delta}{t}+\frac{C}{K^2\Delta^2}.
\end{align*}
For the second term, we use Lemma \ref{intensity}, and we have on $\Omega_{N,K}$
\begin{align*}
    &\mathbb{E}_{\theta}\Big[\Big|\mathcal{A}^{u,2}_{N,K}\Big|\Big]\\
    \le& \frac{1}{K^{2}\Delta t}\int_{0}^{tu}\mathbb{E}_{\theta}\Big[\Big(\sum_{j=1}^{N}c_{N}^{K}(j)\Big(M_{s}^{j,N}-M_{\phi_{t,\Delta}(s)}^{j,N}\Big)\Big)^{4}\Big]^{\frac{1}{2}}\mathbb{E}_{\theta}\Big[\Big|\sum_{i=1}^{N}\Big(c_{N}^{K}(i)\Big)^{2}\Big(\lambda_{s}^{i,N}-\mu\ell_{N}(i)\Big)\Big|^{2}\Big]^{\frac{1}{2}}ds\\
    \le& \frac{1}{K^{2}\Delta t}\int_{0}^{tu}\Et\Big[\Big(\sum_{j=1}^{N}\Big(c_{N}^{K}(j)\Big)^2\Big(Z_{s}^{j,N}-Z_{\phi_{t,\Delta}(s)}^{j,N}\Big)\Big)^{2}\Big]^{\frac{1}{2}}\mathbb{E}_{\theta}\Big[\Big|\sum_{i=1}^{N}\Big(c_{N}^{K}(i)\Big)^{2}\Big(\lambda_{s}^{i,N}-\mu\ell_{N}(i)\Big)\Big|^{2}\Big]^{\frac{1}{2}}ds\\
    \le& \frac{1}{K^{2}\Delta t}\int_{0}^{tu}\Et\Big[\Big(K\Big(\bar{Z}_{s}^{N,K}-\bar{Z}_{\phi_{t,\Delta}(s)}^{N,K}\Big)
    +\frac{K^2}{N}\Big(\bar{Z}_{s}^{N}-\bar{Z}_{\phi_{t,\Delta}(s)}^{N}\Big)\Big)^{2}\Big]^{\frac{1}{2}}\\&\qquad\qquad\qquad\times\mathbb{E}_{\theta}\Big[\Big|\sum_{i=1}^{N}\Big(c_{N}^{K}(i)\Big)^{2}\Big(\lambda_{s}^{i,N}-\mu\ell_{N}(i)\Big)\Big|^{2}\Big]^{\frac{1}{2}}ds\\
    \le& \frac{C}{K t}\int_{0}^{tu}\sum_{i=1}^{N}\Big(c_{N}^{K}(i)\Big)^{2}\mathbb{E}_{\theta}\Big[\Big|(\lambda_{s}^{i,N}-\mu\ell_{N}(i))\Big|^{2}\Big]^{\frac{1}{2}}ds\\
    \le& \frac{C}{Kt}\int_{0}^{t}\sum_{i=1}^{K}\mathbb{E}_{\theta}\Big[\Big|\lambda_{s}^{i,N}-\mu\ell_{N}(i)\Big|^{2}\Big]^{\frac{1}{2}}ds+\frac{C}{Nt}\int_{0}^{t}\sum_{i=1}^{N}\mathbb{E}_{\theta}\Big[\Big|\lambda_{s}^{i,N}-\mu\ell_{N}(i)\Big|^{2}\Big]^{\frac{1}{2}}ds\\
     \le& \frac{C}{\sqrt{N}}+\frac{C}{t^{q}}.
\end{align*}
For the third term, we write
\begin{align*}
    &\mathbb{V}ar_{\theta}(\mathcal{A}^{u,3}_{N,K})=\frac{(A^{N,K}_{\infty,\infty})^2}{K^{4}\Delta^{2}t^{2}}\mathbb{V}ar_{\theta}\Big[\int_{0}^{ut}\Big(\sum_{j=1}^{N}c_{N}^{K}(j)(M_{s}^{j,N}-M_{\phi_{t,\Delta}(s)}^{j,N})\Big)^{2}ds\Big]
    \\ &=\frac{(A^{N,K}_{\infty,\infty})^2}{K^{4}\Delta^{2}t^{2}}\int_{0}^{ut}\!\!\int_{0}^{ut}\!\!
    \mathbb{C}ov_{\theta}\Big[\Big(\sum_{i=1}^{N}c_{N}^{K}(i)(M_{s}^{i,N}-M_{\phi_{t,\Delta}(s)}^{i,N})\Big)^{2},\Big(\sum_{j=1}^{N}c_{N}^{K}(j)(M_{s'}^{j,N}-M_{\phi_{t,\Delta}(s')}^{j,N})\Big)^{2}\Big]dsds'\\
    &=\frac{(A^{N,K}_{\infty,\infty})^2}{K^{4}\Delta^{2}t^{2}}\int_{0}^{ut}\int_{0}^{ut}
    \sum_{1\le i,i',j,j'\le N}\mathbb{C}ov_{\theta}\Big[c_{N}^{K}(i)c_{N}^{K}(i')(M_{s}^{i,N}-M^{i,N}_{\phi_{t,\Delta}(s)})(M_{s}^{i',N}-M^{i',N}_{\phi_{t,\Delta}(s)}),\\
    &\hskip5cm c_{N}^{K}(j)c_{N}^{K}(j')(M_{s'}^{j,N}-M^{j,N}_{\phi_{t,\Delta}(s')})(M_{s'}^{j',N}-M^{j',N}_{\phi_{t,\Delta}(s')})\Big]dsds'\\
    &= \frac{(A^{N,K}_{\infty,\infty})^2}{K^{4}\Delta^{2}t^{2}}\int_{0}^{ut}\int_{0}^{ut} \Big( \boldsymbol{1}_{|s-s'|> 3\Delta}+\boldsymbol{1}_{|s-s'|\le 3\Delta}\Big)\sum_{1\le i,i',j,j'\le N}\\
    &\hskip3cm \mathbb{C}ov_{\theta}\Big[c_{N}^{K}(i)c_{N}^{K}(i')(M_{s}^{i,N}-M^{i,N}_{\phi_{t,\Delta}(s)})(M_{s}^{i',N}-M^{i',N}_{\phi_{t,\Delta}(s)}),\\
    &\hskip5cm c_{N}^{K}(j)c_{N}^{K}(j')(M_{s'}^{j,N}-M^{j,N}_{\phi_{t,\Delta}(s')})(M_{s'}^{j',N}-M^{j',N}_{\phi_{t,\Delta}(s')})\Big]dsds'.
    \end{align*}
But on $\Omega_{N,K},$ we have
    \begin{align*}
    &\sum_{i,i',j,j'=1}^{N}\int_{0}^{ut}\int_{0}^{ut}\boldsymbol{1}_{|s-s'|\le 3\Delta}\mathbb{C}ov_{\theta}\Big[c_{N}^{K}(i)c_{N}^{K}(i')(M_{s}^{i,N}-M^{i,N}_{\phi_{t,\Delta}(s)})(M_{s}^{i',N}-M^{i',N}_{\phi_{t,\Delta}(s)}),\\
    &\hskip3cm c_{N}^{K}(j)c_{N}^{K}(j')(M_{s'}^{j,N}-M^{j,N}_{\phi_{t,\Delta}(s')})(M_{s'}^{j',N}-M^{j',N}_{\phi_{t,\Delta}(s')})\Big]dsds'\\
    \le &\int_{0}^{ut}\int_{0}^{ut}\boldsymbol{1}_{|s-s'|\le 3\Delta}\mathbb{E}_{\theta}\Big[\Big(\sum_{i=1}^{N}c_{N}^{K}(i)(M_{s}^{i,N}-M^{i,N}_{\phi_{t,\Delta}(s)})\Big)^{4}\Big]^\frac{1}{2}\\
    &\mathbb{E}_{\theta}\Big[\Big(\sum_{i=1}^{N}c_{N}^{K}(i)(M_{s'}^{i,N}-M^{i,N}_{\phi_{t,\Delta}(s')})\Big)^{4}\Big]^\frac{1}{2}dsds'\\
    \le &\int_{0}^{ut}\int_{0}^{ut}\boldsymbol{1}_{|s-s'|\le 3\Delta}\mathbb{E}_{\theta}\Big[\Big(\sum_{i=1}^{N}(c_{N}^{K}(i))^2(Z_{s}^{i,N}-Z^{i,N}_{\phi_{t,\Delta}(s)})\Big)^{2}\Big]^\frac{1}{2}\\
    &\mathbb{E}_{\theta}\Big[\Big(\sum_{i=1}^{N}(c_{N}^{K}(i))^2(Z_{s'}^{i,N}-Z^{i,N}_{\phi_{t,\Delta}(s')})\Big)^{2}\Big]^\frac{1}{2}dsds'\\
    \le& C\int_{0}^{ut}\int_{0}^{ut} \boldsymbol{1}_{|s-s'|\le 3\Delta}\Et\Big[\Big(K\Big(\bar{Z}_{s}^{N,K}-\bar{Z}_{\phi_{t,\Delta}(s)}^{N,K}\Big)+\frac{K^2}{N}\Big(\bar{Z}_{s}^{N}-\bar{Z}_{\phi_{t,\Delta}(s)}^{N}\Big)\Big)^{2}\Big]^{\frac{1}{2}}\\
    &\hskip3cm\Et\Big[\Big(K\Big(\bar{Z}_{s'}^{N,K}-\bar{Z}_{\phi_{t,\Delta}(s')}^{N,K}\Big)+\frac{K^2}{N}\Big(\bar{Z}_{s'}^{N}-\bar{Z}_{\phi_{t,\Delta}(s')}^{N}\Big)\Big)^{2}\Big]^{\frac{1}{2}}dsds'\\
    \le& Ct\Delta^{3}K^{2}
\end{align*}
By \cite[Step 6 of the proof of Lemma 30]{A}, we already have,
when $|s-s'|\ge 3\Delta,$ that
\begin{align*}
    &\mathbb{C}ov_{\theta}[(M_{s}^{i,N}-M^{i,N}_{\phi_{t,\Delta}(s)})(M_{s}^{i',N}-M^{i',N}_{\phi_{t,\Delta}(s)}),
    (M_{s'}^{j,N}-M^{j,N}_{\phi_{t,\Delta}(s')})(M_{s'}^{j',N}-M^{j',N}_{\phi_{t,\Delta}(s')})]\\
    \le&  C (\indiq_{\{i=i'\}}+\indiq_{\{j=j'\}})  t^{1/2}\Delta^{1-q}.
\end{align*}
Hence
\begin{align*}
     &\boldsymbol{1}_{\Omega_{N,K}}\sum_{i,i',j,j'=1}^{N}\int_{0}^{t}\int_{0}^{t}\boldsymbol{1}_{|s-s'|\ge 3\Delta}\mathbb{C}ov_{\theta}\Big[c_{N}^{K}(i)c_{N}^{K}(i')\Big(M_{s}^{i,N}-M^{i,N}_{\phi_{t,\Delta}(s)}\Big)\Big(M_{s}^{i',N}-M^{i',N}_{\phi_{t,\Delta}(s)}\Big),\\
&\hskip4cm c_{N}^{K}(j)c_{N}^{K}(j')\Big(M_{s'}^{j,N}-M^{j,N}_{\phi_{t,\Delta}(s')}\Big)\Big(M_{s'}^{j',N}-M^{j',N}_{\phi_{t,\Delta}(s')}\Big)\Big]dsds'
   \\
\le& \indiq_{\Omega_{N,K}}  C t^{5/2}\Delta^{1-q}\Big(\sum_{i=1}^{N}(c_{N}^{K}(i))^2\Big)\Big(\sum_{i=1}^{N}c_{N}^{K}(i)\Big)^{2}\le CK^{3} t^{5/2}\Delta^{1-q}.
\end{align*}
Overall,  we have, on $\Omega_{N,K}$
\begin{align*}
    &\mathbb{V}ar_{\theta}(\mathcal{A}^{u,3}_{N,K})
    \le \frac{1}{K^{4}\Delta^{2}t^{2}}\Big(A^{N,K}_{\infty,\infty}\Big)^{2}\Big(\frac{K^{3} t^{5/2}}{\Delta^{q-1}}+t\Delta^{3}K^{2}\Big)
    \le C\Big(\frac{K\sqrt{t}}{\Delta^{q+1}}+\frac{\Delta}{t}\Big)
\end{align*}
Recalling (\ref{ee3}),  by Lemma \ref{constant}, we have on $\Omega_{N,K},$
\begin{align*}
    & \Big|\mathbb{E}_{\theta}[\mathcal{A}^{u,3}_{N,K}]-\frac{u(A^{N,K}_{\infty,\infty})^{2}}{2K^{2}}\Big|\\
    =&\frac{1}{\Delta tK^{2}}\Big|A^{N,K}_{\infty,\infty}\int_{0}^{ut}\sum_{j=1}^{N}\Big\{\Big(c_{N}^{K}(j)\Big)^{2}\mathbb{E}_{\theta}\Big[Z_{s}^{j,N}-Z_{\phi_{t,\Delta}}(s)\Big]\Big\}ds-\frac{u\Delta t(A^{N,K}_{\infty,\infty})^{2}}{2}\Big|\\
    \le& \frac{C}{\Delta^q}+\frac{C}{t}.
    \end{align*}
Gathering the previous results, we obtain:
\begin{align*}
    &\mathbb{E}\Big[\boldsymbol{1}_{\Omega_{N,K}}\Big|\Big[\mathcal{L}^{t,\Delta}_{N,K}(.),\mathcal{L}^{t,\Delta}_{N,K}(.)]_{u}-\frac{u(A^{N,K}_{\infty,\infty})^{2}}{2K^{2}}\Big|\Big]\\
    \le& \mathbb{E}\Big[\boldsymbol{1}_{\Omega_{N,K}}\Big\{\Big|\mathcal{A}^{u,1}_{N,K}\Big|+\Big|\mathcal{A}^{u,2}_{N,K}\Big|+\mathbb{V}ar_{\theta}(\mathcal{A}^{u,3}_{N,K})^\frac{1}{2}+\Big|\mathbb{E}_{\theta}[\mathcal{A}^{u,3}_{N,K}]-\frac{u(A^{N,K}_{\infty,\infty})^{2}}{2K^{2}}\Big|\Big\}\Big]\\
\le&  \frac{C}{K\Delta}+C\sqrt{\frac{\Delta}{t}}+\frac{C}{\sqrt{N}}+\frac{C}{t^q}+C\Big(\frac{K\sqrt{t}}{\Delta^{q+1}}+\frac{\Delta}{t}\Big)^{\frac{1}{2}}+\frac{C}{\Delta^q}+\frac{C}{t}\\ 
\le& C\Big(\frac{1}{K\Delta}+\frac{1}{\sqrt{N}}+\Big(\frac{K\sqrt{t}}{\Delta^{q+1}}\Big)^\frac{1}{2}+\sqrt{\frac{\Delta}{t}}\Big).
\end{align*}
The proof is finished.
\end{proof}

By Lemmas \ref{Agamma}, \ref{jum}, \ref{bro} and \cite[Theorem VIII.3.8]{B}, we get the following corollary.
\begin{cor}\label{main2}
Assume $K\le N$. For $t\geq 1$, set $\Delta_t= t/(2 \lfloor t^{1-4/(q+1)}\rfloor) \sim t^{4/(q+1)}/2$ (for $t$ large).
We always work in the asymptotic  $(N,K,t)\to (\infty,\infty,\infty)$
and in the regime where 
$\frac 1{\sqrt K} + \frac NK \sqrt{\frac{\Delta_t}t}+ \frac{N}{t\sqrt K}+Ne^{-c_{p,\Lambda}K} \to 0$ and where 
$\frac{K}{N}\to \gamma\le 1$. It holds true that
$$(\mathcal{L}^{t,\Delta_t}_{N,K}(u))_{u\ge 0} \stackrel{d}{\longrightarrow}\frac{1}{\sqrt{2}}\Big(\frac{1-\gamma}{(1-\Lambda p)}+\frac{\gamma}{(1-\Lambda p)^3}\Big)(B_{u})_{u\ge 0}$$
for the Skorokhod topology, where $B$ is a standard Brownian motion.
\end{cor}
Next, we are going to give the proof of Lemma \ref{D33}.
\begin{proof}

By Corollaries \ref{LmathbX} and \ref{main2}, we conclude that  
$$
\frac{K}{\sqrt{t\Delta_t}}(\mathbb{X}_{\Delta_t,t,v}^{N,K})_v \stackrel{d}{\longrightarrow}\frac{1}{\sqrt{2}}\Big(\frac{1-\gamma}{(1-\Lambda p)}+\frac{\gamma}{(1-\Lambda p)^3}\Big)(B_{2v}-B_v)_v$$
as desired.
\end{proof}

\subsection{Proof of theorem \ref{corX}}

We notice that: $\mathbb{X}_{2\Delta,t,1}^{N,K}=\mathbb{X}_{\Delta,t,\frac{1}{2}}^{N,K}.$ By Lemma \ref{D33}, we have
$$
\frac{K}{N}\sqrt{\frac{t}{\Delta_t}}\frac{N}{t}\Big(\mathbb{X}_{\Delta_t,t,1}^{N,K}+2\mathbb{X}_{\Delta_t,t,\frac{1}{2}}^{N,K}\Big)\stackrel{d}{\longrightarrow}\mathcal{N}\Big(0,\frac{3}{2}\Big(\frac{1-\gamma}{(1-\Lambda p)}+\frac{\gamma}{(1-\Lambda p)^3}\Big)^2\Big).
$$
By Lemma \ref{D3}, we conclude that
\begin{align*}
&\frac{K}{N}\sqrt{\frac{t}{\Delta}}\frac{N}{t}\mathbb{E}\Big[\indiq_{\Omega_{N,K}}\Big|D_{\Delta,t}^{N,K,3}+2D_{2\Delta,t}^{N,K,3}-2\mathbb{X}_{2\Delta,t,1}^{N,K}-\mathbb{X}_{\Delta,t,1}^{N,K}\Big|\Big]\\
\le& \frac{CK}{N\Delta}+\frac{CKt^{\frac{3}{4}}}{\Delta^{1+\frac{q}{2}}\sqrt{N}}+\frac{C\sqrt{K}}{\sqrt{N\Delta}}+\frac{Ct^\frac{3}{4}}{\Delta^{1+\frac{q}{2}}}.
\end{align*}
Finally, by  Corollary \ref{D124}, we have
\begin{align*}
&\lim \omg\frac{K}{N}\sqrt{\frac{t}{\Delta_t}}\Big(\mathcal{X}_{\Delta_t,t}^{N,K}-\cX^{N,K}_{\infty,\infty}\Big)\\
=&\lim \omg\frac{K}{N}\sqrt{\frac{t}{\Delta_t}}\Big\{D_{\Delta_t,t}^{N,K,3}+2D_{2\Delta_t,t}^{N,K,3}\Big\}\\
=&\lim \frac{K}{N}\sqrt{\frac{t}{\Delta_t}}\frac{N}{t}\Big(\mathbb{X}_{\Delta_t,t,1}^{N,K}+2\mathbb{X}_{\Delta_t,t,\frac{1}{2}}^{N,K}\Big),
\end{align*}
which goes in distribution to $
\mathcal{N}\Big(0,\frac{3}{2}\Big(\frac{1-\gamma}{(1-\Lambda p)}+\frac{\gamma}{(1-\Lambda p)^3}\Big)^2\Big)$.

\section{The final result in the subcritical case}\label{finsecsub}

We can, at last, give the proof of Theorem \ref{mainsubcr}.

\begin{proof}
One can directly check that 
$\Psi^{(3)}\Big(\frac{\mu}{1-\Lambda p},\frac{(\mu\Lambda)^{2}p(1-p)}{(1-\Lambda p)^{2}},
\frac{\mu}{(1-\Lambda p)^{3}}\Big)=p.$
By the Lagrange mean value theorem, there exist some vectors $\boldsymbol{C_{N,K,t}^{i}}$ for $i=1,2,3,$ lying 
in the segment joining the points 
$(\varepsilon_{t}^{N,K},\mathcal{V}_{t}^{N,K},\mathcal{X}_{\Delta,t}^{N,K})$ and 
$\boldsymbol{C}:=\Big(\frac{\mu}{1-\Lambda p},\frac{(\mu\Lambda)^2 p(1-p)}{(1-\Lambda p)^2},\frac{\mu}
{(1-\Lambda p)^3}\Big)$, such that:
\begin{align*}
\hat p_{N,K,t}-p=&\Psi^{(3)}(\varepsilon_{t}^{N,K},\mathcal{V}_{t}^{N,K},\mathcal{X}_{\Delta,t}^{N,K})-p\\
     =&\Psi^{(3)}(\varepsilon_{t}^{N,K},\mathcal{V}_{t}^{N,K},\mathcal{X}_{\Delta,t}^{N,K})-\Psi^{(3)}\Big(\frac{\mu}{1-\Lambda p},\frac{(\mu\Lambda)^{2}p(1-p)}{(1-\Lambda p)^{2}},\frac{\mu}{(1-\Lambda p)^{3}}\Big)\\
    =&\frac{\partial\Psi^{(3)}}{\partial x}(\boldsymbol{C_{N,K,t}^{1}})\Big(\varepsilon_{t}^{N,K}-\frac{\mu}{1-\Lambda p}\Big)+
\frac{\partial\Psi^{(3)}}{\partial y}(\boldsymbol{C_{N,K,t}^{2}})\Big(\mathcal{V}_{t}^{N,K}-\frac{(\mu\Lambda)^{2}p(1-p)}{(1-\Lambda p)^{2}}\Big)\\
     &+\frac{\partial\Psi^{(3)}}{\partial z}(\boldsymbol{C_{N,K,t}^{3}})\Big(\mathcal{X}_{\Delta,t}^{N,K}-\frac{\mu}{(1-\Lambda p)^{3}}\Big).
     \end{align*}
From the first paragraph of \cite[Section 9]{D}, we know that in the asymptotic  
$(N,K,t)\to (\infty,\infty,\infty)$
and in the regime $\frac 1{\sqrt K} + \frac NK \sqrt{\frac{\Delta_t}t}+ \frac{N}{t\sqrt K}+Ne^{-c_{p,\Lambda}K} \to 0$, it holds
that $(\varepsilon_{t}^{N,K},\mathcal{V}_{t}^{N,K},\mathcal{X}_{\Delta,t}^{N,K})\to \boldsymbol{C}$ in probability.  
This implies that the three vectors $\boldsymbol{C}_{N,K,t}^{i}$, $i=1,2,3$, all converge to 
$\boldsymbol{C}:=\Big(\frac{\mu}{1-\Lambda p},\frac{(\mu\Lambda)^2 p(1-p)}{(1-\Lambda p)^2},\frac{\mu}{(1-\Lambda p)^3}\Big)$ in probability in the same regime.

\vip
We define from $D':= \{(u, v, w)\in \mathbb{R}^{3}: w >u> 0\  and\  v> 0\}$ to $\mathbb{R}^{3}$ 
$$\Psi^{(1)}(u,v,w)=u\sqrt{\frac{u}{w}},\quad  \Psi^{(2)}(u,v,w)=\frac{v+(u-\Psi^{(1)})^{2}}{u(u-\Psi^{(1)})}.$$
Then we have $\Psi^{(3)}(u,v,w)=\frac{1-u^{-1}\Psi^{(1)}}{\Psi^{(2)}}$ in $D'$.

Some tedious but direct computations show that
\begin{gather*}
\frac{\partial\Psi^{(1)}}{\partial y}(C)=0,\quad \frac{\partial\Psi^{(1)}}{\partial z}(C)=\frac{-(1-\Lambda p)^3}{2},\quad  \frac{\partial\Psi^{(2)}}{\partial y}(C)=\frac{(1-\Lambda p)^2}{\mu^2\Lambda p},\\ 
\frac{\partial\Psi^{(2)}}{\partial z}(C)=\Big\{\frac{-2\frac{\partial\Psi^{(1)}}{\partial z}}{u}+\frac{\Psi^{(2)}\frac{\partial\Psi^{(1)}}{\partial z}}{(u-\Psi^{(1)})}\Big\}(C)=\frac{(1-\Lambda p)^4(2p-1)}{2\mu p},\\
\frac{\partial\Psi^{(3)}}{\partial y}(C)=-\frac{\frac{\Psi^{(2)}\frac{\partial\Psi^{(1)}}{\partial y}}{u}+(1-\frac{\Phi^{(1)}}{u})\frac{\partial\Psi^{(2)}}{\partial y}}{(\Psi^{(2)})^2}(C)=-\frac{(\Lambda p-1)^2}{\mu^2\Lambda^2},\\
\frac{\partial\Psi^{(3)}}{\partial z}(C)=-\frac{\frac{\Psi^{(2)}\frac{\partial\Psi^{(1)}}{\partial z}}{u}+(1-\frac{\Phi^{(1)}}{u})\frac{\partial\Psi^{(2)}}{\partial z}}{(\Psi^{(2)})^2}(C)=\frac{(1-\Lambda p)^4(1-p)}{\mu \Lambda}.
\end{gather*}

{\it Case 1.}
In the regime with dominating term $\frac 1{\sqrt K}$, i.e. when
$[\frac 1{\sqrt K}]/[ \frac NK \sqrt{\frac{\Delta_t}t}+ \frac{N}{t\sqrt K}]\to \infty$,
we write
\begin{align*}
\sqrt{K}\Big[\hat p_{N,K,t}-p\Big] =&\sqrt{K}\frac{\partial\Psi^{(3)}}{\partial x}(\boldsymbol{C_{N,K,t}^{1}})\Big(\varepsilon_{t}^{N,K}-\frac{\mu}{1-\Lambda p}\Big)\\
&+
\sqrt{K}\frac{\partial\Psi^{(3)}}{\partial y}(\boldsymbol{C_{N,K,t}^{2}})\Big(\mathcal{V}_{t}^{N,K}-\frac{(\mu\Lambda)^{2}p(1-p)}{(1-\Lambda p)^{2}}\Big)\notag\\
&+\sqrt{K}\frac{\partial\Psi^{(3)}}{\partial z}(\boldsymbol{C_{N,K,t}^{3}})\Big(\cX_{t,\Delta_{t}}^{N,K}-\frac{\mu}{(1-\Lambda p)^{3}}\Big).\notag
\end{align*}
By Lemmas \ref{barell} and \ref{W} and by Theorem \ref{corX}, we have
$$
\sqrt{K}\frac{\partial\Psi^{(3)}}{\partial x}(\boldsymbol{C_{N,K,t}^{1}})\Big(\varepsilon_{t}^{N,K}-\frac{\mu}{1-\Lambda p}\Big)+\sqrt K \frac{\partial\Psi^{(3)}}{\partial z}(\boldsymbol{C_{N,K,t}^{3}})\Big(\cX_{t,\Delta_{t}}^{N,K}-\frac{\mu}{(1-\Lambda p)^{3}}\Big)\stackrel{d}{\longrightarrow} 0.
$$
Next, we notice that 
$$\frac{\partial\Psi^{(3)}}{\partial y}(\boldsymbol{C_{N,K,t}^{2}})\stackrel{t,K,N\to\infty}{\longrightarrow}\frac{(1-\Lambda p)^{2}}{(\mu\Lambda)^{2}}.$$
So by Theorems \ref{21} and \ref{VVNK}, we conclude that
$$
\sqrt{K}\frac{\partial\Psi^{(3)}}{\partial y}(\boldsymbol{C_{N,K,t}^{2}})\Big(\mathcal{V}_{t}^{N,K}-\frac{(\mu\Lambda)^{2}p(1-p)}{(1-\Lambda p)^{2}}\Big)\stackrel{d}{\longrightarrow} \mathcal{N}\Big(0,\frac{p^2(1-p)^{2}}{\mu^{4}}\Big).
$$

{\it Case 2.}
In the regime with dominating term $\frac{N}{t\sqrt K}$, i.e. when
$[\frac{N}{t\sqrt K}]/[\frac 1{\sqrt K}+\frac NK \sqrt{\frac{\Delta_t}t}]\to \infty$,
we write 
\begin{align*}
\frac{t\sqrt{K}}N\Big[\hat p_{N,K,t}-p\Big] =&\frac{t\sqrt{K}}N\frac{\partial\Psi^{(3)}}{\partial x}(\boldsymbol{C_{N,K,t}^{1}})\Big(\varepsilon_{t}^{N,K}-\frac{\mu}{1-\Lambda p}\Big)\\
&+
\frac{t\sqrt{K}}N\frac{\partial\Psi^{(3)}}{\partial y}(\boldsymbol{C_{N,K,t}^{2}})\Big(\mathcal{V}_{t}^{N,K}-\frac{(\mu\Lambda)^{2}p(1-p)}{(1-\Lambda p)^{2}}\Big)\notag\\
&+\frac{t\sqrt{K}}N\frac{\partial\Psi^{(3)}}{\partial z}(\boldsymbol{C_{N,K,t}^{3}})\Big(\cX_{t,\Delta_{t}}^{N,K}-\frac{\mu}{(1-\Lambda p)^{3}}\Big).\notag
\end{align*}
By Lemmas \ref{barell} and \ref{W} and by Theorem \ref{corX}, we have
$$
\frac{t\sqrt{K}}N\frac{\partial\Psi^{(3)}}{\partial x}(\boldsymbol{C_{N,K,t}^{1}})\Big(\varepsilon_{t}^{N,K}-\frac{\mu}{1-\Lambda p}\Big)+\frac{t\sqrt{K}}N\frac{\partial\Psi^{(3)}}{\partial z}(\boldsymbol{C_{N,K,t}^{3}})\Big(\cX_{t,\Delta_{t}}^{N,K}-\frac{\mu}{(1-\Lambda p)^{3}}\Big)\stackrel{d}{\longrightarrow} 0.
$$
Finally, using Theorems \ref{21} and \ref{VVNK}, we find that
\begin{align*}
&\frac{t\sqrt{K}}{N}\Big[\hat p_{N,K,t}-p\Big]\stackrel{d}{\longrightarrow} 
\mathcal{N}\Big(0,\frac{2 (1-\Lambda p)^2}{\mu^2 \Lambda^4}\Big)
\end{align*}

\vip
{\it Case 3.}
In the regime with dominating term $\frac{N}{K}\sqrt{\frac{\Delta_t}t}$, i.e. when
$[\frac{N}{K}\sqrt{\frac{\Delta_t}t}]/[\frac 1{\sqrt K}+ \frac{N}{t\sqrt K} ]\to \infty$, $\lim_{N,K\to\infty} \frac{K}{N}=\gamma\le 1$,
by Lemmas \ref{barell} and \ref{W} and Theorem \ref{corX}, we have
\begin{align*}
&\frac{N}{K}\sqrt{\frac{\Delta_t}t}\Big\{\frac{\partial\Psi^{(3)}}{\partial x}(\boldsymbol{C_{N,K,t}^{1}})\Big(\varepsilon_{t}^{N,K}-\frac{\mu}{1-\Lambda p}\Big)+\frac{\partial\Psi^{(3)}}{\partial y}(\boldsymbol{C_{N,K,t}^{2}})\Big(\cV_{t}^{N,K}-\frac{(\mu\Lambda)^{2}p(1-p)}{(1-\Lambda p)^{2}}\Big)\\
&\hskip6cm+\frac{\partial\Psi^{(3)}}{\partial z}(\boldsymbol{C_{N,K,t}^{3}})\Big(\cX_{\infty,\infty}^{N,K}-\frac{\mu}{(1-\Lambda p)^{3}}\Big)\Big\}\stackrel{d}{\longrightarrow} 0.
\end{align*}
Hence it suffices to study 
$$
\frac{N}{K}\sqrt{\frac{\Delta_t}t} \frac{\partial\Psi^{(3)}}{\partial z}(\boldsymbol{C_{N,K,t}^{3}})
\Big( \cX_{t,\Delta_t}^{N,K} - \cX_{\infty,\infty}^{N,K}\Big).
$$
But by Theorem \ref{corX}, we conclude that
\begin{align*}
&\frac{K}{N}\sqrt{\frac{t}{\Delta_{t}}}\Big[\hat p_{N,K,t}-p\Big]\stackrel{d}{\longrightarrow} \mathcal{N}\Big(0,\frac{3(1-p)^2}{2\mu^2\Lambda^2}\Big((1-\gamma)(1-\Lambda p)^{3}+\gamma(1-\Lambda p)\Big)^2\Big).
\end{align*}
\end{proof}

Next, we are going to prove the proposition \ref{pzero}.

\begin{proof}
We remark the for the case $p=0,$ the result of Theorem \ref{corX} holds true (we do not need the limit of $\frac{K}{N}$ anymore). For matrix, it is easy to check $\bar{\ell}_{N}^{K}=1$, $\cV_\infty^{N,K}=0$ and  $\cX^{N,K}_{\infty,\infty}=\mu.$ 
 We define 
 \begin{align*}
f(u,v,w):=\frac{u(w-u)}{w+\sqrt{wu}}\qquad when \text{ $u>0$, $w>0$}\qquad \text{and  $f=0$ otherwise.}
  \end{align*}

 By \cite[Lemma 6.3]{D}, we have 
 $$
\lim \omg\frac{K}{N}\sqrt{\frac{t}{\Delta_{t}}}(\varepsilon_{t}^{N,K}-\mu)=0.
 $$

 Hence by Theorem \ref{corX},  we have:
 \begin{align*}
    \omg\frac{K}{N}\sqrt{\frac{t}{\Delta_{t}}}(\mathcal{X}_{\Delta_t,t}^{N,K}-\varepsilon_{t}^{N,K})\stackrel{d}{\to}\mathcal{N}\Big(0,\frac{3}{2}\Big).
 \end{align*}
 
 By \cite[Lemma 6.3, corollary 8.9]{D}, we have in the regime $ \frac NK \sqrt{\frac{\Delta_t}t}+ \frac{N}{t\sqrt K}+Ne^{-c_{p,\Lambda}K} \to 0$,
 $$
 \lim \varepsilon_{t}^{N,K}=\lim \mathcal{X}_{\Delta_t,t}^{N,K}=\mu\ \textit{ in probability.}
 $$
Overall, we conclude that 
 
$$
\omg\frac{K}{N}\sqrt{\frac{t}{\Delta_{t}}}f(\varepsilon_{t}^{N,K},\mathcal{V}_{t}^{N,K},\mathcal{X}_{t,\Delta_{t}}^{N,K})\stackrel{d}{\longrightarrow} \mathcal{N}\Big(0,\frac{3}{8}\Big).
$$ 
By theorem \ref{VVNK}, we have
$$\omg\frac{t\sqrt{K}}{N}\cV_{t}^{N,K}\stackrel{d}{\longrightarrow}\mathcal{N}
\Big(0,2\mu^2\Big).$$
Hence in the regime $[\frac{N}{t\sqrt K}]/[\frac NK \sqrt{\frac{\Delta_t}t}]^2\to \infty,$ we have 
$$
\lim [\frac NK \sqrt{\frac{\Delta_t}t}]^{-2}\Big|\mathcal{V}_{t}^{N,K}\Big|=\infty\qquad \textit{in probability}
$$
and since $\Psi^{(3)}(u,v,w)=\frac{f^2}{v+f^2} 1_{\{ v>0\}}$, we deduce that
$$
\hat p_{N,K,t} = \Psi^{(3)}(\varepsilon_{t}^{N,K},\mathcal{V}_{t}^{N,K},\mathcal{X}_{t,\Delta_{t}}^{N,K})\stackrel{d}{\longrightarrow} 0.
$$

In the regime $[\frac{N}{K}\sqrt{\frac{\Delta_t}t}]^2/[ \frac{N}{t\sqrt K} ]\to \infty$, we have
$$ \frac{f^2(\varepsilon_{t}^{N,K},\mathcal{V}_{t}^{N,K},\mathcal{X}_{t,\Delta_{t}}^{N,K})}{\mathcal{V}_{t}^{N,K}+ f^2(\varepsilon_{t}^{N,K},\mathcal{V}_{t}^{N,K},\mathcal{X}_{t,\Delta_{t}}^{N,K})} \to 1 \quad\text{in probability}$$
and
$$
\lim P\Big(\mathcal{V}_{t}^{N,K}>0\Big)=\frac{1}{2}.
$$
Therefore
$$
\hat p_{N,K,t} \stackrel{d}{\longrightarrow} X.
$$
\end{proof}

\section{Matrix analysis for the supercritical case}\label{secsup}

We now turn to the supercritical case and thus assume $(A)$.

\vip

We recall that the matrix $A_N$ is defined by $A_{N}(i,j):=N^{-1}\theta_{ij}$ for $i,j\in\{1,...,N\}$.
We assume here that $p\in (0,1]$ and we introduce the events:
\begin{gather*}
\Omega_{N}^{2}:=\Big\{\frac{1}{N}\sum_{i=1}^{N}\sum_{j=1}^{N}A_{N}(i,j)>\frac{p}{2} \quad\hbox{and}
\quad |NA^{2}_{N}(i,j)-p^{2}|<\frac{p^{2}}{2N^{3/8}} \quad 
 \text{for all  $i,j=1,...,N$ \Big\}},\\
\Omega_{N}^{K,2}:=\Big\{\frac{1}{K}\sum_{i=1}^{K}\sum_{j=1}^{N}A_{N}(i,j)>\frac{p}{2}\Big\} \cap \Omega_{N}^{2}.
\end{gather*}
We know from \cite[Lemma 10.1]{D} that
\begin{lemma}\label{xyz}
There are some constants $C,c>0$ such that for all $1\leq K \leq N$, we have
$$P(\Omega_N^{K,2}) \geq 1- C e^{-c N^{1/4}}.$$
\end{lemma}

Next, we recall the definition of $\rho_N$ of \cite[Lemma 34]{A}.

\begin{lemma}\label{defsupf}
Assume that $p\in(0,1]$. On $\Omega_N^2$, 
the spectral radius $\rho_N$ of $A_N$ is a simple eigenvalue
of $A_N$ and $\rho_N \in [p(1-1/(2N^{3/8})),p(1+1/(2N^{3/8}))]$. There is a unique eigenvector
$\bV_N \in (\rr_+)^N$ of $A_N$ for the eigenvalue $\rho_N$ such that $||V_N||_2=\sqrt N$. 
We also have $V_N(i)>0$ for all $i=1,\dots,N$.
\end{lemma}

\begin{remark}\label{rhotop}
On $\Omega_N^2$, we define $\alpha_N$ as the unique number such that $\rho_N\int^\infty_0 e^{-\alpha_N s}\phi(s)ds=1.$ 
From $\rho_N \in [p(1-1/(2N^{3/8})),p(1+1/(2N^{3/8}))]$, we conclude that $\lim_{N\to\infty}\alpha_N=\alpha_0,$ 
where $\alpha_0>0$ is defined by 
$p\int^\infty_0 e^{-\alpha_0 s}\phi(s)ds=1.$ Recall that $\Lambda p=p\int^\infty_0 \phi(s)ds>1$ by $A$.
Furthermore, under assumption $(A)$, we have $\alpha_N=\rho_N-b$ 
and $\alpha_0=p-b.$
\end{remark}

We introduce $\bV_N^K:=I_K\bV_N$, as well as $\bar{V}^K_N=\frac{1}{K}\sum_{i=1}^KV_N(i)$ and we write 
$\bar{V}_N:=\bar{V}^N_N$.
We first recall the following result, see \cite[Proposition 10.6]{D}.

\begin{lemma}\label{Up11}
There is $N_0\geq 1$ and $C>0$ (depending only on $p$) such that for all $N\geq N_0$ and all
$K\in \{1,\dots,N\}$,
$$
\mathbb{E}\Big[\indiq_{\Omega_{N}^{K,2}}\Big|\mathcal{U}_{\infty}^{N,K}-\Big(\frac{1}{p}-1\Big)\Big|\Big]\le \frac{C}{\sqrt{K}},
$$
where $\mathcal{U}_{\infty}^{N,K}:=\frac{N}{K(\bar{V}_{N}^{K})^{2}}\sum_{i=1}^{K}(V_{N}(i)-\bar{V}_{N}^{K})^{2}.$
\end{lemma}

Next we state some properties of the vector $\bV_N^K.$
\begin{lemma}\label{VNK}
Assume that $p\in(0,1]$. The following properties hold true.
\vip
(i) $\lim_{N,K\to\infty} \indiq_{\Omega_N^{K,2}}\sqrt{K}\bar{V}_{N}^{K}/\|\boldsymbol{V}_{N}^{K}\|_{2}=1$ a.s. 
For any fixed $i\in \mathbb{N}$, $\lim_{N\to\infty}\indiq_{\Omega_N^{K,2}}V_N(i)=1$ a.s.
\vip
(ii) There is $C>0$ such that on $\Omega_N^{K,2}$, we have 
$\frac{N}{K}\sum_{i=1}^K(V_N(i)-\bar{V}_N^K)^2\le CN^{\frac{1}{4}}$ a.s.
\vip
(iii) There is $C>0$ such that
$\mathbb{E}[\boldsymbol{1}_{\Omega_N^{K,2}}\frac{N}{K}\sum_{i=1}^K(V_N(i)-\bar{V}_N^K)^2]\le C$.
\end{lemma}

\begin{proof}
By \cite[Step 1 of the proof of Lemma 39]{A}, on $\Omega_N^{K,2}\subset \Omega_N^2$, we  have, for any $i=1,...,N$,
\begin{equation}\label{zyx}
V_N(i) \in [\kappa_N(1-1/(2N^{3/8})),\kappa_N(1+1/(2N^{3/8}))],
\end{equation}
where $\kappa_N=p^2 \rho_N^{-2}N^{-1}\sum_{j=1}^N V_N(j).$
Then the first result of $(i)$ is obvious, by Lemma \ref{xyz}. 
For the second result in $(i),$ we only need to verify that $\lim_{N\to\infty}\kappa_N=1.$  But 
we easily deduce from \eqref{zyx} that $\lim_{N\to\infty} \kappa_N =\lim_{N\to\infty} \bar V_N$, so that
$\lim_{N\to\infty}\kappa_N =\lim_{N\to\infty}\|\bV_N\|_2 / \sqrt{N} =1$.

\vip

By \cite[Lemma 35-(ii)]{A}, we already know that $V_N(i)\in[\frac{1}{2},2]$, for any $i=1,...,N$, 
on $\Omega_N^2$. By \eqref{zyx}, we deduce that $\kappa_N\in[\frac{1}{4},4]$. Using again
\eqref{zyx} gives us $(ii)$.

\vip

For (iii), since $\bar{V}_N^K$ is bounded by some constant $C$ (on $\Omega_N^{K,2}$) 
and by Lemma \ref{Up11}, we have 
\begin{align*}
\mathbb{E}[\boldsymbol{1}_{\Omega_N^{K,2}}\frac{N}{K}\sum_{i=1}^K(V_N(i)-\bar{V}_N^K)^2]=&\mathbb{E}[\boldsymbol{1}_{\Omega_N^{K,2}}(\bar{V}_N^K)^2\mathcal{U}_{\infty}^{N,K}]
\le C\mathbb{E}[\boldsymbol{1}_{\Omega_N^{K,2}}\mathcal{U}_{\infty}^{N,K}]\le C
\end{align*}
as desired.
\end{proof}

\begin{lemma}\label{ANVK}
Assume that $p\in(0,1]$. On $\Omega_N^{K,2}$, there exists some constant $N_0$, such that for all $N\ge N_0$, 
all $K \in \{1,\dots,N\}$,  for all $n\geq 2,$
$$
||I_KA_N^{n}\bun||_2\in \Big[\|\boldsymbol{V}_N^K\|_2 \rho_N^ne^{-CN^{-\frac{3}{16}}},\|\boldsymbol{V}_N^K\|_2 \rho_N^ne^{CN^{-\frac{3}{16}}}\Big].
$$
\end{lemma}
\begin{proof}

We write $A_N^n \bun = ||A_N^n\bun||_2(\|\boldsymbol{V}_N\|_2^{-1}\boldsymbol{V}_N + \bZ_{N,n})$,
where $\bZ_{N,n}=||A_N^n\bun||_2^{-1}A_N^n\bun - \|\boldsymbol{V}_N\|_2^{-1}\boldsymbol{V}_N$.  
By \cite[Lemma 35-(v)]{A} (with $r=2$), we know that 
$||\bZ_{N,n}||_2 \leq 3(2N^{-3/8})^{\lfloor n/2\rfloor+1}$.
We next write, for each $n\geq 0$, $A_N^{n+1} \bun= ||A_N^n\bun||_2(\|\boldsymbol{V}_N\|_2^{-1} \rho_N \boldsymbol{V}_N + A_N\bZ_{N,n})$ (recall that $A_N \boldsymbol{V}_N=\rho_N \boldsymbol{V}_N$).
Using the fact that $|||A_N|||_2 \leq 1$ , we conclude that
$$
\big| ||A_N^{n+1}\bun||_2- \rho_N ||A_N^{n}\bun||_2\big| 
\leq 3 ||A_N^{n}\bun||_2 ||A_N \bZ_{N,n}||_2 
\leq 3 ||A_N^{n}\bun||_2 (2N^{-3/8})^{\lfloor n/2\rfloor+1}.
$$
We now set $x_n=||A_N^{n}\bun||_2 / (\|\boldsymbol{V}_N\|_2 \rho_N^n)$. 
Dividing the previous inequality by $\rho_N^{n+1}\|\boldsymbol{V}_N\|_2$ and using that
$\rho_N \geq p/2$ on 
$\Omega_N^2$, see Lemma \ref{defsupf}, we have, for all $n\geq 0$,
$$
|x_{n+1}-x_n|\leq 3x_n (2N^{-3/8})^{\lfloor n/2\rfloor+1} /\rho_N \leq  6 x_n (2N^{-3/8})^{\lfloor n/2\rfloor+1} /p.
$$
We easily conclude, using that $x_0=1$, that for all $n\geq 0$,
$$
x_n \in \Big[\prod_{k=1}^n\Big(1-6(2N^{-3/8})^{\lfloor k/2 \rfloor+1} /p\Big),
\prod_{k=1}^n\Big(1+6(2N^{-3/8})^{\lfloor k/2\rfloor +1} /p\Big)\Big].
$$
We conclude that there exists a constant $C$ such that for all $N$ large enough, for all $n\geq 0$,
\begin{align*}
    e^{-CN^{-\frac{3}{16}}}\le\prod_{k=1}^n\Big(1-6(2N^{-3/8})^{\lfloor k/2 \rfloor+1} /p\Big)\le\prod_{k=1}^n\Big(1+6(2N^{-3/8})^{\lfloor k/2\rfloor +1} /p\Big)\le e^{CN^{-\frac{3}{16}}}.
\end{align*}
We obtain that for all $N$ large enough, on $\Omega_N^2$, for all $n\geq 0$, 
$x_n \in [e^{-CN^{-\frac{3}{16}}},e^{CN^{-\frac{3}{16}}}].$
We also have 
$x_n\in [1/2,2]$ for all $n\geq 0$. 
By definition of $x_n$, we conclude that for all $n\geq 0$,
$$
||A_N^{n}\bun||_2\in [\|\boldsymbol{V}_N\|_2 \rho_N^ne^{-CN^{-\frac{3}{16}}},\|\boldsymbol{V}_N\|_2 \rho_N^ne^{CN^{-\frac{3}{16}}}].
$$
Moreover, by \cite[end of the proof of Lemma 10.3-(vii)]{D}, we know that for all 
$N$ large enough, all $K \in \{1,\dots,N\}$, on $\Omega_N^{K,2}$, for all $n\geq 0$,
$$
\Big| \frac{||I_K A_N^n \bun||_2}{||A_N^n \bun||_2} - \frac{\|\boldsymbol{V}_N^K\|_2}{\|\boldsymbol{V}_N\|_2}
\Big| \leq 3 (2N^{-3/8})^{\lfloor n/2\rfloor +1}.
$$
Gathering the two previous estimates and using the fact that $||A_N^n \bun||_2 \leq N^{1/2}$, 
we thus conclude that, still  on $\Omega_N^{K,2}$, for all $n\geq 0$,
\begin{align*}
&||I_KA_N^{n}\bun||_2\in \Big[\|\boldsymbol{V}_N^K\|_2 \rho_N^ne^{-CN^{-\frac{3}{16}}}-3 N^{1/2}(2N^{-3/8})^{\lfloor n/2\rfloor +1},
\\
&\hskip6cm \|\boldsymbol{V}_N^K\|_2 \rho_N^ne^{CN^{-\frac{3}{16}}}+3 N^{1/2}(2N^{-3/8})^{\lfloor n/2\rfloor +1}\Big].
\end{align*}
The conclusion easily follows, because one can find a constant $C>0$ such that, on $\Omega_N^{K,2}$,
$$
3 N^{1/2}(2N^{-3/8})^{\lfloor n/2\rfloor +1} \leq C\|\boldsymbol{V}_N^K\|_2 \rho_N^n (1+ C N^{-3/16}).
$$
Indeed, $\|\boldsymbol{V}_N^K\|_2 \geq 1/2$ (because, as already seen, $V_N(1) \geq 1/2$), 
and $\rho_N \geq p/2 \geq 2N^{-3/8}$, see Lemma \ref{defsupf}. Since $n\geq 2$ (recall the statement), 
this is sufficient.
\end{proof}

Next, we define 
$$v_{t}^{N,K}:=\mu \sum_{n\ge 0}\frac{\|I_{K}A_{N}^{n}\boldsymbol{1}_{N}\|_{2}}{\|\bV_{N}^{K}\|_{2}}
\int_{0}^{t}s\phi^{*n}(t-s)ds.$$
and we recall the definition of $\alpha_N$, see Remark \ref{rhotop}.

\begin{cor}\label{vtnk}
Assume $(A)$. There are some constants $C>0$ and  $N_0 \geq 1$, such that for all $N\ge N_0$, $t\ge 1$ on  $\Omega_N^{K,2},$
 $$ v_{t}^{N,K}\in \Big[\mu e^{-CN^{-\frac{3}{16}}} \frac{\rho_N}{(\alpha_N)^2}(e^{\alpha_N t}-1)-Ct,
\mu e^{CN^{-\frac{3}{16}}}\frac{\rho_N}{(\alpha_N)^2}(e^{\alpha_N t}-1)+Ct\Big].$$
\end{cor}
\begin{proof}
For $n=0,1$ we always have, on $\Omega_N^{K,2}$, since $\|I_{K}A_{N}^{n}\boldsymbol{1}_{N}\|_{2} \leq \sqrt K$,
$\|\bV_{N}^{K}\|_{2}\geq c \sqrt K$ (since $V_N(i)\geq 1/2$ as already seen) and $\rho_N\leq 2 p$
(see Lemma \ref{defsupf})
$$
\Big|\frac{\|I_{K}A_{N}^{n}\boldsymbol{1}_{N}\|_{2}}{\|\bV_{N}^{K}\|_{2}}\int_{0}^{t}s\phi^{*n}(t-s)ds\Big|
+\Big|\mu e^{-CN^{-\frac{3}{16}}}\rho_N^n\int_{0}^{t}s\phi^{*n}(t-s)ds\Big|\le Ct.
$$ 
Hence by Lemma \ref{ANVK}, we have
$$ v_{t}^{N,K}\in \Big[\mu e^{-CN^{-\frac{3}{16}}}\sum_{n\ge 2}\rho_N^n\int_{0}^{t}s\phi^{*n}(t-s)ds - C t,
\mu e^{CN^{-\frac{3}{16}}}\sum_{n\ge 2}\rho_N^n\int_{0}^{t}s\phi^{*n}(t-s)ds + C t\Big].$$
But, recalling Assumption $A$, we have $\phi^{*n}(t)=t^{n-1}e^{-bt}/(n-1)!$ for $n\ge 1,$ which implies that 
(recall that $\alpha_N=\rho_N-b$)
$$
\sum_{n\ge 2}\rho_N^n\int_{0}^{t}s\phi^{*n}(t-s)ds=
\rho_N \int_0^t  s e^{-b(t-s)}(e^{\rho_N(t-s)}-1) ds =
\frac{\rho_N}{(\alpha_N)^2}(e^{\alpha_N t}-1)-\frac{\rho_N}{b^2}(1-e^{-bt}).
$$
Since $\frac{\rho_N}{b^2}(1-e^{-bt})\le C$, the conclusion follows.
\end{proof}

It is easy to deduce the following statement.

\begin{cor}\label{vtea}
Assume $(A)$. There are some constants $C,c>0$, $t_0 \geq 1$ and  $N_0 \geq 1$, 
such that for all $N\ge N_0$ and $t\ge t_0$, on  $\Omega_N^{K,2},$
$$ ce^{\alpha_N t}\le v_{t}^{N,K}\le C e^{\alpha_N t}.$$
\end{cor}
\begin{proof}
We work on $\Omega_N^{K,2}$.
From Remark \ref{rhotop}, we already have $\alpha_N-\alpha_0=\rho_N-p$ and $p/2\leq \rho_N \leq 2p$. 
By Lemma \ref{defsupf}, we conclude that
$$
\alpha_N \in \Big[\alpha_0-\frac{C}{N^{3/8}},\alpha_0+\frac{C}{N^{3/8}}\Big].
$$
By Assumption $(A)$, we know that  $\alpha_0=p-b>0$. We can choose $N_0$ large enough so that 
for all $N\ge N_0,$ $\alpha_N>\frac{\alpha_0}{2}>0$ and $\alpha_N<\frac{3\alpha_0}{2}$. 
So there are some constants $C>0$, $c>0$ and $t_0\ge 1$ such that for all $t\ge t_0$,
\begin{align*}
  \mu e^{CN^{-\frac{3}{16}}}\frac{\rho_N}{(\alpha_N)^2}(e^{\alpha_N t}-1)+Ct\le C e^{\alpha_N t}\\
  ce^{\alpha_N t}\le\mu e^{-CN^{-\frac{3}{16}}}\frac{\rho_N}{(\alpha_N)^2}(e^{\alpha_N t}-1)-Ct.
\end{align*}
Then the result follows from Corollary \ref{vtnk}.
\end{proof}

\section{Analysis of the process in the supercritical case}\label{misecsup}
The aim of this subsection is to give some analysis of the process $\boldsymbol{Z}_{t}^{N,K},$ which will be used in the proof of the main theorem in supercritical case. Recalling \cite[Equations (20)-(21)]{D}, we write
\begin{gather}
\label{defI}
 \mathbb{E}_{\theta}[\boldsymbol{Z}_{t}^{N,K}]=\mu\sum_{n\ge 0}\Big[\int^{t}_{0}s\phi^{*n}(t-s)ds\Big]I_{K}A^{n}_{N}\boldsymbol{1}_{N}=v_{t}^{N,K}\boldsymbol{V}_{N}^{K}+\boldsymbol{I}_{t}^{N,K},\\
 \label{defJ}
 \boldsymbol{U}_{t}^{N,K}=\boldsymbol{Z}_{t}^{N,K}-\mathbb{E}_{\theta}[\boldsymbol{Z}_{t}^{N,K}]=\sum_{n\ge 0}\int^{t}_{0}\phi^{*n}(t-s)I_{K}A^{n}_{N}\boldsymbol{M}_{s}^{N}ds=\boldsymbol{M}_{t}^{N,K}+\boldsymbol{J}_{t}^{N,K},
 \end{gather}
where 
 \begin{gather}
     \label{vtNK}v_{t}^{N,K}=\mu \sum_{n\ge 0}\frac{\|I_{K}A_{N}^{n}\boldsymbol{1}_{N}\|_{2}}{\|\bV_{N}^{K}\|_{2}}\int_{0}^{t}s\phi^{*n}(t-s)ds,\\
     \label{ItNKD}\boldsymbol{I}_{t}^{N,K}=\mu\sum_{n\ge 0}\Big[\int^{t}_{0}s\phi^{*n}(t-s)ds\Big]\Big[I_{K}A^{n}_{N}\boldsymbol{1}_{N}-\frac{\|I_{K}A_{N}^{n}\boldsymbol{1}_{N}\|_{2}}{\|\boldsymbol{V}_{N}^{K}\|_{2}}\boldsymbol{V}_{N}^{K}\Big],\\
\label{JtNK}\boldsymbol{J}_{t}^{N,K}=\sum_{n\ge 1}\int_{0}^{t}\phi^{*n}(t-s)I_{K}A^{n}_{N}\boldsymbol{M}_{s}^{N}ds.
\end{gather}
As usual, we denote by $I^{i,N}_t$ and $J^{i,N}_t$ the coordinates of $\bI^{N,K}_t$ and $\bJ^{N,K}_t$
and by $\bar I^{N,K}_t$ and $\bar J^{N,K}_t$ their empirical mean.
\begin{lemma}\label{supEZ}
Assume $(A)$. There are some constants $C>0$ and $N_0\geq 1$ such that for any $N \geq N_0$, any 
$K\in \{1,\dots,N\}$, any $t\ge 0$,
$$\indiq_{\Omega_{N}^{K,2}}\|\mathbb{E}_\theta[\boldsymbol{Z}_{t}^{N,K}]\|_\infty\le Ce^{\alpha_{N}t}.$$
\end{lemma}

\begin{proof}
By \cite[Lemma 35-(vii)]{A}, we already have, for all $n\geq 0$, $||A_N^n\bun||_\infty \leq C\rho_N^n.$ Then by
(\ref{defI}), on $\Omega_{N}^{K,2},$
\begin{align*}
||\Et[\bZ^{N,K}_t]||_\infty\leq& \mu\sum_{n\ge 0}\Big[\int^{t}_{0}s\phi^{*n}(t-s)ds\Big]\|I_{K}A^{n}_{N}\boldsymbol{1}_{N}\|_\infty\\
\leq& \mu\sum_{n\ge 0}\Big[\int^{t}_{0}s\phi^{*n}(t-s)ds\Big]\|A^{n}_{N}\boldsymbol{1}_{N}\|_\infty\\
\leq& C \sum_{n\geq 0} \rho_N^n \intot s \phi^{\star n}(t-s)ds.
\end{align*}
Recalling that $\phi^{\star n}(s)=s^{n-1}e^{-bs}/(n-1)!$,
\begin{align*}
||\Et[\bZ^{N,K}_t]||_\infty\leq& Ct+ C\rho_N \int_0^t  s e^{\alpha_N(t-s)} ds\\
=& C\Big(t+\frac{\rho_N}{(\alpha_N)^2}(e^{\alpha_N t}-1)\Big)\le  Ce^{\alpha_{N}t}.
\end{align*}
For the last inequality, we used that $\rho_N/\alpha_N^2 \leq C$ (for $N\geq N_0$),
as in the proof of Corollary \ref{vtea}.
\end{proof}

Next, we give the bound of $\|\boldsymbol{I}_{t}^{N,K}\|_{2}.$
\begin{lemma}\label{ItNK}
Assume $(A)$. There exists $N_0\ge 1$ and $C>0$ such that for all $N\ge N_0$, 
on $\Omega_{N}^{K,2}$, we have for all $t\geq 0$
$$\|\boldsymbol{I}_{t}^{N,K}\|_{2}\le C t\sqrt{K}N^{-\frac{3}{8}}.$$
\end{lemma}
\begin{proof}
See \cite[Lemma 11.1]{D}.
\end{proof}

\begin{cor}\label{INbar}
Assume $(A)$. There exists $N_0\ge 1$ and $C>0$ such that for all $N\ge N_0$, on
$\Omega_{N}^{K,2}$, we have, for all $t\geq 0$,
$$(\bar{I}_t^{N,K})^2\le C t^2N^{-\frac{3}{4}}.$$
\end{cor}

\begin{proof}
By using the Cauchy-Schwarz inequality and Lemma \ref{ItNK}, we directly have:
$$
(\bar{I}_t^{N,K})^2\le K^{-1}\|\boldsymbol{I}_{t}^{N,K}\|^2_{2}\le C t^2N^{-\frac{3}{4}}
$$
which completes the proof.
\end{proof}

\begin{cor}\label{Zteat}
Assume $(A)$. There exist $N_0\ge 1$, $t_0\geq 1$, and $C>c>0$ such that for all $N\ge N_0$, 
on the set $\Omega_{N}^{K,2}$, we have, for all $t\geq t_0$,
$$ c e^{\alpha_N t}\le \mathbb{E}_\theta[\bar{Z}_{t}^{N,K}]\le C e^{\alpha_N t}.$$
\end{cor}
\begin{proof}
By  (\ref{defI}), we have $\mathbb{E}_\theta[\bar{Z}_{t}^{N,K}]= v_t^{N,K}\bar V_N^K+\bar{I}_t^{N,K}.$
By \cite[Lemma 35-(ii)]{A}, we have $V_{N}(i)\in [\frac{1}{2},2]$ for all $i=1,...,N,$ whence
$\bar V_N^K\in[\frac{1}{2},2].$ Then the conclusion foolows from Corollaries \ref{vtea} and \ref{INbar}.
\end{proof}

\begin{lemma}\label{supEZ2}
Assume $(A)$. There exist $N_0\geq 1$ and $C>0$ such that for any $N\ge N_0$, any $t\ge 0$
and any $i=1,\dots,N$, 
\vip
(i) $\boldsymbol{1}_{\Omega_{N}^{K,2}}\Et[(J^{i,N}_t)^2] \leq C N^{-1}e^{2\alpha_N t}$ and 
$\boldsymbol{1}_{\Omega_{N}^{K,2}}\mathbb{E}_\theta[(Z_{t}^{i,N})^{2}]\le Ce^{2\alpha_{N}t}$,
\vip
(ii) $\boldsymbol{1}_{\Omega_{N}^{K,2}}\mathbb{E}_\theta[(U_{t}^{i,N})^{2}]\le C[N^{-1}e^{2\alpha_N t}+e^{\alpha_N t}]$
and $\boldsymbol{1}_{\Omega_{N}^{K,2}}\Et[(J^{i,N}_t)^4] \leq C e^{4\alpha_N t}$,
\vip
(iii) $\boldsymbol{1}_{\Omega_{N}^{K,2}}\mathbb{E}_\theta[(Z_{t}^{i,N})^{4}]\le Ce^{4\alpha_{N}t}$ and 
$\boldsymbol{1}_{\Omega_{N}^{K,2}}\mathbb{E}_\theta[(U_{t}^{i,N})^{4}]\le Ce^{4\alpha_N t}.$
\end{lemma}

\begin{proof}
First we prove  for any $i=1,\dots, N$, 
$\boldsymbol{1}_{\Omega_{N}^{K,2}}\Et[(J^{i,N}_t)^2] \leq C N^{-1}e^{2\alpha_N t}.$ 
By \cite[Remark 10]{A}, we already have 
 \begin{align}
\label{M7}\Et[M^{j,N}_rM^{k,N}_s]=\indiq_{\{j=k\}}\Et[Z^{j,N}_{r\land s}].
 \end{align}
Recalling (\ref{JtNK}),  we deduce that
\begin{align*}
\Et[(J^{i,N}_t)^2]=&\sum_{m,n\geq 1}\intot\intot \phi^{\star m}(t-r) \phi^{\star n}(t-s)\sum_{j,k=1}^N 
A_N^m(i,j)A_N^n(i,k)\Et[M^{j,N}_rM^{k,N}_s]drds\\
=&\sum_{m,n\geq 1}\intot\intot \phi^{\star m}(t-r) \phi^{\star n}(t-s)\sum_{j=1}^N 
A_N^m(i,j)A_N^n(i,j)\Et[Z^{j,N}_{r\land s}]drds.
\end{align*}
And by \cite[Lemma 35-(iv)]{A}, for all $n\ge 2$,  
we have $A_{N}^{n}(i,j)\in [\rho_{N}^{n}/(3N),3\rho_{N}^{n}/N]$ on $\Omega_N^{K,2}$,
while by Lemma \ref{supEZ}, we know that $\Et[Z^{j,N}_{r\land s}] \leq C e^{\alpha_N (r\land s)}$.
Hence
\begin{align*}
\Et[(J^{i,N}_t)^2]\le& 9N^{-1}\sum_{m,n\geq 1}\intot\intot \phi^{\star m}(t-r) \phi^{\star n}(t-s)\rho^{m+n}_{N} 
e^{\alpha_N(r\land s)}drds\\
\le& 9N^{-1}\sum_{m,n\geq 1}\intot\intot \phi^{\star m}(t-r) \phi^{\star n}(t-s)\rho^{m+n}_{N} 
e^{\alpha_N(\frac{r+s}{2})}drds
\\
=& 9N^{-1}\Big(\sum_{m\geq 1}\intot \phi^{\star m}(t-r) \rho^{m}_{N} 
e^{\frac{\alpha_Nr}{2}}dr\Big)^2\\
=& 9N^{-1}\Big(\rho_N\intot e^{\alpha_N (t-r)} 
e^{\frac{\alpha_Nr}{2}}dr\Big)^2\le CN^{-1}e^{2\alpha_N t},
\end{align*}
since $\rho_N \leq 2p$, see Remark \ref{rhotop}. This finishes the proof of the first part of (i). 

\vip

By (\ref{M7}) and Lemma \ref{supEZ}, $\indiq_{\Omega_{N}^{K,2}}\Et[(M^{i,N}_t)^2]\le \boldsymbol{1}_{\Omega_{N}^{K,2}}\|\mathbb{E}_\theta[\boldsymbol{Z}_{t}^{N}]\|_\infty\le Ce^{\alpha_{N}t}$, whence, recall \eqref{defJ},
\begin{align*}
\boldsymbol{1}_{\Omega_{N}^{K,2}}\mathbb{E}_\theta[(U_{t}^{i,N})^{2}]\le 2\indiq_{\Omega_{N}^{K,2}}\Big\{\Et[(J^{i,N}_t)^2]+\Et[(M^{i,N}_t)^2]\Big\}\le C[N^{-1}e^{2\alpha_N t}+e^{\alpha_N t}],
\end{align*}
By (\ref{defJ}), we write $Z_{t}^{i,N}=\mathbb{E}_\theta[Z_{t}^{i,N}]+U_{t}^{i,N}$. 
And by Lemma \ref{supEZ}, we have $\indiq_{\Omega_{N}^{K,2}}\mathbb{E}_\theta[Z_{t}^{i,N}]^2\le Ce^{2\alpha_{N}t}$, whence
$$
\boldsymbol{1}_{\Omega_{N}^{K,2}}\mathbb{E}_\theta[(Z_{t}^{i,N})^{2}]\le \boldsymbol{1}_{\Omega_{N}^{K,2}} 2\Big\{\mathbb{E}_\theta[Z_{t}^{i,N}]^2+\Et[(U^{i,N}_t)^2]\Big\}\le Ce^{2\alpha_{N}t}.
$$
We have finished the proof of (i) and of the first part of (ii).

\vip

To verify the second part of (ii), we use that by Doob's inequality, 
\begin{equation}\label{rft}
\boldsymbol{1}_{\Omega_{N}^{K,2}}\Et[(M^{i,N}_s)^4]\le \boldsymbol{1}_{\Omega_{N}^{K,2}}\mathbb{E}_\theta[(Z_{t}^{i,N})^2]\le Ce^{2\alpha_{N}t}.
\end{equation}
Then  by Minkowski's inequality, we have, on $\Omega_{N}^{K,2},$ (recalling \eqref{JtNK} for the expression
of $J^{i,N}_t$)
\begin{align*}
    \Et[(J^{i,N}_t)^4]^\frac{1}{4}\le& \sum_{n\geq 1}\intot  \phi^{\star n}(t-s)\sum_{k=1}^N 
A_N^n(i,k)\Et[(M^{k,N}_s)^4]^\frac{1}{4}ds\\
\le& C\sum_{n\geq 1}\intot \rho_{N}^{n} \phi^{\star n}(t-s)e^{\frac{\alpha_Ns}{2}}ds\\
=&C\rho_N\intot e^{\alpha_N (t-s)} 
e^{\frac{\alpha_Ns}{2}}ds
\le Ce^{\alpha_N t},
\end{align*}
which ends the proof of (ii). 

\vip

Point (iii) easily follows from (ii) and \eqref{rft}, using \eqref{defJ} and Lemma \ref{supEZ}:
\begin{gather*}
\boldsymbol{1}_{\Omega_{N}^{K,2}}\mathbb{E}_\theta[(U_{t}^{i,N})^{4}]\le \indiq_{\Omega_{N}^{K,2}}4\Big\{\Et[(J^{i,N}_t)^4]+\Et[(M^{i,N}_t)^4]\Big\}\le C[e^{4\alpha_N t}+e^{2\alpha_N t}] \\
    \boldsymbol{1}_{\Omega_{N}^{K,2}}\mathbb{E}_\theta[(Z_{t}^{i,N})^{4}]\le \boldsymbol{1}_{\Omega_{N}^{K,2}} 4\Big\{\mathbb{E}_\theta[Z_{t}^{i,N}]^4+\Et[(U^{i,N}_t)^4]\Big\}\le Ce^{4\alpha_{N}t}.
\end{gather*}
We finished the proof.
\end{proof}

\begin{lemma}\label{UtNK2}
Assume $(A)$. There exist $C>0$ and $N_0 \geq 1$ such that for all $N\geq N_0$,
all $K=1,\dots,N_0$, all $t\ge 0$
$$\boldsymbol{1}_{\Omega_{N}^{K,2}}\Et[|\bar{U}_{t}^{N,K}|^2]\le CK^{-1}e^{2\alpha_N t}.$$
\end{lemma}

\begin{proof}
By \eqref{defJ}, Lemma \ref{supEZ2} and the Cauchy-Schwarz inequality, we have
\begin{align*}
\boldsymbol{1}_{\Omega_{N}^{K,2}}\Et[|\bar{U}_{t}^{N,K}|^2]\le \boldsymbol{1}_{\Omega_{N}^{K,2}}2\Big\{\Et[(\bar{J}^{N,K}_t)^2]+\Et[(\bar{M}^{N,K}_t)^2]\Big\}\le C\Big[\frac{e^{2\alpha_N t}}{N}+\frac{e^{\alpha_N t}}{K}\Big]\le \frac{Ce^{2\alpha_N t}}{K}
\end{align*}
\end{proof}

\begin{lemma}\label{PZtNK}
Assume $(A)$. There are $C>0$ and $N_0\geq 1$ and $t_0\geq 1$ such that, for all $N\geq N_0$,
all $K=1,\dots,N$, on the set $\Omega_{N}^{K,2}$, for all $t\geq t_0$,
$$P_{\theta}\Big(\bar{Z}_{t}^{N,K}\le \frac{1}{4}v_{t}^{N,K}\Big)\le C\Big(\frac{1}{\sqrt{K}}+\frac{t}{e^{\alpha_{N}t}}\Big).$$
\end{lemma}

\begin{proof}
We work on $\Omega_{N}^{K,2}$.
By \cite[Lemma 35-(ii)]{A}, we have  $\bar{V}_{N}^{K}\ge \frac{1}{2}$. Thus 
$\bar{Z}_{t}^{N,K}\le v_{t}^{N,K}/4$ implies that
$|\frac{\bar{Z}_{t}^{N,K}}{v_{t}^{N,K}}-\bar{V}_{N}^{K}|\ge \frac{1}{4}$, so that
$$P_{\theta}\Big(\bar{Z}_{t}^{N,K}\le \frac{1}{4}v_{t}^{N,K}\Big)\le 4\mathbb{E}_{\theta}\Big[\Big|\frac{\bar{Z}_{t}^{N,K}}{v_{t}^{N,K}}-\bar{V}_{N}^{K}\Big|\Big]\leq 4\mathbb{E}_{\theta}\Big[\frac{|\bar{U}_{t}^{N,K}|+|\bar{I}_t^{N,K}|}{v_{t}^{N,K}}\Big]
$$
by (\ref{defI}) and (\ref{defJ}). The conclusion follows, since
$v^{N,K}_t \geq ce^{\alpha_N t}$ by Corollary \ref{vtea}, $\Et[|\bar{I}_t^{N,K}|] \leq C t N^{-3/8}$
by Corollary \ref{INbar}, and $\Et[|\bar{U}_t^{N,K}|] \leq C K^{-1/2}e^{\alpha_N t}$ by 
Lemma \ref{UtNK2}.
\end{proof}

The following statement is then clear from Lemmas \ref{xyz} and \ref{PZtNK} and Corollary \ref{vtea}.

\begin{cor}\label{PvtZt}
Assume $(A)$. It holds that,
\begin{align*}
\lim_{(N,K,t)\to(\infty,\infty,\infty)} P \Big(\Omega_{N}^{K,2} \cap \{ \bar{Z}_{t}^{N,K}\ge \frac{1}{4}v_{t}^{N,K}>0\}\Big)=1.
\end{align*}
\end{cor}

We conclude the subsection with the following statement.

\begin{lemma}\label{reg}
We assume $(A)$. In the regime where $(N,K,t)\to(\infty,\infty,\infty)$ with $1\le K\le N$ and
\begin{equation}\label{regime}
\frac N{\sqrt K e^{\alpha_0 t}} + \frac 1{\sqrt K} \to 0  \quad \hbox{with} \quad 
\frac N{e^{\alpha_0 t}} \to \infty,
\end{equation}
we have
$\lim \indiq_{\Omega_{N}^{K,2}}(\frac N{\sqrt K e^{\alpha_N t}} + \frac 1{\sqrt K})=0$ and
$\lim \indiq_{\Omega_{N}^{K,2}} \frac N{e^{\alpha_N t}} =\infty$ a.s.
\end{lemma}

\begin{proof}
On $\Omega_N^{K,2}$, we have $\alpha_N=\rho_N-b$ and $\alpha_0=p-b$
(see Remark \ref{rhotop}), whence 
$\frac{e^{\alpha_N t}}{e^{\alpha_0 t}}=e^{(\rho_N-p) t}$ so that 
$\frac{e^{\alpha_N t}}{e^{\alpha_0 t}}\in [e^{-\frac{p}{2N^{3/8}} t},e^{\frac{p}{2N^{3/8}} t}]$ by Lemma \ref{defsupf}.
The conclusion follows.
\end{proof}

\section{Proof of the main result in the supercritical case}\label{finsecsup}

We recall that $\mathcal{U}_{t}^{N,K}$ was defined in (\ref{UP}). The main result, in the supercritical case,
will easily follow from the following statement.

\begin{prop}\label{mainsup}
We assume $(A)$. In the regime $(\ref{regime})$, it holds 
\begin{align*}
\boldsymbol{1}_{\Omega_{N}^{K,2}\cap \{\bar{Z}_{t}^{N,K}\ge \frac{1}{4}v_{t}^{N,K}>0\}}\frac{e^{\alpha_{N} t}\sqrt{K}}{N}\Big[\mathcal{U}_{t}^{N,K}-\Big(\frac{1}{p}-1\Big)\Big]\stackrel{d}{\longrightarrow}\mathcal{N}\Big(0,\frac{2(\alpha_0)^4}{(\mu p)^2}\Big).
\end{align*}
\end{prop}

We set $\cD_{t}^{N,K}=\mathcal{U}_{t}^{N,K}-(\frac{1}{p}-1)$ and we use the following decomposition, on the event
$\Omega_{N}^{K,2}\cap \{\bar{Z}_{t}^{N,K}\ge \frac{1}{4}v_{t}^{N,K}>0\}$:
\begin{align}
\label{Ddeco}\cD_{t}^{N,K}= \cD_{t}^{N,K,1}+\cD_{t}^{N,K,2}+\cD_{t}^{N,K,3}
\end{align}
where, recalling that $\mathcal{U}_{\infty}^{N,K}$ was introduced in Lemma \ref{Up11} and that
$v^{N,K}_t$ was defined in \eqref{vtNK},
\begin{align*}
    \cD_{t}^{N,K,1}=&\frac{1}{(\bar{Z}_{t}^{N,K})^{2}}\Big(\frac{N}{K}\|\boldsymbol{Z}_{t}^{N,K}-\bar{Z}_{t}^{N,K}\boldsymbol{1}_{K}\|_{2}^{2}-N\bar{Z}_{t}^{N,K}-(v_{t}^{N,K})^{2}\frac{N}{K}\|\boldsymbol{V}_{N}^{K}-\bar{V}_{N}^{K}\boldsymbol{1}_{K}\|_{2}^{2}\Big),\\
   \cD_{t}^{N,K,2}=&\frac{N}{K}\|\boldsymbol{V}_{N}^{K}-\bar{V}_{N}^{K}\boldsymbol{1}_{K}\|_{2}^{2}\Big[\Big(\frac{v_{t}^{N,K}}{\bar{Z}_{t}^{N,K}}\Big)^{2}-\frac{1}{(\bar{V}_{N}^{K})^{2}}\Big],\\
   \cD^{N,K,3}=&\Big[\mathcal{U}_{\infty}^{N,K}-\Big(\frac{1}{p}-1\Big)\Big].
\end{align*}

\begin{lemma}\label{Up111}
We assume $(A)$. In the regime $(\ref{regime})$, in probability,
$$
\lim  \boldsymbol{1}_{\Omega_{N}^{K,2}}\frac{e^{\alpha_{N} t}\sqrt{K}}{N}|\cD^{N,K,3}|=0.
$$
\end{lemma}
\begin{proof}
It suffices to gather Lemma \ref{reg}, from which $\boldsymbol{1}_{\Omega_{N}^{K,2}}e^{\alpha_N t}/N \to 0$ a.s.
and Lemma \ref{Up11}, from which $\sqrt K \boldsymbol{1}_{\Omega_{N}^{K,2}}|\mathcal{U}_{\infty}^{N,K}-
(\frac{1}{p}-1)|$ is bounded in $L^1$.
\end{proof}

Next, we consider the term $\cD_{t}^{N,K,2}.$

\begin{lemma}\label{DTNK2}
Assume $(A)$. There are some constants $C>0$, $t_0 \geq 1$ and  $N_0 \geq 1$, 
such that for all $N\ge N_0$ and $t\ge t_0$,  on the event 
$\Omega_{N}^{K,2}\cap \{\bar{Z}_{t}^{N,K}\ge \frac{1}{4}v_{t}^{N,K}>0\}$,
$$
\mathbb{E}_{\theta}[|\cD_{t}^{N,K,2}|]\le C\frac{N}{K}\|\boldsymbol{V}_{N}^{K}-\bar{V}_{N}^{K}\boldsymbol{1}_{K}\|_{2}^{2}\Big(\frac{1}{\sqrt{K}}+\frac{t}{N^{\frac{3}{8}}e^{\alpha_{N}t}}\Big).
$$
\end{lemma}

\begin{proof}
We work on $\Omega_{N}^{K,2}\cap \{\bar{Z}_{t}^{N,K}\ge \frac{1}{4}v_{t}^{N,K}>0\}$.
Recalling \eqref{defI} and \eqref{defJ}, we can write
 \begin{align*}
\Big|\bar{Z}_{t}^{N,K}(v_{t}^{N,K})^{-1}-\bar{V}_{N}^{K}\Big|
 \le (v_{t}^{N,K})^{-1}\Big(|\bar{I}_{t}^{N,K}|+|\bar{U}_{t}^{N,K}|\Big).
 \end{align*}
According to Corollary \ref{vtea},  there exists some positive constant  $c$ such that 
$ ce^{\alpha_N t}\le v_{t}^{N,K}.$ 
On the event $\Omega_{N}^{K,2}$, we already have  $\bar{V}_{N}^{K}\ge \frac{1}{2}$ by \cite[Lemma 35-(ii)]{A}.  
Since $|\frac{1}{x^{2}}-\frac{1}{y^{2}}|=|\frac{(x-y)(x+y)}{x^{2}y^{2}}|\le 128|x-y|$, for $x,y \ge \frac{1}{4}$,
it holds that
$$
|\cD_{t}^{N,K,2}|\leq \frac{128N}{K}\|\boldsymbol{V}_{N}^{K}-\bar{V}_{N}^{K}\boldsymbol{1}_{K}\|_{2}^{2}
\frac{|\bar{I}_{t}^{N,K}|+|\bar{U}_{t}^{N,K}|}{v_t^{N,K}} \leq \frac{C N e^{-\alpha_N t}}{K}
\|\boldsymbol{V}_{N}^{K}-\bar{V}_{N}^{K}\boldsymbol{1}_{K}\|_{2}^{2}(|\bar{I}_{t}^{N,K}|+|\bar{U}_{t}^{N,K}|).
$$
By Corollary \ref{INbar} and Lemma \ref{UtNK2}, we finally obtain
\begin{align*}
\mathbb{E}_{\theta}[|\cD_{t}^{N,K,2}|]
\le& C\frac{N}{K}\|\boldsymbol{V}_{N}^{K}-\bar{V}_{N}^{K}\boldsymbol{1}_{K}\|_{2}^{2}\Big(\frac{t}{N^{\frac{3}{8}}e^{\alpha_{N}t}}+\frac{1}{\sqrt{K}}\Big)
\end{align*}
which completes the result.
 \end{proof}

\begin{cor}\label{DNK2}
We assume $(A)$.  In the regime $(\ref{regime})$, in probability,
$$
\lim \indiq_{\Omega_{N}^{K,2}\cap \{\bar{Z}_{t}^{N,K}\ge \frac{1}{4}v_{t}^{N,K}>0\}}\frac{e^{\alpha_{N} t}\sqrt{K}}{N}|\cD_{t}^{N,K,2}|=0.
$$
\end{cor}
\begin{proof}
By Lemma \ref{DTNK2}, we have $\mathbb{E}_{\theta}[|\cD_{t}^{N,K,2}|]
\le C\frac{N}{K}\|\boldsymbol{V}_{N}^{K}-\bar{V}_{N}^{K}\boldsymbol{1}_{K}\|_{2}^{2}(\frac{t}{N^{\frac{3}{8}}e^{\alpha_{N}t}}
+\frac{1}{\sqrt{K}})$ on the event $\Omega_{N}^{K,2}\cap \{\bar{Z}_{t}^{N,K}\ge \frac{1}{4}v_{t}^{N,K}>0\}$.
By Lemma \ref{VNK}-(iii), we know that 
$\indiq_{\Omega_{N}^{K,2}}\frac{N}{K}\|\boldsymbol{V}_{N}^{K}-\bar{V}_{N}^{K}\boldsymbol{1}_{K}\|_{2}^{2}$ is bounded 
in $L^1$, whence the conclusion, since $\frac{e^{\alpha_{N} t}\sqrt{K}}{N}(\frac{t}{N^{\frac{3}{8}}e^{\alpha_{N}t}}
+\frac{1}{\sqrt{K}})\indiq_{\Omega_N^{K,2}} \to 0$ by Lemma \ref{reg}.
\end{proof}

Next, we consider the term $\cD_{t}^{N,K,1},$ starting from 
$$
\cD_{t}^{N,K,1}=\cD_{t}^{N,K,11}+\cD_{t}^{N,K,12}+2\cD_{t}^{N,K,13}+2\cD_{t}^{N,K,14},
$$
where 
\begin{align*}
\cD_{t}^{N,K,11}=&\frac{N}{K(\bar{Z}_{t}^{N,K})^{2}}\|\boldsymbol{I}_{t}^{N,K}-\bar{I}_{t}^{N,K}\boldsymbol{1}_{K}+\cJ_{t}^{N,K}-\bar{J}_{t}^{N,K}\boldsymbol{1}_{K}\|_{2}^{2},\\
\cD_{t}^{N,K,12}=&\frac{N}{K(\bar{Z}_{t}^{N,K})^{2}}\Big[\|\bM_{t}^{N,K}-\bar{M}_{t}^{N,K}\boldsymbol{1}_{K}\|_{2}^{2}-NZ_{t}^{N,K}\Big],\\
\cD_{t}^{N,K,13}=&\frac{N}{K(\bar{Z}_{t}^{N,K})^{2}}\Big(\boldsymbol{I}_{t}^{N,K}-\bar{I}^{N,K}_{t}\boldsymbol{1}_{K}+\cJ_{t}^{N,K}-\bar{J}_{t}^{N,K}\boldsymbol{1}_{K},\\
&\hskip5cm v_{t}^{N,K}(V_{N}^{K}-\bar{V}_{N}^{K}\boldsymbol{1}_{K})+\boldsymbol{M}_{t}^{N,K}-\bar{M}_{t}^{N,K}\boldsymbol{1}_{K}\Big),\\
\cD_{t}^{N,K,14}=&\frac{N}{K(\bar{Z}_{t}^{N,K})^{2}}v_{t}^{N,K}\Big(\boldsymbol{V}_{N}^{K}-\bar{V}_{N}^{K}\boldsymbol{1}_{K},\boldsymbol{M}_{t}^{N,K}-\bar{M}_{t}^{N,K}\boldsymbol{1}_{K}\Big).
\end{align*}
First, we  study $\cD_{t}^{N,K,11}.$ In order to obtain its limit theorem, we need the following lemme.

\begin{lemma}\label{MJUtNK}
Assume $(A)$. There exist $N_0\ge 1$ and $C>0$ such that for all $N\ge N_0$, any 
$K\in \{1,\dots,N\}$, on the set $\Omega_{N}^{K,2}$, for any $t\ge 0$, we have 
\begin{align*}
&(i)\quad  \mathbb{E}_{\theta}[\|\boldsymbol{M}_{t}^{N,K}-\bar{M}_{t}^{N,K}\boldsymbol{1}_{K}\|_{2}^{2}]\le CKe^{\alpha_{N}t},\\
&(ii)\quad \mathbb{E}_{\theta}[\|\boldsymbol{J}_{t}^{N,K}-\bar{J}_{t}^{N,K}\boldsymbol{1}_{K}\|^2_{2}]^\frac{1}{2}\le C\sqrt{\frac{K}{N}}\Big[e^{\frac{1}{2}\alpha_{N}t}+\frac{\|\boldsymbol{V}_{N}^{K}-\bar{V}_{N}^{K}\boldsymbol{1}_{K}\|_{2}}{\|\boldsymbol{V}_{N}^{K}\|_{2}}e^{\alpha_{N}t}\Big],\\
&(iii)\quad  \mathbb{E}_{\theta}[\|\boldsymbol{U}_{t}^{N,K}-\bar{U}_{t}^{N,K}\boldsymbol{1}_{K}\|_{2}^{2}]\le C\Big(Ke^{\alpha_{N}t}+\frac{e^{2\alpha_N t}}{N}\|\boldsymbol{V}_{N}^{K}-\bar{V}_{N}^{K}\boldsymbol{1}_{K}\|^2_{2}\Big).
 \end{align*}
\end{lemma}

\begin{proof} We work on $\Omega_N^{K,2}$.
Recalling $(\ref{M7})$, we write
\begin{align*}
\mathbb{E}_{\theta}[\|\boldsymbol{M}_{t}^{N,K}-\bar{M}_{t}^{N,K}\boldsymbol{1}_{K}\|_{2}^{2}]
=\sum_{i=1}^K\mathbb{E}_\theta[(M_t^{i,N})^2]-K\mathbb{E}_\theta[(\bar{M}_t^{N,K})^2]
=\sum_{i=1}^K\mathbb{E}_\theta[Z_t^{i,N}]-K\mathbb{E}_\theta[(\bar{M}_t^{N,K})^2].
\end{align*}
Hence we deduce from Lemma \ref{supEZ} that
\begin{align*}
\mathbb{E}_{\theta}[\|\boldsymbol{M}_{t}^{N,K}-\bar{M}_{t}^{N,K}\boldsymbol{1}_{K}\|_{2}^{2}]
\leq \sum_{i=1}^K\mathbb{E}_\theta[Z_t^{i,N}] 
\le CKe^{\alpha_{N}t},
\end{align*}
which proves $(i)$. For $(ii)$, 
in view of (\ref{JtNK}), by the Minkowski inequality, we have
 \begin{align*}
\mathbb{E}_{\theta}\Big[\|\boldsymbol{J}_{t}^{N,K}-\bar{J}_{t}^{N,K}\boldsymbol{1}_{K}\|_{2}^{2}\Big]^{\frac{1}{2}}
&\le \sum_{n\ge 1}\int_{0}^{t}\phi^{*n}(t-s)\mathbb{E}_{\theta}\Big[\|I_{K}A_{N}^{n}\boldsymbol{M}_{s}^{N}-\overline{I_{K}A_{N}^{n}\boldsymbol{M}_{s}^{N}}\boldsymbol{1}_{K}\|^{2}_{2}\Big]^{\frac{1}{2}}ds,
 \end{align*}
where $\overline{I_{K}A_{N}^{n}\boldsymbol{M}_{s}^{N}}:=\frac{1}{K}\sum_{j=1}^N\sum_{i=1}^KA_N^n(i,j)M_s^{j,N}.$
Using again \eqref{M7}, we see that
\begin{align*}
\mathbb{E}_{\theta}\Big[\|I_{K}A_{N}^{n}\boldsymbol{M}_{s}^{N}-\overline{I_{K}A_{N}^{n}\boldsymbol{M}_{s}^{N}}\boldsymbol{1}_{K}\|^{2}_{2}\Big]&= \sum_{i=1}^K \Et\Big[\Big(\sum_{j=1}^N A_N^n(i,j)M^{j,N}_s
- \frac 1 K \sum_{k=1}^K\sum_{j=1}^N A_N^n(k,j)M^{j,N}_s\Big)^2\Big]\\
&=\sum_{i=1}^{K}\sum_{j=1}^{N}\Big(A_{N}^{n}(i,j)-\frac{1}{K}\sum_{k=1}^{K}A_{N}^{n}(k,j)\Big)^{2}\mathbb{E}_{\theta}[Z_{s}^{j,N}]\\
&\le Ce^{\alpha_Ns}\sum_{j=1}^{N}\|I_{K}A_{N}^{n}\boldsymbol{e}_{j}-\overline{I_{K}A_{N}^{n}\boldsymbol{e}_{j}}\boldsymbol{1}_{K}\|_{2}^{2}.
\end{align*}
For the last inequality, we used that $\max_{i=1,...,N}\mathbb{E}_\theta[(Z^{i,N}_t)^2] \leq C e^{2\alpha_Nt}$
by Lemma \ref{supEZ2} and we introduced $\boldsymbol{e}_{j} \in \R^N$ with coordinates 
$\boldsymbol{e}_{j}(i)=\indiq_{\{i=j\}}$.
Using the inequality $||\bx-\bar x \indiq_N||_2-||\boldsymbol{y}-\bar y \indiq_N||_2
\leq ||\bx-\boldsymbol{y}||_2$ for all $\bx,\boldsymbol{y}\in \mathbb{R}^N,$
\begin{align*}
&\|I_{K}A_{N}^{n}\boldsymbol{e}_{j}-\overline{I_{K}A_{N}^{n}\boldsymbol{e}_{j}}\boldsymbol{1}_{K}\|_{2}\\
\le& \Big\|I_{K}A_{N}^{n}\boldsymbol{e}_{j}-\frac{1}{\|\boldsymbol{V}_{N}^{K}\|_{2}}\|I_{K}A_{N}^{n}\boldsymbol{e}_{j}\|_{2}\boldsymbol{V}_{N}^{K}\Big\|_{2}+\frac{\|I_{K}A^{n}_{N}\boldsymbol{e}_{j}\|_{2}}{\|\boldsymbol{V}_{N}^{K}\|_{2}}\|\boldsymbol{V}_{N}^{K}-\bar{V}_{N}^{K}\boldsymbol{1}_{K}\|_{2}\\
=&\|I_{K}A^{n}_{N}\boldsymbol{e}_{j}\|_{2}\Big(\Big\|\frac{I_{K}A^{n}_{N}\boldsymbol{e}_{j}}{\|I_{K}A^{n}_{N}\boldsymbol{e}_{j}\|_{2}}-\frac{\boldsymbol{V}_{N}^{K}}{\|\boldsymbol{V}_{N}^{K}\|_{2}}\Big\|_{2}+\frac{\|\boldsymbol{V}_{N}^{K}-\bar{V}_{N}^{K}\boldsymbol{1}_{K}\|_{2}}{\|\boldsymbol{V}_{N}^{K}\|_{2}}\Big)
\\ \le& C\|I_{K}A^{n}_{N}\boldsymbol{e}_{j}\|_{2}\Big(N^{-\frac{3}{8}\lfloor\frac{n}{2}\rfloor}+\frac{\|\boldsymbol{V}_{N}^{K}-\bar{V}_{N}^{K}\boldsymbol{1}_{K}\|_{2}}{\|\boldsymbol{V}_{N}^{K}\|_{2}}\Big).
\end{align*}
by \cite[Lemma 10.3-(viii)]{D} with $r=2$.
From \cite[Lemma 10.3-(iv)]{D}, for all $n\ge 2$, we have
\begin{align*}
\|I_{K}A^{n}_{N}\boldsymbol{e}_{j}\|_{2}\le \Big[\sum_{i=1}^K(A^n_N(i,j))^2\Big]^\frac{1}{2}\le \frac{3\sqrt{K}}{N}\rho_{N}^{n}.
\end{align*}
We then conclude that
\begin{align*}
\Et[||I_KA_N^n M^N_s-\overline{I_KA_N^n M^N_s}\indiq_K ||_2^2]^{1/2} \leq&  Ce^{\alpha_N s/2}\sqrt{\frac{K}{N}} \rho_N^n
\Big( (2N^{-3/8})^{\lfloor n/2\rfloor} + \frac{\|\boldsymbol{V}_{N}^{K}-\bar{V}_{N}^{K}\boldsymbol{1}_{K}\|_{2}}{\|\boldsymbol{V}_{N}^{K}\|_{2}}\Big).
\end{align*}
So on the event $\Omega_{N}^{K,2}$, 
\begin{align*}
&\mathbb{E}_{\theta}[\|\boldsymbol{J}_{t}^{N,K}-\bar{J}_{t}^{N,K}\boldsymbol{1}_{K}\|^{2}_{2}]^{\frac{1}{2}}
\le C\sqrt{\frac{K}{N}}\sum_{n\ge 1}\rho_{N}^{n}\Big[(2N^{-\frac{3}{8}})^{\lfloor\frac{n}{2}\rfloor}+\frac{\|\boldsymbol{V}_{N}^{K}-\bar{V}_{N}^{K}\boldsymbol{1}_{K}\|_{2}}{\|\boldsymbol{V}_{N}^{K}\|_{2}}\Big]\int_{0}^{t}\phi^{*n}(t-s)e^{\frac{\alpha_Ns}{2}}ds.
\end{align*}
Using \cite[Lemma 43-(iii)]{A}, we conclude 
\begin{align*}
    \sum_{n\ge 1}\rho_{N}^{n}(2N^{-\frac{3}{8}})^{\lfloor\frac{n}{2}\rfloor}\int_{0}^{t}\phi^{*n}(t-s)e^{\frac{\alpha_Ns}{2}}ds\le e^{\frac{\alpha_N t}{2}}\sum_{n\ge 1}\rho_{N}^{n}(2N^{-\frac{3}{8}})^{\lfloor\frac{n}{2}\rfloor}\int_{0}^{t}\phi^{*n}(t-s)ds\le C e^{\frac{\alpha_N t}{2}}.
\end{align*}
And we can compute directly, recalling that $\phi(u)=e^{-b u}$ and that $\rho_N=\alpha_N+b$, that 
\begin{align*}
    \sum_{n\ge 1}\rho_{N}^{n}\int_{0}^{t}\phi^{*n}(t-s)e^{\frac{\alpha_Ns}{2}}ds= \rho_N\intot e^{\alpha_N (t-s)} 
e^{\frac{\alpha_N s}{2}}ds\le \frac{2 \rho_N}{\alpha_N} e^{\alpha_N t} \leq C e^{\alpha_N t}
\end{align*}
by Remark \ref{rhotop}. All in all,
$$\mathbb{E}_{\theta}[\|\boldsymbol{J}_{t}^{N,K}-\bar{J}_{t}^{N,K}\boldsymbol{1}_{K}\|_{2}^{2}]^\frac{1}{2}\le C\sqrt{\frac{K}{N}}\Big[e^{\frac{\alpha_Nt}{2}}+\frac{\|\boldsymbol{V}_{N}^{K}-\bar{V}_{N}^{K}\boldsymbol{1}_{K}\|_{2}}{\|\boldsymbol{V}_{N}^{K}\|_{2}}e^{\alpha_N t}\Big].$$
This completes the proof of $(ii)$. 
For $(iii),$ we recall $(\ref{defJ})$ and write 
$\boldsymbol{U}_t^{N,K}=\boldsymbol{M}_{t}^{N,K}+\boldsymbol{J}_{t}^{N,K},$ whence
\begin{align*} 
\mathbb{E}_{\theta}[\|\boldsymbol{U}_{t}^{N,K}-\bar{U}_{t}^{N,K}\boldsymbol{1}_{K}\|_{2}^{2}]\le& 2\Big(\mathbb{E}_{\theta}[\|\boldsymbol{M}_{t}^{N,K}-\bar{M}_{t}^{N,K}\boldsymbol{1}_{K}\|_{2}^{2}]+\mathbb{E}_{\theta}[\|\boldsymbol{J}_{t}^{N,K}-\bar{J}_{t}^{N,K}\boldsymbol{1}_{K}\|_{2}^{2}]\Big).
\end{align*}
By $(i)$, we have 
$\mathbb{E}_{\theta}[\|\boldsymbol{M}_{t}^{N,K}-\bar{M}_{t}^{N,K}\boldsymbol{1}_{K}\|_{2}^{2}]\le CKe^{\alpha_{N}t}.$ 
By $(ii)$, we have $\mathbb{E}_{\theta}[\|\boldsymbol{J}_{t}^{N,K}-\bar{J}_{t}^{N,K}\boldsymbol{1}_{K}\|_{2}^{2}]
\le C(\frac KNe^{\alpha_{N}t}+\frac{e^{2\alpha_N t}}{N}\|\boldsymbol{V}_{N}^{K}-\bar{V}_{N}^{K}\boldsymbol{1}_{K}\|^2_{2})$.
The conclusion follows, since, as already seen, we have $V_N(i)\ge \frac{1}{2}$, whence
$\|\bV_N^K\|_2\ge c\sqrt{K}$, by \cite[Lemma 35-(ii)]{A}.
\end{proof}

\begin{lemma}\label{DNK11}
Assume $(A)$. In the regime $(\ref{regime})$, in probability
$$
\lim\indiq_{\Omega_{N}^{K,2}\cap \{\bar{Z}_{t}^{N,K}\ge \frac{1}{4}v_{t}^{N,K}>0\}}\frac{e^{\alpha_{N} t}\sqrt{K}}{N}|\cD_{t}^{N,K,11}|=0.
$$
\end{lemma}
\begin{proof}
By Corollary \ref{vtea}, we easily deduce that 
\begin{align*}
   \indiq_{\Omega_{N}^{K,2}\cap \{\bar{Z}_{t}^{N,K}\ge \frac{1}{4}v_{t}^{N,K}>0\}}|\cD_{t}^{N,K,11}|\le \indiq_{\Omega^{K,2}_N}\frac{CN}{Ke^{2\alpha_N t}}\Big\{\|\boldsymbol{I}_{t}^{N,K}\|^2_2+\|\cJ_{t}^{N,K}-\bar{J}_{t}^{N,K}\boldsymbol{1}_{K}\|_{2}^{2}\Big\}.
\end{align*}
By Lemma \ref{ItNK}, we have
\begin{align*}
    \indiq_{\Omega^{K,2}_N}\frac{N}{Ke^{2\alpha_N t}}\|\boldsymbol{I}_{t}^{N,K}\|^2_2\le   \indiq_{\Omega^{K,2}_N}\frac{N}{Ke^{2\alpha_N t}}t^2\frac{K}{N^\frac{3}{4}}=\indiq_{\Omega^{K,2}_N}\frac{t^2N^\frac{1}{4}}{e^{2\alpha_N t}}.
\end{align*}
As seen in the proof of Corollary \ref{vtea}, on $\Omega^{K,2}_N$, $\alpha_N>\alpha_0/2$  
(if $N$ is large enough), so that $\lim \indiq_{\Omega_{N}^{K,2}}\frac{t^2}{e^{\alpha_N t}} =0.$   
And $\indiq_{\Omega_{N}^{K,2}}\frac N{\sqrt K e^{\alpha_N t}} =0$, see Lemma \ref{reg},
implies that $\indiq_{\Omega^{K,2}_N}\frac{N^\frac{1}{4}}{e^{\alpha_N t}}=0.$ Hence, we obtain 
$\lim\indiq_{\Omega^{K,2}_N}\frac{N}{Ke^{2\alpha_N t}}\|\boldsymbol{I}_{t}^{N,K}\|^2_2=0.$

\vip

From Lemma \ref{MJUtNK}-(ii) and since $||\bV_N^K||_2\geq c \sqrt K$ on $\Omega_N^{K,2}$
(see the end of the proof of the previous lemma), 
$\mathbb{E}_{\theta}[\|\boldsymbol{J}_{t}^{N,K}-\bar{J}_{t}^{N,K}\boldsymbol{1}_{K}\|_{2}^{2}]
\le C(\frac KNe^{\alpha_{N}t}+\frac{e^{2\alpha_N t}}{N}\|\boldsymbol{V}_{N}^{K}-\bar{V}_{N}^{K}\boldsymbol{1}_{K}\|^2_{2}).$
Hence by Lemma \ref{VNK}-$(ii)$ and since $\alpha_N>\alpha_0/2$, we have:
\begin{align*}
    \indiq_{\Omega^{K,2}_N}\frac{N}{Ke^{2\alpha_N t}}\Et[\|\cJ_{t}^{N,K}-\bar{J}_{t}^{N,K}\boldsymbol{1}_{K}\|_{2}^{2}]\le  C\indiq_{\Omega^{K,2}_N}\Big(\frac{1}{e^{\alpha_{N}t}}+\frac{1}{N^{\frac{3}{4}}}\Big)\le C\indiq_{\Omega^{K,2}_N}\Big(\frac{1}{e^\frac{\alpha_{0}t}{2}}+\frac{1}{N^{\frac{3}{4}}}\Big).
\end{align*}
Finally  we have $\lim \E[\indiq_{\Omega^{K,2}_N}\frac{N}{Ke^{2\alpha_N t}}\Et[\|\cJ_{t}^{N,K}-\bar{J}_{t}^{N,K}\boldsymbol{1}_{K}\|_{2}^{2}]]=0$
which complete the proof.
\end{proof}

\begin{lemma}\label{DNK14}
Assume $(A)$. In the regime $(\ref{regime})$, in probability
$$
\lim\indiq_{\Omega_{N}^{K,2}\cap \{\bar{Z}_{t}^{N,K}\ge \frac{1}{4}v_{t}^{N,K}>0\}}\frac{e^{\alpha_{N} t}\sqrt{K}}{N}|\cD_{t}^{N,K,14}|=0.
$$
\end{lemma}

\begin{proof}
By $(\ref{M7})$ and Lemma \ref{supEZ2}, we have
\begin{align*}
\mathbb{E}_{\theta}\Big[\Big(\boldsymbol{M}_{t}^{N,K}-\bar{M}_{t}^{N,K}\boldsymbol{1}_{K},\bV_{N}^{K}-\bar{V}_{N}^{K}\boldsymbol{1}_{K}\Big)^{2}\Big]
&=\mathbb{E}_{\theta}\Big[\Big(\boldsymbol{M}_{t}^{N,K},\bV_{N}^{K}-\bar{V}_{N}^{K}\boldsymbol{1}_{K}\Big)^{2}\Big]
 \\&=\sum_{i=1}^{K}(V_{N}(i)-\bar{V}_{N}^{K})^{2}\mathbb{E}_{\theta}[Z_{t}^{i,N}]\\
 &\le C\|\bV_{N}^{K}-\bar{V}_{N}^{K}\|_{2}^{2}e^{\alpha_{N} t}.
 \end{align*}
By Corollary \ref{vtea}, we knwo that $ce^{\alpha_N t}\le \frac{1}{4}v_{t}^{N,K}\le\bar{Z}_{t}^{N,K}$ for $t$ big enough,
on the event  $\Omega_N^{K,2}\cap  \{\bar{Z}_{t}^{N,K}\ge \frac{1}{4}v_{t}^{N,K}>0\}$. 
By definition of $\cD_{t}^{N,K,14}$,
\begin{align*}
&\indiq_{\Omega^{K,2}_{N}\cap  \{\bar{Z}_{t}^{N,K}\ge \frac{1}{4}v_{t}^{N,K}>0\}}\frac{e^{\alpha_N t}\sqrt{K}}{N}\Et[|\cD_{t}^{N,K,14}|]\\
\le& \frac{C}{\sqrt{K}} \indiq_{\Omega^{K,2}_{N}}\mathbb{E}_{\theta}\Big[\Big(\boldsymbol{M}_{t}^{N,K}-\bar{M}_{t}^{N,K}\boldsymbol{1}_{K},\bV_{N}^{K}-\bar{V}_{N}^{K}\boldsymbol{1}_{K}\Big)^{2}\Big]^\frac{1}{2}\Big]\\
    \le& \indiq_{\Omega^{K,2}_{N}}\frac{Ce^\frac{\alpha_N t}{2}}{\sqrt{N}}\sqrt{
    \frac{N}{K}}\|\bV_{N}^{K}-\bar{V}_{N}^{K}\|_{2}.
\end{align*}
It suffices to gather Lemma \ref{reg}, from which $\boldsymbol{1}_{\Omega_{N}^{K,2}}e^{\alpha_N t}/N \to 0$ in probability
and Lemma \ref{VNK} (iii), from which $\indiq_{\Omega^{K,2}_{N}}\frac{N}{K}\|\bV_{N}^{K}-\bar{V}_{N}^{K}\|^2_{2}$ is bounded in $L^1$.
\end{proof}

Next, we  rewrite $\cD_{t}^{N,K,13}=\cD_{t}^{N,K,131}+\cD_{t}^{N,K,132}+\cD_{t}^{N,K,133}$,
where 
\begin{align*}
    \cD_{t}^{N,K,131}=&\frac{N}{K(\bar{Z}_{t}^{N,K})^{2}}\Big(\boldsymbol{I}_{t}^{N,K}-\bar{I}^{N,K}_{t}\boldsymbol{1}_{K}+\cJ_{t}^{N,K}-\bar{J}_{t}^{N,K}\boldsymbol{1}_{K},v_{t}^{N,K}(V_{N}^{K}-\bar{V}_{N}^{K}\boldsymbol{1}_{K})\Big),\\
    \cD_{t}^{N,K,132}=&\frac{N}{K(\bar{Z}_{t}^{N,K})^{2}}\Big(\boldsymbol{I}_{t}^{N,K}-\bar{I}^{N,K}_{t}\boldsymbol{1}_{K},\boldsymbol{M}_{t}^{N,K}-\bar{M}_{t}^{N,K}\boldsymbol{1}_{K}\Big),\\
    \cD_{t}^{N,K,133}=&\frac{N}{K(\bar{Z}_{t}^{N,K})^{2}}\Big(\cJ_{t}^{N,K}-\bar{J}_{t}^{N,K}\boldsymbol{1}_{K},\boldsymbol{M}_{t}^{N,K}-\bar{M}_{t}^{N,K}\boldsymbol{1}_{K}\Big).
\end{align*}
\begin{lemma}\label{DNK131}
We assume $(A)$. In the regime $(\ref{regime})$, in probability,
$$
\lim  \indiq_{\Omega^{K,2}_{N}\cap  \{\bar{Z}_{t}^{N,K}\ge \frac{1}{4}v_{t}^{N,K}>0\}}\frac{e^{\alpha_{N} t}\sqrt{K}}{N}|\cD_t^{N,K,131}|=0.
$$
\end{lemma}

\begin{proof}
By Corollary \ref{vtea}, we have 
$v_{t}^{N,K}\geq ce^{\alpha_N t}$ on the event $\Omega^{K,2}_{N}\cap  \{\bar{Z}_{t}^{N,K}\ge \frac{1}{4}v_{t}^{N,K}>0\}$
whence, by definition of $\cD_{t}^{N,K,131}$, 
\begin{align*}
    &\indiq_{\Omega^{K,2}_{N}\cap  \{\bar{Z}_{t}^{N,K}\ge \frac{1}{4}v_{t}^{N,K}>0\}}|\cD_{t}^{N,K,131}|\\
    \le& \indiq_{\Omega^{K,2}_{N}}\frac{N}{Ke^{\alpha_{N}t}}\Big(\|\boldsymbol{I}_{t}^{N,K}-\bar{I}_{t}^{N,K}\boldsymbol{1}_{K}\|_{2}+\|\boldsymbol{J}_{t}^{N,K}-\bar{J}_{t}^{N,K}\boldsymbol{1}_{K}\|_{2}\Big)
    \|\boldsymbol{V}_{N}^{K}-\bar{V}_{N}^{K}\boldsymbol{1}_{K}\|_{2}.
\end{align*}
Hence 
\begin{align*}
    &\indiq_{\Omega^{K,2}_{N}\cap  \{\bar{Z}_{t}^{N,K}\ge \frac{1}{4}v_{t}^{N,K}>0\}}\frac{e^{\alpha_N t}\sqrt{K}}{N}|\cD_{t}^{N,K,131}|\\
    \le& \indiq_{\Omega^{K,2}_{N}}\frac{C}{\sqrt{N}}\Big[\|\boldsymbol{I}_{t}^{N,K}\|_{2}+\|\boldsymbol{J}_{t}^{N,K}-\bar{J}_{t}^{N,K}\boldsymbol{1}_{K}\|_{2}\Big]
    \Big(\sqrt{\frac{N}{K}}\|\boldsymbol{V}_{N}^{K}-\bar{V}_{N}^{K}\boldsymbol{1}_{K}\|_{2}\Big).
\end{align*}
By Lemma \ref{ItNK},  we have
\begin{align*}
    \indiq_{\Omega^{K,2}_{N}}\frac{C}{\sqrt{N}}\|\boldsymbol{I}_{t}^{N,K}\|_{2}
    \Big(\sqrt{\frac{N}{K}}\|\boldsymbol{V}_{N}^{K}-\bar{V}_{N}^{K}\boldsymbol{1}_{K}\|_{2}\Big)
    \le&\indiq_{\Omega^{K,2}_{N}}\frac{Ct\sqrt{K}}{N^\frac{7}{8}}\Big(\sqrt{\frac{N}{K}}\|\boldsymbol{V}_{N}^{K}-\bar{V}_{N}^{K}\boldsymbol{1}_{K}\|_{2}\Big).
\end{align*}
By Lemma \ref{VNK}-(iii), we know that 
$\indiq_{\Omega_{N}^{K,2}}\frac{N}{K}\|\boldsymbol{V}_{N}^{K}-\bar{V}_{N}^{K}\boldsymbol{1}_{K}\|_{2}^{2}$ is bounded 
in $L^1$. In the regime $(\ref{regime}),$ we have $\lim \indiq_{\Omega^{K,2}_{N}}\frac{Ct\sqrt{K}}{N^\frac{7}{8}}=0.$ 
So $\lim \indiq_{\Omega^{K,2}_{N}}\frac{C}{\sqrt{N}}\|\boldsymbol{I}_{t}^{N,K}\|_{2}
    \Big(\sqrt{\frac{N}{K}}\|\boldsymbol{V}_{N}^{K}-\bar{V}_{N}^{K}\boldsymbol{1}_{K}\|_{2}\Big)=0$ in probability. 
%Then we just need to verify that $\lim \frac{1}{\sqrt{N}}\E[\indiq_{\Omega^{K,2}_{N}}\|
%\boldsymbol{J}_{t}^{N,K}-\bar{J}_{t}^{N,K}\boldsymbol{1}_{K}\|_{2}]=0.$

\vip
By Lemma \ref{MJUtNK}-$(ii),$ we conclude that:
\begin{align*}
    \frac{1}{\sqrt{N}}\E[\indiq_{\Omega^{K,2}_{N}}\|\boldsymbol{J}_{t}^{N,K}-\bar{J}_{t}^{N,K}\boldsymbol{1}_{K}\|^2_{2}]^\frac{1}{2}
    \le \frac{C\sqrt{K}}{N}\E\Big[\indiq_{\Omega^{K,2}_{N}}\Big\{e^{\frac{1}{2}\alpha_{N}t}+\frac{\|\boldsymbol{V}_{N}^{K}-\bar{V}_{N}^{K}\boldsymbol{1}_{K}\|_{2}}{\|\boldsymbol{V}_{N}^{K}\|_{2}}e^{\alpha_{N}t}\Big\}\Big].
\end{align*}
As already seen, on $\Omega_{N}^{K,2},$ we have $V_N(i)\ge \frac{1}{2}$, 
whence $\|\bV_N^K\|_2\ge c\sqrt{K}$, by \cite[Lemma 35-(ii)]{A}. And by Remark \ref{rhotop}, 
we see on $\Omega_{N}^{K,2},$ $e^{\alpha_N t}\le e^{\alpha_0 t} e^{\frac{p}{2N^{3/8}} t}. $
Hence
\begin{align*}
    \frac{1}{\sqrt{N}}\E[\indiq_{\Omega^{K,2}_{N}}\|\boldsymbol{J}_{t}^{N,K}-\bar{J}_{t}^{N,K}\boldsymbol{1}_{K}\|^2_{2}]^\frac{1}{2}\le \frac{Ce^{\frac{1}{2}\alpha_{0}t}e^{\frac{pt}{2N^{3/8}} }\sqrt{K}}{N}+\frac{e^{\alpha_{0}t}e^{\frac{pt}{2N^{3/8}}}}{N}\E\Big[\indiq_{\Omega^{K,2}_{N}}\|\boldsymbol{V}_{N}^{K}-\bar{V}_{N}^{K}\boldsymbol{1}_{K}\|_{2}\Big].
\end{align*}
In view of $(\ref{regime}),$ we know that $\lim e^{\frac{pt}{2N^{3/8}} }=1$ a.s., $\lim \frac{e^{\alpha_{0}t}}{N}=0.$
By  Lemma \ref{VNK}-(iii), from which $\indiq_{\Omega^{K,2}_{N}}\frac{N}{K}\|\bV_{N}^{K}-\bar{V}_{N}^{K}\|^2_{2}$ is bounded in $L^1$. Then we finish the proof.
\end{proof}
\begin{lemma}\label{DNK132}
We assume $(A)$. In the regime $(\ref{regime})$, in probability,
$$
\lim  \indiq_{\Omega^{K,2}_{N}\cap  \{\bar{Z}_{t}^{N,K}\ge \frac{1}{4}v_{t}^{N,K}>0\}}\frac{e^{\alpha_{N} t}\sqrt{K}}{N}|\cD_t^{N,K,132}|=0.
$$
\end{lemma}

\begin{proof}
By Corollary \ref{vtea}, we have $ce^{\alpha_N t}\le\bar{Z}_{t}^{N,K}$ on the event 
$\Omega^{K,2}_{N}\cap  \{\bar{Z}_{t}^{N,K}\ge \frac{1}{4}v_{t}^{N,K}>0\}$. By the Cauchy-Schwarz inequality,
\begin{align*}
    \indiq_{\Omega^{K,2}_{N}\cap  \{\bar{Z}_{t}^{N,K}\ge \frac{1}{4}v_{t}^{N,K}>0\}}\frac{e^{\alpha_N t}\sqrt{K}}{N}|\cD_{t}^{N,K,132}|
    \le \indiq_{\Omega^{K,2}_{N}}\frac{C}{e^{\alpha_{N}t}\sqrt{K}}\|\boldsymbol{I}_{t}^{N,K}\|_{2}\|\boldsymbol{M}_{t}^{N,K}-\bar{M}_{t}^{N,K}\boldsymbol{1}_{K}\|_{2}.
\end{align*}
By Lemmas \ref{ItNK} and \ref{MJUtNK}-(i),
$$
\indiq_{\Omega^{K,2}_{N}}\frac{C}{e^{\alpha_{N}t}\sqrt{K}}\Et[\|\boldsymbol{I}_{t}^{N,K}\|_{2}\|\boldsymbol{M}_{t}^{N,K}-\bar{M}_{t}^{N,K}\boldsymbol{1}_{K}\|_{2}]\le \indiq_{\Omega^{K,2}_{N}}\frac{Ct\sqrt{K}}{e^{\frac{\alpha_{N}t}{2}}N^\frac{3}{8}}.
$$
By Remark \ref{rhotop}, we obtain
$$
\E\Big[\indiq_{\Omega^{K,2}_{N}\cap  \{\bar{Z}_{t}^{N,K}\ge \frac{1}{4}v_{t}^{N,K}>0\}}\frac{e^{\alpha_{N} t}\sqrt{K}}{N}|\cD^{N,K,132}|\Big]\le\E\Big[ \indiq_{\Omega^{K,2}_{N}}\frac{Ct\sqrt{K}}{e^{\frac{\alpha_{N}t}{2}}N^\frac{3}{8}}\Big]\le e^{\frac{pt}{2N^{3/8}}}\frac{Ct\sqrt{K}}{e^{\frac{\alpha_{0}t}{2}}N^\frac{3}{8}}.
$$
In the regime $(\ref{regime})$,  we have $\lim e^{\frac{pt}{2N^{3/8}} }=1$ a.s. 
and $\lim \frac{t\sqrt{K}}{e^{\frac{\alpha_{0}t}{2}}N^\frac{3}{8}}=0$ a.s., which ends the proof.
\end{proof}

\begin{lemma}\label{JJM}
Assume $(A)$. There are some constants $C>0$, $t_0 \geq 1$ and  $N_0 \geq 1$, 
such that for all $N\ge N_0$, on the event, for all $t\ge t_0$,
$\Omega_{N}^{K,2}$,
$$
\Et\Big[( \cJ_{t}^{N,K}-\bar{J}^{N,K}_t,\bM_{t}^{N,K} )^2\Big]^\frac{1}{2}\le Cte^{\alpha_N t}\sqrt{\frac{K}{N}}+\frac{Cte^{\frac{3}{2}\alpha_N t}\sqrt{K}}{N^{\frac{7}{8}}}.
$$
\end{lemma}
\begin{proof}
In view of \cite[Proof of Lemma 35, Step 5]{A}, we already know that on $\Omega_N^2,$ for all 
$n\geq 2$, $j=1,...,N$,
$$
\frac{\max_i A_N^n(i,j) }{\min_i A_N^n(i,j)} \leq (1+2N^{-3/8})^{2}(1+8N^{-3/8}) \leq1+8N^{-3/8} .
$$
By \cite[Lemma 35-(iv)]{A}, we know that for all $n\geq 2$, $A_N^n(i,j)\in [\rho_N^n/(3N),3\rho_N^n/N]$. Hence we deduce that : for all $n\geq 2$, $i,\ j=1,...,N$,
$$
\Big|A_{N}^{n}(i,j)-\frac{1}{K}\sum_{k=1}^{K}A_{N}^{n}(k,j)\Big|\le \frac{C\rho_N^n}{N^{1+\frac{3}{8}}}.
$$
We then write, for $n\geq 1$,
\begin{align*}
\mathbb{E}_{\theta}\Big[\Big(I_{K}A_{N}^{n}\boldsymbol{M}_{s}^{N}-\overline{I_{K}A_{N}^{n}\boldsymbol{M}_{s}^{N}}\boldsymbol{1}_{K},\boldsymbol{M}_{s}^{N}\Big)^{2}\Big]
=\sum_{i,i'=1}^{K}\sum_{j,j'=1}^{N}\Big(A_{N}^{n}(i,j)-\frac{1}{K}\sum_{k=1}^{K}A_{N}^{n}(k,j)\Big)\hskip1cm
\\
\qquad\Big(A_{N}^{n}(i',j')-\frac{1}{K}\sum_{k=1}^{K}A_{N}^{n}(k,j')\Big)\Et[M^{j,N}_sM^{i,N}_tM^{j',N}_{s}M^{i',N}_t].
\end{align*}
By Lemma \ref{MMMMMM}-(iii)-(iv) in the appendix, we  conclude that, for $n\ge 2,$
\begin{align*}
\mathbb{E}_{\theta}\Big[\Big(I_{K}A_{N}^{n}\boldsymbol{M}_{s}^{N}-
\overline{I_{K}A_{N}^{n}\boldsymbol{M}_{s}^{N}}\boldsymbol{1}_{K},\boldsymbol{M}_{s}^{N}\Big)^{2}\Big]
\le &\frac{C\rho_N^{2n}}{N^{2+\frac{3}{4}}}\sum_{i,i'=1}^{K}\sum_{j,j'=1}^{N}|\Et[M^{j,N}_sM^{i,N}_tM^{j',N}_{s}M^{i',N}_t]|\\
\le& \frac{CNK}{N^{2+\frac{3}{4}}}\rho^{2n}_N e^{\alpha_N (t+s)}+ \frac{CN^2Kt^2}{N^2N^{2+\frac{3}{4}}}\rho^{2n}_N e^{\alpha_N t}\\
\le& \frac{CKt^2}{N^{1+\frac{3}{4}}}\rho^{2n}_N e^{\alpha_N (t+s)}.
\end{align*}
For $n=1,$ we just use $A_N(i,j)\le \frac{1}{N}$ to write,
\begin{align*}
\mathbb{E}_{\theta}\Big[\Big(I_{K}A_{N}\boldsymbol{M}_{s}^{N}-\overline{I_{K}A_{N}\boldsymbol{M}_{s}^{N}}\boldsymbol{1}_{K},\boldsymbol{M}_{s}^{N}\Big)^{2}\Big]
\le&  \frac{C}{N^2}\sum_{i,i'=1}^{K}\sum_{j,j'=1}^{N}|\Et[M^{j,N}_sM^{i,N}_tM^{j',N}_{s}M^{i',N}_t]|\\
 \le& \frac{CNK}{N^{2}} e^{\alpha_N (t+s)}+ \frac{Ct^2N^2K}{N^2N^{2}} e^{\alpha_N t}\\
\le& \frac{CKt^2}{N} e^{\alpha_N (t+s)}
\end{align*}
by Lemma \ref{MMMMMM}-(iii)-(iv) again.
Then we conclude, recalling \eqref{JtNK},
\begin{align*}
&\Et\Big[( \cJ_{t}^{N,K}-\bar{J}^{N,K}_t,\bM_{t}^{N,K} )^2\Big]^\frac{1}{2}\\
\le&  \sum_{n\ge 1}\int_{0}^{t}\phi^{*n}(t-s)\mathbb{E}_{\theta}\Big[\Big(I_{K}A_{N}^{n}\boldsymbol{M}_{s}^{N}-\overline{I_{K}A_{N}^{n}\boldsymbol{M}_{s}^{N}}\boldsymbol{1}_{K},\boldsymbol{M}_{s}^{N}\Big)^{2}\Big]^{\frac{1}{2}}ds\\
\le& Cte^{\frac{\alpha_N t}{2}}\Big\{\sqrt{\frac{K}{N}}\int_{0}^{t}e^{\frac{\alpha_N s}{2}}\phi(t-s)ds+\frac{C\sqrt{K}}{N^{\frac{7}{8}}}\sum_{n\ge 1}\int_{0}^{t}\rho^{n}_N e^{\frac{\alpha_N s}{2}}\phi^{*n}(t-s)ds\Big\}\\
=& Cte^{\frac{\alpha_N t}{2}}\Big\{\sqrt{\frac{K}{N}}\int_{0}^{t}e^{\frac{\alpha_N s}{2}}e^{-b(t-s)}ds+\frac{C\rho_N\sqrt{K}}{N^{\frac{7}{8}}}\sum_{n\ge 0}\int_{0}^{t}\frac{\rho^{n}_N(t-s)^n}{n!} e^{\frac{\alpha_N s}{2}}e^{-b(t-s)}ds\Big\}\\
\le& Cte^{\alpha_N t}\sqrt{\frac{K}{N}}+\frac{Cte^{\frac{3}{2}\alpha_N t}\sqrt{K}}{N^{\frac{7}{8}}}.
\end{align*}
We used that $\phi^{*n}(t)=t^{n-1}e^{-b t} /(n-1)!$ for all $n\geq 1$.
\end{proof}

\begin{lemma}\label{DNK133}
We assume $(A)$.  In the regime $(\ref{regime})$, in probability,
$$
\lim \indiq_{\Omega_{N}^{K,2}\cap \{\bar{Z}_{t}^{N,K}\ge \frac{1}{4}v_{t}^{N,K}>0\}}\frac{e^{\alpha_{N} t}\sqrt{K}}{N}|\cD_{t}^{N,K,133}|=0.
$$
\end{lemma}

\begin{proof}
By Corollary \ref{vtea}, we know that on the event 
$\Omega^{K,2}_{N}\cap  \{\bar{Z}_{t}^{N,K}\ge \frac{1}{4}v_{t}^{N,K}>0\},$ we always have 
$ce^{\alpha_N t}\le \bar{Z}_{t}^{N,K}$. With Lemma \ref{JJM}, we conclude that, by Definition of $\cD_{t}^{N,K,133}$,
\begin{align*}
    \indiq_{\Omega_{N,K}\cap  \{\bar{Z}_{t}^{N,K}\ge \frac{1}{4}v_{t}^{N,K}>0\}}\frac{e^{\alpha_N t}\sqrt{K}}{N}\mathbb{E}_\theta[|\cD_{t}^{N,K,133}|]\le& \frac{Ct}{\sqrt{N}}+\frac{Cte^{\frac{\alpha_N t}{2}}}{N^\frac{7}{8}}
    \le \frac{Ct}{\sqrt{N}}+e^{\frac{pt}{2N^{3/8}}}\frac{Cte^{\frac{\alpha_0 t}{2}}}{N^\frac{7}{8}}.
\end{align*}
In view of $(\ref{regime}),$ the proof is complete.
\end{proof}
Summarizing Lemmas 
\ref{DNK131}, \ref{DNK132} and \ref{DNK133}, we have the following corollary:
 
 \begin{cor}\label{DNK13}
Assume $(A)$. In the regime $(\ref{regime})$,  in probability
$$
\lim\indiq_{\Omega^{K,2}_{N}\cap  \{\bar{Z}_{t}^{N,K}\ge \frac{1}{4}v_{t}^{N,K}>0\}}\frac{e^{\alpha_N t}\sqrt{K}}{N}|\cD_{t}^{N,K,13}|=0.
$$
\end{cor}

By Corollaries \ref{DNK2}, \ref{DNK11}, \ref{DNK14} \ref{DNK13} and  Lemma \ref{Up111}, we conclude:

\begin{cor}\label{D1234}
Assume $(A)$. In the regime $(\ref{regime})$, in probability
 $$
 \lim\frac{e^{\alpha_N t}\sqrt{K}}{N}\Big\{|\cD_{t}^{N,K,2}|+|\cD_{t}^{N,K,3}|+|\cD_{t}^{N,K,11}|+|\cD_{t}^{N,K,13}|+|\cD_{t}^{N,K,14}|\Big\}=0
 $$
 \end{cor}

It only remains to study $\cD_{t}^{N,K,12}$. We need some preparation.

\begin{lemma}\label{a0b}Assume $(A)$. In the regime $(\ref{regime})$, in probability
$$\lim\indiq_{\Omega_{N}^{K,2}}\frac{\mathbb{E}_\theta[\bar{Z}_{t}^{N,K}]}{e^{\alpha_N t}}=\frac {\mu p}{(\alpha_0)^2},\quad
\lim\indiq_{\Omega_{N}^{K,2}}\frac{\bar{Z}_{t}^{N,K}}{e^{\alpha_N t}}=\frac{\mu p}{(\alpha_0)^2},\quad
\lim\indiq_{\Omega_{N}^{K,2}}\frac{\sum_{i=1}^K(Z_{t}^{i,N})^2}{Ke^{2\alpha_N t}}=\frac{(\mu p)^2}{(\alpha_0)^4} .$$
\end{lemma}

\begin{proof}
Corollary \ref{INbar} tells us that 
$\lim\indiq_{\Omega_{N}^{K,2}}\frac{\mathbb{E}_\theta[|\bar{I}_{t}^{N,K}|]}{e^{\alpha_N t}}=0$ in probability.
In view of $(\ref{defI}),$  we have, in probability,
$$
\lim\indiq_{\Omega_{N}^{K,2}}\frac{\mathbb{E}_\theta[\bar{Z}_{t}^{N,K}]}{e^{\alpha_N t}}=\lim\indiq_{\Omega_{N}^{K,2}}\frac{v_t^{N,K}\bar{V}_N^K}{e^{\alpha_N t}}.
$$
From $(\ref{zyx})$ (and a few lines after) 
$\lim \indiq_{\Omega_{N}^{K,2}}\bar{V}_N^K=\lim \indiq_{\Omega_{N}^{K,2}}\kappa_N=1.$ From Remark \ref{rhotop}, we already have $\lim\alpha_N=\alpha_0 $ and $\lim\rho_N=p.$ 
And by Corollary \ref{vtnk}, we know that 
$\lim \indiq_{\Omega_{N}^{K,2}}\frac{v_t^{N,K}}{e^{\alpha_N t}}=\lim \indiq_{\Omega_{N}^{K,2}}\frac{\mu\rho_N}{(\alpha_N)^2}
=\frac{\mu p}{(\alpha_0)^2}.$
This finishes the proof of the first part.
  
By Lemma \ref{UtNK2},  we have $\lim\boldsymbol{1}_{\Omega_{N}^{K,2}}\frac{\Et[|\bar{U}_{t}^{N,K}|]}{e^{\alpha_N t}}=0$ in $L^1$.  
Then we conclude that, in probability,
\begin{align*}
\lim\indiq_{\Omega_{N}^{K,2}}\bar{Z}_{t}^{N,K}/e^{\alpha_N t}=&\lim\indiq_{\Omega_{N}^{K,2}}\Big(\bar{U}_{t}^{N,K}+\mathbb{E}_\theta[\bar{Z}_{t}^{N,K}]\Big)/e^{\alpha_N t }\\
=&\lim\indiq_{\Omega_{N}^{K,2}}\mathbb{E}_\theta[\bar{Z}_{t}^{N,K}]/e^{\alpha_N t}=(\mu p)/(\alpha_0)^2 .
 \end{align*}
 For the third part,  we start from:
 $$
 \frac{\sum_{i=1}^K(Z_{t}^{i,N})^2}{Ke^{2\alpha_N t}}=\frac{\sum_{i=1}^K\Et[Z_{t}^{i,N}]^2+\sum_{i=1}^K(U_{t}^{i,N})^2+2\sum_{i=1}^K\Et[Z_{t}^{i,N}]U_{t}^{i,N}}{Ke^{2\alpha_N t}}.
 $$
 As seen in the proof of Corollary \ref{vtea}, on $\Omega^{K,2}_N$, $\alpha_N>\alpha_0/2$  
(if $N$ is large enough). By Lemma \ref{supEZ2}-(ii), we have:
 \begin{align*}
     \E\Big[\indiq_{\Omega_{N}^{K,2}}\frac{\sum_{i=1}^K(U_{t}^{i,N})^2}{Ke^{2\alpha_N t}}\Big]\le  \E\Big[\indiq_{\Omega_{N}^{K,2}}\Big(\frac{C}{N}+\frac{C}{e^{\alpha_N t}}\Big)\Big]\le \frac{C}{N}+\frac{C}{e^{\frac{1}{2}\alpha_0t}},
 \end{align*}
 which implies $\lim \indiq_{\Omega_{N}^{K,2}}\frac{\sum_{i=1}^K(U_{t}^{i,N})^2}{Ke^{2\alpha_N t}}=0$ in probability. Recall Lemma \ref{supEZ}: we already have $\indiq_{\Omega_{N}^{K,2}}\|\mathbb{E}_\theta[\boldsymbol{Z}_{t}^{N,K}]\|_\infty\le Ce^{\alpha_{N}t}.$ Hence, by Lemma \ref{supEZ2}-(ii), we have
 \begin{align*}
     \E\Big[\indiq_{\Omega_{N}^{K,2}}\frac{\sum_{i=1}^K\Et[Z_{t}^{i,N}]|U_{t}^{i,N}|}{Ke^{2\alpha_N t}}\Big]\le&  \E\Big[\indiq_{\Omega_{N}^{K,2}}\frac{\sum_{i=1}^K|U_{t}^{i,N}|}{Ke^{\alpha_N t}}\Big]\\\le& \E\Big[\indiq_{\Omega_{N}^{K,2}}\Big(\frac{C}{\sqrt{N}}+\frac{C}{e^{\frac{1}{2}\alpha_N t}}\Big)\Big]\le \frac{C}{\sqrt{N}}+\frac{C}{e^{\frac{1}{4}\alpha_0t}},
 \end{align*}
which tends to $0$ in probability.
By Lemma \ref{VNK}-(i), we have $\lim \frac{\|\boldsymbol{V}^K_N\|^2_2}{K}=\lim  (\bar V_N^K)^2=1.$
By Lemma \ref{ItNK}, we see that $\lim\frac{\|\boldsymbol{I}_t^{N,K}\|^2_2}{Ke^{2\alpha_N t}}=0.$ 
Hence, in view of \eqref{defI} by Lemma \ref{vtnk}, 
 \begin{align*}
     \lim\frac{\sum_{i=1}^K\Et[Z_{t}^{i,N}]^2}{Ke^{2\alpha_N t}}=\lim\frac{(v_{t}^{N,K})^2\|\boldsymbol{V}^K_N\|^2_2
+\|\boldsymbol{I}_t^{N,K}\|^2_2+ 2v_{t}^{N,K} (\boldsymbol{V}^K_N,\boldsymbol{I}_t^{N,K})}
{Ke^{2\alpha_N t}}=\frac{(\mu p)^2}{(\alpha_0)^4},
 \end{align*}
 which complete the proof.
\end{proof}

Next, we consider the term $\cD_t^{N,K,12}.$ By It\^o's formula and \eqref{M7}, we have:
\begin{align*}
    \sum_{i=1}^{K}(M_{t}^{i,N}-\bar{M}_{t}^{N,K})^2-K\bar{Z}_{t}^{N,K}=\sum_{i=1}^{K}2\int_{0}^t M_{s-}^{i,N}dM_{s}^{i,N}-K(\bar{M}_{t}^{N,K})^2.
\end{align*}
On $\Omega^{K,2}_{N}\cap  \{\bar{Z}_{t}^{N,K}\ge \frac{1}{4}v_{t}^{N,K}>0\}$, we write
$\cD_{t}^{N,K,12}=2\cD_{t}^{N,K,121}-\cD_{t}^{N,K,122}$, where
\begin{gather*}
    \cD_{t}^{N,K,121}=\frac{N}{K(\bar{Z}_{t}^{N,K})^{2}}\sum_{i=1}^{K}\int_{0}^t M_{s-}^{i,N}dM_{s}^{i,N},\\
    \cD_{t}^{N,K,122}=\frac{N(\bar{M}_{t}^{N,K})^2}{(\bar{Z}_{t}^{N,K})^{2}}.
\end{gather*}
\begin{lemma}\label{DNK122}
Assume $(A)$. In the regime $(\ref{regime})$,  in probability
 $$
 \lim\indiq_{\Omega^{K,2}_{N}\cap  \{\bar{Z}_{t}^{N,K}\ge \frac{1}{4}v_{t}^{N,K}>0\}}\frac{e^{\alpha_N t}\sqrt{K}}{N}|\cD_{t}^{N,K,122}|=0.
 $$
\end{lemma}
\begin{proof}
In view of $(\ref{M7}),$ by Lemma \ref{supEZ}, we have 
\begin{align*}
\indiq_{\Omega_{N}^{K,2}}\Et[(\bar{M}_{t}^{N,K})^2]= \indiq_{\Omega_{N}^{K,2}}\frac{1}{K^2}\sum_{i=1}^K\Et[Z_t^{i,N}]
\le  \frac{C}{K}\indiq_{\Omega_{N}^{K,2}}e^{\alpha_N t}.
\end{align*}
By Corollary \ref{vtea},$\bar{Z}_{t}^{N,K} \geq ce^{\alpha_N t}$, fro $t$ large enough, 
on the event  $\Omega_N^{K,2}\cap  \{\bar{Z}_{t}^{N,K}\ge \frac{1}{4}v_{t}^{N,K}>0\}$.
Hence
$$\mathbb{E}\Big[\boldsymbol{1}_{\Omega_{N}^{K,2}\cap\{\bar{Z}_{t}^{N,K}\ge \frac{1}{4}v_{t}^{N,K}>0\}}e^{\alpha_{N} t}\frac{\sqrt{K}(\bar{M}_{t}^{N,K})^2}{(\bar{Z}_{t}^{N,K})^2}\Big]\le \frac{C}{\sqrt{K}}.$$
This completes the proof.
\end{proof}

For $t\geq 1$ and $1\leq K \leq N$, we set, for $m\in[0,1],$ 
$$
\mathcal{E}^{t}_{N,K}(m):=e^{-2\alpha_N t}\sum_{i=1}^{K}\int_{0}^{\varphi_t(m)}M^{i,N}_{s-}dM_{s}^{i,N} \quad
\hbox{where} \quad \varphi_t(m)=t + \frac{1}{2\alpha_0} \log[(1-e^{-2\alpha_0 t}) m +e^{-2\alpha_0 t} ],
$$
so that $\mathcal{E}^{t}_{N,K}$ is a martingale (in the filtration ${\mathcal G}_m=\cF_{\varphi_t(m)})$ issued from $0$.
Recalling the definition of $\cD_t^{N,K,121}$, we are interested in the convergence of
$\mathcal{E}^{t}_{N,K}(1)$.

\begin{lemma}\label{jump2}
Assume $(A)$. In the regime $(\ref{regime})$,  in probability,
$$\lim\boldsymbol{1}_{\Omega_{N}^{K,2}}\frac{e^{\alpha_N t}}{\sqrt{K}}\sup_{0\le m\le 1}\Big|\mathcal{E}^{t}_{N,K}(m)-\mathcal{E}^{t}_{N,K}(m-)\Big|=0.$$
\end{lemma}
\begin{proof}
Recall that $\alpha_0>0$.
By Doob's inequality and Lemma \ref{supEZ}, we have:
$$
\indiq_{\Omega^{K,2}_{N}}\max_{i=1,\dots,K}\mathbb{E}_\theta\Big[\sup_{[0,t]}(M_{s}^{i,N})^2\Big]\le C \indiq_{\Omega_{N}^{K,2}}\max_{i=1,...,K}\|\mathbb{E}_\theta[\boldsymbol{Z}_{t}^{N,K}]\|_\infty\le Ce^{\alpha_{N}t}.
$$
Hence, since the jumps of all our martingales have size $1$ and since they never jump simultaneously,
\begin{align*}
\boldsymbol{1}_{\Omega^{K,2}_{N}}\mathbb{E}_{\theta}\Big[\sup_{0\le m\le 1}\Big|\mathcal{E}^{t}_{N,K}(m)-\mathcal{E}^{t}_{N,K}(m-)\Big|\Big]
    \le& \boldsymbol{1}_{\Omega^{K,2}_{N}}e^{-2\alpha_N t}\mathbb{E}_{\theta}\Big[\sup_{0\le s\le t}\max_{i=1,...,K}|M^{i,N}_{s}|\Big]\\
    \le& \boldsymbol{1}_{\Omega^{K,2}_{N}}e^{-2\alpha_N t}\mathbb{E}_{\theta}\Big[\sqrt{\sup_{0\le s\le t}\sum_{i=1}^K|M^{i,N}_{s}|^2}\Big]\\
    \le& \boldsymbol{1}_{\Omega^{K,2}_{N}}e^{-2\alpha_N t}\mathbb{E}_{\theta}\Big[\sup_{0\le s\le t}\sum_{i=1}^K|M^{i,N}_{s}|^2\Big]^\frac{1}{2}\\
    \le& C\boldsymbol{1}_{\Omega^{K,2}_{N}}\sqrt{K}e^{\frac{-3\alpha_N t}{2}}.
\end{align*}
As seen in the proof of Corollary 
\ref{vtea}, on $\Omega^{K,2}_N$, $\alpha_N>\alpha_0/2$  (if $N$ is large enough). We finally conclude that
\begin{align*}
    \E\Big[\boldsymbol{1}_{\Omega_{N}^{K,2}}\frac{e^{\alpha_N t}}{\sqrt{K}}\sup_{0\le m \le 1}\Big|\mathcal{E}^{t}_{N,K}(m)-\mathcal{E}^{t}_{N,K}(m-)\Big|\Big]\le Ce^{-\frac{\alpha_0 t}{4}},
\end{align*}
which ends the proof.
\end{proof}

\begin{lemma}\label{mainlemmasup}
We assume $(A)$. In the regime $(\ref{regime})$, it holds
\begin{align*}
\boldsymbol{1}_{\Omega_{N}^{K,2}}\Big(\frac{e^{\alpha_N t}\mathcal{E}^{t}_{N,K}(m)}{\sqrt{K}}\Big)_{0\le m\le 1}\stackrel{d}{\longrightarrow}\frac{\mu p}{\sqrt{2}(\alpha_0)^2}(B_m)_{0\le m\le 1}
\end{align*}
for the Skorohod topology, where $B$ is a standard Brownian motion.
\end{lemma}

\begin{proof}
By \cite[Chapter VIII, Theorem 3.8-(b)]{B} and thanks to Lemma \ref{jump2}, it suffices to verify that
$\lim \boldsymbol{1}_{\Omega_{N}^{K,2}}C_{N,K}^t(m)=\frac{(\mu p)^2}{2(\alpha_0)^4} m$
in probability,  for all $0\le m\le 1$, where
$$
C_{N,K}^t(m)=\Big[\frac{e^{\alpha_N t}\mathcal{E}^{t}_{N,K}(.)}{\sqrt{K}}
,\frac{e^{\alpha_N t}\mathcal{E}^{t}_{N,K}(.)}{\sqrt{K}}\Big]_m.
$$

\vip

We start from
\begin{align*}
C_{N,K}^t(m)=\frac{e^{2\alpha_N t}}{K} \Big[\mathcal{E}^{t}_{N,K}(.),\mathcal{E}^{t}_{N,K}(.)\Big]_m
=\frac1{Ke^{2\alpha_N t}}\sum_{i=1}^K \Big[\int_{0}^{.}M_{s-}^{i,N}dM_{s}^{i,N},\int_{0}^{.}M_{s-}^{i,N}dM_{s}^{i,N}
\Big]_{\varphi_t(m)}
\end{align*}
by $(\ref{M7})$, from which we also have
$$
C_{N,K}^t(m)=\frac1{Ke^{2\alpha_N t}}\sum_{i=1}^K \int_0^{\varphi_t(m)}(M_{s-}^{i,N})^2 dZ_{s}^{i,N}.
$$
Using now It\^o's formula, we find
\begin{align*}
C_{N,K}^t(m)=&\frac1{Ke^{2\alpha_N t}}\sum_{i=1}^K\Big[2\int_0^{\varphi_t(m)}\Big(\int_0^{s-}M_{l-}^{i,N}dM^{i,N}_l\Big) dZ_{s}^{i,N}
+ \int_0^{\varphi_t(m)}Z_{s-}^{i,N}dZ_{s}^{i,N} \Big]\\
=&\frac2{Ke^{2\alpha_N t}}\sum_{i=1}^K\int_0^{\varphi_t(m)}\Big(\int_0^{s-}M_{l-}^{i,N}dM^{i,N}_l\Big) dZ_{s}^{i,N}
+ \frac1{2Ke^{2\alpha_N t}}\sum_{i=1}^K \Big[
 [Z_{\varphi_t(m)}^{i,N}]^2 - Z_{\varphi_t(m)}^{i,N} \Big] \\
=& C_{N,K}^{t,1}(m) +  C_{N,K}^{t,2}(m),
\end{align*}
the last equality standing for a definition.
\vip

{\bf Step 1.}
In this step, we prove that $\lim\boldsymbol{1}_{\Omega_{N}^{K,2}}C_{N,K}^{t,1}(m)$
tends to $0$ in probability, for each fixed $m \in [0,1]$.
By integration by parts, we have
\begin{align*}
&\int_{0}^{u}\Big(\int_{0}^{s-}M_{l-}^{i,N}dM_{l}^{i,N}\Big)dZ_{s}^{i,N} \\
   =&Z_{u}^{i,N}\int_{0}^{u}M_{l-}^{i,N}dM_{l}^{i,N}-\int_{0}^{u}M_{l-}^{i,N}Z_{l-}^{i,N}dM_{l}^{i,N}-\int_{0}^{u}M_{l-}^{i,N}dZ_{l}^{i,N}\\
   =&U_{u}^{i,N}\int_{0}^{u}M_{l-}^{i,N}dM_{l}^{i,N}+\int_{0}^{u}M_{l-}^{i,N}(\Et[Z_{u}^{i,N}]-Z_{l-}^{i,N})dM_{l}^{i,N}
-\int_{0}^{u}M_{l-}^{i,N}dZ_{l}^{i,N}.
\end{align*}
We write $C_{N,K}^{t,1}(m)=C_{N,K}^{t,11}(m)+C_{N,K}^{t,12}(m)+C_{N,K}^{t,13}(m)$ accordingly.

\vip

{\bf Step 1.1.} Concerning $C_{N,K}^{t,11}(m)$, we first observe, using again an integration by parts, that
\begin{align*}
\Big|U_{u}^{i,N}\int_{0}^{u}M_{l-}^{i,N}dM_{l}^{i,N}\Big|=&\frac{1}{2} \Big|U_{u}^{i,N}\Big((M_{u}^{i,N})^2-Z_{u}^{i,N}\Big)\Big|
\le Ct \Big(U_{u}^{i,N}\Big)^2+\frac{C}{t}\Big((M_{u}^{i,N})^4+(Z_{u}^{i,N})^2\Big).
\end{align*}
As seen in the proof of Corollary \ref{vtea}, on $\Omega^{K,2}_N$, $\alpha_N>\alpha_0/2$  
(if $N$ is large enough). By Lemma \ref{supEZ2}-(ii), we have: for $0\le u\le t$,
\begin{align*}
    \indiq_{\Omega_N^{K,2}}\frac{t}{Ke^{2\alpha_N t}}\sum_{i=1}^K\Et\Big[\Big(U_{u}^{i,N}\Big)^2\Big]
\le C\indiq_{\Omega_N^{K,2}}\Big(\frac{t}{N}+\frac{t}{e^{\alpha_N t}}\Big)\le C\Big(\frac{t}{N}
+\frac{t}{e^{\frac{1}{2}\alpha_0 t}}\Big).
\end{align*}
In view of $(\ref{M7}),$ by Lemma \ref{supEZ2}-(i), we conclude that, for $u \in [0,t]$,
\begin{align*}
    \indiq_{\Omega_N^{K,2}}\frac{1}{Ke^{2\alpha_N t}}\sum_{i=1}^K\Et\Big[\frac{C}{t}\Big((M_{u}^{i,N})^4+(Z_{u}^{i,N})^2\Big)\Big]\le \frac{C}{t}.
\end{align*}
This concludes the sub-step.

\vip

{\bf Step 1.2}
To study $C_{N,K}^{t,12}(m)$, we first use 
$(\ref{M7})$: on $\Omega_{N}^{K,2}$,
\begin{align*}
    &\Et\Big[\Big(\sum_{i=1}^K\int_{0}^{u}M_{l-}^{i,N}(\Et[Z_{u}^{i,N}]-Z_{l-}^{i,N})dM_{l}^{i,N}\Big)^2\Big]\\
    =&\Et\Big[\sum_{i=1}^K\int_{0}^{u}\Big(\Et[Z_{u}^{i,N}]-Z_{l-}^{i,N}\Big)^2\Big(M_{l-}^{i,N}\Big)^2dZ_{l}^{i,N}\Big]\\
    \le& \sum_{i=1}^K\Et\Big[\Big(\Et[Z_{u}^{i,N}]+Z_{u}^{i,N}\Big)^2\int_{0}^{u}\Big(M_{l-}^{i,N}\Big)^2dZ_{l}^{i,N}\Big]\\
    \le& 2\sum_{i=1}^K\Et\Big[\Big(\Et[Z_{u}^{i,N}]+Z_{u}^{i,N}\Big)^4+\Big(\int_{0}^{u}\Big(M_{l-}^{i,N}\Big)^2dZ_{l}^{i,N}\Big)^2\Big]\\
    \le& C\sum_{i=1}^K\Et\Big[\Big(Z_{u}^{i,N}\Big)^4+\Big(\sup_{0\le l\le u}\Big(M_{l}^{i,N}\Big)^2Z_{u}^{i,N}\Big)^2\Big]\\
   \le& C\sum_{i=1}^K\Et\Big[\Big(Z_{u}^{i,N}\Big)^4+\sup_{0\le l\le u}\Big(M_{l}^{i,N}\Big)^8+\Big(Z_{u}^{i,N}\Big)^4\Big].
\end{align*}
By Doob's inequality, we have on $\Omega_{N}^{K,2}$, $\Et[\sup_{0\le l\le u}(M_{l}^{i,N})^8]\le C\Et[(Z_{u}^{i,N})^4]
\leq CKe^{4\alpha_N t}$ by Lemma \ref{supEZ2} (iii). We conclude that on $\Omega_{N}^{K,2}$, for any $m\in [0,1]$,
$$
\Et[C_{N,K}^{t,12}(m)]\leq \frac C{\sqrt K}.
$$

{\bf Step 1.3}
We next write, for $u\in [0,t]$,
$$
\Big|\int_{0}^{u}M_{l-}^{i,N}dZ_{l}^{i,N}\Big|\le Z_{u}^{i,N}\sup_{0\le s\le u}|M_{s}^{i,N}|\le t^{-1}(Z_{u}^{i,N})^2
+t\sup_{0\le s\le u}(M_{s}^{i,N})^2.
$$
In view of $(\ref{M7})$, we have, on $\Omega_N^{K,2}$,
\begin{align*}
\sum_{i=1}^K\Et\Big[\Big|\int_{0}^{u}M_{l-}^{i,N}dZ_{l}^{i,N}\Big|\Big]\le 
t^{-1}\Et[(Z_{u}^{i,N})^2] + C t \E[Z_{u}^{i,N}] \leq 
\frac{CKe^{2\alpha_N t}}{t}+CKte^{\alpha_N t} 
\end{align*}
by Lemma \ref{supEZ2}-(i). This implies that, still on $\Omega_N^{K,2}$,
$$
\Et[|C_{N,K}^{t,13}(m)|] \leq C \Big(\frac 1 t + \frac t {e^{\alpha_N t}} \Big),
$$
which tends to $0$.

\vip

{\bf Step 2.}
We next study $C_{N,K}^{t,2}(m)$.
Using Lemma \ref{a0b} and that $\varphi_t(m) \in [0,t]$, we conclude that, in probability
\begin{align*}
 \lim\boldsymbol{1}_{\Omega_{N}^{K,2}}\frac{\bar{Z}_{\varphi_t(m)}^{N,K}}{e^{2\alpha_N t}}=0.
\end{align*}
in probability.
Since furthermore, by Corollary \ref{PvtZt}, we have $\lim \boldsymbol{1}_{\Omega_{N}^{K,2}}=1$ in probability,
we are reduced to check that, for all $m \in (0,1]1$,
$$
\lim \frac1{2Ke^{2\alpha_N t}}\sum_{i=1}^K  [Z_{\varphi_t(m)}^{i,N}]^2 = \frac{(\mu p)^2}{2 \alpha_0^4}m.
$$
We write
$$
\lim \frac1{2Ke^{2\alpha_N t}}\sum_{i=1}^K  [Z_{\varphi_t(m)}^{i,N}]^2 = 
\lim \Big( \frac1{2Ke^{2\alpha_N \varphi_t(m)}}\sum_{i=1}^K  [Z_{\varphi_t(m)}^{i,N}]^2\Big) e^{2\alpha_N (\varphi_t(m)-t)}
= \frac{(\mu p)^2}{2 \alpha_0^4} m
$$
by Lemma \ref{a0b} and since $\varphi_t(m)=t+\frac1{2\alpha_0}\log(m(1-e^{-2\alpha_0 t})+ e^{-2\alpha_0 t})$
and since $\alpha_N \to \alpha_0$, see Remark \ref{rhotop}.
The proof is complete.
\end{proof}

\begin{cor}\label{D12h}
Assume $(A)$. In the regime $(\ref{regime})$, we have the following convergence in distribution
\begin{align*}
\boldsymbol{1}_{\Omega_{N}^{K,2}\cap\{\bar{Z}_{t}^{N,K}\ge \frac{1}{4}v_{t}^{N,K}>0\}}\frac{e^{\alpha_N t}\sqrt{K}\cD_{t}^{N,K,12}}
{N}\to\mathcal{N}\Big(0,\frac{2(\alpha_0)^4}{(\mu p)^2}\Big).
\end{align*}
\end{cor}
\begin{proof}
First, we know from Corollary \ref{PvtZt} that 
$\lim\indiq_{\Omega^{K,2}_{N}\cap  \{\bar{Z}_{t}^{N,K}\ge \frac{1}{4}v_{t}^{N,K}>0\}}=1$ in probability. 
Also, we recall that $\cD_{t}^{N,K,12}=2\cD_{t}^{N,K,121}-\cD_{t}^{N,K,122}$ and that 
by Lemma $\ref{DNK122},$
$$\lim\indiq_{\Omega^{K,2}_{N}\cap  \{\bar{Z}_{t}^{N,K}\ge \frac{1}{4}v_{t}^{N,K}>0\}}\frac{e^{\alpha_N t}\sqrt{K}}{N}|\cD_{t}^{N,K,122}|=0$$ in probability. 
Since next 
$$
\cD_{t}^{N,K,121} = \frac{N}{K(\bar{Z}_{t}^{N,K})^{2}}\sum_{i=1}^{K}\int_{0}^t M_{s-}^{i,N}dM_{s}^{i,N} = 
\frac NK \frac{ e^{2\alpha_N t}}{(\bar Z^{N,K}_t)^2} \cE^{t}_{N,K}(1)
$$
and since $\lim \frac{ e^{2\alpha_N t}}{(\bar Z^{N,K}_t)^2} = (\alpha_0^2/(\mu p))^2$ by Lemma 
\ref{a0b}, we conclude from Lemma \ref{mainlemmasup} that
$$
\lim \boldsymbol{1}_{\Omega_{N}^{K,2}\cap\{\bar{Z}_{t}^{N,K}\ge \frac{1}{4}v_{t}^{N,K}>0\}}\frac{e^{\alpha_N t}\sqrt{K}\cD_{t}^{N,K,12}}
{N}\to \frac{\sqrt 2 \alpha_0^2}{\mu p} B_1
$$
in distribution.
\end{proof}

\begin{proof}[{Proof of Proposition \ref{mainsup}}]
We recall that we have written
$$
\cD_{t}^{N,K}=\cD_{t}^{N,K,11}+\cD_{t}^{N,K,12}+\cD_{t}^{N,K,13}+\cD_{t}^{N,K,14}+\cD_{t}^{N,K,2}+\cD_{t}^{N,K,3}.
$$
Gathering Corollaries \ref{D1234} and \ref{D12h} ends the proof.
\end{proof}

\begin{proof}[{Proof of Theorem \ref{mainsupsup}}]
We work in the regime $(\ref{regime}).$
By Corollary \ref{PvtZt}, we know that, in probability, 
$\lim\boldsymbol{1}_{\Omega_{N}^{K,2}\cap\{\bar{Z}_{t}^{N,K}\ge \frac{1}{4}v_{t}^{N,K}>0\}}=1$.
Also, we know from Proposition \ref{mainsup} that
$$
\lim\boldsymbol{1}_{\Omega_{N}^{K,2}\cap\{\bar{Z}_{t}^{N,K}\ge \frac{1}{4}v_{t}^{N,K}>0\}}\mathcal{U}_{t}^{N,K}=\frac{1}{p}-1.
$$
in probability. Since $p<1$, we deduce that 
$\lim\boldsymbol{1}_{\{\mathcal{U}_{t}^{N,K}\geq 0\}}\boldsymbol{1}_{\Omega_{N}^{K,2}\cap\{\bar{Z}_{t}^{N,K}\ge \frac{1}{4}v_{t}^{N,K}>0\}}=1$ in probability.  In view of $(\ref{PU}),$ we have 
\begin{align*}
    \Big(\mathcal{P}_{t}^{N,K}-p\Big)\boldsymbol{1}_{\{\mathcal{U}_{t}^{N,K}\ge 0\}}=\frac{-p\Big[\mathcal{U}_{t}^{N,K}-\Big(\frac{1}{p}-1\Big)\Big]}{\mathcal{U}_{t}^{N,K}+1}\boldsymbol{1}_{\{\mathcal{U}_{t}^{N,K}\ge 0\}}.
\end{align*}
Hence 
\begin{align*}
&\lim\frac{e^{\alpha_0 t}\sqrt{K}}{N}\Big(\mathcal{P}_{t}^{N,K}-p\Big)\\
=&\lim\boldsymbol{1}_{\Omega_{N}^{K,2}\cap\{\bar{Z}_{t}^{N,K}\ge \frac{1}{4}v_{t}^{N,K}>0\}\cap \{\mathcal{U}_{t}^{N,K}\geq 0\}}\frac{-pe^{(\alpha_0-\alpha_N) t}e^{\alpha_N t}\sqrt{K}}{N}\frac{\mathcal{U}_{t}^{N,K}-\Big(\frac{1}{p}-1\Big)}{\mathcal{U}_{t}^{N,K}+1}.
\end{align*}
By Remark \ref{rhotop}, we have $|\alpha_N-\alpha_0|=|\rho_N-p|\le \frac{C}{N^\frac{3}{8}}.$ 
Since $\lim \frac{Ct}{N^\frac{3}{8}}=0$ in the regime $(\ref{regime})$,
we conclude that $\lim e^{(\alpha_0-\alpha_N) t}=1.$
Finally, by Proposition \ref{mainsup}, we deduce that
 \begin{align*}
\lim\frac{e^{\alpha_0 t}\sqrt{K}}{N}\Big(\mathcal{P}_{t}^{N,K}-p\Big)=\lim\frac{-pe^{\alpha_N t}\sqrt{K}}{N}\frac{\mathcal{U}_{t}^{N,K}-\Big(\frac{1}{p}-1\Big)}{\mathcal{U}_{t}^{N,K}+1}
\stackrel{d}{\longrightarrow}\mathcal{N}\Big(0,\frac{2(\alpha_0)^4p^2}{\mu ^2}\Big).
\end{align*}
This ends the proof.
\end{proof}

\section{Appendix}

We first write down two lemmas concerning the convolution of the function $\phi$ that will
be useful in the subcritical case.

\begin{lemma}\label{phisq}
We consider 
$\phi:[0,\infty)\to [0,\infty)$ such that $\Lambda=\int_0^\infty \phi(s)ds <\infty$ and, for some $q\ge 1$, 
$\int_0^\infty s^q\phi(s)ds <\infty.$ 
Then, for all $n\ge 1$ and $r\ge 1,$
$$
\int_{r}^{\infty}\sqrt{s}\phi^{*n}(s)ds\le  C \Lambda^{n}n^qr^{\frac{1}{2}-q} 
\quad \hbox {and} \quad
\int^{\infty}_{0}\sqrt{s}\phi^{*n}(s)ds\le \sqrt{n}\Lambda^{n}.
$$
\end{lemma}

\begin{proof}
We introduce some i.i.d. random variables $X_1,X_2,\dots$ with density $\Lambda^{-1}\phi$ and
set $S_0=0$ as well as $S_n=X_1+\dots+X_n$ for all $n\geq 1$.
We observe that
\begin{align*}
\int_{r}^{\infty}\sqrt{s}\phi^{*n}(s)ds=
\Lambda^{n}\E[\sqrt{S_n}\indiq_{S_n\ge r}] \leq \Lambda^{n}r^{\frac{1}{2}-q}\E[S_n^q\indiq_{S_n\ge r}]\leq 
\Lambda^{n}n^qr^{\frac{1}{2}-q} \E[X_1^q]\leq C \Lambda^{n}n^qr^{\frac{1}{2}-q}.
\end{align*}
We used the Minkowski inequality and that $\E[X_1^q]=\Lambda^{-1}\int_0^\infty s^q \phi(s)ds<\infty$ by assumption.
For the second part, we write
$$
\int^{\infty}_{0}\sqrt{s}\phi^{*n}(s)ds=\Lambda^{n}\E[\sqrt{S_n}] \le \sqrt{n}\Lambda^{n}\sqrt{\E[{X_1}]}\le \sqrt{n}\Lambda^{n}
$$
by the Cauchy-Schwarz inequality.
\end{proof}
\begin{lemma}\label{phi}
Under the same conditions as in Lemma \ref{phisq}, we have,
for $n\in \mathbb{N}_+$ and $r\ge 1,$
$$\Big|\int_{0}^{t}\phi^{*n}(s)ds-\Lambda^{n}\Big|\le n\Lambda^{n-1} \int^\infty_{\frac{t}{n}}\phi(s)ds.$$
\end{lemma}

\begin{proof}
Then consider $n$ i.i.d random variables $\{X_i\}_{i=1,...,n}$ with density $\phi(s)/\Lambda$ and write
\begin{align*}
\Big|\int_{0}^{t}\phi^{*n}(s)ds-\Lambda^{n}\Big|=\Lambda^n P \Big(\sum_{i=1}^n X_i\ge t\Big)
\leq \Lambda^n P \Big(\max_{i=1,\dots,n} X_i\ge t/n\Big) \leq 
n \Lambda^{n} P (X_1 \geq t/n),
\end{align*}
which complete the proof.
\end{proof}

We next adopt the notation of the supercritical case and study the martingales $M^{i,N}_t$.

\begin{lemma}\label{MMMMMM}
Assume $(A)$. For any $ s\in [0,t],\ i\neq k,$ on $\Omega_N^{K,2},$
\vip
(i) $\mathbb{E}_\theta[M_{t}^{i,N}M_{s}^{j,N}M_{t}^{i',N}M_{s'}^{j',N}]=0 $ if $\#\{i,j,i',j'\}=4.$
\vip
(ii) $|\mathbb{E}_\theta[M_{t}^{i,N}M_{s}^{i,N}M_{t}^{k,N}]|\le \frac{Cse^{\alpha_N s}}{N} $ if $i\neq k$.
\vip
(iii) $|\mathbb{E}_\theta[M_{t}^{i,N}M_{s}^{j,N}M_{t}^{i',N}M_{s}^{j',N}]|\le \frac{C(e^{\alpha_N t}t^2)}{N^2}$ if $\#\{i,j,i',j'\}=3.$
\vip
(iv) $|\mathbb{E}_\theta[M_{t}^{i,N}M_{s}^{j,N}M_{t}^{i',N}M_{s}^{j',N}]|\le Ce^{\alpha_N (t+s)}$ without any conditions.
\end{lemma}

\begin{proof}
First, we recall that under $(A)$, we have $\phi^{*n}(t)=t^{n-1}e^{-bt}/(n-1)!$ for all $n\ge 1.$

\vip

Point (i) follows from the fact that, since when $i,j,i',j'$ are pairwise different, the martingales
$M^{i,N}$, $M^{i',N}$, $M^{j,N}$ and $M^{j',N}$ are orthogonal by \eqref{M7}.

\vip

For point (ii), we first use that $M^{i,N}$ and $M^{k,N}$ are orthogonal and that $t\geq s$ to write
$$
\mathbb{E}_\theta[M_{t}^{i,N}M_{s}^{i,N}M_{t}^{k,N}]=\Et[M^{i,N}_s\Et[M^{i,N}_tM^{k,N}_t |\cF_s]]=\Et[(M^{i,N}_s)^2M^{k,N}_{ s}].
$$
Since $[M^{i,N},M^{i,N}]_s=Z^{i,N}_s$, it holds that
$(M^{i,N}_s)^2=2\int_0^s M^{i,N}_{r-}dM^{i,N}_r + Z^{i,N}_s$. 
Using that
$\int_0^{\cdot} M^{i,N}_{r-}dM^{i,N}_r$ and $M^{k,N}$ are orthogonal, we deduce that
$\Et[(M^{i,N}_s)^2M^{k,N}_{s}]|=\Et[Z^{i,N}_s M^{k,N}_{s}]$ whence, by $(\ref{defJ}),$
\begin{align*}
|\Et[(M^{i,N}_s)^2M^{k,N}_{s}]|=&|\Et[U^{i,N}_s M^{k,N}_{s}]|\\
\le&\sum_{n\ge 0}\sum_{j=1}^N\int_0^{s}\phi^{*n}(s-l)A^n_N(i,j)|\mathbb{E}_\theta[M_{l}^{j,N}M_{s}^{k,N}]|dl\\
=&\sum_{n\ge 0}\int_0^{s}\phi^{*n}(s-l)A^n_N(i,k)\mathbb{E}_\theta[Z_{l}^{k,N}]dl
\end{align*}
by \eqref{M7}. By \cite[Lemma 35-(iv)]{A}, we 
have $A^n_N(i,k)\le \frac{C\rho_N^n}{N}$ for all $n\geq 1$ and $A^n_N(i,k)=0$ for $n=0$ (since $i\ne k$).
We also know that $\mathbb{E}_\theta[Z_{l}^{k,N}]\leq C e^{\alpha_N l}$ (on $\Omega_N^{K,2}$), so that
\begin{align*}
|\Et[(M^{i,N}_s)^2M^{k,N}_{s}]|\le& \frac{C}{N}\sum_{n\ge 1}\int_0^{s}e^{-b(s-l)}\frac{(\rho_N)^n(s-l)^{n-1}}{(n-1)!}
e^{\alpha_N l}dl\\
=& \frac{C \rho_N}N \int_0^s e^{(\rho_N-b)(s-l)}e^{\alpha_N l} dl= \frac{C \rho_N se^{(\rho_N-b) s}}{N}\leq \frac{C 
se^{(\rho_N-b) s}}{N},
\end{align*}
since $\alpha_N=\rho_N-b$ and $\rho_N\leq 2p$ on $\Omega_N^{K,2}$ by Remark \ref{rhotop}.

\vip

For point (iii), we first consider the case  $j=j'$ (and $i,i',j$ are pairwise different).
We have $\mathbb{E}_\theta[M_{t}^{i,N}M_{s}^{j,N}M_{t}^{i',N}M_{s}^{j,N}]=\Et[M_{s}^{i,N}M_{s}^{i',N} (M_{s}^{j,N})^2]$
because $t\geq s$ and $M^{i,N}$ and $M^{i',N}$ are orthogonal. Using the It\^o formula as
in (ii), we find $(M^{j,N}_s)^2=2\int_0^s M^{j,N}_{r-}dM^{j,N}_r + Z^{j,N}_s$,
with $\int_0^{\cdot} M^{i,N}_{r-}dM^{i,N}_r$ orthogonal to $M^{i,N}M^{i',N}$. As a consequence,
$\mathbb{E}_\theta[M_{t}^{i,N}M_{s}^{j,N}M_{t}^{i',N}M_{s}^{j,N}]=\Et[M_{s}^{i,N}M_{s}^{i',N} Z^{j,N}_s]
=\Et[M_{s}^{i,N}M_{s}^{i',N} U^{j,N}_s]$, recall \eqref{defJ}, so that
\begin{align*}
&|\mathbb{E}_\theta[M_{t}^{i,N}M_{s}^{j,N}M_{t}^{i',N}M_{s}^{j,N}]|\\
\le&  \sum_{n\ge 0}\sum_{q=1}^N\int_0^{s}\phi^{*n}(s-l)A^n_N(j,q)|\mathbb{E}_\theta[M_{s}^{i,N}M_{s}^{i',N}M_{l}^{q,N}]|dl\\
=&  \sum_{n\ge 0}\int_0^{s}\phi^{*n}(s-l)\Big\{A^n_N(j,i)|\mathbb{E}_\theta[M_{s}^{i,N}M_{s}^{i',N}M_{l}^{i,N}]|
+A^n_N(j,i')|\mathbb{E}_\theta[M_{s}^{i,N}M_{s}^{i',N}M_{l}^{i',N}]|\Big\}dl\\
\le& \frac{C}{N^2}\sum_{n\ge 1}\int_0^{s}\frac{\rho^n_N(s-l)^{n-1}}{(n-1)!}e^{-b(s-l)}le^{\alpha_N l}dl
\end{align*}
by point (ii) and since, as in the proof of (ii), 
$A^n_N(j,i)\le \frac{C\rho_N^n}{N}$ for all $n\geq 1$ and $A^n_N(j,i)=0$ for $n=0$ 
(since $i\ne j$), and similar considerations for $i'$. Using as usual that,
on $\Omega_N^{K,2}$, we have $\rho_N \leq 2p$ and $\alpha_N=\rho_N-b$, we easily conclude that
\begin{align*}
|\mathbb{E}_\theta[M_{t}^{i,N}M_{s}^{j,N}M_{t}^{i',N}M_{s}^{j,N}]|
=\frac{C \rho_N}{N^2}\int_0^{s}le^{\alpha_N s}dl
\le \frac{Cs^2e^{\alpha_N s}}{N^2}.
\end{align*}
For the case where $i=i'$ (and where $i,j,j'$ are pairwise different), we have,
proceeding as previously, $\mathbb{E}_\theta[M_{t}^{i,N}M_{s}^{j,N}M_{t}^{i,N}M_{s}^{j',N}]=
\mathbb{E}_\theta[Z^{i,N}_t M_{s}^{j,N}M_{s}^{j',N}]=\mathbb{E}_\theta[U^{i,N}_t M_{s}^{j,N}M_{s}^{j',N}]$.
Hence, using \eqref{defJ},
\begin{align*}
    &|\mathbb{E}_\theta[M_{t}^{i,N}M_{s}^{j,N}M_{t}^{i,N}M_{s}^{j',N}]|\\
    \le& \sum_{n\ge 0}\sum_{q=1}^N\int_0^{t}\phi^{*n}(t-l)A^n_N(i,q)|\mathbb{E}_\theta[M_{l}^{q,N}M_{s}^{j,N}M_{s}^{j',N}]|dl\\
    \leq& \sum_{n\ge 0}\int_0^{t}\phi^{*n}(t-l)\Big\{A^n_N(i,j)|\mathbb{E}_\theta[M_{l}^{j,N}M_{s}^{j,N}M_{s}^{j',N}]|+A^n_N(i,j')|\mathbb{E}_\theta[M_{l}^{j',N}M_{s}^{j,N}M_{s}^{j',N}]|\Big\}dl.
\end{align*}
If $l\geq s$, we see that $\mathbb{E}_\theta[M_{l}^{j,N}M_{s}^{j,N}M_{s}^{j',N}]=
\mathbb{E}_\theta[M_{s}^{j,N}M_{s}^{j,N}M_{s}^{j',N}]$. So in any case, we can apply point (ii) 
and we find
\begin{align*}
|\mathbb{E}_\theta[M_{t}^{i,N}M_{s}^{j,N}M_{t}^{i,N}M_{s}^{j',N}]|
\leq \frac C N \sum_{n\ge 0}\int_0^{t}\phi^{*n}(t-l)[A^n_N(i,j)+A^n_N(i,j')]  (l\land s)e^{\alpha_N (l\land s)} dl.
\end{align*}
Using the same argument as in the previous case, we conclude that
$$
|\mathbb{E}_\theta[M_{t}^{i,N}M_{s}^{j,N}M_{t}^{i,N}M_{s}^{j',N}]|
\leq \frac{C}{N^2}\sum_{n\ge 1}\int_0^{t}\frac{l(\rho_N)^n(t-l)^{n-1}}{(n-1)!}e^{-b(t-l)}e^{\alpha_N l}dl\leq 
\frac{C t^2 e^{\alpha_N t}}{N^2}
$$
as usual. Finally, if $i=j$ (and $i,i',j'$ are pairwise different), we 
first write, using that the three involved martingales are orthogonal, 
$\mathbb{E}_\theta[M_{t}^{i,N}M_{s}^{i,N}M_{t}^{i',N}M_{s}^{j',N}]=\mathbb{E}_\theta[M_{s}^{i,N}M_{s}^{i,N}M_{s}^{i',N}M_{s}^{j',N}]$,
so that, arguing as in the previous cases, 
$\mathbb{E}_\theta[M_{t}^{i,N}M_{s}^{i,N}M_{t}^{i',N}M_{s}^{j',N}]=\mathbb{E}_\theta[U_{s}^{i,N}M_{s}^{i',N}M_{s}^{j',N}]$.
It then suffices to copy the previous case (when $t=s$ and replacing $j$ by $i'$) to find
\begin{align*}
|\mathbb{E}_\theta[M_{t}^{i,N}M_{s}^{i,N}M_{t}^{i',N}M_{s}^{j',N}]| \frac{C s^2 e^{\alpha_N s}}{N^2}
\end{align*}
This completes the proof of (iii).

\vip

By Doob's inequality, \eqref{M7} and Lemma \ref{supEZ2}-(i), we have 
$\mathbb{E}_\theta[(M_{t}^{i,N})^4]\le C\mathbb{E}_\theta[(Z_{t}^{i,N})^2]\le Ce^{2\alpha_N t}$, 
whence
\begin{align*}
|\mathbb{E}_\theta[M_{t}^{i,N}M_{s}^{j,N}M_{t}^{i',N}M_{s}^{j',N}]|\le \mathbb{E}_\theta[(M_{t}^{i,N})^4]^\frac{1}{4}\mathbb{E}_\theta[(M_{s}^{j,N})^4]^\frac{1}{4}\mathbb{E}_\theta[(M_{t}^{i',N})^4]^\frac{1}{4}\mathbb{E}_\theta[(M_{s}^{j',N})^4]^\frac{1}{4}
\end{align*}
is bounded by $Ce^{\alpha_N (t+s)}$ as desired.
\end{proof}

%\begin{lemma}\label{ito}
%\begin{align*}
%\frac{1}{\Delta t K^{2}}\mathbb{E}_{\theta}[(\int_{0}^{2t}\sum_{j=1}^{N}c_{N}^{K}(j)M^{j,N}_{s}d(\sum_{j=1}^{N}c_{N}^{K}(j)(M^{j,N}_{s}))^{2}]
%\end{align*}
%\end{lemma}

%\begin{proof}
%By the final of lemma 5.3 in [2], we have:
%\begin{align*}
    %&\mathbb{E}_{\theta}[(\int_{0}^{2t}\sum_{j=1}^{N}c_{N}^{K}(j)(M^{j,N}_{s})d(\sum_{j=1}^{N}c_{N}^{K}(j)(M^{j,N}_{s}))^{2}]\\
    %&=\mathbb{E}_{\theta}[\int_{0}^{2t}\sum_{j=1}^{N}(c_{N}^{K}(j)M^{j,N}_{s})^{2}d(\sum_{j=1}^{N}(c_{N}^{K}(j))^{2}Z_{s}(j)]\\
    %&\le \mathbb{E}_{\theta}[\int_{0}^{2t}\{\sum_{i=1}^{N}(c_{N}^{K}(i))^{2}\}^{2}sds]
    %\end{align*}
%\end{proof}

%\begin{lemma}\label{ito}
%\begin{align*}
 %   \frac{1}{K\sqrt{t\Delta}}\int_{0}^{2t}\sum_{j=1}^{N}c_{N}^{K}(j)M^{j,N}_{s}d(\sum_{j=1}^{N}c_{N}^{K}(j)M^{j,N}_{s})\stackrel{d}{\longrightarrow}0
%\end{align*}

%\end{lemma}

%\begin{proof}

%\begin{align*}
 %   \int_{0}^{2t}\sum_{j=1}^{N}c_{N}^{K}(j)M^{j,N}_{s}d(\sum_{j=1}^{N}c_{N}^{K}(j)M^{j,N}_{s})=(\sum_{j=1}^{N}c_{N}^{K}(j)M^{j,N}_{2t})^{2}\\
  %  -\frac{1}{2}\sum_{j=1}^{N}(c_{N}^{K}(j))^{2}Z^{j,N}_{2t}
%\end{align*}

%\end{proof}

\bibliographystyle{amsplain}

\end{document}